%% file: final.tex
\documentclass[10pt]{article}

\makeatletter
\newcommand\RedeclareMathOperator{%
  \@ifstar{\def\rmo@s{m}\rmo@redeclare}{\def\rmo@s{o}\rmo@redeclare}%
}
\newcommand\rmo@redeclare[2]{%
  \begingroup \escapechar\m@ne\xdef\@gtempa{{\string#1}}\endgroup
  \expandafter\@ifundefined\@gtempa
     {\@latex@error{\noexpand#1undefined}\@ehc}%
     \relax
  \expandafter\rmo@declmathop\rmo@s{#1}{#2}}
\newcommand\rmo@declmathop[3]{%
  \DeclareRobustCommand{#2}{\qopname\newmcodes@#1{#3}}%
}
\@onlypreamble\RedeclareMathOperator
\makeatother

\usepackage{witharrows}
\usepackage{tikz}
\usepackage{graphicx,wrapfig,lipsum}   

\usepackage{xspace}

\usepackage{framed}

\usepackage{amsmath,amssymb}
\usepackage{amsthm}
\usepackage{mathtools}
\usepackage[noend]{algorithmic}
\usepackage[ruled,vlined]{algorithm2e}
\usepackage{url}
\usepackage{fullpage}
\usepackage{makeidx}
\usepackage{enumerate}
\usepackage[top=1in, bottom=1.1in, left=0.8in, right=0.8in]{geometry}
\usepackage{graphicx,float,psfrag,epsfig,caption}

\usepackage{hyperref}
\usepackage{epstopdf}
\usepackage{color}
\usepackage{xr}
\usepackage{subcaption}
\usepackage[utf8]{inputenc}

\usepackage{scalerel,stackengine}

\usepackage{thm-restate}

\usepackage{enumitem}
\usepackage{bbm}

\stackMath
\newcommand\reallywidehat[1]{%
\savestack{\tmpbox}{\stretchto{%
  \scaleto{%
    \scalerel*[\widthof{\ensuremath{#1}}]{\kern.1pt\mathchar"0362\kern.1pt}%
    {\rule{0ex}{\textheight}}
  }{\textheight}%
}{2.4ex}}%
\stackon[-6.9pt]{#1}{\tmpbox}%
}
\usepackage[mathscr]{euscript} 

\DeclareSymbolFont{rsfs}{U}{rsfs}{m}{n}
\DeclareSymbolFontAlphabet{\mathscrsfs}{rsfs}
\usepackage{mathrsfs}

\numberwithin{equation}{section}

\renewcommand{\paragraph}[1]{\noindent\textbf{#1}.\quad}

\newtheoremstyle{myexample} 
    {\topsep}                    
    {\topsep}                    
    {\rm }                   
    {}                           
    {\bf }                   
    {.}                          
    {.5em}                       
    {}  

\newtheoremstyle{myremark} 
    {\topsep}                    
    {\topsep}                    
    {\rm}                        
    {}                           
    {\bf}                        
    {.}                          
    {.5em}                       
    {}  

\newtheorem{claim}{Claim}[section]
\newtheorem{lemma}[claim]{Lemma}
\newtheorem{fact}[claim]{Fact}
\newtheorem{assumption}{Assumption}[]

\newtheorem{theorem}{Theorem}
\newtheorem{proposition}[claim]{Proposition}
\newtheorem{corollary}[claim]{Corollary}

\theoremstyle{definition}
\newtheorem{definition}[claim]{Definition}

\theoremstyle{myremark}
\newtheorem{remark}{Remark}[section]

\theoremstyle{myremark}

\theoremstyle{myexample}

\definecolor{darkgreen}{rgb}{0.0, 0.5, 0.0}

\hypersetup{
  colorlinks   = true, 
  urlcolor     = blue, 
  linkcolor    = blue, 
  citecolor   = red 
}

\input{commands.tex}
\title{Algorithmic Threshold for Multi-Species Spherical Spin Glasses}
\author{
    Brice Huang\thanks{Department of Electrical Engineering and Computer Science, Massachusetts Institute of Technology. Email: \texttt{bmhuang@mit.edu}.} 
    \and 
    Mark Sellke\thanks{Department of Statistics, Harvard University.
    Email: \texttt{msellke@fas.harvard.edu}.}
}
\date{}

\begin{document}

\maketitle

\input{final-tex/0-abstract}

\setcounter{tocdepth}{1}
\tableofcontents

\input{final-tex/1-intro}

\input{final-tex/2-bogp}
\input{final-tex/3-uc}
\input{final-tex/4-alg}

\input{final-tex/7-ack}

\small
\bibliographystyle{alpha}
\bibliography{all-bib}
\normalsize

\appendix
\input{final-tex/a0-equivalence-of-bogps}
\input{final-tex/a1-sk-to-alg}
\input{final-tex/a2-details-of-alg}

\end{document}

%% file: commands.tex
\newcommand{\bea}{\begin{eqnarray}}
\newcommand{\eea}{\end{eqnarray}}
\newcommand{\<}{\langle}
\renewcommand{\>}{\rangle}

\newcommand{\wt}{\widetilde}
\newcommand{\op}{\text{op}}

\def\I{{\rm I}}
\def\II{{\rm II}}
\def\III{{\rm III}}

\def\eps{{\varepsilon}}

\def\bh{\boldsymbol{h}}

\def\supp{{\rm supp}}

\def\ind{{\mathbbm 1}}

\def\btau{{\boldsymbol{\tau}}}

\def\bZ{{\boldsymbol{Z}}}

\def\bM{{\boldsymbol{M}}}

\def\bg{{\boldsymbol{g}}}
\def\hbg{{\hat \bg}}
\def\bx{{\boldsymbol{x}}}

\def\bzero{{\mathbf 0}}
\def\bone{{\mathbf 1}}

\def\cE{{\mathcal E}}

\def\sS{{\mathscr S}}

\def\op{\mbox{\tiny\rm op}}

\def\vv{\vec v}
\def\vw{\vec w}
\def\vx{\vec x}

\def\vA{\vec A}
\def\vB{\vec B}
\def\vC{\vec C}
\def\vD{\vec D}

\def\wtH{\wt{H}}

\def\wtM{\wt{M}}

\def\wtp{\wt{p}}
\def\wtq{\wt{q}}
\def\wtq{\wt{q}}
\def\vone{\vec 1}
\def\vzero{\vec 0}

\def\bsig{{\boldsymbol {\sigma}}}
\def\vbsig{{\vec \bsig}}
\def\ubx{{\underline \bx}}
\def\uby{{\underline \by}}
\def\ubsig{{\underline \bsig}}
\def\ubrho{{\underline \brho}}
\def\ubtau{{\underline \btau}}

\def\vh{{\vec h}}
\def\hv{{\hat v}}
\def\vL{{\vec L}}
\def\vlam{{\vec \lambda}}
\def\va{{\vec a}}
\def\vb{{\vec b}}
\def\ve{{\vec e}}

\def\ubsig{\underline{\bsig}}

\def\brho{{\boldsymbol \rho}}

\def\vbrho{{\vec \brho}}

\def\by{{\boldsymbol y}}
\def\vy{{\vec y}}

\def\vbx{{\vec \bx}}

\def\bz{{\boldsymbol{z}}}
\def\bx{{\boldsymbol{x}}}

\def\bA{\boldsymbol{A}}

\def\va{\vec{a}}
\def\vb{\vec{b}}
\def\vc{\vec{c}}

\def\de{{\rm d}}

\def\bY{\boldsymbol{Y}}

\def\<{\langle}
\def\>{\rangle}

\def\range{{\rm range}}

\def\rank{{\rm rank}}
\def\diag{{\rm diag}}

\def\cH{{\cal H}}
\def\cM{{\cal M}}
\def\cN{{\cal N}}

\def\by{{\boldsymbol{y}}}
\def\tby{{\widetilde \by}}

\def\P{\mathbb{P}}
\def\cE{{\mathcal E}}

\def\bD{{\boldsymbol{D}}}

\def\cD{{\cal D}}

\def\b0{{\boldsymbol{0}}}

\def\hH{\widehat{H}}

\def\Var{{\rm Var}}

\def\bG{{\boldsymbol G}}

\def\bR{{\boldsymbol R}}

\def\OPT{{\sf OPT}}

\def\uq{\underline{q}}

\def\cD{{{\mathcal D}}}

\def\cA{{\mathcal A}}
\def\cI{{\mathcal I}}
\def\cS{{\mathcal S}}

\def\cB{{\mathcal B}}

\def\bzero{\boldsymbol{0}}

\renewcommand{\b}{\mathbf{b}}

\def\fr{\frac}
\def\lt{\left}
\def\rt{\right}

\def\la{\langle}
\def\ra{\rangle}

\def\eps{\varepsilon}

\def\bbA{{\mathbb{A}}}

\def\bbC{{\mathbb{C}}}
\def\bbE{{\mathbb{E}}}
\def\bbH{{\mathbb{H}}}
\def\bbI{{\mathbb{I}}}
\def\hbbI{{\hat{\mathbb{I}}}}
\def\bbL{{\mathbb{L}}}
\def\bbN{{\mathbb{N}}}
\def\bbP{{\mathbb{P}}}
\def\bbR{{\mathbb{R}}}

\def\bbT{{\mathbb{T}}}

\def\bbZ{{\mathbb{Z}}}

\def\tbbI{{\widetilde \bbI}}

\def\cA{{\mathcal{A}}}
\def\cB{{\mathcal{B}}}

\def\cN{{\mathcal{N}}}

\def\cQ{{\mathcal{Q}}}
\def\cR{{\mathcal{R}}}

\def\sH{{\mathscr{H}}}

\def\bg{{\mathbf{g}}}
\def\bh{{\boldsymbol{h}}}

\def\bM{\mathbf{M}}

\def\bR{\mathbf{R}}

\def\ALG{{\mathsf{ALG}}}
\def\hALG{{\widehat{\mathsf{ALG}}}}

\def\OPT{{\mathsf{OPT}}}

\def\BOGP{{\mathsf{BOGP}}}

\def\GS{{\mathrm{GS}}}
\def\bGS{{\boldsymbol{\mathrm{GS}}}}

\def\Ssolve{S_{\mathrm{solve}}}

\def\Soverlap{S_{\mathrm{overlap}}}

\def\Sogp{S_{\mathrm{ogp}}}

\DeclareMathOperator*{\E}{\bbE}
\RedeclareMathOperator*{\P}{\bbP}
\DeclareMathOperator*{\argmax}{\arg\max}
\DeclareMathOperator*{\argmin}{\arg\min}

\newcommand{\Gp}[1]{\mathbf{G}^{(#1)}}

\newcommand{\wtHNp}[1]{{\wt{H}_N^{(#1)}}}

\newcommand{\HNp}[1]{{H_N^{(#1)}}}

\newcommand{\norm}[1]{{\lt\|#1\rt\|}}
\newcommand{\tnorm}[1]{{\|#1\|}}

\newcommand{\Span}{\mathrm{span}}

\newcommand{\loc}{\mathrm{loc}}
\newcommand{\den}{\mathrm{den}}

\newcommand{\tp}{{\widetilde{p}}}

\newcommand{\tPhi}{{\widetilde{\Phi}}}

\newcommand{\free}{\mathrm{free}}

\newcommand{\diff}[1]{{\mathrm{d}#1}}
\newcommand{\deriv}[1]{{\fr{\diff{}}{\diff{#1}}}}

\def\vchi{\vec{\chi}}
\def\vphi{\vec{\phi}}
\def\up{\underline{p}}

\def\uvphi{\underline{\vphi}}

\def\pplus{p_+}
\def\pminus{p_-}
\def\vphiplus{\vphi_+}
\def\vphiminus{\vphi_-}

\def\vR{\vec R}

\def\oH{\overline H}
\def\Cov{\mathrm{Cov}}

\def\sym{{\rm sym}}

\def\sph{{\mathrm{sp}}}

\def\hbz{\hat{\boldsymbol z}}

\def\Adm{\mathrm{Adm}}
\def\hAdm{\widehat{\mathrm{Adm}}}
\def\hM{\widehat M}
\def\oM{\overline M}

\newcommand{\tay}{{\mathrm{tay}}}

%% file: final-tex/0-abstract.tex
\begin{abstract}
    We study efficient optimization of the Hamiltonians of multi-species spherical spin glasses.
    Our results characterize the maximum value attained by algorithms that are suitably Lipschitz with respect to the disorder through a variational principle that we study in detail. 
    We rely on the branching overlap gap property introduced in our previous work and develop a new method to establish it that does not require the interpolation method. Consequently our results apply even for models with non-convex covariance, where the Parisi formula for the true ground state remains open.
    As a special case, we obtain the algorithmic threshold for all single-species spherical spin glasses, which was previously known only for even models. 
    We also obtain closed-form formulas for pure models 
    which coincide with the $E_{\infty}$ value previously determined by the Kac-Rice formula.
    %
\end{abstract}

%% file: final-tex/1-intro.tex
\section{Introduction}

This paper studies the efficient optimization of a family of random functions $H_N$ which are high-dimensional and extremely non-convex. 
The computational complexity of such random optimization problems remains poorly understood in the majority of cases as most impossibility results concern worst-case rather than average-case behavior.

We focus on a general class of such problems: the Hamiltonians of multi-species spherical spin glasses. 
Mean-field spin glasses have been studied since \cite{sherrington1975solvable} as models for disordered magnetic systems and are also closely linked to random combinatorial optimization problems \cite{krzakala2007gibbs, dembo2017extremal, panchenko2018k}. Simply put, their Hamiltonians are certain polynomials in many variables with independent centered Gaussian coefficients.

Multi-species spin glasses such as the bipartite SK model \cite{kincaid1975phase,korenblit1985spin,fyodorov1987antiferromagnetic,fyodorov1987phase} open the door to yet richer behavior and as discussed below remain poorly understood from a rigorous viewpoint.
Our main result gives, for all multi-species spherical spin glasses, an exact \emph{algorithmic threshold} $\ALG$ for the maximum Hamiltonian value obtained by a natural class of \emph{stable} optimization algorithms.

For the more well-known single-species spin glasses, the celebrated Parisi formula \cite{parisi1979infinite,talagrand2006parisi,talagrand2006spherical,auffinger2017parisi} gives the limiting maximum value of $H_N$ as a certain variational formula.
In previous work \cite{huang2021tight} we obtained the algorithmic threshold for these models restricted to have only even degree interactions, given by the same variational formula over an extended state space.
The central idea was to show $H_N$ obeys a branching version of the \emph{overlap gap property} (OGP): the absence of a certain geometric configuration of high-energy inputs \cite{gamarnik2014limits,gamarnik2021survey}. 
The proofs of the Parisi formula \cite{talagrand2006parisi,talagrand2006spherical,auffinger2017parisi}, the branching OGP in \cite{huang2021tight}, and other results (e.g. \cite{guerra2002thermodynamic,bayati2010combinatorial}) require the so-called interpolation method, which is known to fail when the model's covariance is not convex.
Due to this limitation of the interpolation method, the proof of our previous result does not generalize to single-species spin glasses with odd interactions, nor to multi-species spin glasses. 
For the same reason, the Parisi formula for the ground state of a multi-species spin glass 
is known only in restricted cases \cite{barra2015multi, panchenko2015free, baik2020free, subag2021tap, bates2022free}. 

We develop a new method to establish the branching OGP which does not use the interpolation method.
Instead, we recursively apply a uniform concentration idea introduced in \cite{subag2018free}.
Consequently we are able to determine $\ALG$ for all multi-species spherical spin glasses, including those whose ground state energy is not known. 
As a special case, this removes the even interactions condition from \cite{huang2021tight} for spherical models and is the first OGP that applies to mean-field spin glasses with odd interactions.

Our results strengthen a geometric picture put forth in \cite[Section 1.4]{huang2021tight} that in mean-field random optimization problems, the tractability of optimization to value $E$ coincides with the presence of densely branching ultrametric trees within the super-level set at value $E$.
On the hardness side, such trees are precisely what the branching OGP forbids. 
On the algorithmic side, it will be clear from our methods (see the end of Subsection~\ref{subsec:bogp-summary}) that efficient algorithms can be designed to descend such trees whenever they exist, thereby achieving value $E$.

Our algorithmic threshold for multi-species models is expressed as the maximum of a somewhat different variational principle.
We analyze our algorithmic variational principle in detail, showing that maximizers are formed by joining the solutions to a pair of differential equations, and are explicit and unique for single-species and pure models.
To our surprise the maximizers are not unique in general, a behavior we term \emph{algorithmic symmetry breaking}.

\subsection{Problem Description and the Value of $\ALG$}

Fix a finite set $\sS = \{1,\ldots,r\}$. 
For each positive integer $N$, fix a deterministic partition $\{1,\ldots,N\} = \sqcup_{s\in\sS}\, \cI_s$ with $\lim_{N\to\infty} |\cI_s| / N =\lambda_s$ where $\vlam = (\lambda_1,\ldots,\lambda_r) \in \bbR_{>0}^\sS$ sum to $1$.
For $s\in \sS$ and $\bx \in \bbR^N$, let $\bx_s \in \bbR^{\cI_s}$ denote the restriction of $\bx$ to coordinates $\cI_s$.
We consider the state space 
\[
    \cB_N = \lt\{
        \bx \in \bbR^N : 
        \norm{\bx_s}_2^2 \le \lambda_s N
        ~\forall s\in \sS
    \rt\}.
\]
Fix $\vh = (h_1,\ldots,h_r) \in \bbR_{\ge 0}^\sS$ and let $\bone = (1,\ldots,1) \in \bbR^N$.
For each $k\ge 2$ fix a symmetric tensor $\Gamma^{(k)} = (\gamma_{s_1,\ldots,s_k})_{s_1,\ldots,s_k\in \sS} \in (\bbR_{\ge 0}^{\sS})^{\otimes k}$ with $\sum_{k\ge 2} 2^k \norm{\Gamma^{(k)}}_\infty < \infty$, and let $\Gp{k} \in (\bbR^N)^{\otimes k}$ be a tensor with i.i.d. standard Gaussian entries.
For $A\in (\bbR^\sS)^{\otimes k}$, $B\in (\bbR^N)^{\otimes k}$, define $A\diamond B \in (\bbR^N)^{\otimes k}$ to be the tensor with entries
\begin{equation}
    \label{eq:def-diamond}
    (A\diamond B)_{i_1,\ldots,i_k} = A_{s(i_1),\ldots,s(i_k)} B_{i_1,\ldots,i_k},
\end{equation}
where $s(i)$ denotes the $s\in \sS$ such that $i\in \cI_s$.
Let $\bh = \vh \diamond \bone$.
We consider the mean-field multi-species spin glass Hamiltonian
\begin{align}
    \label{eq:def-hamiltonian}
    H_N(\bsig) &= \la \bh, \bsig \ra + \wtH_N(\bsig), \quad \text{where}\\
    \label{eq:def-hamiltonian-no-field}
    \wtH_N(\bsig) &=
	\sum_{k\ge 2}
	\fr{1}{N^{(k-1)/2}}
	\la \Gamma^{(k)} \diamond \bG^{(k)}, \bsig^{\otimes k} \ra \\
	\notag
	&=
	\sum_{k\ge 2}
	\fr{1}{N^{(k-1)/2}}
	\sum_{i_1,\ldots,i_k=1}^N 
	\gamma_{s(i_1),\ldots,s(i_k)} \bG^{(k)}_{i_1,\ldots,i_k} \sigma_{i_1}\cdots \sigma_{i_k}
\end{align}
with inputs $\bsig = (\sigma_1,\ldots,\sigma_N) \in \cB_N$.
For $\bsig,\brho\in \cB_N$, define the species $s$ overlap and overlap vector
\begin{equation}
    \label{eq:R}
    R_s(\bsig, \brho)
    =
    \fr{ \la \bsig_s, \brho_s \ra}{\lambda_s N},
    \qquad
    \vR(\bsig, \brho) 
    = 
    \lt(R_1(\bsig, \brho), \ldots, R_r(\bsig, \brho)\rt).
\end{equation}
Let $\odot$ denote coordinate-wise product. 
For $\vx = (x_1,\ldots,x_r) \in \bbR^\sS$, let 
\begin{align*}
    \xi(\vx) 
    &= \sum_{k\ge 2} \la \Gamma^{(k)}\odot \Gamma^{(k)}, (\vlam \odot \vx)^{\otimes k}\ra \\
    &= \sum_{k\ge 2}
	\sum_{s_1\ldots,s_k\in \sS}
	\gamma_{s_1,\ldots,s_k}^2
	(\lambda_{s_1} x_{s_1})
	\cdots
	(\lambda_{s_k} x_{s_k}).
\end{align*}
The random function $\wtH_N$ can also be described as the Gaussian process on $\cB_N$ with covariance
\[
	\bbE \wtH(\bsig)\wtH(\brho)
	=
	N\xi(\vR(\bsig, \brho)).
\]
It will be useful to define, for $s\in \sS$,
\[
    \xi^s(\vx) 
    = 
    \lambda_s^{-1} 
    \partial_{x_s} 
    \xi(\vx).
\]

Our main result is a characterization of the largest energy attainable by algorithms with $O(1)$-Lipschitz dependence on the disorder coefficients.
To define this class of algorithms, we consider the following distance on the space $\sH_N$ of Hamiltonians $H_N$.
We identify $H_N$ with its disorder coefficients $(\Gp{k})_{k\ge 2}$, which we concatenate in an arbitrary but fixed order into an infinite vector $\bg(H_N)$. 
We equip $\sH_N$ with the (possibly infinite) distance
\[
    \norm{H_N - H'_N}_2
    =
    \norm{\bg(H_N) - \bg(H'_N)}_2
\]
and $\cB_N$ with the $\ell_2$ distance.
For each $\tau > 0$, these distances define a class of $\tau$-Lipschitz functions $\cA_N : \sH_N \to \cB_N$, satisfying 
\[
    \norm{\cA_N(H_N) - \cA_N(H'_N)}_2 \le \tau \norm{H_N - H'_N}_2,\quad\forall~H_N,H_N'\in \sH_N.
\]
Note that this inequality holds vacuously for pairs $(H_N,H_N')$ where the latter distance is infinite.
As explained in \cite[Section 8]{huang2021tight}, the class of $O(1)$-Lipschitz algorithms includes gradient descent and Langevin dynamics for the Gibbs measure $e^{\beta H_N(\bsig)} \de \bsig$ (with suitable reflecting boundary conditions) run on constant time scales.\footnote{Up to modification on a set of probability at most $e^{-cN}$, which suffices just as well for our purposes.} 
The behavior of such dynamics has been a major focus of study in its own right, see e.g. \cite{sompolinsky1981dynamic,cugliandolo1994out,arous1995large,arous1997symmetric,arous2001aging,arous2006cugliandolo,arous2020bounding,dembo2020dynamics,dembo2021diffusions,dembo2021universality,celentano2021high,sellke2023threshold}.

We will characterize the largest energy attainable by a $\tau$-Lipschitz algorithm, where $\tau$ is an arbitrarily large constant independent of $N$, in terms of the following variational principle.
For $0 \le q_0 \le q_1 \le 1$, let $\bbI(q_0,q_1)$ be the set of increasing, continuously differentiable functions $f : [q_0,q_1] \to [0,1]$.
Let $\Adm(q_0,q_1) \subset \bbI(q_0,q_1)^\sS$ be the set of coordinate-wise increasing, continuously differentiable functions $\Phi:[q_0,q_1]\to [0,1]^{\sS}$ which satisfy, for all $q\in [q_0,q_1]$,
\begin{equation}
    \label{eq:admissible}
    \la \vlam, \Phi(q)\ra = q.
\end{equation}
We say $\Phi$ is \emph{admissible} if it satisfies \eqref{eq:admissible}.
For $p \in \bbI(q_0,1)$, $\Phi \in \Adm(q_0,1)$, define the algorithmic functional
\begin{equation}
    \label{eq:alg-functional}
    \bbA(p,\Phi; q_0)
    \equiv
    \sum_{s\in \sS}
    \lambda_s \lt[
        h_s \sqrt{\Phi_s(q_0)}
        +
        \int_{q_0}^1
        \sqrt{\Phi'_s(q) (p\times \xi^s \circ \Phi)'(q)}
        ~\de q
    \rt]
\end{equation}
where $(p\times \xi^s\circ \Phi)(q) = p(q) \xi^s(\Phi(q))$. (See the end of this subsection for an interpretation of this formula.)
We can now state the algorithmic threshold for multi-species spherical spin glasses:
\begin{equation}
    \label{eq:alg}
    \ALG
    \equiv
    \sup_{q_0\in [0,1]}
    \sup_{\substack{p\in \bbI(q_0,1) \\ \Phi \in \Adm(q_0,1)}} \bbA(p,\Phi; q_0).
\end{equation}

The following theorem is our main result. Together with Theorem~\ref{thm:main-alg} in our companion work \cite{huang2023optimization}, we find that $\ALG$ is the largest energy attained by an $O(1)$-Lipschitz algorithm.
Here and throughout, all implicit constants may depend also on $(\xi,\vh,\vlam)$.

\begin{theorem}
\label{thm:main-ogp}
    Let $\tau, \eps > 0$ be constants.
    For $N\geq N_0$ sufficiently large, any $\tau$-Lipschitz $\cA_N : \sH_N \to \cB_N$ satisfies
    \[
        \bbP[H_N(\cA_N(H_N))/N \ge \ALG + \eps]
        \le 
        \exp(-cN),
        \qquad
        c = c(\eps, \tau) > 0.
    \]
\end{theorem}

\begin{theorem}[{\cite[Theorem 1]{huang2023optimization}}]
\label{thm:main-alg}
For any $\eps>0$, there exists an efficient and $O_{\eps}(1)$-Lipschitz algorithm $\cA_N:\sH_N\to \cB_N$ such that
    \[
        \bbP[H_N(\cA_N(H_N))/N \geq \ALG-\eps]
        \ge 1-\exp(-cN),
        \quad 
        c = c(\eps) > 0.
    \] 
\end{theorem}

Our proof of Theorem~\ref{thm:main-alg} in \cite{huang2023optimization} uses approximate message passing (AMP), a general family of gradient-based algorithms, following a recent line of work \cite{subag2018following,mon18,ams20,alaoui2022algorithmic,sellke2021optimizing}.

In fact, in Theorem~\ref{thm:main-ogp} we will not require the full Lipschitz assumption on $\cA_N$.
Theorem~\ref{thm:main-ogp} holds for all algorithms satisfying an \emph{overlap concentration} property (see Definition~\ref{defn:oc}, Theorem~\ref{thm:main-ogp-oc}), that for any fixed correlation $p\in [0,1]$ between the disorder coefficients of $H_N^1$ and $H_N^2$, the overlap vector $\vR(\cA_N(H_N^1), \cA_N(H_N^2))$ concentrates tightly around its mean.
This property holds automatically for $O(1)$-Lipschitz $\cA_N$ due to Gaussian concentration of measure.

\paragraph{Interpretation of the Algorithmic Functional $\bbA$}

Suppose first that $\vh = \vzero$. 
We will see (Theorem~\ref{thm:alg-optimizer}) that $\ALG$ is maximized at $q_0=0$, $p\equiv 1$, in which case 
\begin{equation}
    \label{eq:alg-no-field}
    \bbA(p,\Phi; q_0) = \sum_{s\in \sS} \lambda_s \int_0^1 \sqrt{\Phi'_s(q) (\xi^s \circ \Phi)'(q)} ~\de q.
\end{equation}
In a single-species spherical spin glass, we have $\lambda_1=1$ and $\Phi(q)=q$, so \eqref{eq:alg-no-field} reduces to the formula $\ALG = \int_0^1 \xi''(q)^{1/2}~\de q$ derived in \cite{huang2021tight}.
This energy is attained by the algorithm of Subag \cite{subag2018following}, which starts from the origin and explores to the surface of the sphere by small orthogonal steps in the direction of the largest eigenvector of the local tangential Hessian.

In multi-species models, \eqref{eq:alg-no-field} is the energy attained by a generalization of Subag's algorithm, which is essentially shown in Proposition~\ref{prop:uc-bogp}.
Instead of computing a maximal eigenvector at each step, given the current iterate $\bx^t$ this algorithm chooses $\bx^{t+1}$ to maximize $\langle \nabla^2 H_N(\bx^t), (\bx^{t+1}-\bx^t)^{\otimes 2}\rangle$
on a product of $r$ small spheres centered at $\bx^t$. 
This algorithm may advance through different species at different speeds by tuning the radii of the spheres at each step, and the function $\Phi$ is a ``radius schedule" whose image specifies the path of depths $(\norm{\bx_s^t}_2^2/\lambda_s N)_{s\in \sS}$ traced by the iterates $\bx^t$.
Thus each $\Phi \in \Adm(0,1)$ corresponds to an algorithm, and Theorem~\ref{thm:main-ogp} essentially states that the algorithmic threshold is the energy attained by the multi-species Subag algorithm with the best $\Phi$.

The function $p$ arises from a further generalization of this algorithm, which becomes necessary in the presence of external field $\vh\neq \vzero$.
The idea is to reveal the disorder coefficients of $H_N$ gradually (in the sense of progressively less noisy Gaussian observations, see \eqref{eq:def-correlated-disorder}) and in tandem with the iterates $\bx^t$.
Though counterintuitive, this allows the algorithm to take advantage of the gradients of the newly revealed part of $H_N$ at each step.
The iterate $\bx^{t+1}$ is now chosen to maximize the sum
\begin{equation}
    \label{eq:compound-objective}
    \langle \nabla (H_N^{t+1} - H_N^t)(\bx^t), \bx^{t+1}-\bx^t \rangle
    + \frac{1}{2}\langle \nabla^2 H_N^t(\bx^t), (\bx^{t+1}-\bx^t)^{\otimes 2}\rangle
\end{equation}
of a gradient contribution from the new component and a Hessian contribution from the previously revealed components.
The function $p$ is an ``information schedule" that determines the rate at which entries of $H_N$ are revealed.
Moreover, to take advantage of the external field, the algorithm starts from a point $\bx^0$ correlated with $\bh$ whose norm is $q_0 \sqrt{N}$; the first term in \eqref{eq:alg-functional} is exactly the value $\langle \bh,\bx^0\rangle/N$ (see \eqref{eq:root-energy}).

Technically it is not obvious whether these generalized Subag algorithms can be directly made suitably Lipschitz. This is one reason we prove Theorem~\ref{thm:main-alg} using AMP in \cite{huang2023optimization}. 

\subsection{Description of Maximizers to the Algorithmic Variational Problem}

In this subsection we describe the detailed properties of the maximizers $(p,\Phi,q_0)$ of \eqref{eq:alg}, culminating in an explicit description in Theorem~\ref{thm:alg-optimizer} as a piecewise combination of solutions to two ordinary differential equations.

For intuition, it may help to recall the famous ansatz that spin glass Gibbs measures are asymptotically ultrametric, corresponding to orthogonally branching trees in $\bbR^N$ (see e.g. \cite{mezard1985microstructure,panchenko2013parisi,jagannath2017approximate,chatterjee2019average}). When $\vh=\vzero$, the associated tree is rooted at the origin; otherwise the root's location is correlated with $\bh$ but random.
Theorem~\ref{thm:alg-optimizer} below shows that maximizers of $\bbA$ consist of a ``root-finding" component and a ``tree-descending" component; the corresponding algorithms first locate an analogous root, and then descend an algorithmic analog of a low-temperature ultrametric tree.

This description holds under the following generic assumption.
\begin{assumption}
    \label{as:nondegenerate}
    All quadratic and cubic interactions participate in $H$, i.e. $\Gamma^{(2)}, \Gamma^{(3)} > 0$ coordinate-wise.
    We will call such models \textbf{non-degenerate}.
\end{assumption}
Note that $\ALG$ is continuous in the parameters $\xi,\vh$ (for a simple proof, first observe that $\bbA$ and hence $\ALG$ are monotone and subadditive in $(\xi,\vh)$).
Since Assumption~\ref{as:nondegenerate} is a dense condition, to determine the value of $\ALG$ it suffices to do so under this assumption.
In fact we will describe in detail the maximizing triples $(p,\Phi;q_0)$ under this assumption, which always exist but need not be unique.
Non-degeneracy removes extraneous symmetries among the maximizers of $\bbA$ which arise when e.g. $\xi$ is a sum of polynomials in disjoint sets of variables.

\begin{definition}
    \label{defn:diag-signed}
    A symmetric matrix $M\in \bbR^{\sS \times \sS}$ is \textbf{diagonally signed} if $M_{i,i}\ge 0$ and $M_{i,j}<0$ for all $i\neq j$.
\end{definition}
\begin{definition}
    \label{defn:solvable}
    A diagonally signed matrix $M$ is \textbf{super-solvable} if it is positive semidefinite, and \textbf{solvable} if it is furthermore singular; otherwise $M$ is \textbf{strictly sub-solvable}.
    A point $\vx \in (0,1]^\sS$ is super-solvable, solvable, or strictly sub-solvable if $M^*_\sym(\vx)$ is, where
    \begin{equation}
        \label{eq:M*sym}
        M^*_\sym(\vx) = 
        \diag\lt(\lt(\fr{\partial_{x_s}\xi(\vx) + \lambda_s h_s^2}{x_s}\rt)_{s\in \sS}\rt) 
        - \lt(\partial_{x_s,x_{s'}}\xi(\vx)\rt)_{s,s'\in \sS}
        .
    \end{equation}
    We also adopt the convention that $\vzero$ is always super-solvable, and solvable if $\vh=\vzero$. 
\end{definition}

\begin{remark}
    It is possible to extend the notions of (super, strict sub)-solvability to all of $[0,1]^\sS$ by using the alternative characterization from Corollary~\ref{cor:solvability-equivalent}. 
    However this will not be necessary, as our results only use these notions for $\vx \in (0,1]^\sS \cup \{\vzero\}$.
\end{remark}

\begin{definition}
\label{defn:root-finding-trajectory}
    Suppose $\vx \in (0,1]^\sS$ is super-solvable with $\la \vlam, \vx\ra = q_1$.
    A \textbf{root-finding trajectory} with endpoint $\vx$ is a pair $(p,\Phi) \in \bbI(q_0,q_1) \times \Adm(q_0,q_1)$, for some $q_0 \in [0, q_1]$, satisfying $p(q_1)=1$, $\Phi(q_1) = \vx$, $p(q_0)=0$, and for all $q\in [q_0,q_1]$:
    \begin{equation}
        \label{eq:root-finding-ode}
        \fr{(p\times \xi^s \circ \Phi)'(q)}{\Phi'_s(q)}
        =
        L_s\equiv \fr{\xi^s(\vx) + h_s^2}{x_s},\quad\forall s\in\sS.
    \end{equation}
\end{definition}

Assuming for now that $p,\Phi_s\in C^1([q_0,1])$,
\eqref{eq:root-finding-ode} together with admissibility can be written for each $q\in [q_0,q_1]$ as the ordinary differential equation

\begin{align}
\label{eq:root-finding-system-1}
    p'(q) 
    \xi^s(\Phi(q))
    +
    p(q)
    \sum_{s'\in\sS}
    \partial_{x_{s'}}\xi^s(\Phi(q))
    \,
    \Phi_{s'}'(q)
    &=
    L_s \Phi_s'(q),\quad\forall s\in\sS;
    \\
\label{eq:root-finding-system-2}
    \sum_{s\in\sS}
    \lambda_s \Phi_s'(q)
    &=
    1;
    \\
\label{eq:root-finding-system-3}
    p'(q),\Phi_s'(q)&\geq 0.
\end{align}
Here $\vL$ is treated as fixed, as it is determined by the boundary condition at $q_1$. Note that equation~\eqref{eq:root-finding-system-1} does not depend on $q$, except that $\Phi(q)$ determines $q$ via admissibility. In fact \eqref{eq:root-finding-system-1} is equivalent to a well-posed ordinary differential equation (away from $\vzero$, which it never reaches by Proposition~\ref{prop:root-finding-trajectory}). Moreover as shown in Proposition~\ref{prop:root-finding-trajectory}(\ref{it:unique-root-finding}), solving this ODE from a super-solvable initial condition always yields a valid root-finding trajectory (e.g. the resulting $p$ is actually increasing on $[q_0,q_1]$).

\begin{proposition}
\label{prop:root-finding-well-posed}
    For $(p(q),\Phi(q))$ in compact subsets of $[0,1]\times (0,1]^\sS$ the equation \eqref{eq:root-finding-system-1} has a unique solution $(p'(q),\Phi'(q))$ which is locally Lipschitz in $(p(q),\Phi(q))$.
\end{proposition}

\begin{proposition}
    \label{prop:root-finding-trajectory}
    $\vh \neq \vzero$ if and only if there exists a super-solvable $\vx \in (0,1]^\sS$. 
    If this holds, for each such $\vx$: 
    \begin{enumerate}[label=(\alph*), ref=\alph*]
        \item
        \label{it:unique-root-finding}
        Let $q_1 = \la \vlam, \vx\ra > 0$.
        There is a unique root-finding trajectory $(p,\Phi)$ with endpoint $\vx$. It is obtained by solving \eqref{eq:root-finding-system-1} backward in time from initial condition $p(q_1)=1$, $\Phi(q_1)=\vx$ until reaching $p(q_0)=0$. Moreover the resulting $p$ is increasing and concave on $[q_0,q_1]$.
        \item 
        \label{it:unique-root-finding-q0}
        $q_0>0$, and in fact $\Phi_s(q_0)>0$ if and only if $h_s>0$.
    \end{enumerate}
\end{proposition}

\begin{definition}
    \label{defn:tree-descending-trajectory}
    Suppose $\vx \in (0,1]^\sS \cup \{\vzero\}$ is solvable with $\la \vlam, \vx\ra = q_1$.
    A \textbf{tree-descending trajectory} with endpoint $\vx$ is a pair $(p,\Phi) \in \bbI(q_1,q_2) \times \Adm(q_1,q_2)$ satisfying $p\equiv 1$, $\Phi(q_1)=\vx$, $M^*_\sym(\vx) \Phi'(q_1) = \vzero$, $\norm{\Phi_s(q_2)}_\infty = 1$ and
    \begin{equation}
        \label{eq:tree-descending-ode}
        \fr{1}{\Phi'_s(q)}
    	\deriv{q}
    	\sqrt{\fr{\Phi'_s(q)}{(\xi^s \circ \Phi)'(q)}}
    	=
		\fr{1}{\Phi'_{s'}(q)}
    	\deriv{q}
    	\sqrt{\fr{\Phi'_{s'}(q)}{(\xi^{s'} \circ \Phi)'(q)}}
    \end{equation}
    for all $s,s'\in \sS$ and $q\in [q_1,q_2]$.
    Moreover, $(p,\Phi)$ is \emph{targeted} if $\Phi(1)=\vone$ (i.e. $q_2=1$).
\end{definition}

Similarly to \eqref{eq:root-finding-system-1}, assuming $\Phi''$ is defined, \eqref{eq:tree-descending-ode} together with the admissibility constraint 
\begin{equation}
\label{eq:admissible-second-order}
\sum_{s\in\sS} \lambda_s\Phi_s''(q)=0
\end{equation}
is equivalent to a second order differential equation. We show in Subsection~\ref{subsec:type-II} and Appendix~\ref{subsec:type-II-Lipschitz} that this equation is suitably well-posed and obtain the following results.

\begin{proposition}
    \label{prop:tree-descending-trajectory}
    Suppose Assumption~\ref{as:nondegenerate} holds and $\vx \in (0,1]^\sS \cup \{\vzero\}$ is solvable with $\la \vlam, \vx\ra = q_1$.
    \begin{enumerate}[label=(\alph*), ref=\alph*]
        \item \label{itm:tree-descending-with-field} If $\vh \neq \vzero$, then $\vx \in (0,1]^\sS$ and $q_1 > 0$. 
        There is a unique $\vv \in \bbR_{\ge 0}^\sS$ satisfying
        \begin{align}
            \label{eq:start-velocity-direction}
            M^*_\sym(\vx)\vv &= \vzero, \\
            \label{eq:start-velocity-magnitude}
            \la \vlam, \vv\ra &= 1.
        \end{align}
        There is a unique tree-descending trajectory with endpoint $\vx$, which is obtained by solving \eqref{eq:tree-descending-ode} forward in time from $\Phi(q_1)=\vx$, $\Phi'(q_1)=\vv$ until reaching $\norm{\Phi_s(q_2)}_\infty = 1$.
        \item \label{itm:tree-descending-no-field} If $\vh = \vzero$, then $\vx = \vzero$ and $q_1 = 0$. 
        For any $\vv \in \bbR_{\ge 0}^{\sS}$ satisfying \eqref{eq:start-velocity-magnitude}, there is a unique tree-descending trajectory with $\Phi(0)=\vzero$ and $\Phi'(0)=\vv$, which is obtained by solving \eqref{eq:tree-descending-ode} forward in time from these conditions until reaching $\norm{\Phi_s(q_2)}_\infty = 1$.
    \end{enumerate}
\end{proposition}

The following theorem is our main result describing maximizers of \eqref{eq:alg}. 
\begin{theorem}
    \label{thm:alg-optimizer}
    Suppose Assumption~\ref{as:nondegenerate} holds.
    Then a maximizer $(p,\Phi,q_0)$ of \eqref{eq:alg} exists, and all maximizers are continuously differentiable on $[q_0,1]$. 
    There exists $q_1\in [q_0,1]$ such that $\Phi(q_1) \in (0,1]^\sS \cup \{\vzero\}$ and furthermore 
    $(p,\Phi)$ is the root-finding trajectory with endpoint $\Phi(q_1)$ on $[q_0,q_1]$ and a (targeted) tree-descending trajectory with endpoint $\Phi(q_1)$ on $[q_1,1]$. $\ALG$ is given by
    \begin{equation}
        \label{eq:alg-for-optimizer}
        \ALG = 
        \bbA(p,\Phi; q_0) =
        \sum_{s\in \sS}
        \lambda_s \lt[
            \sqrt{\Phi_s(q_1) (\xi^s(\Phi(q_1)) + h_s^2)}  + 
            \int_{q_1}^1 \sqrt{\Phi'_s(q)(\xi^s\circ \Phi)'(q)}~\de q
        \rt].
    \end{equation}
    Finally the value of $q_1$ is described as follows:
    \begin{enumerate}[label=(\alph*), ref=\alph*]
        \item \label{itm:supsolvable} If $\vone$ is super-solvable then $q_1=1$, i.e. $(p,\Phi)$ is the root-finding trajectory with endpoint $\vone$.
        \item \label{itm:subsolvable-with-field} If $\vone$ is sub-solvable and $\vh \neq \vzero$, then $q_1 \in (q_0,1)$, i.e. $(p,\Phi)$ contains both root-finding and tree-descending trajectories. 
        \item \label{itm:subsolvable-no-field} If $\vh = \vzero$, then $\vone$ is sub-solvable and $q_1=0$, i.e. $(p,\Phi)$ is a (targeted) tree-descending trajectory with endpoint~$\vzero$. 
    \end{enumerate}
\end{theorem}

Note that in case (\ref{itm:subsolvable-with-field}), $\Phi(q_1) \in (0,1]^\sS$ if $h_s>0$ for \textbf{any} $s$.
Examples of each of these cases are given in Figure~\ref{fig:ode}.

\begin{remark} 
    \label{rem:normalization}
    The choice of state space $\cB_N$ is a natural though arbitrary normalization.
    For any $\va \in \bbR_{>0}^\sS$, we could just as well consider the state space 
    \begin{equation}
    \label{eq:BN-va}
        \cB_N(\va) = \lt\{
            \bx \in \bbR^N : \tnorm{\bx_s}_2^2 \le a_s \lambda_s N ~\forall s\in \sS
        \rt\}.
    \end{equation}
    Clearly optimizing the model described by $\xi,\vh$ over this space is equivalent to optimizing the model described by\footnote{Here and throughout this paper, powers of vectors such as $\sqrt{\va}$ are taken coordinate-wise.}
    \begin{equation}
        \label{eq:rescale-problem-transformation}
        \tilde \xi(\vx) = \xi(\vx \odot \sqrt{\va}), \qquad
        \tilde \vh = \vh \odot \sqrt{\va}
    \end{equation} 
    over $\cB_N$, so changing the problem in this way does not add any complexity.
    However, from this point of view we can see that the requirement in the equation \eqref{eq:alg} and Theorem~\ref{thm:alg-optimizer} that $\Phi(1) = \vone$ is merely a product of the normalization.
    If we wished to optimize over $\cB_N(\va)$, equation \eqref{eq:alg} and Theorem~\ref{thm:alg-optimizer} still hold with the right endpoint of $\Phi$ changed to $\va$, which is easily proved by the transformation \eqref{eq:rescale-problem-transformation}.
    Thus the non-targeted trajectories in Figure~\ref{fig:ode} describe optimal algorithms for other state spaces $\cB_N(\va)$.
\end{remark}

\begin{remark}
    Because the root-finding and tree-descending ODEs are well-posed, the results above give a natural approach to solve the $N$-independent problem of approximately maximizing $\bbA$ to $\eps$ error. If $\vone$ is super-solvable then $\ALG$ is given directly by \eqref{eq:alg-for-optimizer}. If $\vh\neq \vzero$, then it suffices to brute-force search for the value $\Phi(q_1)$ over a $\delta$-net of solvable $\vx\in [0,1]^{\sS}$ and solve each of the two ODEs above; note that the vector $\Phi'(q_1)$ is determined by \eqref{eq:root-finding-system-1}. Finally if $\vh=\vzero$, since $q_1=0$ it suffices to brute-force search over all $\Phi'(0)$ satisfying \eqref{eq:root-finding-system-2}.
\end{remark}

\begin{remark}
    In models where $\vone$ is super-solvable, the formula \eqref{eq:alg-for-optimizer} simplifies to
    \begin{equation}
        \label{eq:alg-in-trivial-phase}
        \ALG =  
        \sum_{s\in \sS}
        \lambda_s 
        \sqrt{\Phi_s(1) (\xi^s(\Phi(1)) + h_s^2)}.
    \end{equation}
    As shown in our companion paper \cite[Theorem 1.6]{huang2023strong}, this coincides with the true maximum value $\OPT$.
    Moreover the models where $\vone$ is strictly super-solvable are precisely the \emph{topologically trivial} ones, where with high probability the number of critical points is exactly $2^r$, the minimum number possible for a Morse function on a product of $r$ spheres.
    This generalizes an observation from \cite{huang2021tight} that in an analogous regime of single-species models, $\ALG=\OPT$ and, as shown in \cite{fyodorov2013high, belius2021triviality}, the model is topologically trivial.
\end{remark}

\begin{remark}
    \label{rem:type1-partway}
    Recall the algorithmic interpretation of $(p,\Phi,q_0)$ discussed around \eqref{eq:compound-objective}.
    For any $q \in [q_0,q_1]$, the iterate of this algorithm at radii $\Phi(q)$ is an approximate maximizer of the Hamiltonian revealed up to that point (whose disorder coefficients have variance $p(q)$) on the product of spheres $\cB_N(\Phi(q))$.
    Indeed the energy attained by these iterates is calculated in Corollary~\ref{cor:type1-partway} and coincides with \eqref{eq:alg-in-trivial-phase} with $\Phi(q), p(q)\xi$ in place of $\Phi(1), \xi$.
\end{remark}

\begin{figure}[t]
    \begin{subfigure}[b]{.32\textwidth}
        \centering
        \includegraphics[width=.9\linewidth]{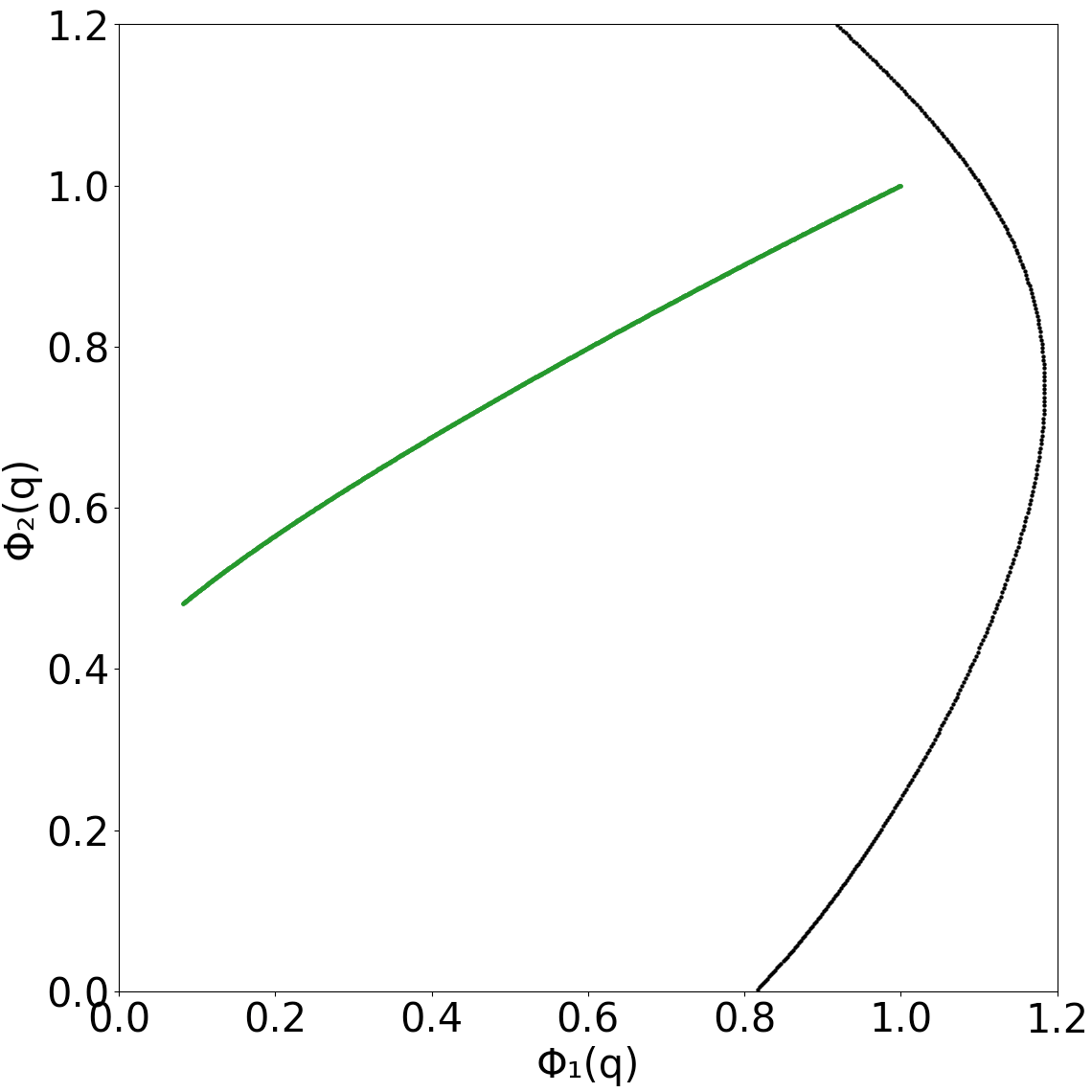} \\
        \includegraphics[width=.9\linewidth]{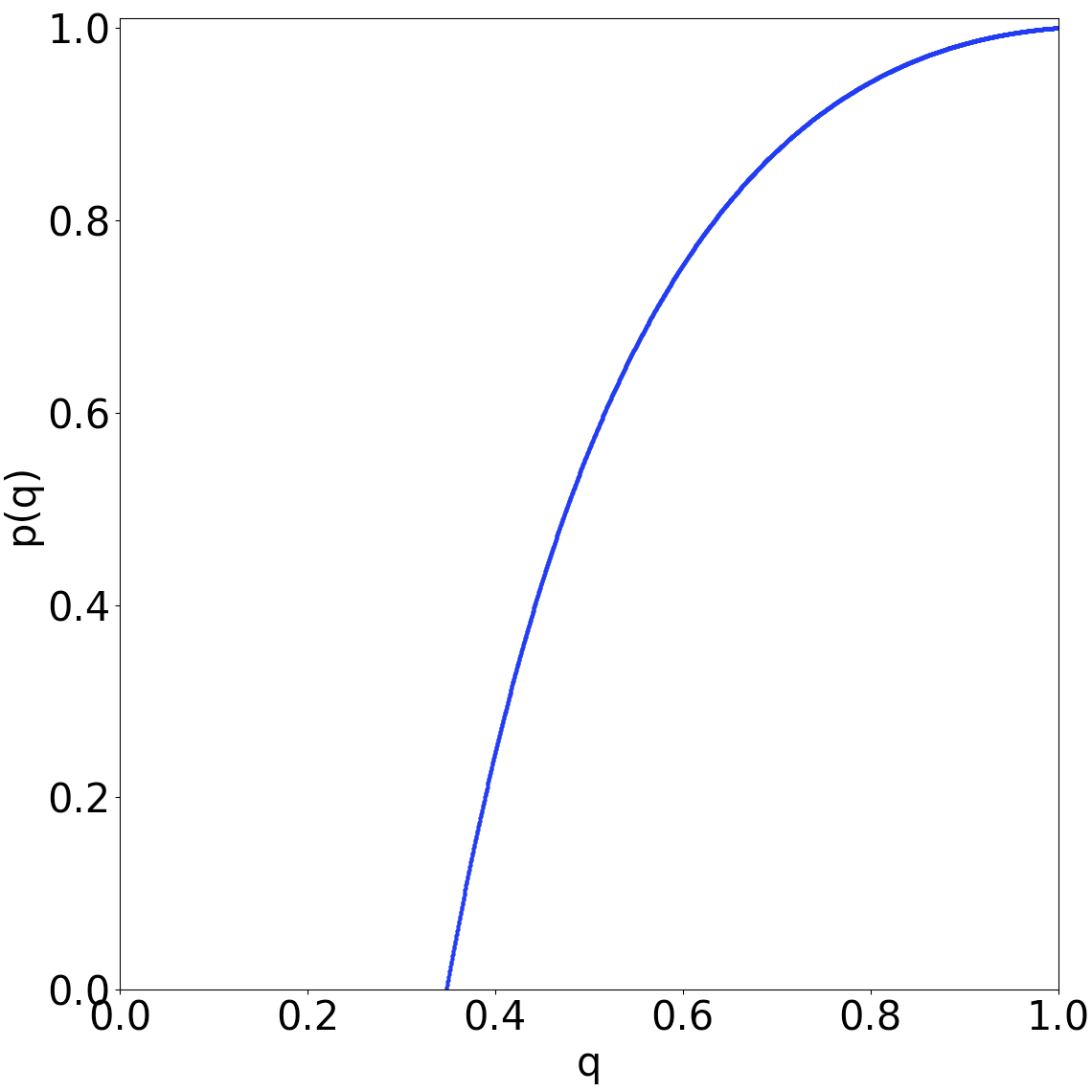}
        \caption{$\vh = (0.4,1.4)$, $\vone$ super-solvable.}
        \label{subfig:supersolvable}
    \end{subfigure}
    \hfill
    \begin{subfigure}[b]{.32\textwidth}
        \centering
        \includegraphics[width=.9\linewidth]{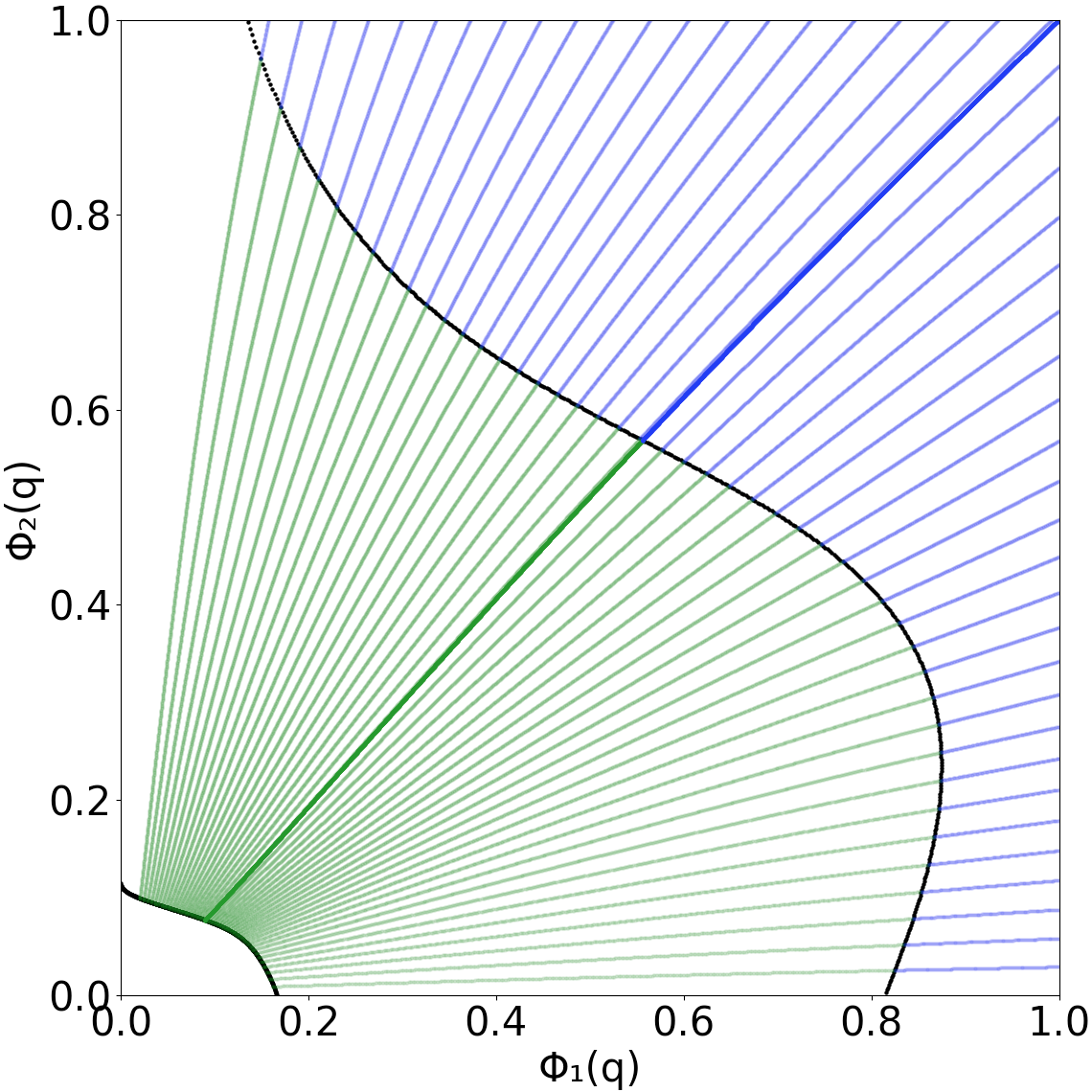} \\
        \includegraphics[width=.9\linewidth]{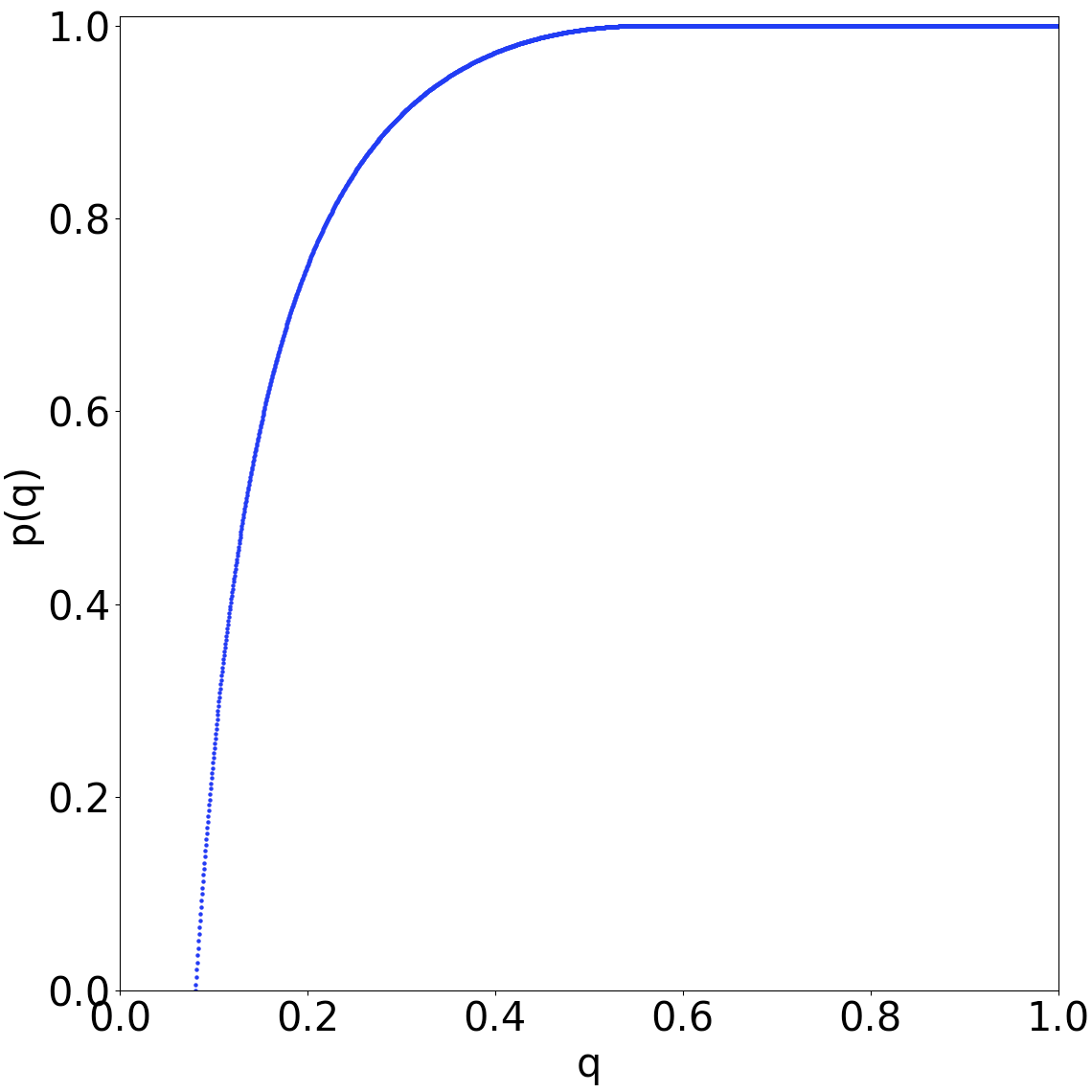}
        \caption{$\vh = (0.4,0.4)$, $\vone$ sub-solvable.}
        \label{subfig:with-field}
    \end{subfigure}
    \hfill
    \begin{subfigure}[b]{.32\textwidth}
        \centering
        \includegraphics[width=.9\linewidth]{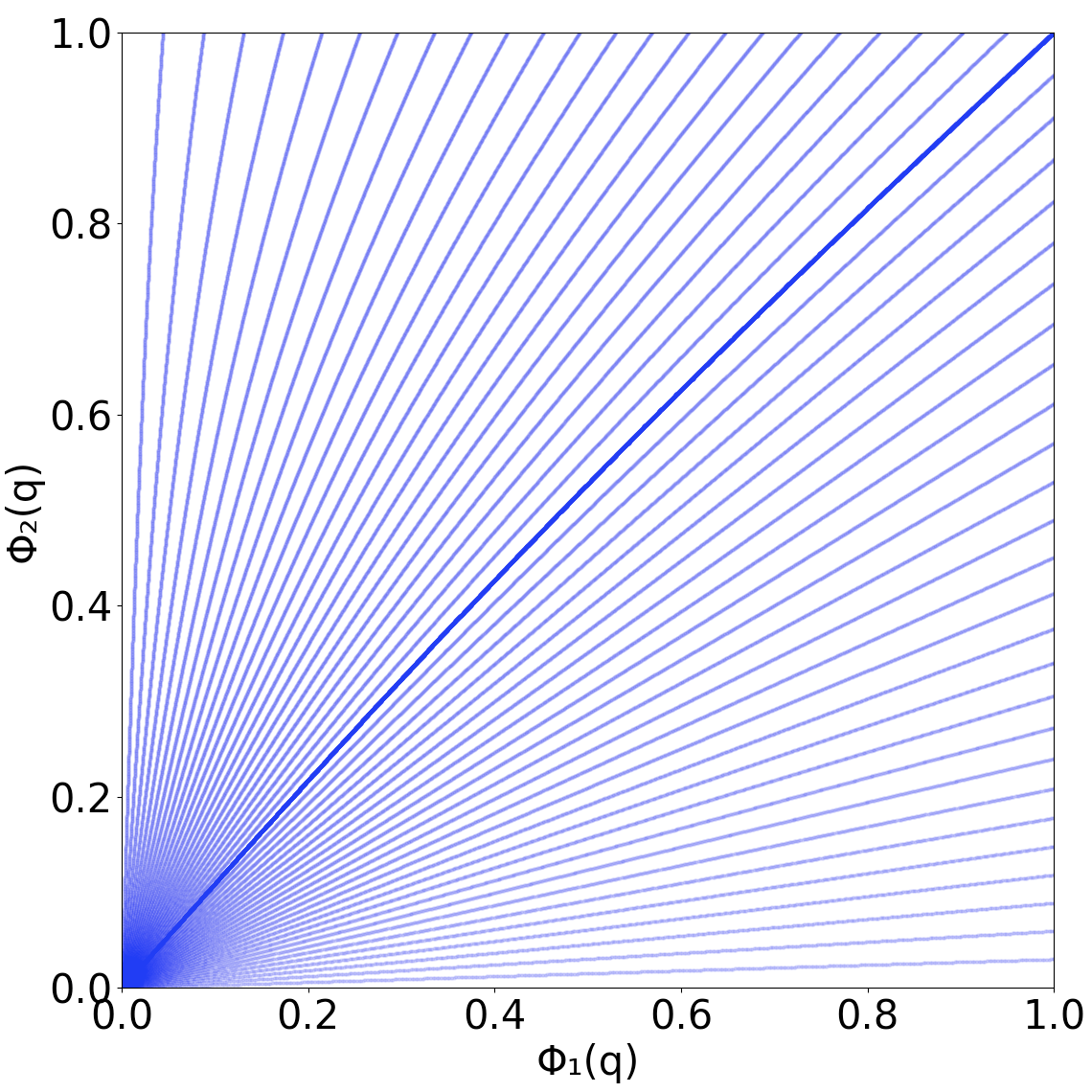} \\
        \includegraphics[width=.9\linewidth]{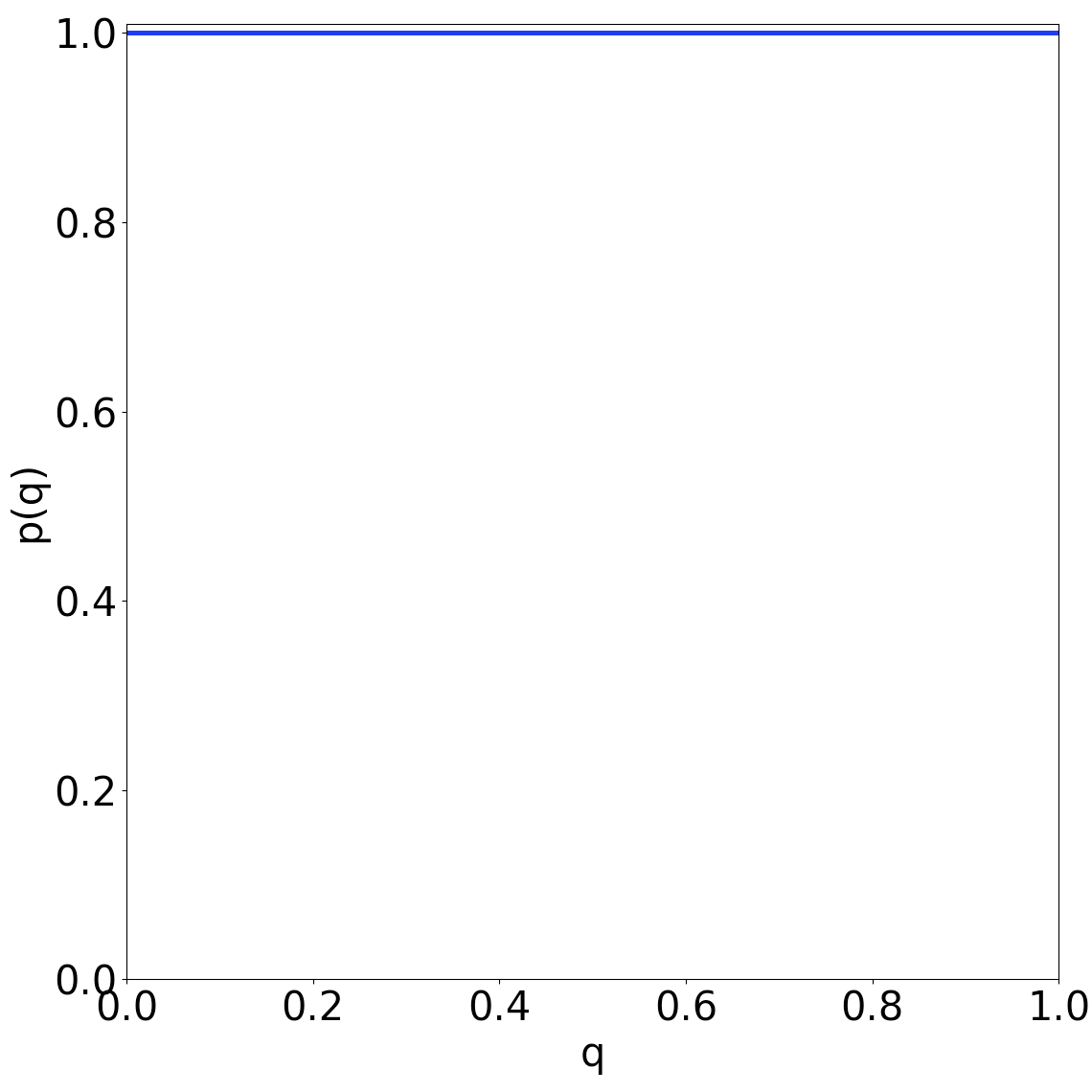}
        \caption{$\vh = (0,0)$.}
        \label{subfig:no-field}
    \end{subfigure}
    \caption{
    \small
        Examples of Theorem~\ref{thm:alg-optimizer}.
        Consider the model $\vlam = (\fr13, \fr23)$, $\xi(x_1,x_2) = \nu(\lambda_1x_1,\lambda_2x_2)$, and various $\vh$ specified in the captions above, where $\nu(x_1,x_2) = x_1^2 + x_1x_2 + x_2^2 + x_1^4 + x_1x_2^3$.
        These are described by parts (\ref{itm:supsolvable}), (\ref{itm:subsolvable-with-field}), and (\ref{itm:subsolvable-no-field}) of Theorem~\ref{thm:alg-optimizer}, respectively.
        The top diagrams plot $\Phi(q)$ with root-finding components green and tree-descending components blue.
        The optimal $\Phi$, which passes through $(1,1)$, is bold.
        Figures~\ref{subfig:with-field} and \ref{subfig:no-field} show non-targeted trajectories otherwise described by Theorem~\ref{thm:alg-optimizer}.
        In Figure~\ref{subfig:supersolvable} the black curve comprises the solvable points and $(1,1)$ is inside of this curve.
        In Figure~\ref{subfig:with-field} the outer black curve comprises the solvable points, which are the possible values of $\Phi(q_1)$, and $(1,1)$ is outside of this curve. The inner black curve of Figure~\ref{subfig:with-field} comprises the corresponding values of $\Phi(q_0)$.
        The bottom diagrams plot $(q, p(q))$ for the optimal $p$.
    }
    \label{fig:ode}
\end{figure}

\subsection{Explicit Solutions in Special Cases}

While the formulas \eqref{eq:alg}, \eqref{eq:alg-for-optimizer} for $\ALG$ involve the solution to a variational problem, $\ALG$ can be written explicitly in the important special cases of single-species models where $r=1$ and $\vlam = (1)$, and pure models where $\xi$ is a monomial.

\subsubsection{Single-Species Models}

In single-species models, $\xi(q)$ is a univariate function and \eqref{eq:admissible} implies $\Phi(q)=q$.
Let $\vh = (h)$. 
\begin{corollary}[Algorithmic threshold of single-species models]
    \label{cor:alg-one-species}
    If $\xi'(1) + h^2 \ge \xi''(1)$, then 
    \[
        \ALG = (\xi'(1) + h^2)^{1/2}.
    \]
    The variational formula \eqref{eq:alg} is maximized at $q_0=\fr{h^2}{\xi'(1)+h^2}$, $p(q) = \fr{q(\xi'(1)+h^2)-h^2}{\xi'(q)}$ for $q\in [q_0,1]$.
    Otherwise there is a unique $q_1\in [0,1)$ satisfying $\xi'(q_1) + h^2 = q_1 \xi''(q_1)$, and 
    \[  
        \ALG = q_1 \xi''(q_1)^{1/2} + \int_{q_1}^1 \xi''(q)^{1/2}~\de q.
    \]
    The variational formula \eqref{eq:alg} is maximized at
    \[
        q_0 = \fr{h^2}{\xi''(q_1)},
        \qquad
        p(q) =
        \begin{cases}
            \fr{q\xi''(q_1)-h^2}{\xi'(q)} & q\in [q_0,q_1], \\
            1 & q\in [q_1,1].
        \end{cases}    
    \]
\end{corollary}
Except for the formulas for $q_0$ and $p(q)$, this corollary follows readily from Theorem~\ref{thm:alg-optimizer}; note that super-solvability of $\vone$ generalizes the inequality $\xi'(1) + h^2 \ge \xi''(1)$ and solvability of $\Phi(q_1)$ generalizes $\xi'(q_1) + h^2 = q_1 \xi''(q_1)$.
The formulas for $q_0$ and $p(q)$ follow from \eqref{eq:type1-q0-formula} and \eqref{eq:type1-p-formula}.

The formula for $\ALG$ in Corollary~\ref{cor:alg-one-species} matches \cite[Proposition 2.2]{huang2021tight}.
Whereas \cite{huang2021tight} proves this formula for even $\xi$, we obtain it in full generality.
This formula also matches the ground state energy in full replica symmetric breaking models as obtained in \cite[Proposition 2]{chen2017parisi}.

\subsubsection{Direct Proof for Single Species Models without External Field}

In the case $h=0$, the formula for $\ALG$ can be directly recovered from the variational formula \eqref{eq:alg}.
First, we should clearly take $q_0=0$, so $\ALG = \sup_{p\in \bbI(0,1)} \int_0^1 (p\xi')'(q)^{1/2}~\de q$.
Then, because
\[
    \int_t^1 (p\xi')'(q)~\de q
    =
    \xi'(1) - p(t)\xi'(t)
    \ge 
    \xi'(1) - \xi'(t)
    =
    \int_t^1 \xi''(q)~\de q
\]
for all $t\in [0,1]$ with equality at $t=0$, the function $(p\xi')'$ majorizes $\xi''$ (see e.g. \cite{joe1992generalized} for precise definitions of majorization in non-discrete settings). Here we use that $\xi''$ is increasing, but do not assume that $(p\xi')'$ is.
By Karamata's inequality,
\[
    \int_0^1 (p\xi')'(q)^{1/2}~\de q \le \int_0^1 \xi''(q)^{1/2}~\de q
\]
with equality at $p\equiv 1$.

\subsubsection{Pure Models}

Finally we give in Theorem~\ref{thm:pure} below an explicit formula for $\ALG$ for \emph{pure} $\xi$ consisting of a single monomial, and moreover identify the unique maximizer to $\bbA$. 
Our proof in Subsection~\ref{subsec:pure} takes advantage of scale invariance to relate values of $\ALG$ at different radii (see Remark~\ref{rem:normalization}). 
Recently \cite{subag2021tap} used a similar scale invariance (and other ideas) to compute the free energy in such models under the mild assumption of convergence as $N\to\infty$.
Intriguingly for all pure models, the value $\ALG$ agrees with the threshold $E_{\infty}$ arising from critical point asymptotics in \cite{auffinger2013random} and determined in the multi-species setting by \cite{mckenna2021complexity}.

It should be noted that Assumption~\ref{as:nondegenerate} on non-degeneracy is false for pure models, so we cannot rely on the structural results of Theorem~\ref{thm:alg-optimizer}. 
Additionally, note that although the optimal trajectories $\Phi$ stated in Theorem~\ref{thm:pure} are not admissible, this does not present a problem; Lemma~\ref{lem:admissible-optional} shows that admissibility is just a convenient choice of time parametrization and deviating from it does not affect the value of $\bbA$.

\begin{theorem}
\label{thm:pure}
    Suppose $\vh=\vzero$ and
    \[
    \xi(x_1,\dots,x_r)=\prod_{s\in\sS} x_s^{a_s}
    \]
    for positive integers $a_1,\dots,a_r$ with $r\geq 2$ and $\sum_{s\in\sS}a_s\geq 3$. Define the exponents $b_s$ by
    \begin{equation}
    \label{eq:pure-exponents-formula}
    b_s=
    \frac{1- \sqrt{\frac{a_s}{a_s+L\lambda_s}}}{2},
    \quad s\in \sS
    \end{equation}
    where $L=L(\va)>0$ is the unique value such that $\sum_{s\in\sS}a_s b_s=1$. Then $\ALG$ and the $(p,\Phi,q_0)$ maximizing $\bbA$ are 
    \begin{align*}
    \ALG&=\sum_{s\in\sS} \frac{\lambda_s \sqrt{L a_s}}{\sqrt{a_s+L\lambda_s}},
    \\
    \big(p(q)\,, \Phi(q),\,q_0\big)
    &=
    \big(1,\,(q^{b_1},\dots,q^{b_r}),\,0\big)
    .
    \end{align*}
    In the case $\xi(x_1,x_2)=x_1 x_2$ we have
    \begin{align*}
    \ALG&=\sqrt{\lambda_1}+\sqrt{\lambda_2},
    \\
    \big(p(q)\,, \Phi(q),\,q_0\big)
    &=
    \big(1,\,(q,q),\,0\big)
    .
    \end{align*}
    Moreover the optimal $(p,\Phi,q_0)$ is always unique up to reparametrization.
\end{theorem}

Theorem~\ref{thm:pure} simplifies in the special case that $\frac{a_s}{\lambda_s}$ is independent of $s$, i.e. $\lambda_s=\frac{a_s}{\sum_{s\in\sS}a_s}$. In particular $\ALG$ depends only on the total degree $\sum_{s\in\sS}a_s$. Note that the formula \eqref{eq:pure-exponents-formula} gives $b_s=\frac{1}{\sum_{s\in\sS} a_s}$, which is equivalent by reparametrization to $b_s=1$ as stated below.

\begin{corollary}
\label{cor:pure}
    For pure models with $\lambda_s=\frac{a_s}{\sum_{s'\in\sS}a_{s'}}$, $\Phi(q)=(q,\dots,q)$ uniquely maximizes $\bbA$ and
    \[
    \ALG=2\sqrt{\frac{\big(\sum_{s\in\sS} a_s\big)-1}{\sum_{s\in\sS} a_s}}.
    \]
\end{corollary}

For all pure models, the value $\ALG$ in Theorem~\ref{thm:pure} agrees with the threshold $E_{\infty}$ defined as follows. We denote by $\nabla_{\sph}$ the gradient on the product of spheres $\cS_N\equiv\{\bx\in\cB_N~:~\vR(\bx,\bx)=\vone\}$, and $\nabla^2_{\sph}$ the Riemannian Hessian.
Below the \emph{index} of a square matrix denotes the number of non-negative eigenvalues.

\begin{definition}
\label{defn:E-infty}
    For $\vh=\vzero$ and any $\xi$, the value $E_{\infty}$ is given by $E_{\infty}=\lim_{k\to\infty} E_k\geq 0$. Here $E_k\geq 0$ is the minimal value such that for any $E>E_k$,
    \[
    \lim_{N\to\infty}
    \frac{1}{N}
    \log
    \bbE\lt[
    \lt|
    \lt\{
    \bsig\in\cS_N~:~
    H_N(\bsig)\geq EN,~
    \nabla_{\sph}H_N(\bsig)=0,~
    \text{index}(\nabla^2_{\sph} H_N(\bsig))\ge k
    \rt\}
    \rt|
    \rt]
    <0.
    \]
\end{definition}

Informally, $E_{\infty}$ is the threshold above which critical points of unbounded index cease to exist in an annealed sense. 
For multi-species spin glasses, $E_{\infty}$ is given by the somewhat complicated formula \cite[Equation (2.7)]{mckenna2021complexity} which involves the solution to a matrix Dyson equation, recalled in Subsection~\ref{subsec:pure}. This generalizes the single-species formulas in \cite{auffinger2013random,arous2020geometry}.
We note that for pure \emph{single-species models}, \cite[Theorem 1.4]{auffinger2020number} claims (without a full proof yet) that for any $E<E_{\infty}$, critical points of bounded index (depending only on $E$) exist above energy $E$ with high probability. 

\begin{corollary}
\label{cor:E-infty}
    For all pure $\xi$, we have
    \[
    \ALG=E_{\infty}.
    \]
\end{corollary}

In the single species case, Corollary~\ref{cor:E-infty} holds for the pure $p$-spin model $\xi(x)=x^p$ with
$\ALG=E_{\infty}=2\sqrt{\frac{p-1}{p}}$ identified in \cite{auffinger2013random}, as discussed in \cite[Section 2.3]{huang2021tight}.
While the single-species formula is simple, Corollary~\ref{cor:E-infty} is much less obvious in general. 
In our companion works \cite{huang2023optimization,huang2023strong} we give a more general approach to this connection by showing that the top of the bulk spectrum of $\nabla^2_{\sph} H_N(\bsig)$ is approximately $0$ for $\bsig$ the output of an explicit optimization algorithm attaining value $\ALG$. This statement holds for all $\xi$ and implies that $\ALG$ in general lies in an interval denoted $[E_{\infty}^-,E_{\infty}^+]$ in \cite{auffinger2013complexity}.
Also relatedly, \cite{sellke2023threshold} shows that low-temperature Langevin dynamics (run for large dimension-free time) suffices to attain energy $\ALG=E_{\infty}$ in pure models. (The result is stated for $1$ species but extends with almost no changes to multi-species pure models.)
This is not expected to generalize to mixed models as discussed at the end of Subsection 1.1 therein.

\subsection{Non-Uniqueness of Maximizers and Algorithmic Symmetry Breaking}
\label{subsec:ASB}

In cases (\ref{itm:subsolvable-with-field}) and (\ref{itm:subsolvable-no-field}) of Theorem~\ref{thm:alg-optimizer}, the ODE description of maximizers does \emph{not} uniquely determine $(p,\Phi)$.
In case (\ref{itm:subsolvable-with-field}), each $(p,\Phi)$ described by Theorem~\ref{thm:alg-optimizer} is specified by the point $\vx = \Phi(q_1)$, which must be solvable and have the property that the tree-descending trajectory with endpoint $\vx$ (unique by Proposition~\ref{prop:tree-descending-trajectory}) is targeted. 
In case (\ref{itm:subsolvable-no-field}), each $(p,\Phi)$ is specified by the velocity $\vv = \Phi'(0)$, which must satisfy \eqref{eq:start-velocity-magnitude} and have the property that the tree-descending trajectory with endpoint $\vzero$ and starting velocity $\vv$ (unique by Proposition~\ref{prop:tree-descending-trajectory}) is targeted.
There may be multiple possible $\vx$ or $\vv$; see Figure~\ref{fig:asb} for examples.

In fact, even in \emph{symmetric} two-species models -- where $\vlam = (\fr12,\fr12)$, $\vh = (h,h)$, and $\xi(q_1,q_2)$ is symmetric in $q_1,q_2$ -- there may be many $(p,\Phi)$ described by Theorem~\ref{thm:alg-optimizer}.
Moreover, surprisingly, the maximizer of \eqref{eq:alg} need not be symmetric!
The only possible symmetric maximizer is $\Phi(q)=(q,q)$, which (for suitable $p$) satisfies the properties in Theorem~\ref{thm:alg-optimizer}.
In Figures~\ref{subfig:with-field-asb} and \ref{subfig:no-field-asb} we give examples of models, corresponding to cases (\ref{itm:subsolvable-with-field}) and (\ref{itm:subsolvable-no-field}) of Theorem~\ref{thm:alg-optimizer}, where a pair of asymmetric $\Phi$ numerically outperform the symmetric $\Phi$.
We name this phenomenon \emph{algorithmic symmetry breaking}.\footnote{While we don't prove rigorously that these examples exhibit algorithmic symmetry breaking, it can be verified explicitly that for the model $\xi(x,y)=x^4+y^4+24xy$, $\vh = (0,0)$ with endpoint $\va = (5,5)$ (cf. Remark~\ref{rem:normalization}), the symmetric path $\Phi(q)=(q,q)$ is not even a local optimum as witnessed by $\Phi^{\eps}(q)=(q+\eps\sin(\pi q/10),q)$.}
The presence of algorithmic symmetry breaking implies that there exist symmetric models where the best instantiation of the multi-species Subag algorithm advances through the species asymmetrically.
Note that it is impossible for solutions to a first order ODE to cross, but the tree-descending ODE is second order which enables this behavior.

It is also possible to have 
several trajectories satisfying the ODE description in Theorem~\ref{thm:alg-optimizer} and we expect an unbounded number can coexist, see Figure~\ref{subfig:many-asb}.
While it is a priori unclear that the extremal trajectories attaining value $E_1$ (defined in the caption) outperform the diagonal trajectory, there is a simple reason the diagonal-crossing trajectories attaining $E_2$ cannot be optimal: if these two trajectories were optimal, then 
joining their above-diagonal parts would yield another global maximizer which is not $C^1$ and in particular does not satisfy the ODE description of Theorem~\ref{thm:alg-optimizer}. (Note also that different trajectories must have different derivatives where they meet, given their description by a second order ODE.)
We leave the question of characterizing global maximizers in the presence of algorithmic symmetry breaking for future work.

We emphasize that algorithmic symmetry breaking is not a barrier to any algorithm, as the optimal $(p,\Phi,q_0)$ for the variational principle needs to be computed only once.
Moreover $\xi$ is convex in the examples shown in Figure~\ref{fig:asb}, so algorithmic symmetry breaking is not related to the failure of the interpolation method to determine the free energy (obtained for convex $\xi$ in \cite{bates2022free}).

\begin{figure}
    \begin{subfigure}[b]{.32\textwidth}
        \centering
        \includegraphics[width=.9\linewidth]{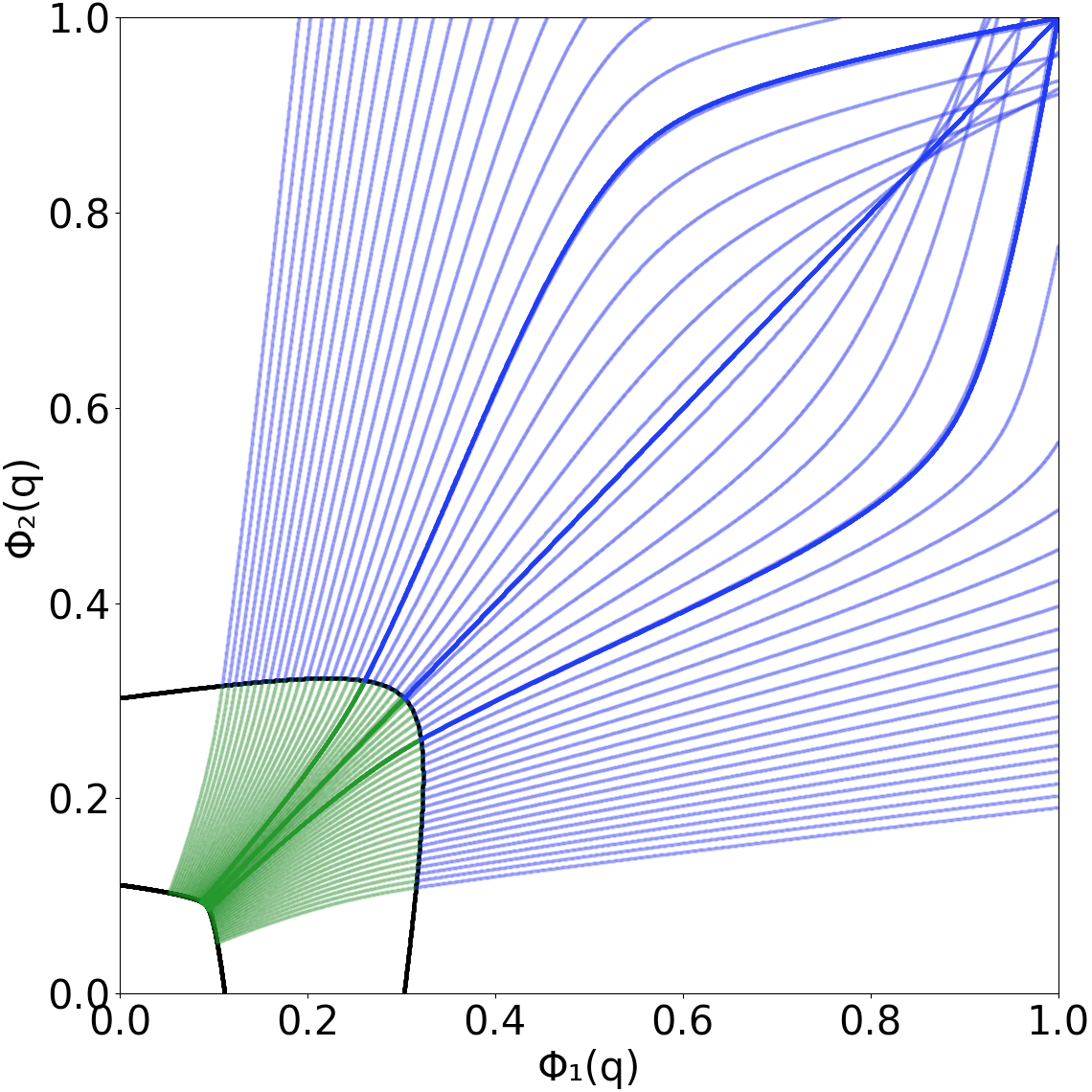}
        \caption{
            $h=1.5$, $a=3$.
            Here $E_0 \approx 7.1755$, $E_1 \approx 7.1767$. 
        }
        \label{subfig:with-field-asb}
    \end{subfigure}
    \hfill
    \begin{subfigure}[b]{.32\textwidth}
        \centering
        \includegraphics[width=.9\linewidth]{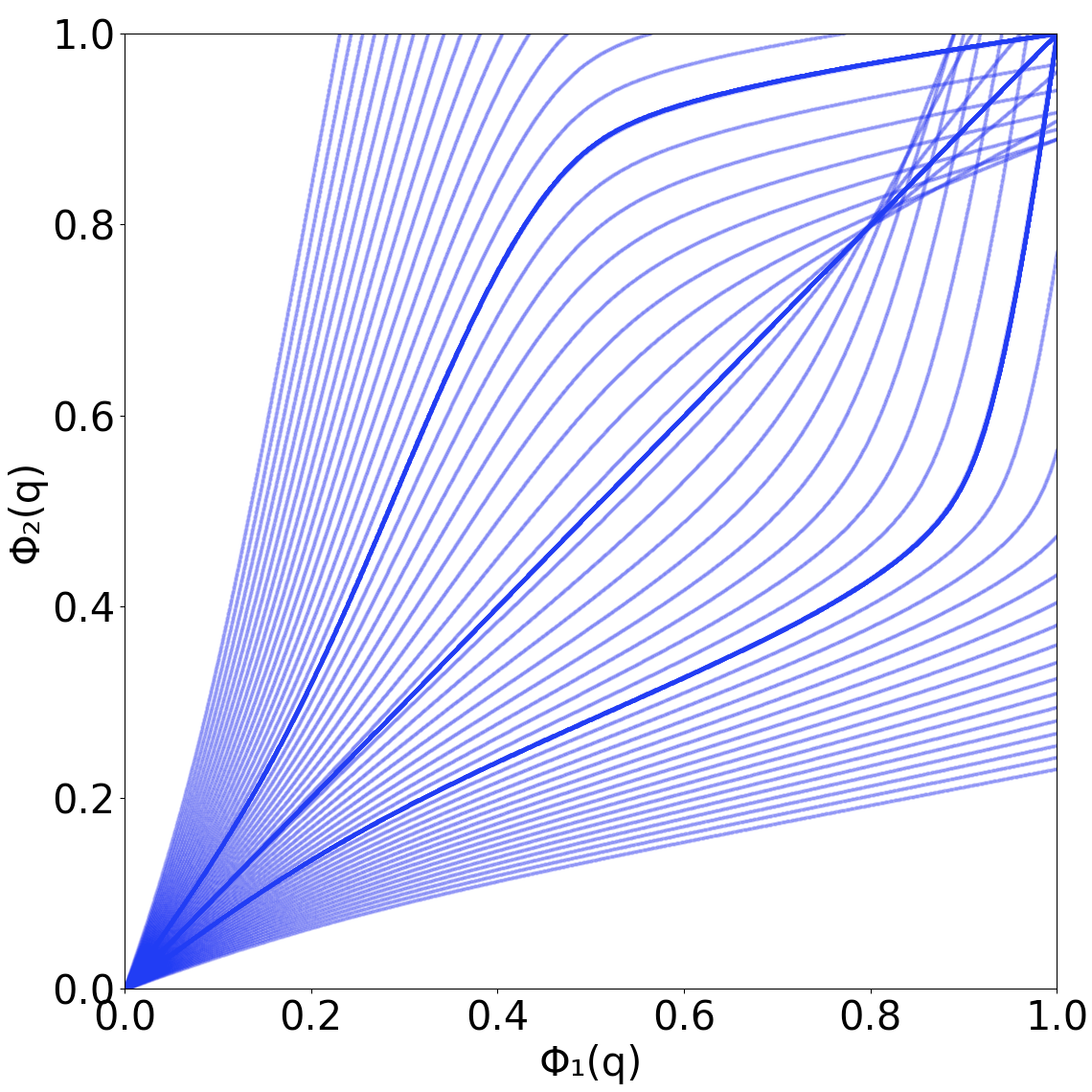}
        \caption{
            $h=0$, $a=3$.
            Here $E_0 \approx 6.9230$, $E_1 \approx 6.9254$.
        }
        \label{subfig:no-field-asb}
    \end{subfigure}
    \hfill
    \begin{subfigure}[b]{.32\textwidth}
        \centering
        \includegraphics[width=.9\linewidth]{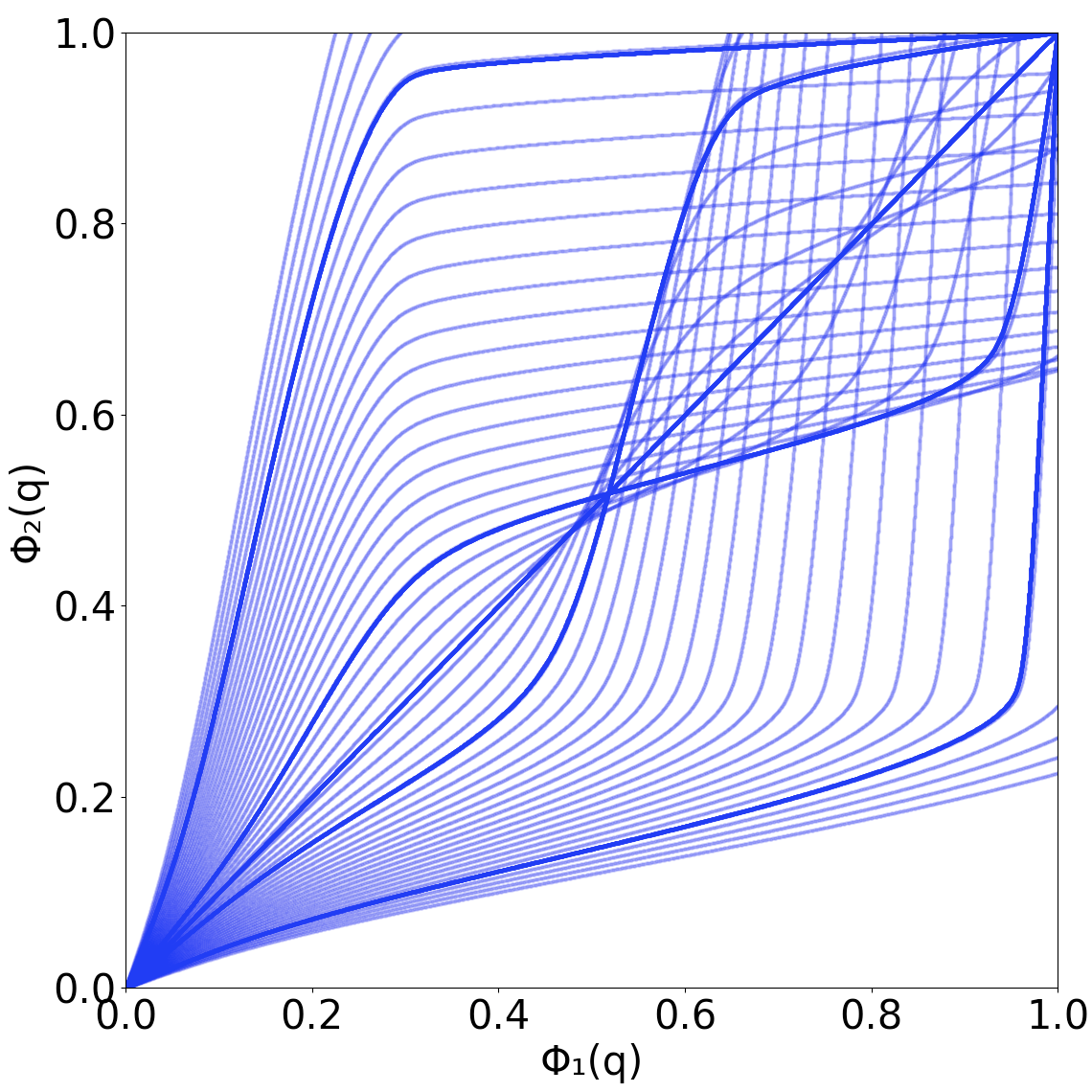}
        \caption{
            $h=0$, $a=5$.
            Here $E_0 \approx 17.0286$, $E_1 \approx 17.0642$, $E_2 \approx 17.0292$.
        }
        \label{subfig:many-asb}
    \end{subfigure}
    \caption{
    \small
        Plots of $\Phi(q)$ with algorithmic symmetry breaking.
        Consider $\vlam = (\fr12,\fr12)$, $\vh = (h,h)$, and $\xi(x_1,x_2) = \nu(a\lambda_1x_1,a\lambda_2x_2)$ for $h,a$ given in the captions above, where $\nu(x_1,x_2) = x_1^2 + x_1x_2 + x_2^2 + x_1^4 + x_2^4$.
        Figure~\ref{subfig:with-field-asb} shows an example with external field (Theorem~\ref{thm:alg-optimizer}(\ref{itm:subsolvable-with-field})), Figure~\ref{subfig:no-field-asb} shows an example without external field (Theorem~\ref{thm:alg-optimizer}(\ref{itm:subsolvable-no-field})), and Figure~\ref{subfig:many-asb} shows an example with several symmetry-breaking trajectories.
        Targeted trajectories are bold and colors have the same meaning as in Figure~\ref{fig:ode}.
        Numerical estimates of the energy $\bbA(p,\Phi; q_0)$ attained by each bold path are given in the captions: $E_0$ is the energy of the diagonal trajectory and $E_k$ is the energy of the asymmetric trajectories that intersect the diagonal $k$ times not including $(0,0)$.
        In all cases the asymmetric trajectories outperform the symmetric trajectory, and in Figure~\ref{subfig:many-asb} the asymmetric trajectories farthest from diagonal perform the best.
    }
    \label{fig:asb}
\end{figure}

Assuming non-degeneracy, we show that algorithmic symmetry breaking does not occur sufficiently close to $\vzero$. 
To make this precise, let $\Delta^r = \{\vv \in \bbR_{\ge0}^\sS : \la \vlam, \vv\ra = 1\}$ denote the simplex of admissible $\Phi'$ vectors.
Then if $\vh=\vzero$, we define a map $F_{t}:\Delta^r\to \Delta^r$ given by
\begin{equation}
\label{eq:F-t}
    F_{t}(\vv)=\Phi(t)/t
\end{equation}
where $\Phi$ is the tree-descending trajectory with endpoint $\Phi(0)=\vzero$, $\Phi'(q)=\vv$. The next proposition shows that $F_{t}$ is injective for small $t$, i.e. algorithmic symmetry breaking is absent sufficiently close to the origin, and is surjective for all $t$.

\begin{proposition}
\label{prop:type-II-locally-unique}
Assume $\xi$ is non-degenerate and $\vh=\vzero$. There exists $\eps>0$ such that the map $F_{t}$ defined in \eqref{eq:F-t} is injective for $t\in (0,\eps]$. Moreover $F_{t}$ is surjective for all $t>0$.
\end{proposition}

\subsection{Branching Overlap Gap Property as a Tight Barrier to Algorithms}
\label{subsec:bogp-summary}

Mean-field spin glasses, including the multi-species models we focus on here, are natural examples of random optimization problems.
Other examples are random constraint satisfaction problems such as random (max)-$k$-SAT and random perceptron models. 
For any such problem, a basic property to understand is the maximum objective that an efficient algorithm can find.

Since the early 2000s, there has been extensive heuristic work in the physics and computer science communities aiming to understand this question in terms of geometric properties of these problems' solution spaces \cite{krzakala2007gibbs, zdeborova2007phase, achlioptas2008phasetransitions}.
The first rigorous link from solution geometry to hardness was obtained by Gamarnik and Sudan \cite{gamarnik2014limits}, in the form of the Overlap Gap Property (OGP).
An OGP argument shows that the absence of a certain geometric constellation in the super-level set $S_E(H_N) = \{\bsig : H_N(\bsig) / N \ge E\}$ implies that suitably stable algorithms cannot find objectives larger than $E$.
The proof is by contradiction, showing that a stable algorithm attaining value $E$ can construct the forbidden constellation.

The value $E$ at which the constellation disappears (and at which hardness is shown) depends on the constellation and does not generally equal the value $\ALG$ found by the best efficient algorithm.
The first OGP works used as the constellation a pair of solutions with medium overlap \cite{gamarnik2014limits, gamarnik2019overlap, chen2019suboptimality, gamarnik2020optimization}.
Subsequent work considered constellations with more points, arranged in a ``star" \cite{rahman2017independent, gamarnik2017performance, gamarnik2021partitioning, gamarnik2022algorithms} or ``ladder" \cite{wein2020independent, bresler2021ksat} configuration; these constellations vanish at smaller $E$, thereby showing hardness closer to $\ALG$.
In particular, \cite{rahman2017independent, wein2020independent} identify the computational threshold of maximum independent set on $G(N,d/N)$ within a $1+o_d(1)$ factor, and \cite{bresler2021ksat} identifies that of random $k$-SAT within a constant factor clause density.
We refer the reader to \cite[Sections 1.2 and 1.3]{huang2021tight} for a more detailed discussion and \cite{gamarnik2021survey} for a survey of OGP.

Our previous work \cite{huang2021tight} introduced the \emph{branching OGP}, where the forbidden constellation is a densely branching ultrametric tree.
For mixed even $p$-spin models, this work showed that this constellation is absent for any $E>\ALG$, and therefore Lipschitz algorithms cannot surpass $\ALG$.
It was further shown that for these models, any ultrametric constellation that is not densely branching is \emph{not} forbidden at all $E>\ALG$, and thus the branching OGP is necessary to show hardness at $\ALG$.
As discussed previously, the hardness proof of \cite{huang2021tight} uses interpolation to upper bound the maximum energy of the ultrametric constellation, and hence does not apply with odd interactions or more generally in multi-species models.

In Section~\ref{sec:uc}, we develop a new method to establish the branching OGP which does not rely on interpolation. 
Instead we recursively apply a uniform concentration idea of Subag \cite{subag2018free} (see Lemma~\ref{lem:unif-main}) to show that among all densely branching ultrametric constellations, the highest energy ones can be constructed \emph{greedily}.
Roughly speaking, in such constellations the children $\bx^1,\ldots,\bx^k$ of a point $\bx$ lie on a small sphere centered at $\bx$ such that the increments $\bx^i-\bx$ are orthogonal to $\bx$ and to each other, and approximately maximize $H_N$ on this set.
Because the aforementioned generalized Subag algorithm traces a root-to-leaf path of this tree, this method automatically finds a matching algorithm and lower bound (again modulo that the greedy algorithm is not clearly Lipschitz; our AMP algorithm in \cite{huang2023optimization} also descends this tree).
In other words, the optimal algorithm can be read off from the proof of the lower bound.

We remark that in the branching OGP (and many previous OGPs) one must actually consider a family of correlated Hamiltonians. 
In the branching OGP the correlation structure of these Hamiltonians is also ultrametric.
The function $p$ in \eqref{eq:alg-functional} enters to parametrize the correlation structure of this Hamiltonian family, see Subsection~\ref{subsec:ultra-corr-H}.

Finally, let us point out that the branching OGP is somewhat of a counterpart to the ultrametricity of low-temperature Gibbs measures mentioned previously. 
One essentially expects that $\ALG=\OPT$ holds whenever the Gibbs measure branches at all depths in a suitable zero-temperature limit, which is a strong form of \emph{full replica symmetry breaking}.
However, in general the true Gibbs measures may not exhibit full RSB and may even have finite combinatorial depth, whereas the algorithmic trees we consider must always branch continuously.

\subsection{Other Related Work}

Following the introduction of mean-field spin glasses in \cite{sherrington1975solvable}, a great deal of effort has been devoted to computing their free energy.
In \cite{parisi1979infinite}, Parisi conjectured the value of the free energy based on his celebrated ultrametric ansatz. 
Following progress by \cite{mezard1985microstructure,ruelle1987mathematical,guerra2002thermodynamic,aizenman2003extended}, the Parisi formula was confirmed by \cite{talagrand2006parisi,talagrand2006spherical,panchenko2013parisi}, and the zero-temperature Parisi formula for the ground state energy by \cite{auffinger2017parisi,chen2017parisi}.
An understanding of the high temperature regime was obtained earlier in \cite{aizenman1987some,comets1995sherrington} and through Talagrand's cavity method \cite{TalagrandVolI}.

Another important line of work is the landscape complexity, i.e. the determination of the exponential growth rate of critical points of $H_N$ at each energy level. Such asymptotics were put forward in \cite{crisanti2003complexity,crisanti2005complexity,parisi2006computing} followed by much rigorous progress in \cite{auffinger2013random,auffinger2013complexity,subag2017complexity,arous2020geometry,mckenna2021complexity,kivimae2021ground,subag2021concentration}. 
The dynamical behavior of spin glasses is also of great interest; as previously mentioned, the behavior of e.g. Langevin dynamics has been described on dimension-free time-scales.
At high temperature, fast mixing has been recently established in
\cite{eldan2021spectral,anari2021entropic,adhikari2022spectral}.

The first multi-species spin glass to be introduced was the bipartite Sherrington-Kirkpatrick model in \cite{kincaid1975phase}. It was later studied further in \cite{korenblit1985spin,fyodorov1987antiferromagnetic,fyodorov1987phase}. While the analogous lower bound to the Parisi formula applies in general with a similar proof \cite{panchenko2015free}, the upper bound is known only in special cases: models where $\xi$ is convex in the positive orthant \cite{barra2015multi,baik2020free}, pure spherical models assuming the $N\to\infty$ limit exists \cite{subag2021tap}, and spherical models for which $\vone$ is super-solvable \cite{huang2023strong}. A different free energy upper bound, in the form of an infinite-dimensional Hamilton-Jacobi equation, was recently proved by Mourrat \cite{mourrat2020free}.

In the large degree limit, the maxima of random constraint satisfaction problems such as max-$k$-SAT and MaxCut are described by Ising mean-field models \cite{dembo2017extremal, panchenko2018k}. See \cite{alaoui2021local,jones2022random} for algorithmic analogs.

\subsection{Notations and Preliminaries}
\label{subsec:notation}

Throughout this paper we adopt the following notational conventions. 
For $\bx \in \bbR^N$, $\bx_s \in \bbR^{\cI_s}$ denotes the restriction of $\bx$ to the coordinates $\cI_s$.
The symbol $\odot$ denotes coordinate-wise product, and the symbol $\diamond$ denotes the operation defined in \eqref{eq:def-diamond}.
The all-$0$ and all-$1$ vectors in $\bbR^\sS$ are denoted $\vzero, \vone$, and those in $\bbR^N$ are denoted $\bzero, \bone$.
For vectors $\vx,\vy \in \bbR^{\sS}$, $\vx \preceq \vy$ denotes the coordinate-wise inequality, and for matrices $\preceq$ denotes the Loewner order. 
Vector operations such as $\sqrt{\vx}$ are always coordinate-wise.

Let $S_N = \{\bx \in \bbR^N : \norm{\bx}_2^2 = N\}$.
For any tensor $\bA \in (\bbR^N)^{\otimes k}$, we define the operator norm
\[
    \tnorm{\bA}_{\op} = 
    \fr1N \max_{\bsig^1,\ldots,\bsig^k \in S_N} 
    \lt|\la \bA, \bsig^1 \otimes \cdots \otimes \bsig^k \ra\rt|.
\]
The following proposition shows that with all but exponentially small probability, the operator norms of all constant-order gradients of $H_N$ are bounded and $O(1)$-Lipschitz.
\begin{proposition}
\label{prop:gradients-bounded}
    For any fixed model $(\xi, \vh)$ there exists a constant $c>0$, sequence $(K_N)_{N\geq 1}$ of convex sets $K_N\subseteq \sH_N$, and sequence of constants $(C_{k})_{k\geq 1}$ independent of $N$, such that the following properties hold.
    \begin{enumerate}[label=(\alph*)]
        \item 
        \label{it:KN-high-prob}
        $\P[H_N\in K_N]\geq 1-e^{-cN}$;
        \item For all $H_N\in K_N$ and 
        $\bx, \by\in \cB_N$,
        \begin{align}
            \label{eq:gradient-bounded}
            \norm{\nabla^k H_N(\bx)}_{\op}
            &\le 
            C_{k}, \\
            \label{eq:gradient-lipschitz}
            \norm{\nabla^k H_N(\bx) - \nabla^k H_N(\by)}_{\op}
            &\le 
            \fr{C_{k+1}}{\sqrt{N}} \norm{\bx - \by}_2.
        \end{align}
    \end{enumerate}
\end{proposition}

\begin{proof}
    Note that the conditions \eqref{eq:gradient-bounded} and \eqref{eq:gradient-lipschitz} are convex in $H_N$. Defining $K_N$ to be the set of $H_N$ such that the estimates \eqref{eq:gradient-bounded}, \eqref{eq:gradient-lipschitz} hold with suitably large implicit constants, it remains to show point~\ref{it:KN-high-prob}.
    For this, by Slepian's lemma it suffices to consider the case where $\gamma_{s_1,\dots,s_k}$ is replaced by the maximal entry in $\Gamma^{(k)}$. The result then follows by \cite[Proposition 2.3]{huang2021tight} since we assumed at the outset that $\sum_{k\ge 2} 2^k \norm{\Gamma^{(k)}}_\infty < \infty$.
\end{proof}

%% file: final-tex/2-bogp.tex
\section{Algorithmic Thresholds from Branching OGP}

We begin this section by recalling some fundamental definitions and constructions from \cite{huang2021tight}. 
We then review the details of the branching overlap gap property introduced in \cite{huang2021tight}, and in particular the link to hardness for overlap concentrated algorithms. 

\subsection{Correlation Functions and Overlap Concentration}
\label{subsec:correlation}

For any $p\in [0,1]$, we may construct two correlated copies $\HNp{1}, \HNp{2}$ of $H_N$ as follows.
Construct three i.i.d. copies $\wtH_N^{[0]}, \wtH_N^{[1]}, \wtH_N^{[2]}$ of $\wtH$ as in \eqref{eq:def-hamiltonian-no-field}.
For $i=1,2$ define
\begin{align*}
    \HNp{i}(\bsig) 
    &= 
    \la \bh, \bsig \ra 
    + \wtHNp{i}(\bsig), \quad \text{where} \\
    \wtHNp{i}(\bsig) 
    &= 
    \sqrt{p} \wtH_N^{[0]}(\bsig) + 
    \sqrt{1-p} \wtH_N^{[i]}(\bsig).
\end{align*}
We say $\HNp{1}, \HNp{2}$ are $p$-correlated.
Note that pairs of corresponding entries in $\bg(\HNp{1})$ and $\bg(\HNp{2})$ are Gaussian with covariance $\lt[\begin{smallmatrix}1 & p \\ p & 1\end{smallmatrix}\rt]$.

Given a function $\cA_N : \sH_N \to \cB_N$ (always assumed to be measurable) define $\vchi:[0,1] \to \bbR^{\sS}$ by
\begin{equation}
    \label{eq:def-correlation-fn}
    \vchi(p) 
    = 
    \E \vR \lt(
        \cA(\HNp{1}), 
        \cA(\HNp{2})
    \rt),
\end{equation}
where $\HNp{1}, \HNp{2}$ are $p$-correlated copies of $H_N$. 
We say that $\vchi$ is the \textbf{correlation function} of $\cA$.
Let $\chi_s$ denote the $s$-coordinate of $\vchi$.

\begin{proposition}
    \label{prop:correlation-fn-properties}
    We have $\vchi \in \bbI(0,1)^\sS$.
\end{proposition}
\begin{proof}
    Identically to \cite[Proposition 3.1]{huang2021tight}, Hermite expanding $R_s \lt(\cA(\HNp{1}), \cA(\HNp{2}) \rt)$ shows that $\chi_s$ is continuous and increasing.
    The same Hermite expansion shows $\chi_s$ is continuously differentiable.
\end{proof}
The other properties of correlation functions proved in \cite[Proposition 3.1]{huang2021tight} also hold, namely that $\chi_s$ is convex and either strictly increasing or constant; however they are not needed in this paper. 

We will determine the maximum energy attained by algorithms $\cA_N : \sH_N \to \cB_N$ obeying the following overlap concentration property. 

\begin{definition}
    \label{defn:oc}
    Let $\eta, \nu > 0$. 
    An algorithm $\cA = \cA_N$ is $(\eta,\nu)$ overlap concentrated if for any $p\in [0,1]$ and $p$-correlated Hamiltonians $\HNp{1}, \HNp{2}$,
    \begin{equation}
        \label{eq:overlap-concentrated}
        \P
        \lt[
            \norm{ 
                \vR \lt( \cA(\HNp{1}), \cA(\HNp{2}) \rt) -
                \vchi(p)
            }_{\infty}
            \ge \eta
        \rt]
        \le \nu.
    \end{equation}
\end{definition}

Our main hardness result is the following bound on the performance of overlap concentrated algorithms.

\begin{theorem}
    \label{thm:main-ogp-oc}
    Consider a multi-species spherical spin glass Hamiltonian $H_N$ with parameters $(\xi,\vh)$. 
    Let $\ALG$ be given by \eqref{eq:alg}.
    For any $\eps > 0$ there are $\eta, c, N_0$ depending only on $\xi, \vh, \eps$ such that the following holds for any $N\ge N_0$ and $\nu \in [0,1]$. 
    For any $(\eta,\nu)$-overlap concentrated $\cA_N : \sH_N \to \cB_N$, 
    \[
        \bbP\lt[H_N(\cA_N(H_N))/N \ge \ALG + \eps\rt]
        \le 
        \exp(-cN) + 
        \nu^c.
    \]
\end{theorem}
By Gaussian concentration of measure (see \cite[Propositon 8.2]{huang2021tight}), any $\tau$-Lipschitz algorithm is $(\eta, e^{-c(\eta,\tau)N})$-overlap concentrated for any $\eta>0$ and appropriate $c(\eta,\tau)>0$.
Thus Theorem~\ref{thm:main-ogp-oc} implies Theorem~\ref{thm:main-ogp}.

\subsection{Ultrametrically Correlated Hamiltonians}
\label{subsec:ultra-corr-H}

Next we define the hierarchically correlated ensemble of Hamiltonians used to define the branching overlap gap property.
Let $k\ge 2$, $D\ge 1$ be positive integers.
For each $0\le d \le D$, let $V_d = [k]^d$ denote the set of length $d$ sequences of elements of $[k]$. 
The set $V_0$ consists of the empty tuple, which we denote $\emptyset$.
Let $\bbT(k,D)$ denote the depth $D$ tree rooted at $\emptyset$ with depth $d$ vertex set $V_d$, where $u\in V_d$ is the parent of $v\in V_{d+1}$ if $u$ is the length $d$ initial substring of $v$.
For nodes $u^1,u^2\in \bbT(k,D)$, let 
\[
    u^1 \wedge u^2
    =
    \max \lt\{
        d \in \bbZ_{\ge 0}: 
        \text{$u^1_{d'} = u^2_{d'}$ for all $1\le d' \le d$}
    \rt\},
\]
where the set on the right-hand side always contains $0$ vacuously.
This is the depth of the least common ancestor of $u^1$ and $u^2$.
Let $\bbL(k,D) = V_D$ denote the set of leaves of $\bbT(k,D)$.
When $k,D$ are clear from context, we denote $\bbT(k,D)$ and $\bbL(k,D)$ by $\bbT$ and $\bbL$.
Finally, let $K = |\bbL| = k^D$.

Let the sequences $\up = (p_0, p_1, \ldots, p_D)\in\bbR^{D+1}$ 
and $\uvphi = (\vphi_0, \vphi_1, \ldots, \vphi_D)\in (\bbR^{\sS})^{D+1}$ 
satisfy
\begin{align*}
    0 = p_0 \le p_1 \le \cdots \le p_D &= 1, \\
    \vzero \preceq \vphi_0 \preceq \vphi_1 \preceq \cdots \preceq \vphi_D &\preceq \vone.
\end{align*}
The sequence $\up$ controls the correlation structure of our ensemble of Hamiltonians while the sequence $\uvphi$ controls the overlap structure of their inputs.
For each $u\in \bbT$, including interior nodes, let $\wtH_N^{[u]}$ be an independent copy of $\wtH_N$ generated by \eqref{eq:def-hamiltonian-no-field}, and let
\begin{equation}
    \label{eq:def-correlated-disorder}
    \wtHNp{u} = 
    \sum_{d=1}^{|u|}
    \sqrt{p_d - p_{d-1}} \cdot
    \wtH_N^{[(u_1,\ldots,u_d)]}
\end{equation}
where $|u|$ is the length of $u$ and $(u_1,\ldots,u_d)$ is the length-$d$ prefix of $u$.
For $u\in \bbL$, define
\[
    \HNp{u}(\bsig) = 
    \la \bh, \bsig \ra + 
    \wtHNp{u}(\bsig).
\]
This constructs a Hamiltonian ensemble $(\HNp{u})_{u\in \bbL}$ where each $\HNp{u}$ is marginally distributed as $H_N$ and each pair of Hamiltonians $\HNp{u^1}, \HNp{u^2}$ is $p_{u^1\wedge u^2}$-correlated.
We define a grand Hamiltonian on states
\[
    \ubsig = (\bsig(u))_{u\in \bbL} \in (\bbR^N)^\bbL.
\]
by
\begin{equation}
    \label{eq:grand-hamiltonian}
    \cH_N^{k,D,\up}(\ubsig) 
    =
    \fr{1}{K}
    \sum_{u\in \bbL}
    \HNp{u}(\bsig(u)).
\end{equation}
We denote this by $\cH_N$ when $k,D,\up$ are clear from context.
Note that we have thus far not used the definition of $\wtHNp{u}$ for interior nodes $u\in \bbT \setminus \bbL$; these Hamiltonians will be useful in our analysis of the branching OGP threshold in Section~\ref{sec:uc}.
The branching OGP is defined by a maximization of $\cH_N$ over the overlap-constrained set
\begin{equation}
\label{eq:cQ}
    \cQ^{k,D,\uvphi}(\eta)
    =
    \lt\{
        \ubsig \in \cB_N^\bbL : 
        \norm{\vR(\bsig(u^1),\bsig(u^2)) - \vphi_{u^1\wedge u^2}}_\infty 
        \le \eta,
        ~\forall u^1,u^2 \in \bbL
    \rt\}.
\end{equation}
We denote this set $\cQ(\eta)$ when $k,D,\uvphi$ are clear from context.

\subsection{The Branching OGP Threshold}

We will show that overlap concentrated algorithms cannot outperform a \emph{branching OGP} energy $\BOGP$ defined as the ground state energy of the grand Hamiltonian \eqref{eq:grand-hamiltonian} in the limit of ``continuously branching" ultrametrics.

\begin{definition}[Branching OGP energy]
    \label{defn:bogp}
    The energy $\BOGP = \BOGP(\xi,\vh)$ is the infimum of energies $E$ such that the following holds.
    Choose sufficiently large $D$, followed by small $\eta$ and then large $k$. For any $\vchi \in \bbI(0,1)^\sS$ there exists $\up$ such that for $\uvphi = \vchi(\up)$ element-wise (i.e. $\vphi_d=\vchi(p_d)$),
    \begin{equation}
        \label{eq:bogp}
        \limsup_{N\to\infty}
        \fr{1}{N} 
        \bbE \sup_{\ubsig \in \cQ(\eta)}
        \cH_N(\ubsig)
        \le 
        E.
    \end{equation}
    More explicitly,
        \begin{equation}
        \label{eq:bogp-explicit}
        \BOGP(\xi,\vh)
        \equiv
        \lim_{D\to\infty}
        \lim_{\eta\to 0}
        \lim_{k\to\infty}
        \sup_{\vchi \in \bbI(0,1)^\sS}
        \inf_{\uvphi=\vchi(\up)}
        \limsup_{N\to\infty}
        \fr{1}{N} 
        \bbE \sup_{\ubsig \in \cQ^{k,D,\uvphi}(\eta)}
        \cH_N^{k,D,\up}(\ubsig).
    \end{equation}
\end{definition}
Our previous work \cite{huang2021tight} implicitly considered the same quantity. Note that the limits in $(D,k,\eta)$ are decreasing, so they could actually be taken in any order (and moreover the limiting value $\BOGP$ exists apriori). Additionally the role of the infimum over $(\uvphi,\up)$ is quite simple: the only important thing is to ensure both sequences increase in uniformly small steps (see Definition~\ref{def:delta-dense}).

Section~\ref{sec:uc} proves the following proposition identifying $\BOGP$ with the formula \eqref{eq:alg} for $\ALG$.
\begin{proposition}
    \label{prop:bogp-alg}
    For all $(\xi,\vh)$, we have $\BOGP = \ALG$.
\end{proposition}
Let us first prove Theorem~\ref{thm:main-ogp-oc} assuming Proposition~\ref{prop:bogp-alg}.
Let $\eps > 0$ be arbitrary and $k,D,\eta$ be given by Definition~\ref{defn:bogp} for $E = \ALG + \eps/4$. 
Let $\cA = \cA_N : \sH_N \to \cB_N$ be a $(\eta,\nu)$-overlap concentrated algorithm with correlation function $\vchi$.
Let $\up$ and $\uvphi$ be given by Definition~\ref{defn:bogp} (depending on $\vchi$). 
Since $\BOGP = \ALG$ by Proposition~\ref{prop:bogp-alg}, for sufficiently large $N$
\[
    \fr{1}{N} 
    \bbE \sup_{\ubsig \in \cQ(\eta)}
    \cH_N(\ubsig)
    \le 
    \ALG
    + \eps/2.
\]
Let 
\[
    \alpha_N = \bbP\lt[
        H_N(\cA(H_N)) \ge \ALG + \eps      
    \rt].
\]
Let $\bsig(u) = \cA(\HNp{u})$ and $\ubsig = (\bsig(u))_{u\in \bbL}$. 
Define the events
\begin{equation}
\label{eq:S-events}
\begin{aligned}
    \Ssolve &= \lt\{\HNp{u}(\bsig(u)) / N \ge \ALG + \eps ~\forall u\in \bbL\rt\}, \\
    \Soverlap &= \lt\{\ubsig \in \cQ(\eta)\rt\}, \\
    \Sogp &= \lt\{ \sup_{\ubsig \in \cQ(\eta)} \cH_N(\ubsig) / N < \ALG + \eps \rt\}.
\end{aligned}
\end{equation}
\begin{proposition}
    \label{prop:prob-ineqs}
    The following inequalities hold. 
    \begin{enumerate}[label=(\alph*), ref=\alph*]
        \item \label{itm:ssolve} $\bbP(\Ssolve) \ge \alpha_N^K$.
        \item \label{itm:soverlap} 
        $\bbP(\Soverlap) \ge 1 - K^2\nu$.
        \item \label{itm:sogp} $\bbP(\Sogp) \ge 1 - 2 \exp(-cN)$ for suitable $c=c(\eps) > 0$.
    \end{enumerate}
\end{proposition}
\begin{proof}[Proof of (\ref{itm:ssolve})]
    Use Jensen's inequality $D$ times as in \cite[Proof of Proposition 3.6(a)]{huang2021tight}. 
\end{proof}
\begin{proof}[Proof of (\ref{itm:soverlap})]
    For each $u^1,u^2\in\bbL$, $\bbE \vR(\bsig(u^1),\bsig(u^2)) = \vchi(p_{u^1\wedge u^2}) = \vphi_{u^1\wedge u^2}$.
    So,
    \[
        \bbP\lt[
            \norm{\vR(\bsig(u^1),\bsig(u^2)) - \vphi_{u^1\wedge u^2}}_\infty \le \eta
        \rt]
        \ge 1-\nu.
    \]
    The result follows by a union bound on $u^1,u^2$.
\end{proof}
\begin{proof}[Proof of (\ref{itm:sogp})]
    Use the Borell-TIS inequality on the random variable $Y = \fr{1}{N} \sup_{\ubsig \in \cQ(\eta)} \cH_N(\ubsig)$, as in \cite[Proof of Proposition 3.6(d)]{huang2021tight}.
\end{proof}

\begin{proof}[Proof of Theorem~\ref{thm:main-ogp-oc}]
    Note that $\Ssolve \cap \Soverlap \cap \Sogp = \emptyset$.
    So, $\bbP(\Ssolve) + \bbP(\Soverlap) + \bbP(\Sogp) \le 2$.
    The bounds in Proposition~\ref{prop:prob-ineqs} imply
    \[
        \alpha_N^K \le 2\exp(-cN) + K^2\nu
    \]
    By adjusting the constant $c$,
    \[
        \alpha_N 
        \le 
        \exp(-cN) + \nu^c.
    \]
\end{proof}

\subsection{An Alternate Definition for the $\BOGP$ Threshold}
\label{subsec:alt-bogp}

The overlap-constrained input set $\cQ(\eta)$ used to define $\BOGP$ was designed to capture the properties of $\ubsig=(\cA(H_N^{(u)}))_{u\in\bbL}$.
In this set, overlap constraints are enforced \emph{globally}, between each pair of states, and the constraints are \emph{approximate}, within a tolerance $\eta > 0$.

In this subsection, we define a variant $\BOGP_{\loc,0}$ of $\BOGP$, based on an input set $\cQ_{\loc}(0)$, in which overlap constraints are enforced \emph{locally}, between only adjacent and sibling nodes in $\bbT$, and the constraints are \emph{exact}.
We also enforce that the sequences $p_d$, $\vphi_d$ increase in small steps.
To define the local constraints, we introduce the extended states
\[
    \ubrho = (\brho(u))_{u\in \bbT} \in \cB_N^{\bbT}
\]
whose indices now also include interior $u\in \bbT$. 
For $u,v\in \bbT$, let $u\sim v$ indicate that $u=v$, or one of $u,v$ is the parent of the other, or $u,v$ are siblings.
Define
\begin{align*}
    \cQ_{\loc+}^{k,D,\uvphi}(\eta)
    &= \lt\{
        \ubrho \in \cB_N^\bbT : 
        \norm{\vR(\brho(u), \brho(v)) - \vphi_{u\wedge v}}_\infty \le \eta,~\forall u\sim v
    \rt\} \\
    \cQ_{\loc}^{k,D,\uvphi}(\eta)
    &= \lt\{
        \ubsig \in \cB_N^\bbL : 
        \exists \ubrho \in \cQ_{\loc+}^{k,D,\uvphi}(\eta)~\text{such that}~(\brho(u))_{u\in \bbL} = \ubsig
    \rt\}.
\end{align*}
We similarly omit the superscript $k,D,\uvphi$ when this is clear from context.
The following definition captures the property that $p_d$, $\vphi_d$ increase in small steps. 
\begin{definition}
    \label{def:delta-dense}
    The pair of sequences $(\up,\uvphi)$ is \textbf{$\delta$-dense} if $p_d-p_{d-1} \le \delta$ and $\vphi_d - \vphi_{d-1} \preceq \delta \vone$ for all $d$.
\end{definition}
The following technical condition ensures continuous dependence of orthogonal bands on their centers. 
\begin{definition}
    \label{def:separated}
    The function $\vchi \in \bbI(0,1)^\sS$ is \textbf{$\delta$-separated} if $\vchi(0) \succeq \delta \vone$.
\end{definition}
Define
\begin{equation}
    \label{eq:bogp-loc}
        \BOGP_{\loc,0}
        =
        \lim_{D\to\infty}
        \lim_{k\to\infty}
        \sup_{\substack{
            \vchi \in \bbI(0,1)^\sS \\ 
            \text{$1/D^2$-separated}
        }}
        \inf_{\substack{
            \uvphi=\vchi(\up) \\ 
            \text{$6r/D$-dense}
        }}
        \limsup_{N\to\infty}
        \fr{1}{N} 
        \bbE \sup_{\ubsig \in \cQ_{\loc}(0)}
        \cH_N(\ubsig).
\end{equation}
Note that the limit in $D$ is no longer obviously decreasing, so the existence of this limit also needs to be proven. 

The following proposition, which we prove in Appendix~\ref{sec:equivalence-of-bogps}, shows that $\BOGP_{\loc,0}$ is an equivalent characterization of $\BOGP$.
This characterization will be more convenient for the proof of Proposition~\ref{prop:bogp-alg} carried out in the next section. 
We note that in the proof we define several more variants of $\BOGP$ and show all are equal, and it also follows that the average in the definition \eqref{eq:grand-hamiltonian} of $\cH_N$ can be replaced by a minimum with no change. This illustrates some flexibility in using the branching OGP.

\begin{proposition}
    \label{prop:bogp-equivalent}
    The limit $\BOGP_{\loc,0}$ exists and $\BOGP=\BOGP_{\loc,0}$.
\end{proposition}

Finally we record two useful facts.
\begin{lemma}
    \label{lem:loc0-barycenter}
    If $\ubrho \in \cQ_{\loc,+}(0)$ and $\bar \brho = \fr{1}{K} \sum_{u\in \bbL} \brho(u)$, then $\fr{1}{\sqrt{N}} \norm{\brho(\emptyset) - \bar \brho}_2 \le \sqrt{D/k}$.
\end{lemma}
\begin{proof}
    Define $\ubtau \in (\bbR^N)^\bbT$ by $\btau(u) = \brho(u)$ for $u\in \bbL$ and otherwise recursively $\btau(u) = \fr{1}{k} \sum_{i=1}^k \btau(ui)$. 
    By bilinearity of $\vR$, for all $u\in \bbT \setminus \bbL$ with $|u|=d$,
    \[
        \vR\lt(
            \brho(u) - \fr1k \sum_{i=1}^k \brho(ui),
            \brho(u) - \fr1k \sum_{i=1}^k \brho(ui)
        \rt)
        = \fr1k (\vphi_{d+1} - \vphi_d),
    \]
    so
    \[
        \fr{1}{\sqrt{N}} \norm{\brho(u) - \fr1k \sum_{i=1}^k \brho(ui)}_2 
        = \sqrt{\fr{q_{d+1}-q_d}{k}},
    \]
    where $q_d = \la \vlam, \vphi_d \ra$.
    It is easy to see by induction on $d$ that 
    \begin{align*}
        \fr{1}{\sqrt{N}} \norm{\brho(u) - \btau(u)}_2 
        &\le \fr{1}{\sqrt{N}} \norm{\brho(u) - \fr1k \sum_{i=1}^k \brho(ui)}_2 
        + \fr{1}{k} \sum_{i=1}^k \fr{1}{\sqrt{N}} \norm{\brho(ui) - \btau(ui)}_2 \\
        &\le \sum_{\ell=d}^{D-1} \sqrt{\fr{q_{\ell+1}-q_\ell}{k}}.
    \end{align*}
    Since $\bar \brho = \btau(\emptyset)$, 
    \[
        \fr{1}{\sqrt{N}} \norm{\brho(\emptyset) - \bar \brho}_2
        \le \sum_{d=0}^{D-1} \sqrt{\fr{q_{d+1}-q_d}{k}}
        \le \sqrt{\fr{D}{k}}
    \]
    by Cauchy-Schwarz.
\end{proof}

\begin{lemma}
    \label{lem:bogp-subgaussian}
    For any $S \subseteq \cB_N^{\bbL}$, $\fr{1}{N} \sup_{\ubsig \in S} \cH_N(\ubsig)$ is $O(N^{-1/2})$-subgaussian, in particular
    \[
    \bbP\lt[
    \lt|
    \sup_{\ubsig \in S} \cH_N(\ubsig)
    -
    \bbE[
    \sup_{\ubsig \in S} \cH_N(\ubsig)
    ]
    \rt|
    \geq tN^{1/2}
    \rt]
    \leq 
    Ce^{-t^2/C}
    \]
    for a constant $C$ and all $t\geq 0$.
\end{lemma}
\begin{proof}
    We calculate identically to \cite[Proof of Proposition 3.6(d)]{huang2021tight} that for any fixed $\ubsig \in \cB_N^\bbL$, $\Var \cH_N(\ubsig) = O(N)$.
    The result follows from the Borell-TIS inequality, whose statement and proof hold for noncentered Gaussian processes with no modification.
\end{proof}

%% file: final-tex/3-uc.tex
\section{Branching OGP from Uniform Concentration}
\label{sec:uc}

We now turn to the proof of Proposition~\ref{prop:bogp-alg}. 
In light of Proposition~\ref{prop:bogp-equivalent}, it suffices to prove $\BOGP_{\loc,0}=\ALG$.
We begin with a very general argument that due to the ``many orthogonal increments'' property at each layer of the branching tree, it suffices to consider ``greedy'' embeddings in some sense. 
This argument is essentially elementary and relies on an idea of Subag \cite{subag2018free} applied recursively down the tree.

\subsection{Uniform Concentration}

For $\vx \in [0,1]^\sS$, define the product of spheres 
\[
    \cS_N(\vx) 
    = \lt\{
        \bsig \in \bbR^N : 
        \vR(\bsig,\bsig) = \vx
    \rt\}
    = \lt\{
        \bsig \in \bbR^N : 
        \norm{\bsig_s}_2^2 = \lambda_s x_s N
        ~\forall s\in \sS
    \rt\}.
\]
For $\bsig^0 \in \cS_N(\vx)$ and $\vy \succeq \vx$, define 
\begin{equation}
    \label{eq:band-defn}
    B(\bsig^0, \vy, k)
    =
    \lt\{\begin{array}{l}
        \ubsig =
        (\bsig^1,\bsig^2,\dots,\bsig^k) \in \cS_N(\vy)^k: 
        \\
        \vR(\bsig^i-\bsig^0,\bsig^0)
        = \vR(\bsig^i-\bsig^0,\bsig^j-\bsig^0)
        = \vzero
        \quad \forall i,j\in [k], i\neq j
    \end{array}\rt\}.
\end{equation}
Let $0\le \pminus < \pplus \le 1$. 
Generate $k+1$ i.i.d. copies $\hH_N^{[0]}, \hH_N^{[1]}, \ldots, \hH_N^{[k]}$ of $\wtH_N$ as in \eqref{eq:def-hamiltonian-no-field}.
Set 
\begin{align}
    \label{eq:def-hH(0)}
    \hH_N^{(0)}(\bsig) &= \sqrt{\pminus} \hH_N^{[0]}(\bsig) \quad \text{and} \\
    \label{eq:def-hH(i)}
    \hH_N^{(i)}(\bsig) &= \sqrt{\pminus} \hH_N^{[0]}(\bsig) + \sqrt{\pplus-\pminus} \hH_N^{[i]}(\bsig),\quad 1\le i\le k.
\end{align}
Define
\[
    F_{\pminus,\pplus}(\bsig^0,\vy, k)
    = 
    \fr{1}{kN}
    \max_{\ubsig \in B(\bsig^0, \vy, k)}
    \sum_{i=1}^k
    \lt(\hH_N^{(i)}(\bsig^i) - \hH_N^{(0)}(\bsig^0)\rt).
\]

\begin{lemma}
    \label{lem:F-lip}
    There exists $C$ such that the following holds. 
    Suppose that $\delta \vone \preceq \vx \preceq \vy \preceq \vone$ and $\bsig^0, \brho^0 \in \cS_N(\vx)$ satisfy $\norm{\bsig^0-\brho^0}_2 \le \iota\sqrt{N}$.
    If $\hH_N^{[0]}, \ldots, \hH_N^{[k]} \in K_N$ for the event $K_N$ in Proposition~\ref{prop:gradients-bounded}, then
    \begin{equation}
        \label{eq:F-lip}
        |F_{\pminus,\pplus}(\bsig^0,\vy,k) - F_{\pminus,\pplus}(\brho^0,\vy,k)|
        \leq
        \frac{C\iota}{\sqrt{\delta}}. 
    \end{equation}
\end{lemma}

\begin{proof}
    Let $T:\bbR^N\to\bbR^N$ be a product of rotation maps in the $r$ factors $\bbR^{\cI_s}$ such that $T(\bsig^0)=\brho^0$. Then
    \[
        T\lt(B(\bsig^0, \vy, k)\rt)
        = B(T(\bsig^0), \vy, k)
        = B(\brho^0, \vy, k).
    \]
    In particular, we take $T$ to be obtained using geodesic rotations from each $\bsig^0_s$\ to $\brho^0_s$.
    Thus, if $\ubsig \in B(\bsig^0,\vy,k)$ and $\ubrho = (\brho^1,\ldots,\brho^k) \in B(\brho^0,\vy,k)$ for $\brho^i = T\bsig^i$, then for all $i\in [k]$
    \[
        \fr{\norm{\brho^i-\bsig^i}_2}{\norm{\bsig^i}_2}
        \le \fr{\norm{\brho^0-\bsig^0}_2}{\norm{\bsig^0}_2}
        \le \fr{\iota}{\sqrt{\delta}},
    \]
    so $\fr{1}{\sqrt{N}} \norm{\brho^i-\bsig^i}_2 \le \iota / \sqrt{\delta}$.
    On the event $\hH_N^{[0]}, \ldots, \hH_N^{[k]} \in K_N$, it follows that
    \[
        \lt|\hH_N^{(i)}(\bsig^i) - \hH_N^{(i)}(\brho^i)\rt|
        \le \fr{C\iota}{\sqrt{\delta}}
    \]
    for $1\le i\le k$ and 
    \[
        \lt|\hH_N^{(0)}(\bsig^0) - \hH_N^{(0)}(\brho^0)\rt|
        \le C\iota,
    \]
    which implies the conclusion (after adjusting $C$).
\end{proof}

\begin{lemma}
\label{lem:unif-main}
    There exist constants $c,C>0$ such that for all $k\in \bbN$ and $\delta,\eps > 0$ the following holds. For any $\vx,\vy$ satisfying $\delta \vone \preceq\vx \preceq \vy$,
    \begin{align*}
        &\bbP\lt(
            \sup_{\bsig^0 \in \cS_N(\vx)}
            |F_{\pminus,\pplus}(\bsig^0, \vy, k) - \bbE F_{\pminus,\pplus}(\bsig^0, \vy, k)|
            \le \eps
        \rt) \\
        &\qquad 
        \ge 
        1 - 
        \exp\lt(
            C\log\lt(\frac{1}{\delta\eps}\rt) N -
            ck\eps^2 N
        \rt)
        - e^{-cN}
    \end{align*}
\end{lemma}
\begin{proof}
    Fix for now $\bsig^0 \in \cS_N(\vx)$ and $\ubsig = (\bsig^1, \ldots, \bsig^k) \in B(\bsig^0, \vy, k)$. 
    Using the definition \eqref{eq:band-defn} in the final step, we find that for small $c>0$,
    \begin{align*}
        &\bbE \lt[\lt(
            \sum_{i=1}^k
            (\hH_N^{(i)}(\bsig^i) - \hH^{(0)}(\bsig^0))
        \rt)^2\rt] \\
        &= 
        \bbE \lt[\lt(
            \sum_{i=1}^k
            \sqrt{\pminus} (\hH_N^{[0]}(\bsig^i) - \hH_N^{[0]}(\bsig^0))
            + 
            \sqrt{\pplus-\pminus} \hH_N^{[i]} (\bsig^i)
        \rt)^2\rt] \\
        &=
        \pminus 
        \sum_{i,j=1}^k
        \bbE\lt[
            (\hH_N^{[0]}(\bsig^i) - \hH_N^{[0]}(\bsig^0))
            (\hH_N^{[0]}(\bsig^j) - \hH_N^{[0]}(\bsig^0))
        \rt]
        +
        (\pplus-\pminus)
        \sum_{i=1}^k
        \bbE\lt[
            \hH_N^{[i]} (\bsig^i)^2
        \rt]
        \\
        &=
        \pminus
        \sum_{i,j=1}^k
        \xi(\vR(\bsig^i,\bsig^j))
        -
        \xi(\vR(\bsig^i,\bsig^0))
        -
        \xi(\vR(\bsig^0,\bsig^j))
        +
        \xi(\vR(\bsig^0,\bsig^0))
        +
        (\pplus-\pminus) \sum_{i=1}^k \xi(\vR(\bsig^i,\bsig^i))
        \\
        &\le
        \frac{k}{8c}.
    \end{align*}
    By the Borell-TIS inequality, for each fixed $\bsig^0 \in \cS_N(\vx)$
    \begin{equation}
        \label{eq:concentration-one-sigma}
        \bbP\lt[
            |F_{\pminus,\pplus}(\bsig^0,\vy,k) 
            -\bbE F_{\pminus,\pplus}(\bsig^0,\vy,k)|
            \le \eps/2
        \rt]
        \ge 
        1 - 
        2\exp\lt(-ck\eps^2 N\rt).
    \end{equation}
    Choose $\iota = \Theta(\eps \sqrt{\delta})$ so that the right-hand side of \eqref{eq:F-lip} is bounded by $\eps/2$, and let $\cN$ be an $\iota \sqrt{N}$-net of $\cS_N(\vx)$ with size $|\cN|\le (1/(\delta\eps))^{CN}$.
    Define the events 
    \begin{align*}
        S_{\mathrm{conc}}
        &= 
        \lt\{
            \,|F_{\pminus,\pplus}(\brho^0, \vy, k) - \bbE F_{\pminus,\pplus}(\brho^0, \vy, k)|
            \le \eps/2
            ~~\forall~
            \brho^0 \in \cN
        \rt\}, \\
        S_{\mathrm{lip}}
        &= 
        \lt\{
            \,\hH_N^{[0]},\ldots,\hH_N^{[k]} \in K_N
        \rt\},
    \end{align*}
    where $K_N$ is defined in Proposition~\ref{prop:gradients-bounded}.
    By a union bound (after adjusting $c,C$), 
    \begin{equation}
        \label{eq:concentration-many-sigma}
        \bbP\lt(
            S_{\mathrm{conc}}
            \cap
            S_{\mathrm{lip}}
        \rt)
        \ge 
        1 - 
        \exp\lt(C\log \lt(\fr{1}{\delta\eps}\rt) N - ck\eps^2 N\rt)-e^{-cN}.
    \end{equation}
    Suppose $S_{\mathrm{conc}} \cap S_{\mathrm{lip}}$ holds.
    For any $\bsig^0 \in \cS_N(\vx)$, there exists $\brho^0 \in \cN$ such that $\tnorm{\bsig^0-\brho^0}_2 \le \iota \sqrt{N}$, and so
    \begin{align*}
        &|F_{\pminus,\pplus}(\bsig^0, \vy, k) - \bbE F_{\pminus,\pplus}(\bsig^0, \vy, k)| \\
        &\le |F_{\pminus,\pplus}(\bsig^0, \vy, k) - F_{\pminus,\pplus}(\brho^0, \vy, k)| 
        + |F_{\pminus,\pplus}(\brho^0, \vy, k) - \bbE F_{\pminus,\pplus}(\brho^0, \vy, k)|
        \le \fr{\eps}{2} + \fr{\eps}{2} = \eps.
    \end{align*}
\end{proof}

For now, let $k, D, \uvphi, \up$ (recall Definition~\ref{defn:bogp}) be arbitrary.
In Proposition~\ref{prop:uc-bogp} below, we obtain an estimate for
$\fr{1}{N} \bbE \max_{\ubsig\in \cQ_{\loc}(0)}\cH_N(\ubsig)$ by applying Lemma~\ref{lem:unif-main} repeatedly at each internal vertex $u\in \bbT\backslash \bbL$.
This maximum will take the form of an abstract sum of energy increments.
In the next subsection we will take a continuum limit of this bound, which will yield the variational formula \eqref{eq:alg} for $\ALG$ and prove Proposition~\ref{prop:bogp-alg}.

Spherical symmetry implies that $\bbE F_{\pminus,\pplus}\lt(\bsig,\vy,k\rt)$ depends on $\bsig$ only through $\vR(\bsig,\bsig)$. Hence for $\vphiminus=\vR(\bsig,\bsig)$ we may define
\begin{equation}
    \label{eq:f}
    f(\vphiminus, \vphiplus; \pminus, \pplus; k)
    =
    \bbE 
    F_{\pminus,\pplus}\lt(\bsig,\vphiplus,k\rt).
\end{equation}

\begin{proposition}
    \label{prop:uc-bogp}
    Fix $D \in \bbN$ and $\eps, \delta > 0$. 
    Suppose that $\vphi_0 \succeq \delta \vone$. 
    There exists $k_0 = k_0(D,\eps,\delta)$ such that for all $k\ge k_0$, there exists $c = c(D,\eps,\delta,k)$ such that
    \[
        \bbP\lt[\lt|
            \fr1N
            \sup_{\ubsig \in \cQ_{\loc}(0)}
            \cH_N(\ubsig)
            - \lt(
                \sum_{s\in \sS}  h_s\lambda_s \sqrt{\phi_0^s} +
                \sum_{d=0}^{D-1}
                f\lt(\vphi_d, \vphi_{d+1}; p_d, p_{d+1}; k\rt)
            \rt)
        \rt|\le 2D\eps\rt]
        \ge
        1-e^{-cN}.
    \]
\end{proposition}

\begin{proof}
    Let $C,c$ be as in Lemma~\ref{lem:unif-main}, and $k_0$ large enough that
    \begin{align}
        \label{eq:uc-exp-rate}
        C \log \lt(\fr{1}{\delta\eps}\rt) - ck_0\eps^2
        &\le -c, \\
        \label{eq:uc-barycenter-close}
        \tnorm{\vh}_\infty / \sqrt{k_0} &\le \eps.
    \end{align}
    Recall the construction of $\wtHNp{u}$ from \eqref{eq:def-correlated-disorder}.
    For any $u\in V_d$, $0\le d\le D-1$, let $\cE_u$ denote the event in Lemma~\ref{lem:unif-main}, with $(\pminus,\pplus) = (p_d, p_{d+1})$, $(\vx,\vy) = (\vphi_d,\vphi_{d+1})$, and
    \begin{equation}
        \label{eq:uc-hamiltonian-choice}
        (\hH_N^{(0)},\hH_N^{(1)},\ldots,\hH_N^{(k)})
        = \lt(\wtHNp{u},\wtHNp{u1},\ldots,\wtHNp{uk}\rt).
    \end{equation}
    Let $\cE = \bigcap_{u\in \bbT \setminus \bbL} \cE_u$. 
    Lemma~\ref{lem:unif-main} and equation \eqref{eq:uc-exp-rate} imply $\bbP(\cE^u) \ge 1-2e^{-cN}$ for all $u\in \bbL$. 
    By a union bound, $\bbP(\cE) \ge 1 - e^{-cN}$ (after adjusting $c$).
    
    Denote by $F^u_{p_d,p_{d+1}}$ the function $F_{p_d,p_{d+1}}$ defined with Hamiltonians \eqref{eq:uc-hamiltonian-choice}.
    Let $\ubsig \in \cQ_{\loc}(0)$, so there exists $\ubrho \in \cQ_{\loc+}(0)$ with $(\brho(u))_{u\in \bbL} = \ubsig$. 
    On the event $\cE$,
    \begin{align*}
        \fr{1}{N} \cH_N(\ubsig) 
        - \fr{1}{KN} \sum_{v\in \bbL} \la \bh, \bsig(u) \ra
        &=
        \fr{1}{KN}
        \sum_{u\in \bbL} \wtHNp{u}(\bsig(u)) \\
        &=
        \sum_{d=0}^{D-1}
        \fr{1}{k^{d}}
        \sum_{u\in V_d} 
        \fr{1}{kN}
        \sum_{i=1}^{k}
        \lt( 
            \wtHNp{ui}(\brho(ui))-\wtHNp^{u}(\brho(u))
        \rt) \\
        &\le
        \sum_{d=0}^{D-1}
        \fr{1}{k^{d}}
        \sum_{u\in V_d} 
        F_{p_d,p_{d+1}}^{u}
        \lt(\brho(u),\vphi_{d+1},k\rt) \\
        &\stackrel{Lem.~\ref{lem:unif-main}}{\le}
        D\eps
        + \sum_{d=0}^{D-1}
        f(\vphi_d, \vphi_{d+1}; p_d, p_{d+1}; k).
    \end{align*}
    In the telescoping sum, we used that $\wtH_N^{(\emptyset)}$ is the zero function. 
    By Lemma~\ref{lem:loc0-barycenter} and equation \eqref{eq:uc-barycenter-close},
    \begin{align*}
        \lt|
            \fr{1}{KN} \sum_{v\in \bbL} \la \bh, \bsig(u) \ra
            - \fr{1}{N} \la \bh, \brho(\emptyset) \ra
        \rt|
        &\le 
        \fr{1}{\sqrt{N}} \tnorm{\bh}_2 \cdot 
        \fr{1}{\sqrt{N}} \norm{\brho(\emptyset) - \fr{1}{K} \sum_{u\in \bbL} \bsig(u)}_2 \\
        &\le 
        \tnorm{\vh}_\infty \sqrt{\fr{D}{k}} 
        \le D\eps.
    \end{align*}
    Finally,
    \begin{equation}
        \label{eq:root-energy}
        \fr{1}{N} \la \bh, \brho(\emptyset) \ra
        = \fr{1}{N} \sum_{s\in \sS} h_s \norm{\brho(\emptyset)_s}_1 
        \le \fr{1}{N} \sum_{s\in \sS} h_s \sqrt{|\cI_s|} \norm{\brho(\emptyset)_s}_2 
        = \sum_{s\in \sS}  h_s \lambda_s \sqrt{\phi_0^s}.
    \end{equation}
    This completes the proof of the upper bound for $\fr1N \sup_{\ubsig\in \cQ_{\loc}(0)}\cH_N(\ubsig)$. 
    Finally, observe that equality holds above (up to the same $2D\eps$ error) if we choose $\brho(\emptyset) = \sqrt{\vphi_0} \diamond \bone$ and then recursively choose $(\brho(ui))_{i\in[k]}$ given $\brho(u)$ so that, for $|u|=d$, 
    \[
        \fr{1}{Nk}
        \sum_{i=1}^k \lt(\wtHNp{ui}(\brho(ui))-\wtHNp{u}(\brho(u))\rt)
        =
        F^u_{p_d,p_{d+1}}(\brho(u),\vphi_{d+1},k).
    \] 
\end{proof}

\subsection{The Algorithmic Functional}

Our next objective is to estimate the terms $f(\vphiminus, \vphiplus; \pminus, \pplus; k)$ appearing in Proposition~\ref{prop:uc-bogp}. 
The key point is that when the differences $\vphiplus-\vphiminus$ and $\pplus-\pminus$ are small, which is ensured by $\delta$-denseness of $(\up, \uvphi)$, this estimate only requires Taylor approximating the relevant Hamiltonians to second order. 
We take advantage of this using the following lemma, which (for $k=1$) gives the ground state energy $\bGS(W, \vv, 1)$ of a quadratic multi-species spin glass with Gaussian external field.
For general $k$, this lemma gives the limiting ground state energy $\bGS(W, \vv, k)$ of a $k$-replica Hamiltonian \eqref{eq:k-replica-hamiltonian} with shared quadratic component $W \diamond \bG$ and independent external fields $\vv \diamond \bg^i$, whose inputs \eqref{eq:bbtperp} are $k$ pairwise orthogonal elements of $\cS_N(\vone)$.
Note that $\bGS(W,\vv,k) \le \bGS(W,\vv,1)$ by definition. 
In fact equality holds, i.e. there exist orthogonal $\bsig^1, \ldots, \bsig^k$ such that each $\bsig^i$ approximately maximizes $H_{N,k}^i(\bsig^i)$.
We prove this lemma in Appendix~\ref{sec:sk-ext-field} by combining a known formula for the case $(k,\vv)=(1,\vzero)$ with an elementary recursive argument along subspaces.

\begin{lemma}
    \label{lem:sk-ext-field}
    Let $W = (w_{s,s'})_{s,s\in \sS} \in \bbR_{\ge 0}^{\sS\times \sS}$ be symmetric and $\vv = (v_s)_{s\in \sS} \in \bbR_{\ge 0}^\sS$.
    Let $k\in \bbN$ and sample independent $\bg^1,\ldots,\bg^k \in \bbR^N$ and $\bG \in \bbR^{N\times N}$ with i.i.d. standard Gaussian entries. 
    Consider the $k$-replica Hamiltonian
    \begin{equation}
        \label{eq:k-replica-hamiltonian}
        H_{N,k}(\ubsig)
        = 
        \fr{1}{k}
        \sum_{i=1}^k
        H_{N,k}^i(\bsig^i),
        \qquad
        H_{N,k}^i(\bsig^i)
        =
        \la \vv \diamond \bg^i, \bsig^i\ra 
        + 
        \fr{1}{\sqrt{N}}
        \la W \diamond \bG, (\bsig^i)^{\otimes 2} \ra
    \end{equation}
    on the input space of orthogonal replicas
    \begin{equation}
        \label{eq:bbtperp}
        \cS_N^{k, \perp}
        = 
        \lt\{
            \ubsig = (\bsig^1, \ldots, \bsig^k) \in \cS_N(\vone):
            \vR(\bsig^i,\bsig^j) = \vzero 
            ~\forall i\neq j
        \rt\}.
    \end{equation}
    Define the $k$-replica ground state energy 
    \begin{equation}
        \label{eq:sk-gsn}
        \GS_N(W,\vv,k)
        \equiv 
        \fr{1}{N}
        \max_{\ubsig \in \cS_N^{k, \perp}}
        H_{N,k}(\ubsig).
    \end{equation}
    Then $\bGS(W,\vv,k) \equiv \lim_{N\to\infty} \bbE \GS_N(W,\vv,k)$ exists, does not depend on $k$, and is given by
    \[
        \bGS(W,\vv,k) 
        = 
        \sum_{s\in \sS}
        \lambda_s
        \sqrt{v_s^2 + 2 \sum_{s'\in \sS} \lambda_s w_{s,s'}^2 }.
    \]
\end{lemma}

\begin{proposition}
    \label{prop:what-F-is}
    Suppose $0\le \pminus \le \pplus \le 1$, $\vzero \preceq \vphiminus \preceq \vphiplus \preceq \vone$ and
    \begin{equation}
        \label{eq:delta-discrete}
        \pplus - \pminus \le \delta,
        \qquad
        \vphiplus - \vphiminus \preceq \delta \vone.
    \end{equation}
    Then,
    \begin{align*}
        f(\vphiminus,\vphiplus; \pminus,\pplus; k)
        &=
        \sum_{s\in\sS}
        \lambda_s
        \sqrt{(\vphiplus^s - \vphiminus^s)\lt((\pplus-\pminus) \xi^s(\vphiminus) + \pminus \sum_{s'\in \sS} \partial_{x_{s'}} \xi^s(\vphiminus) (\vphiplus^s - \vphiminus^s)\rt)}
        \\
        &\qquad + 
        O\big(\delta^{3/2} + (\delta/k)^{1/2}\big)+o_N(1),
    \end{align*}
    where $o_N(1)$ denotes a term tending to $0$ as $N\to\infty$.
\end{proposition}

\begin{proof}
    Fix $\bsig^0$ such that $\vR(\bsig^0, \bsig^0) = \vphiminus$. 
    Let $\ubsig = (\bsig^1,\ldots,\bsig^k) \in B(\bsig^0,\vphiplus,k)$. 
    Let $\Delta \vphi = \vphiplus - \vphiminus$ and $\ubx = (\bx^1,\ldots,\bx^k)$ for $\bx^i = (\Delta \vphi)^{-1/2} \diamond (\bsig^i - \bsig^0)$.
    Define
    \begin{align*}
        \cS_{\bullet} &= \lt\{
            \by \in \cS_N(\vone) : \vR(\by,\bsig^0) = \vzero
        \rt\}, \\
        \cS_\bullet^{k,\perp} &= \lt\{
            \uby = (\by^1,\ldots,\by^k) \in \cS_{\bullet}^k : 
            \vR(\by^i,\by^j) = \vzero ~\forall i\neq j
        \rt\}.
    \end{align*}
    Note that $\ubx \in \cS_\bullet^{k,\perp}$. 
    Recall that $\hH^{[0]}_N,\ldots,\hH^{[k]}_N$ are i.i.d. copies of $\wtH_N$, and that $\hH^{(0)}_N,\ldots,\hH^{(k)}_N$ are defined by \eqref{eq:def-hH(0)}, \eqref{eq:def-hH(i)}.
    Let
    \[
        \oH^i_N(\bx^i) 
        = \hH_N^{[i]}(\bsig^i) - \hH_N^{[i]}(\bsig^0) 
        = \hH_N^{[i]}\lt(\bsig^0 + \sqrt{\Delta \vphi} \diamond \bx^i\rt) - \hH_N^{[i]}(\bsig^0).
    \]
    Then
    \begin{align}
        \notag
        f(\vphiminus,\vphiplus; \pminus,\pplus; k)
        &= 
        \fr{1}{kN}
        \bbE 
        \max_{\ubsig \in B(\bsig^0,\vphiplus,k)}
        \sum_{i=1}^k 
        \lt(\hH^{(i)}_N(\bsig^i) - \hH^{(0)}_N(\bsig^0)\rt) \\
        \notag
        &= 
        \fr{1}{kN}
        \bbE 
        \max_{\ubsig \in B(\bsig^0,\vphiplus,k)}
        \sum_{i=1}^k 
        \bigg(
            \sqrt{\pminus} \lt(\hH_N^{[0]}(\bsig^i) - \hH_N^{[0]}(\bsig^0)\rt) \\
            \notag
            &\qquad 
            + \sqrt{\pplus-\pminus} \lt(\hH_N^{[i]}(\bsig^i) - \hH_N^{[i]}(\bsig^0)\rt) 
            + \sqrt{\pplus-\pminus}\,
            \hH_N^{[i]}(\bsig^0)
        \bigg) \\
        \label{eq:f-expansion-into-hamiltonians}
        &= 
        \fr{1}{kN}
        \bbE 
        \max_{\ubx \in \cS_\bullet^{k,\perp}}
        \sum_{i=1}^k 
        \lt(
            \sqrt{\pminus}\, \oH^0 (\bx^i) + \sqrt{\pplus - \pminus}\, \oH^i (\bx^i)
        \rt) 
    \end{align}
    where we note that $\bbE \hH_N^{[i]}(\bsig^0) = 0$. 
    Let $\oH^{i,\tay}_N$ denote the degree $2$ Taylor expansion of $\oH_N^i$ around $\bzero$. 
    By Proposition~\ref{prop:gradients-bounded} (recalling \eqref{eq:delta-discrete}),
    \[
        \bbE \sup_{\bx \in \cS_{\bullet}} |\oH^i(\bx)_N - \oH^{i,\tay}_N(\bx)| = O(N\delta^{3/2}).
    \]
    So, for all $0\le i\le k$, we have as processes on $\cS_{\bullet}$
    \begin{equation}
        \label{eq:oH-taylor-expansion}
        \oH^i_N(\bx) =_d \la \vv \diamond \bg^i, \bx\ra
        + 
        \la W \diamond \bG^i, \bx^{\otimes 2}\ra
        + O_{\bbP}(N\delta^{3/2}),
    \end{equation}
    where $O_{\bbP}(N\delta^{3/2})$ denotes a $\cS_{\bullet}$-valued process $X(\bx)$ with $\bbE \sup_{\bx \in \cS_{\bullet}} |X(\bx)| = O(N\delta^{3/2})$ and $\vv = (v_s)_{s\in \sS}$ and $W = (w_{s,s'})_{s,s'\in \sS}$ are given by
    \[
        v_s = \sqrt{\xi^s(\vphiminus) (\Delta \vphi)^s }, 
        \qquad
        w_{s,s'} = \fr{1}{\sqrt{2}} \sqrt{\lambda_{s'}^{-1} \partial_{x_{s'}} \xi^s(\vphiminus) (\Delta \vphi)^s (\Delta \vphi)^{s'}}.
    \]
    Next we observe some simplifications. 
    Because $\Delta \vphi \preceq \delta \vone$, we have $v_s = O(\delta^{1/2})$, $w_{s,s'} = O(\delta)$ uniformly over $s,s'$.
    The linear contribution to $\oH^0_N$ in \eqref{eq:oH-taylor-expansion} is small because
    \[
        \fr{1}{kN} \sum_{i=1}^k \la \vv \diamond \bg^0, \bx^i\ra
        \le \fr{1}{kN} \norm{\vv \diamond \bg^0}_2 \norm{\sum_{i=1}^k \bx^i}_2
        = O_{\bbP}((\delta/k)^{1/2})
    \]
    by orthogonality of the $\bx^i$.
    Because $\pplus - \pminus \le \delta$, the quadratic contributions to $\oH^i_N$ for $i\ge 1$ are also small:
    \[
        \fr{\sqrt{\pplus-\pminus}}{N}
        \la W \diamond \bG^i, (\bx^i)^{\otimes 2}\ra
        = O_{\bbP}(\delta^{3/2}).
    \]
    Combining these estimates with \eqref{eq:f-expansion-into-hamiltonians} and \eqref{eq:oH-taylor-expansion}, we find
    \begin{align*}
        f(\vphiminus,\vphiplus; \pminus,\pplus; k)
        &=
        \fr{1}{kN}
        \bbE \max_{\ubx \in \cS_\bullet^{k,\perp}}
        \sum_{i=1}^k 
        \sqrt{\pplus - \pminus} \lt\la \vv \diamond \bg^i, \bx^i\rt\ra +
        \sqrt{\pminus} \lt\la W \diamond \bG^0, (\bx^i)^{\otimes 2} \rt\ra \\
        &\qquad + O((\delta/k)^{1/2} + \delta^{3/2}).
    \end{align*}
    By Lemma~\ref{lem:sk-ext-field} (applied in dimension $N-r$ due to the linear constraint $\vR(\bx^i,\bsig^0)=\vzero$ in $\cS_{\bullet}$), this remaining expectation is given up to $o_N(1)$ error by
    \begin{align*}
        &\sum_{s\in \sS}
        \lambda_s
        \sqrt{(\pplus - \pminus) v_s^2 + 2\sum_{s'\in \sS} \lambda_{s'} \pminus w_{s,s'}^2} \\
        &\quad = 
        \sum_{s\in \sS}
        \lambda_s
        \sqrt{(\Delta \vphi)^s \lt((\pplus - \pminus) \xi^s(\vphiminus) + \pminus \sum_{s'\in \sS} \partial_{x_{s'}} \xi^s(\vphiminus) (\Delta \vphi)^{s'}\rt)}.
    \end{align*}
    This implies the result. 
\end{proof}

We now evaluate $\BOGP_{\loc,0}$ by taking a continuous limit of Propositions~\ref{prop:uc-bogp} and \ref{prop:what-F-is}. 
Fix $D,k$ and $\delta = 6r/D$, and let $(\up,\uvphi)$ be $\delta$-dense. 
We parametrize time by $q_d = \la \vlam, \vphi_d\ra$, so in particular $q_0 = \la \vlam, \vphi_0\ra$.
Let the functions $\tp:[q_0,1]\to [0,1]$ and $\tPhi:[q_0,1]\to [0,1]^{\sS}$ satisfy 
\begin{equation}
    \label{eq:continuous-def}
    \tp(q_d) = p_d,
    \qquad
    \tPhi(q_d) = \vphi_d.
\end{equation}
and be linear on each interval $[q_d,q_{d+1}]$. 
These are piecewise linear approximations of inputs $(p,\Phi)$ to the algorithmic functional $\bbA$.
Define
\begin{equation}
    \label{eq:def-Ads}
    A_d^s = 
    \sqrt{(\phi^s_{d+1} - \phi^s_d)\lt(
        (p_{d+1}-p_d) \xi^s(\vphi_d) + 
        p_d \sum_{s'\in\sS}\partial_{x_{s'}}\xi^s(\vphi_d)(\phi_{d+1}^{s'}-\phi_d^{s'})
    \rt)}.
\end{equation}
This term appears in the estimate of $f\lt(\vphi_d,\vphi_{d+1};p_{d+1},p_d;k\rt)$ obtained from Proposition~\ref{prop:what-F-is}.
\begin{lemma}
    \label{lem:ALG-derivation}
    We have
    \begin{align*}
        \lt|
            \sum_{d=0}^{D-1} A_d^s - 
            \int_{q_0}^{q_D}
            \sqrt{\tPhi_s'(q)(\tp\times \xi^s\circ\tPhi)'(q)}
            ~\de q
        \rt|
        \le CD^{-1/2}
    \end{align*}
    for a constant $C>0$ independent of $D,\up,\uvphi$.
\end{lemma}
\begin{proof}
Until the end, we focus on estimating the difference
\[
    \Delta_d^s \equiv 
    \lt|
        A_d^s - 
        \int_{q_d}^{q_{d+1}}
        \sqrt{\tPhi_s'(q)(\tp\times \xi^s\circ\tPhi)'(q)}
        ~\de q
    \rt|.
\]
Note the general inequality
\begin{align}
    \notag
    \int_{q_d}^{q_{d+1}}\sqrt{a(q)}\cdot |\sqrt{b(q)}-\sqrt{c(q)}| \de q
    &\leq
    \lt(\int_{q_d}^{q_{d+1}}a(q)\de q\rt)^{1/2}
    \cdot
    \lt(\int_{q_d}^{q_{d+1}}\big(\sqrt{b(q)}-\sqrt{c(q)}\big)^2\de q\rt)^{1/2} \\
    \label{eq:integral-bound}
    &\leq
    \lt(\int_{q_d}^{q_{d+1}}a(q)\de q\rt)^{1/2}
    \cdot
    \lt(\int_{q_d}^{q_{d+1}}|b(q)-c(q)|\de q\rt)^{1/2}.
\end{align}
Thus
\begin{align*}
    \Delta^s_d
    &=
    \lt|
        \int_{q_d}^{q_{d+1}}
        \sqrt{\tPhi_s'(q)}
        \lt(
            \sqrt{(\tp\times \xi^s\circ\tPhi)'(q)}
            -
            \sqrt{\fr{
                (p_{d+1}-p_d) \xi^s(\vphi_d) + 
                p_d \sum_{s'\in\sS}
                \partial_{x_{s'}}\xi^s(\vphi_d)
                (\phi_{d+1}^{s'}-\phi_d^{s'})
            }{q_{d+1}-q_d}}
        \rt)
        ~\de q
    \rt| \\
    &\le
    \sqrt{
        (\phi_{d+1}^s-\phi_d^s)
        \int_{q_d}^{q_{d+1}}
        \lt|
            (\tp\times \xi^s\circ\tPhi)'(q)
            -
            \fr{
                (p_{d+1}-p_d) \xi^s(\vphi_d) + 
                p_d \sum_{s'\in\sS}
                \partial_{x_{s'}}\xi^s(\vphi_d)
                (\phi_{d+1}^{s'}-\phi_d^{s'})
            }{q_{d+1}-q_d}
        \rt|
        \de q
    }.
\end{align*}
In the first step we used that $\tPhi'(q)=(\vphi_{d+1}-\vphi_d) / (q_{d+1}-q_d)$ by definition, and in the second we used \eqref{eq:integral-bound}.
Let $(\tp\times \xi^s\circ\tPhi)'(q_d)$ and $(\tp\times \xi^s\circ\tPhi)'(q_{d+1})$ denote the right and left derivatives at these points, respectively.
The definitions of $\tp'$ and $\tPhi'$ imply
\[
    \fr{
        (p_{d+1}-p_d) \xi^s(\vphi_d) + 
        p_d \sum_{s'\in\sS}
        \partial_{x_{s'}}\xi^s(\vphi_d)
        (\phi_{d+1}^{s'}-\phi_d^{s'})
    }{q_{d+1}-q_d}
    = (\tp\times \xi^s\circ\tPhi)'(q_d),
\]
so in fact
\begin{align}
    \notag
    \Delta^s_d
    &\le 
    \sqrt{
        (\phi_{d+1}^s-\phi_d^s)
        \int_{q_d}^{q_{d+1}}
        \lt|
            (\tp\times \xi^s\circ\tPhi)'(q)
            - (\tp\times \xi^s\circ\tPhi)'(q_d)
        \rt|
        \de q
    } \\
    \label{eq:one-step-integral-estimation}
    &\le
    \sqrt{
        (\phi_{d+1}^s-\phi_d^s)
        (q_{d+1}-q_d)
        \lt(
            (\tp\times \xi^s\circ\tPhi)'(q_{d+1})
            - (\tp\times \xi^s\circ\tPhi)'(q_d)
        \rt)
    } .
\end{align}
Let $\nabla \vphi_d = (\vphi_{d+1}-\vphi_d) / (q_{d+1}-q_d)$ be the constant value of $\nabla \tPhi$ on $[q_d,q_{d+1}]$.
Then 
\begin{align*}
    (\tp\times \xi^s\circ\tPhi)'(q_{d+1})-(\tp\times \xi^s\circ\tPhi)'(q_d)
    &=
    \lt(\fr{p_{d+1}-p_d}{q_{d+1}-q_d}\rt)
    \lt(\xi^s(\vphi_{d+1})-\xi^s(\vphi_d)\rt) \\
    &\qquad
    + p_{d+1} \la \nabla \xi^s(\vphi_{d+1}),\nabla \vphi_d \ra
    - p_{d} \la \nabla \xi^s(\vphi_d),\nabla \vphi_d \ra.
\end{align*}
We thus obtain
\begin{align*}
    &(q_{d+1}-q_d)\lt((\tp\times \xi^s\circ\tPhi)'(q_{d+1})-(\tp\times \xi^s\circ\tPhi)'(q_d)\rt) \\
    &= (p_{d+1}-p_d) \lt(\xi^s(\vphi_{d+1})-\xi^s(\vphi_d)\rt) 
    + (q_{d+1}-q_d) (p_{d+1}-p_d) \la \nabla \xi^s(\vphi_{d+1}),\nabla \vphi_d \ra \\
    &\qquad + (q_{d+1}-q_d) p_d 
    \la \nabla \xi^s(\vphi_{d+1}) -  \nabla \xi^s(\vphi_d),\nabla \vphi_d \ra \\
    &\le O\lt(
        (p_{d+1}-p_d) \norm{\vphi_{d+1}-\vphi_d}_2 
        + \norm{\vphi_{d+1}-\vphi_d}_2^2
    \rt) = O(\delta^2).
\end{align*}
Combining with \eqref{eq:one-step-integral-estimation} gives the estimate $\Delta^s_d = O(\delta) \sqrt{\phi_{d+1}^s - \phi_d^s}$.
Summing this over $0\le d\le D-1$ gives the final estimate
\begin{align*}
    \lt|
        \sum_{d=0}^{D-1}
        A_d^s -
        \int_{q_0}^{q_D}
        \sqrt{\tPhi_s'(q)(\tp\times \xi^s\circ\tPhi)'(q)}
        ~\de q
    \rt|
    &\le \sum_{d=0}^{D-1} \Delta^s_d
    \le O(\delta) \sum_{d=0}^{D-1} \sqrt{\phi_{d+1}^s-\phi_d^s} \\
    &\le O(\delta \sqrt{D}) = O(D^{-1/2}).
\end{align*}
by Cauchy-Schwarz.
\end{proof}

We next show that discretizing any $C^1$ functions $(p,\Phi)$ preserves the value of $\bbA$.

\begin{lemma}
    \label{lem:ALG-from-continuum}
    Suppose $q_0\in [0,1]$, $p\in \bbI(q_0,1)$, and $\Phi \in \Adm(q_0,1)$.
    Consider any $\uq = (q_0,\ldots,q_D)$ with $q_0<\cdots<q_D=1$, such that the $(\up, \uvphi)$ defined by $p_d = p(q_d)$ and $\vphi_d = \Phi(q_d)$ is $6r/D$-dense. 
    Then, for all $s\in \sS$,
    \[
        \lt|
            \int_{q_0}^1 \sqrt{\Phi'_s(q) (p\times \xi^s \circ \Phi)'(q)}~\de q 
            - \int_{q_0}^1 \sqrt{\tPhi'_s(q) (\tp\times \xi^s \circ \tPhi)'(q)}~\de q 
        \rt|
        = o_D(1),
    \]
    where $\tp, \tPhi$ are the piecewise linear interpolations defined by \eqref{eq:continuous-def} and $o_D(1)$ is a term tending to $0$ as $D\to\infty$ (for fixed $(p,\Phi)$).
\end{lemma}
\begin{proof}
    The functions $\Phi,\Phi',p,p'$ are uniformly continuous because they are continuous on $[q_0,1]$.
    So,
    \[
        \tnorm{\Phi-\tPhi}_\infty,
        \tnorm{\Phi'-\tPhi'}_\infty,
        \tnorm{p-\tp}_\infty,
        \tnorm{p'-\tp'}_\infty
        = o_D(1).
    \]
    Bounded convergence implies the result. 
\end{proof}

\begin{proof}[Proof of Proposition~\ref{prop:bogp-alg}]
    By Proposition~\ref{prop:bogp-equivalent} it suffices to prove that $\BOGP_{\loc,0}=\ALG$. 
    We will separately show $\BOGP_{\loc,0} \le \ALG$ and $\BOGP_{\loc,0} \ge \ALG$.

    We first show $\BOGP_{\loc,0} \le \ALG$. 
    Let $\iota > 0$.
    Let $D$ be sufficiently large, $\eps = D^{-2}$ and $\delta = 6r/D$, and $k$ be sufficiently large depending on $D$ such that the following holds. 
    First, $k \ge k_0 (D,\eps,\delta)$ for $k_0$ defined in Proposition~\ref{prop:uc-bogp}. 
    Second, for some $1/D^2$-separated $\vchi \in \bbI(0,1)^\sS$ and all $\delta$-dense $(\up,\uvphi)$ with $\uvphi = \vchi(\up)$, we have
    \[
        \fr1N \bbE \sup_{\ubsig \in \cQ_\loc(0)} \cH_N(\ubsig)
        \ge \BOGP_{\loc,0} - \iota/2.
    \]
    Let $q_d = \la \vlam, \vphi_d\ra$.
    Let $\tp,\tPhi$ be the piecewise linear interpolations defined by \eqref{eq:continuous-def} and $A_d^s$ be defined by \eqref{eq:def-Ads}.
    Then 
    \begin{align*}
        &\lt|
            \fr1N \sup_{\ubsig \in \cQ_{\loc}(0)} \cH_N(\ubsig) 
            - \sum_{s\in \sS} \lt(
                h_s \lambda_s \sqrt{\tPhi_s(q_0)}
                + \lambda_s \int_{q_0}^{q_D}
                \sqrt{\tPhi'_s(q) (p\times \xi^s \circ \Phi)(q)}
                ~\de q
            \rt)
        \rt| \\
        &\le \lt|
            \fr1N \sup_{\ubsig \in \cQ_{\loc}(0)} \cH_N(\ubsig) 
            - \sum_{s\in \sS} \lt(
                h_s \lambda_s \sqrt{\phi_0^s}
                + \sum_{d=0}^{D-1} 
                f\lt(\vphi_d,\vphi_{d+1};p_d,p_{d+1};k\rt)
            \rt)
        \rt| \\
        &\quad + 
        \sum_{d=0}^{D-1} 
        \lt|
            f\lt(\vphi_d,\vphi_{d+1};p_d,p_{d+1};k\rt)
            - \sum_{s\in \sS} \lambda_s A_d^s
        \rt|
        + \sum_{s\in \sS} \lambda_s \lt|
            \sum_{d=0}^{D-1} A_d^s 
            - \int_{q_0}^{q_D}
            \sqrt{\tPhi'_s(q) (p\times \xi^s \circ \Phi)(q)}
            ~\de q
        \rt|.
    \end{align*}
    By Propositions~\ref{prop:uc-bogp}, \ref{prop:what-F-is} and Lemma~\ref{lem:ALG-derivation}, on an event with probability $1-e^{-cN}$ this is bounded by
    \[
        2D\eps + O(D\delta^{3/2} + D(\delta/k)^{1/2}) + O(D^{-1/2}) + o_N(1)
        = O(D^{-1/2} + (D/k)^{1/2}) + o_N(1)
        \le \iota/4,
    \]
    for sufficiently large $N,D,k$.
    Because $\fr1N \sup_{\ubsig \in \cQ_{\loc}(0)} \cH_N(\ubsig) $ is subgaussian with fluctuations $O(N^{-1/2})$ by Lemma~\ref{lem:bogp-subgaussian}, the contributions of the complement of this event are $o_N(1)$, and so
    \begin{equation}
        \label{eq:bogp-final-estimate}
        \lt|
            \fr1N \bbE \sup_{\ubsig \in \cQ_{\loc}(0)} \cH_N(\ubsig) 
            - \sum_{s\in \sS} \lt(
                h_s \lambda_s \sqrt{\tPhi_s(q_0)}
                + \lambda_s \int_{q_0}^{q_D}
                \sqrt{\tPhi'_s(q) (p\times \xi^s \circ \Phi)(q)}
                ~\de q
            \rt)
        \rt|
        \le \iota/4.
    \end{equation}
    Let $p \in \bbI(q_0,1)$ and $\Phi \in \Adm(q_0,1)$ approximate the piecewise linear functions $(\tp,\tPhi)$ on $[q_0,q_D]$, in the sense that
    \begin{equation}
        \label{eq:approximate-piecewise-linear}
        \lt|
            \sum_{s\in \sS}
            \lambda_s
            \int_{q_0}^{q_D}
            \lt(\sqrt{\tPhi_s'(q)(\tp\times \xi^s\circ\tPhi)'(q)}
            - \sqrt{\Phi_s'(q)(p\times \xi^s\circ\Phi)'(q)}\rt)
            ~\de q
        \rt|
        \le \iota/4.
    \end{equation}
    It is clear that such $p,\Phi$ exist. 
    Thus
    \[
        \BOGP_{\loc,0} \le \bbA(p,\Phi;q_0) - \iota \le \ALG - \iota.
    \]
    Since $\iota$ was arbitrary, we conclude $\BOGP_{\loc,0} \le \ALG$. 

    Next, we will show $\BOGP_{\loc,0} \ge \ALG$.
    Let $\iota > 0$, and let $D$ be sufficiently large and $k$ be sufficiently large depending on $D$.
    There exist $q_0\in [0,1]$, $p\in \bbI(q_0,1)$, and $\Phi \in \Adm(q_0,1)$ such that 
    \[
        \bbA(p,\Phi;q_0) \ge \ALG - \iota/2.
    \]
    By replacing $\Phi$ with $(1-D^{-2})\Phi + D^{-2} \vone$ we may assume $\Phi(q_0) \succeq \vone/D^2$, as this replacement affects the left-hand side by $o_D(1)$.
    Similarly, by replacing $p(q)$ with $(1-D^{-1})p(q) + D^{-1}q$, we may assume $p$ is strictly increasing.
    We choose $\vchi = \Phi \circ p^{-1}$, which is $1/D^2$-separated.
    Consider any $\uq = (q_0,q_1,\ldots,q_D)$ with $q_0<q_1<\cdots<q_D=1$ such that for $p_d = p(q_d)$, $\vphi_d = \Phi(q_d)$, the pair $(\up,\uvphi)$ is $6r/D$-dense.
    Similarly to above, we have \eqref{eq:bogp-final-estimate} for sufficiently large $N,D,k$.
    By Lemma~\ref{lem:ALG-from-continuum}, \eqref{eq:approximate-piecewise-linear} holds for $D$ sufficiently large. 
    This implies 
    \[
        \bbA(p,\Phi;q_0) \le \BOGP_{\loc,0}+\iota/2,
    \]
    and so $\ALG \le \BOGP_{\loc,0} + \iota$. 
    Because $\iota$ was arbitrary, we have $\ALG \le \BOGP_{\loc,0}$. 
\end{proof}

%% file: final-tex/4-alg.tex
\section{Optimization of the Algorithmic Variational Principle}
\label{sec:alg}

In this section we will prove Propositions~\ref{prop:root-finding-trajectory} and \ref{prop:tree-descending-trajectory} and Theorem~\ref{thm:alg-optimizer}. 
Throughout this section we assume Assumption~\ref{as:nondegenerate} except where stated.

To ensure a priori existence of a maximizer in \eqref{eq:alg}, we work in the following compact space which removes the constraint that $p$ and $\Phi$ are continuously differentiable.

\begin{definition}
\label{defn:cM}
    The space $\cM$ consists of all triples $(p,\Phi,q_0)$ such that:
    \begin{itemize}
    \item $q_0\in [0,1]$.
    \item $p:[q_0,1]\to [0,1]$ is non-decreasing and right-continuous (we write $p\in \hbbI(q_0,1)$).
    \item $\Phi=(\Phi_s)_{s\in\sS}$ consists of $r$ non-decreasing functions $\Phi_s:[q_0,1]\to [0,1]$ satisfying admissibility \eqref{eq:admissible} (we write $\Phi\in \hAdm(q_0,1)$).
    \end{itemize}
\end{definition}

Because we assume almost no regularity for elements of $\cM$, we formally define the integral in \eqref{eq:alg-functional} as follows.
Since $(p\times \xi^s \circ \Phi)$ is a bounded increasing function, it has a positive measure valued distributional derivative 
\begin{equation}
    \label{eq:careful-alg}
    (p\times \xi^s \circ \Phi)'(q) = f(q)~\de q + \de \mu(q)
\end{equation}
where $f\in L^1([q_0,1])$ and $\mu$ is an atomic-plus-singular measure supported in $[q_0,1]$.
Moreover, \eqref{eq:admissible} implies $\Phi_s$ is $\lambda_s^{-1}$-Lipschitz, hence has distributional derivative $\Phi'_s \in L^\infty([q_0,1])$.

\begin{definition}
For $(p,\Phi,q_0)\in\cM$, define 
\begin{equation}
    \label{eq:halg}
    \hALG
    \equiv
    \sup_{(p,\Phi,q_0)\in \cM}
    \bbA(p,\Phi; q_0).
\end{equation}
where the second term of $\bbA$ is given (with $f$ as in \eqref{eq:careful-alg}) by:
\[
    \int_{q_0}^1
    \sqrt{\Phi'_s(q) (p\times \xi^s \circ \Phi)'(q)}
    ~\de q
    =
    \int_{q_0}^1
    \sqrt{\Phi'_s(q) f(q)}
    ~\de q.
\]
\end{definition}

It will follow from our results in this section that for non-degenerate $\xi$, all maximizers to the extended variational problem are continuously differentiable on $[q_0,1]$. The equality $\ALG=\hALG$ follows in general since both are continuous in $(\xi,\vh)$.

\begin{remark}
    A related (for the most part, simpler) variational problem was considered in \cite{deuschel1995limiting}. 
    There, after showing existence and other basic properties, the general result \cite[Theorem 5.1]{cesari2012optimization} was used to derive an ordinary differential equation \cite[Theorem 4]{deuschel1995limiting} for the optimal $\Phi$.
    The same general result applies in our setting, and essentially yields Proposition~\ref{prop:psi}, assuming $f_s$ defined in \eqref{eq:f-s-q} are absolutely continuous for all $s\in\sS$. 
    More precisely, under this assumption one finds (cf. \eqref{eq:Psi-s}, \eqref{eq:derivative-stability}):
    \[
    \sum_{s\in\sS}
    \Psi_s(q)
    \big(p\times \partial_s \xi^{s'}\circ\Phi\big)(q)\Phi_s'(q)
    =
    \Psi_{s'}(q)
    \big(p\times \partial_s \xi^{s'}\circ\Phi\big)'(q).
    \]
    Viewing this as a linear system in the variables $\Psi_s(q)$, Corollary~\ref{cor:rank} implies that if $p'(q)>0$, then $\Psi_1(q)=\Psi_2(q)=\dots=\Psi_r(q)=0$.
    Similarly if $p'(q)=0$, Lemma~\ref{lem:rank-v2} with $\eps=0$ implies $\Psi_1(q)=\Psi_2(q)=\dots=\Psi_r(q)$.
    However the only way we could establish absolute continuity of $f_s$ was by going through the full proof of Proposition~\ref{prop:psi}.
\end{remark}

\subsection{Linear Algebraic and Analytic Preliminaries}

We first prove Corollary~\ref{cor:solvability-equivalent} below, an equivalent characterization of (super, strict sub)-solvability.
\begin{proposition}
    \label{prop:smallest-eigenvalue}
    Let $M \in \bbR^{\sS \times \sS}$ be diagonally signed.
    Then 
    \begin{equation}
        \label{eq:VM-def}
        \Lambda(M)=\sup_{\vv\in \mathbb R_{>0}^{\sS}} \min_{s\in\sS} \frac{(M\vv)_s}{v_s}
    \end{equation}
    equals the smallest eigenvalue $\lambda_{\min}(M)$ of $M$. 
\end{proposition}
\begin{proof}
    Let $\vw$ be a (unit) minimal eigenvector of $M$.
    Note that
    \[
        \vw^\top M \vw
        = \sum_{s,s' \in \sS} M_{s,s'} w_s w_{s'}
        \ge \sum_{s,s' \in \sS} M_{s,s'} |w_s| |w_{s'}|.
    \]
    Since $\vw$ minimizes $\vw^\top M \vw$, all entries of $\vw$ are the same sign.
    We may thus assume $\vw \in \bbR_{\ge 0}^\sS$.
    Moreover, if $w_s=0$ for any $s$, then $(M\vw)_s < 0$ so $\vw$ is not an eigenvector; thus $\vw \in \bbR_{>0}^\sS$.
    Because $M\vw = \lambda_{\min}(M)\vw$, clearly $\Lambda(M) \ge \lambda_{\min}(M)$.
    For any other $\vv \in \bbR_{>0}^\sS$,
    \[
        \min_{s\in \sS}
        \fr{(M\vv)_s}{v_s}
        \le 
        \la \vw,\vv \ra^{-1} \sum_{s\in \sS} w_sv_s \cdot \fr{(M\vv)_s}{v_s}
        = \fr{\la \vw, M\vv \ra}{\la \vw,\vv \ra}
        = \fr{\la M\vw, \vv \ra}{\la \vw,\vv \ra}
        = \lambda_{\min}(M),
    \]
    so $\Lambda(M) \le \lambda_{\min}(M)$.
    Thus $\Lambda(M) = \lambda_{\min}(M)$.
\end{proof}
\begin{corollary}
    \label{cor:solvability-equivalent}
    For $\vx \in (0,1]^\sS \cup \{\vzero\}$ define
    \[
        M^*(\vx) 
        = \diag\lt((\xi^s(\vx) + h_s^2)_{s\in \sS}\rt)
        - \lt(x_s \partial_{x_{s'}} \xi^s(\vx)\rt)_{s,s'\in \sS}.
    \]
    Then $\vx$ is super-solvable (resp. solvable, strictly sub-solvable) if and only if $\Lambda(M^*(\vx)) \ge 0$ (resp. $=0$, $<0$).
\end{corollary}
\begin{proof}
    Suppose first $\vx \in (0,1]^\sS$. 
    By Proposition~\ref{prop:smallest-eigenvalue}, $\vx$ is super-solvable (resp. solvable, strictly sub-solvable) if and only if $\Lambda(M^*_\sym(\vx)) \ge 0$ (resp. $=0$, $<0$). 
    Note that 
    \begin{equation}
        \label{eq:M*sym-to-M*}
        M^*(\vx) = \diag\lt((\lambda_s x_s)_{s\in \sS}\rt) M^*_\sym(\vx),
    \end{equation}
    so $\Lambda(M^*(\vx))$ has the same sign as $\Lambda(M^*_\sym(\vx))$, as desired.
    If $\vx = \vzero$, then clearly $\Lambda(M^*(\vx)) \ge 0$ with equality at $\vh = \vzero$, which agrees with the convention from Definition~\ref{defn:solvable}.
\end{proof}

The following proposition is clear.

\begin{proposition}
\label{prop:VM}
Let $\Lambda$ be as in \eqref{eq:VM-def} and let $M\in\bbR_{\geq 0}^{r\times r}$ (not necessarily diagonally signed). Then $\Lambda(M)$ is non-negative and locally bounded. Moreover if for some $c\in\bbR$ we have $M'_{s,s'}\geq M_{s,s'}+c\cdot 1_{s=s'}$ for all $s,s'$, then $\Lambda(M')\geq \Lambda(M)+c$.
\end{proposition}

Many perturbation arguments used to establish regularity rely on the following basic fact.

\begin{proposition}[{\cite[Theorem 7.7]{rudin1970real}}]
    \label{prop:lebesgue}
    For $f\in L^{1}([0,1])$, almost all $x\in [0,1]$ are \textbf{Lebesgue points}: 
    \[
      \lim_{\eps\to 0}
      \frac{1}{2\eps} 
      \int_{x-\eps}^{x+\eps}
      |f(y)-f(x)|
      ~\de y = 0.
    \]
\end{proposition}

The next fact ensures that Lipschitz ordinary differential equations are well-posed (even if they are only required to hold almost everywhere).

\begin{proposition}[{\cite[Theorem 1.45, Part (ii)]{roubivcek2013nonlinear}}]
\label{prop:ODE-well-posed}
Suppose $Y_1,Y_2:[0,1]\to\bbR^d$ are each absolutely continuous with $Y_1(0)=Y_2(0)$ and solve the ODE $Y_i'(q)=F(Y_i(q))$ at almost all $q$ for $F:\bbR^d\to\bbR^d$ Lipschitz. Then $Y_1,Y_2$ agree and solve the ODE for all $q$.
\end{proposition}

\subsection{A Priori Regularity of Maximizers}
\label{subsec:basic-regularity}

We first show that for the optimization problem \eqref{eq:alg}, admissibility \eqref{eq:admissible} is just a convenient choice of normalization. 
This makes variational arguments more convenient because we do not need to worry about preserving admissibility of $\Phi$ under perturbations.
Let $\tbbI(q_0,1) \subseteq \hbbI(q_0,1)$ be the set of increasing and Lipschitz functions $f:[q_0,1] \to [0,1]$ with no explicit bound on the Lipschitz constant and with $f(1)=1$.
Note that the algorithmic functional $\bbA$ \eqref{eq:alg-functional} remains well-defined for $\Phi \in \tbbI(q_0,1)^\sS$. 
\begin{lemma}
    \label{lem:admissible-optional}
    We have that
    \begin{equation}
        \label{eq:admissible-optional}
        \hALG =
        \sup_{q_0\in [0,1]}
        \sup_{\substack{p\in \hbbI(q_0,1) \\ \Phi \in \tbbI(q_0,1)^\sS}} 
        \bbA(p,\Phi; q_0).
    \end{equation}
\end{lemma}
\begin{proof}
    Let $\hALG'$ be the right-hand side of \eqref{eq:admissible-optional}.
    We will show that $\hALG \geq \hALG'$ (the opposite implication being trivial).
    
    Consider any $q_0 \in [0,1]$, $p\in \hbbI(q_0,1)$, and $\Phi \in \tbbI(q_0,1)^\sS$.
    For small $\delta > 0$, consider
    \[
        \Phi_{\delta} (q) = \delta q \vone + (1-\delta) \Phi(q)
    \]
    and let $\alpha(q) = \la \vlam, \Phi_{\delta}(q)\ra$, so $\alpha'(q) \ge \delta$. 
    Thus $\alpha^{-1}$ exists and is $\delta^{-1}$-Lipschitz. 
    Consider $(\wtp,\tPhi,\wtq_0)$ given by
    \[
        \wtp(q) = p(\alpha^{-1}(q)),
        \quad 
        \tPhi(q) = \Phi(\alpha^{-1}(q)),
        \quad 
        \wtq_0 = \alpha(q_0).
    \]
    By construction, $\tPhi \in \hAdm(\wtq_0,1)$.
    By the chain rule, $\bbA(\wtp,\tPhi;\wtq_0) = \bbA(p,\Phi_{\delta};q_0)$.
    Thus
    \[
        \hALG 
        \ge 
        \limsup_{\delta\downarrow 0}
        \bbA(\wtp,\tPhi,\wtq_0)
        =
        \limsup_{\delta\downarrow 0}
        \bbA(p,\Phi_{\delta};q_0)
        \ge 
        \bbA(p,\Phi;q_0).
    \]
    Since $p,\Phi,q_0$ were arbitrary the conclusion follows.
\end{proof}

A routine compactness argument given in Appendix~\ref{subsec:maximizer-existence} yields the following. 
\begin{restatable}{proposition}{propFmax}
\label{prop:F-max}
There exists a maximizer $(p,\Phi,q_0)\in\cM$ for $\bbA$ and $\bbA(p,\Phi;q_0)<\infty$.
\end{restatable}

From now on, we let $(p,\Phi,q_0)\in\cM$ denote any maximizer and study the behavior of $(p,\Phi,q_0)$. 
While almost no regularity on $(p,\Phi)$ is assumed, it is possible to establish a priori regularity using variational arguments. 
We defer the proofs of the following two propositions to Appendix~\ref{subsec:regularity-for-4.1}. 
Proposition~\ref{prop:basic-regularity} implies that the discussion following \eqref{eq:alg} is not necessary to define $\bbA(p,\Phi;q_0)$.
\begin{restatable}{proposition}{propBasicRegularity}
    \label{prop:basic-regularity}
    The functions $p,\Phi$ are continuously differentiable on $[q_0+\eps,1]$ for any $\eps>0$.
    Moreover, there exists $L>0$ (possibly depending on $(p,\Phi;q_0)$ as well as $\xi$) such that $L^{-1} \vone \preceq \Phi'(q) \preceq L \vone$ for almost all $q\in (q_0,1]$.
\end{restatable}

\begin{restatable}{proposition}{propPBasic}
    \label{prop:p-basic}
    The function $p$ satisfies $p(q)>0$ for all $q>q_0$, $p(1)=1$, and $p(q_0)=0$ if $q_0>0$.
\end{restatable}

Throughout the next subsection we will use $\eps>0$ as in Proposition~\ref{prop:basic-regularity}. Later we slightly improve the result of Proposition~\ref{prop:basic-regularity} to continuity on $[q_0,1]$ using more detailed properties of the maximizers. 

\subsection{Identification of Root-Finding and Tree-Descending Phases}
\label{subsec:Psi-calculation}

In this subsection we will prove the following result. Recall that the Sobolev space $W^{2,\infty}([q_0+\eps,1])$ consists of $C^1$ functions with Lipschitz derivative on the interval.
\begin{proposition}
    \label{prop:type-12}
    The restrictions of $p$ and $\Phi_s$, for all $s\in \sS$, lie in the space $W^{2,\infty}([q_0+\eps,1])$ for any $\eps>0$. 
    There exists $q_1\in [q_0,1]$ such that the following holds. 
    \begin{enumerate}[label=(\alph*), ref=\alph*]
        \item On $[q_0,q_1]$, $p'>0$ almost everywhere and the quantities $\fr{\Phi'_s(q)}{(p\times \xi^s \circ \Phi)'(q)}$ are constant.
        Moreover $p(q_1)=1$.
        \item On $[q_1,1]$, the ODE \eqref{eq:tree-descending-ode} is satisfied for all $s,s'\in \sS$ almost everywhere and $p=1$.
    \end{enumerate}
\end{proposition}

We begin with a result on diagonally dominant matrices. Variants especially with $\eps=0$ have been used many times, see e.g. \cite{taussky1949recurring}.
Related linear algebraic statements will appear later in Lemmas~\ref{lem:positive-linalg-with-p} and \ref{lem:pos-linalg-diagonal-must-grow} as, roughly speaking, $r$-dimensional analogs of monotonicity. 

\begin{lemma}
    \label{lem:rank-v2}
    Let $A=(a_{i,j})_{i,j\in [r]} \in \bbR^{r\times r}$ satisfy $a_{i,i}>0$ and $a_{i,j} < 0$ for all $i\neq j$.
    \begin{enumerate}[label=(\alph*), ref=\alph*]
        \item 
        \label{it:zero-sum-v2}
        If $\sum_{j=1}^r a_{i,j}=0$ for all $i\in [r]$, then all solutions $\vv \in \bbR^r$ to $A\vv \preceq \eps\vone$ satisfy $|v_i-v_j| \le \eps / a_{\min}$ for all $i,j$, where $a_{\min} = \min_{i\neq j} |a_{i,j}|$.
        \item 
        \label{it:positive-sum-v2}
        If $\sum_{j=1}^r a_{i,j}\geq d_{\min}>0$ for all $i\in [r]$, then all solutions $\vv \in \bbR^r$ to $\tnorm{A\vv}_\infty \le \eps$ satisfy $\|v_i\|_{\infty} \le \eps / d_{\min}$.
    \end{enumerate}
\end{lemma}
\begin{proof}[Proof of Lemma~\ref{lem:rank-v2}]
    Assume without loss of generality that $v_1\geq v_s$ for all $s$. 
    If $\sum_{j=1}^r a_{i,j}=0$ for all $i\in [r]$, then
    \[
        \eps 
        \ge (A\vv)_1 
        = a_{1,1}v_1 + \sum_{j=2}^r a_{1,j}v_i
        = \sum_{j=2}^r |a_{1,j}|(v_1-v_j)
        \ge a_{\min} (v_1-v_i)
    \]
    for all $i\ge 2$.
    Thus $v_1-v_i \le \eps / a_{\min}$, proving the first part.
    For the second part, we will first show $v_1 \le \eps / d_{\min}$. 
    If $v_1 < 0$ there is nothing to prove, and otherwise
    \[
        \eps
        \ge 
        (A\vv)_1 
        = 
        a_{1,1}v_1 + \sum_{j=2}^r a_{1,j}v_j 
        \ge 
        \lt(a_{1,1} - \sum_{j=2}^r a_{1,j}\rt) v_1
        \ge 
        d_{\min} v_1.
    \]
    So $v_1 \le \eps / d_{\min}$, as claimed.
    Finally, note that if $\tnorm{A\vv}_\infty \le \eps$, the same is true for $-\vv$.
    By the same argument we find the largest entry of $-\vv$ is at most $\eps / d_{\min}$.
    This implies the second part.
\end{proof}
\begin{corollary}
    \label{cor:rank}
    Let $A=(a_{i,j})_{i,j\in [r]} \in \bbR^{r\times r}$ satisfy $a_{i,i}>0$ and $a_{i,j} < 0$ for all $i\neq j$.
    If $\sum_{j=1}^r a_{i,j}>0$ for all $i\in [r]$, then the only solution to $A\vv = \vzero$ is $\vv=\vzero$.
\end{corollary}
\begin{proof}
    Apply Lemma~\ref{lem:rank-v2}(\ref{it:positive-sum-v2}) with $\eps=0$. 
\end{proof}

To establish additional regularity we use the following fact on distributional derivatives.

\begin{lemma}[{See e.g. \cite[Theorem 2.2.1]{ziemer2012weakly}}]
    \label{lem:manual-integration-by-parts}
    If $A,B \in L^{\infty}([q_0,1])$ satisfy 
    \[
        \int_{q_0}^1 
        A(q) \psi(q) + B(q) \psi'(q)
        ~\de q 
        = 0
    \]
    for all $\psi \in C_c^\infty((q_0,1);\bbR)$, then there exists $C\in \bbR$ such that for all $q\in [q_0,1]$, 
    \[
        B(q) = \int_{q_0}^q A(t)~\de t + C.
    \]
\end{lemma}

We will make use of the functions
\begin{equation}
\label{eq:f-s-q}
    f_s(q) = \sqrt{\fr{\Phi'_s(q)}{(p\times \xi^s \circ \Phi)'(q)}}
\end{equation}
Note that Propositions~\ref{prop:basic-regularity} and \ref{prop:p-basic} imply $f_s$ is continuous on $[q_0+\eps,1]$.

\begin{proposition}
    \label{prop:psi}
    The functions $f_s$ are Lipschitz on $[q_0+\eps,1]$.
    Thus (recall Proposition~\ref{prop:basic-regularity}) the functions 
    \begin{equation}
    \label{eq:Psi-s}
    \Psi_s(q) = f'_s(q)/\Phi'_s(q)
    \end{equation}
    are measurable and locally bounded on $(q_0,1]$.
    Moreover for almost all $q\in (q_0,1]$, the following holds:
    \begin{equation}    
        \label{eq:psi-equality}
        \Psi_1(q)=\cdots=\Psi_s(q),
    \end{equation}
    and furthermore this common value is $0$ if $p'(q)>0$.
\end{proposition}
\begin{proof}
    Let $\psi \in C_c^\infty((q_0,1);\bbR)$.
    Consider the perturbation
    \[
        \tPhi_1(q) = \Phi_1(q) + \delta \psi(q),
    \]
    and let $\tPhi_s(q) = \Phi_s(q)$ for $s\neq 1$.
    By Proposition~\ref{prop:basic-regularity}, $\tPhi$ remains coordinate-wise increasing and Lipschitz for small positive and negative $\delta$.
    Although $\tPhi \not \in \hAdm(q_0,1)$, recalling Lemma~\ref{lem:admissible-optional} we nonetheless have $\bbA(p,\tPhi;q_0) \le \bbA(p,\Phi;q_0)$.
    Thus,
    \[
        F_1 \equiv \fr{\de}{\de \delta} \bbA(p,\tPhi;q_0) \Big|_{\delta=0} = 0.
    \]
    We now calculate $F_1$. 
    Note that
    \begin{equation}
        \label{eq:phi1-deriv}
        \fr{\de}{\de \delta} (p\times \xi^s \circ \tPhi)'(q) 
        \Big|_{\delta=0}
        = 
        (p\psi \times \partial_{x_1}\xi^s \circ \Phi)'(q)
        =
        \fr{\lambda_1}{\lambda_s}
        (p\psi \times \partial_{x_s}\xi^1 \circ \Phi)'(q).
    \end{equation}
    So,
    \begin{align*}
        0 = \fr{2}{\lambda_1} F_1
        &= 
        \int_{q_0}^1
        f_1(q)^{-1}
        \psi'(q)~\de q
        +
        \sum_{s\in \sS}
        \int_{q_0}^1
        f_s(q)
        (p\psi \times \partial_{x_s}\xi^1 \circ \Phi)'(q)
        ~\de q \\
        &= 
        \int_{q_0}^1
        A_1(q) \psi(q) + B_1(q) \psi'(q) 
        ~\de q
    \end{align*}
    where
    \[
        A_1(q)
        \equiv
        \sum_{s\in \sS}
        f_s(q)
        (p\times \partial_{x_s} \xi^1 \circ \Phi)'(q), 
        \quad 
        B_1(q) 
        \equiv
        f_1(q)^{-1}
        + \sum_{s\in \sS}
        f_s(q)
        (p\times \partial_{x_s} \xi^1 \circ \Phi)(q).
    \]
    By Proposition~\ref{prop:basic-regularity}, for all $\eps>0$ $A_1(q)$ and $B_1(q)$ are bounded for $q\in[q_0+\eps,1]$.
    Lemma~\ref{lem:manual-integration-by-parts} implies that $B_1(q)$ is absolutely continuous and $B'_1(q) = A_1(q)$ for all $q\in (q_0,1]$. In fact by Proposition~\ref{prop:basic-regularity}, $A_1$ is bounded and continuous on $[q_0+\eps,1]$, so $B_1\in C^1([q_0+\eps,1])$ (for all $\eps>0$).

    Fix $q\in (q_0,1]$. 
    For $\iota \in \bbR$ with $|\iota|$ small, let $\Delta^\iota_s = f_s(q+\iota) - f_s(q)$. 
    By Proposition~\ref{prop:basic-regularity} all $f_s$ are continuous, so $\Delta^\iota_s =o(1)$; here are below we use $o(\cdot)$ for limits as $\iota\to 0$.
    Thus,
    \begin{align*}
        B_1(q+\iota) - B_1(q)
        &= 
        \fr{1}{f_1(q)+\Delta^\iota_1} - \fr{1}{f_1(q)}
        + \sum_{s\in \sS}
        \Delta^\iota_s
        \cdot
        (p\times \partial_{x_s} \xi^1 \circ \Phi)(q) \\
        &\quad + \sum_{s\in \sS} f_s(q+\iota)\lt((p\times \partial_{x_s} \xi^1 \circ \Phi)(q+\iota) - (p\times \partial_{x_s} \xi^1 \circ \Phi)(q)\rt).
    \end{align*}
    Since $(p\times \partial_{x_s} \xi^1 \circ \Phi)$ is differentiable and $f_s$ is continuous,
    \begin{align*}
        \sum_{s\in \sS} 
        f_s(q+\iota)
        \cdot
        \big(
        (p\times \partial_{x_s} \xi^1 \circ \Phi)(q+\iota) - (p\times \partial_{x_s} \xi^1 \circ \Phi)(q)
        \big)
        &=
        \iota 
        \sum_{s\in \sS} f_s(q)(p\times \partial_{x_s} \xi^1 \circ \Phi)'(q) + o(|\iota|) \\
        &= \iota A_1(q) + o(|\iota|).
    \end{align*}
    Moreover,
    \begin{align*}
        \fr{1}{f_1(q)+\Delta^\iota_1} - \fr{1}{f_1(q)}
        &= 
        -\fr{\Delta^\iota_1}{f_1(q)(f_1(q)+\Delta^\iota_1)} 
        = 
        \fr{(\Delta^\iota_1)^2}{f_1(q)^2(f_1(q)+\Delta^\iota_1)}
        -\fr{\Delta^\iota_1}{f_1(q)^2} \\
        &= 
        \fr{(\Delta^\iota_1)^2}{f_1(q)^2(f_1(q)+\Delta^\iota_1)}
        - \fr{\Delta^\iota_1}{\Phi'_1(q)} \lt(
            p'(q)(\xi^1\circ \Phi)(q) + 
            \sum_{s\in \sS} (p\times \partial_{x_s} \xi^1 \circ \Phi)(q)\Phi'_s(q)
        \rt).
    \end{align*}
    We also have $B_1(q+\iota) - B_1(q) = A_1\iota + o(|\iota|)$ (recall $A_1$ is continuous).
    Thus
    \begin{equation}
        \label{eq:derivative-stability}
        \sum_{s\in \sS}
        (p\times \partial_{x_s} \xi^1 \circ \Phi)(q) \Phi'_s(q)
        \lt[\fr{\Delta^\iota_1}{\Phi'_1(q)} - \fr{\Delta^\iota_s}{\Phi'_s(q)}\rt]
        + p'(q) (\xi^1 \circ \Phi)(q) \fr{\Delta^\iota_1}{\Phi'_1(q)}
        - \fr{(\Delta^\iota_1)^2}{f_1(q)^2(f_1(q)+\Delta^\iota_1)}
        =
        o(|\iota|).
    \end{equation}
    We get similar equations from perturbing any $\Phi_{s}$ instead of $\Phi_1$.
    If $p'(q)>0$, then we can write the last two terms on the left-hand side of \eqref{eq:derivative-stability} as 
    \[
        \fr{\Delta^\iota_1}{\Phi'_1(q)} \lt(
            p'(q) (\xi^1 \circ \Phi)(q)
            - \fr{\Delta^\iota_1 \Phi'_1(q)}{f_1(q)^2(f_1(q)+\Delta^\iota_1)}
        \rt)
        = \fr{\Delta^\iota_1}{\Phi'_1(q)} \lt(p'(q) (\xi^1 \circ \Phi)(q) + o(1)\rt).
    \]
    Then, \eqref{eq:derivative-stability} and its analogs form a linear system in variables $x_s\equiv\Delta^\iota_s/\Phi'_s(q)$ with all row sums positive. (E.g. in \eqref{eq:derivative-stability}, the first term gives zero coefficient sum so the total coefficient sum is just $\lt(p'(q) (\xi^1 \circ \Phi)(q) + o(1)\rt)>0$.) Moreover the diagonal coefficients of this system are e.g.
    \[
    a_{1,1}
    =
    \lt(p'(q) (\xi^1 \circ \Phi)(q) + o(1)\rt)
    +
    \sum_{s\in \sS}
        (p\times \partial_{x_s} \xi^1 \circ \Phi)(q) \Phi'_s(q)>0
    \]
    while the off-diagonal coefficients are e.g.
    \[
    a_{1,s}
    =
    -
    (p\times \partial_{x_s} \xi^1 \circ \Phi)(q) \Phi'_s(q)
    <0.
    \]
    Applying Lemma~\ref{lem:rank-v2}(\ref{it:positive-sum-v2}), we obtain 
    \[
    |\Delta^\iota_s/\Phi'_s(q)| = o(|\iota|)
    \]
    for all $s\in \sS$.
    Taking $\iota \to 0$ we conclude that $f'_s(q)$ is well-defined and equals $0$.
    This implies the conclusion for $p'(q)>0$.

    Otherwise $p'(q)=0$, and \eqref{eq:derivative-stability} implies that
    \[
        \sum_{s\in \sS}
        (p\times \partial_{x_s} \xi^1 \circ \Phi)(q) \Phi'_s(q)
        \lt[\fr{\Delta^\iota_1}{\Phi'_1(q)} - \fr{\Delta^\iota_s}{\Phi'_s(q)}\rt]
        \ge 
        -o(|\iota|)
    \]
    and analogously with any $s\in \sS$ in place of $1$.
    This is a linear system of inequalities in variables 
    $-\Delta^\iota_s/\Phi'_s(q)$, so Lemma~\ref{lem:rank-v2}(\ref{it:zero-sum-v2}) implies that 
    \begin{equation}
        \label{eq:psi-equality-discretized}
        \lt|\fr{\Delta^\iota_s}{\Phi'_s(q)} - \fr{\Delta^\iota_{s'}}{\Phi'_{s'}(q)}\rt| \le o(|\iota|)
    \end{equation}
    for all $s,s'\in \sS$.
    The result now follows if we find a constant $C=C(\eps)$ such that 
    \begin{equation}
        \label{eq:psi-goal}
        \max_{s\in \sS} |\Delta^\iota_s/\Phi'_s(q)| \le C|\iota|
    \end{equation}
    for all sufficiently small $\iota$, and $q\in [q_0+\eps,1]$.
    Indeed, this would imply by Proposition~\ref{prop:basic-regularity} that $f_s$ is Lipschitz on $[q_0+\eps,1]$. It would then follow that $\Psi\in L^{\infty}([q_0+\eps,1])$, and we would conclude from \eqref{eq:psi-equality-discretized} that $\Psi_1=\cdots=\Psi_r$ almost everywhere. 
    
    Since $p$ and $\partial_{x_s}\xi^1$ are differentiable and $p'(q)=0$, we have $p(q+\iota) = p(q) + o(|\iota|)$ and $\partial_{x_s}\xi^1(q+\iota) = \partial_{x_s}\xi^1(q) + O(|\iota|)$.
    Suppose first that $\iota > 0$.
    Using $p'(q)=0$ and $p'(q+\iota)\geq 0$, we find
    \begin{align}
        \label{eq:use-p'-zero}
        \Delta^\iota_1
        &\le
        \sqrt{\fr{\Phi'_1(q+\iota)}{p(q+\iota)(\xi^1\circ \Phi)'(q+\iota)}}
        -\sqrt{\fr{\Phi'_1(q)}{p(q)(\xi^1\circ \Phi)'(q)}} \\
        \notag
        &= 
        \fr{1}{\sqrt{p(q)}} \lt(
            \sqrt{\fr{\Phi'_1(q+\iota)}{\sum_{s\in \sS} (\partial_{x_s}\xi^1 \circ \Phi)(q) \Phi'_s(q+\iota)}}
            - \sqrt{\fr{\Phi'_1(q)}{\sum_{s\in \sS} (\partial_{x_s}\xi^1 \circ \Phi)(q) \Phi'_s(q)}}
        \rt) + O(\iota)
    \end{align}
    where we used that $p'(q)=0$ and $p'(q+\iota) \ge 0$. 
    The hidden constants are uniform on any interval $[q_0+\eps,1]$.
    Analogous bounds hold for $\Delta^\iota_s$. 
    We claim that we cannot have 
    \begin{equation}
        \label{eq:all-ratios-bigger}
        \fr{\Phi'_s(q+\iota)}{\sum_{s'\in \sS} (\partial_{x_{s'}}\xi^s \circ \Phi)(q) \Phi'_{s'}(q+\iota)}
        > \fr{\Phi'_s(q)}{\sum_{s'\in \sS} (\partial_{x_{s'}}\xi^s \circ \Phi)(q) \Phi'_{s'}(q)}
    \end{equation}
    for all $s\in \sS$. 
    Indeed, suppose this holds and let 
    \begin{align*}
        b_s &= \fr{\sum_{s'\in \sS} (\partial_{x_{s'}}\xi^s \circ \Phi)(q) \Phi'_{s'}(q)}{\Phi'_s(q)}, \\
        b'_s &= \fr{\sum_{s'\in \sS} (\partial_{x_{s'}}\xi^s \circ \Phi)(q) \Phi'_{s'}(q+\iota)}{\Phi'_s(q+\iota)},
    \end{align*}
    so $b'_s < b_s$. 
    The linear system given by
    \[
        b'_s\Phi'_s(q+\iota)x_s - \lt(\sum_{s'\in \sS} (\partial_{x_{s'}}\xi^s \circ \Phi)(q) \Phi'_{s'}(q+\iota)x_{s'}\rt) = 0
    \]
    for all $s\in \sS$ has solution $\vx = \vone$, and thus has row sums zero. 
    The linear system given by 
    \[
        b_s\Phi'_s(q+\iota)x_s - \lt(\sum_{s'\in \sS} (\partial_{x_{s'}}\xi^s \circ \Phi)(q) \Phi'_{s'}(q+\iota)x_{s'}\rt) = 0
    \]
    has solution $x_s = \Phi'_s(q)/\Phi'_s(q+\iota)$.
    However, by Corollary~\ref{cor:rank} its only solution is $\vx=\vzero$, contradiction.
    Thus \eqref{eq:all-ratios-bigger} does not hold for all $s\in \sS$. 
    Assume without loss of generality \eqref{eq:all-ratios-bigger} does not hold for $s=1$. 
    Then, $\Delta^\iota_1 \le O(\iota)$.
    In conjunction with \eqref{eq:psi-equality-discretized}, this implies $\max_{s\in \sS} \Delta^\iota_s/\Phi'_s(q) \le C\iota$.
    
    For the matching lower bound, first consider the case $p'(q+\iota)=0$. 
    In this case, the inequality in \eqref{eq:use-p'-zero} is an equality.
    We similarly cannot have
    \[
        \fr{\Phi'_s(q+\iota)}{\sum_{s'\in \sS} (\partial_{x_{s'}}\xi^s \circ \Phi)(q) \Phi'_{s'}(q+\iota)}
        < \fr{\Phi'_s(q)}{\sum_{s'\in \sS} (\partial_{x_{s'}}\xi^s \circ \Phi)(q) \Phi'_{s'}(q)}
    \]
    for all $s\in \sS$, so the same argument implies $\min_{s\in \sS} \Delta^\iota_s/\Phi'_s(q) \ge -C\iota$, which implies \eqref{eq:psi-goal}.
    Otherwise assume $p'(q+\iota)>0$. 
    Let $\iota_1 \in (0, \iota/2)$ be small enough that 
    \begin{equation}
        \label{eq:p'-bdd-from-0}
        p'(q')\ge \fr12 p'(q+\iota) \quad \text{for all}~q'\in [q+\iota-\iota_1,q+\iota]
    \end{equation} 
    which exists by continuity of $p'$. 
    Let $\psi \in C_c^\infty((q_0,1);\bbR)$ satisfy that $|\psi'|\le 1$ and $\psi'$ is supported on $[q,q+\iota_1]\cup [q+\iota-\iota_1,q+\iota]$, positive on $[q,q+\iota_1]$, and negative on $[q+\iota-\iota_1,q+\iota]$.
    (Note that $\psi'$ integrates to zero because $\psi$ has bounded support, and that $\psi$ is clearly nonnegative.)
    Let $\iota_2 = \psi(q+\iota_1)$.
    Consider the perturbation $\wtp = p + \delta \psi$, which is increasing for small $\delta>0$ by \eqref{eq:p'-bdd-from-0}.
    Let $o_{\iota_1}(1)$ denote a term tending to $0$ as $\iota_1\to0$. 
    We compute that
    \begin{align*}
        F &\equiv 
        \fr{\de}{\de \delta} \bbA(\wtp,\Phi;q_0) \Big|_{\delta=0} \\
        &= 
        \sum_{s\in \sS}
        \lambda_s
        \int_{q_0}^1
        f_s(q) (\psi \times \xi^s \circ \Phi)'(q) \\
        &\ge 
        \sum_{s\in \sS}
        \lambda_s
        \int_{q_0}^1
        \psi'(q) f_s(q) (\xi^s \circ \Phi)(q) && \text{(positivity of $\psi$)}\\
        &=
        \sum_{s\in \sS}
        \lambda_s \cdot
        \iota_2 \lt(
            f_s(q) (\xi^s \circ \Phi)(q)
            - f_s(q+\iota) (\xi^s \circ \Phi)(q+\iota)
            +o_{\iota_1}(1)
        \rt) && \text{(continuity of $f_s, \xi^s\circ \Phi$)} \\
        &= \iota_2 \lt(
            -\sum_{s\in \sS} 
            \lambda_s
            \Delta^\iota_s (\xi^s \circ \Phi)(q)
            +o_{\iota_1}(1)
            +O(\iota)
        \rt) && \text{(continuity of $\xi^s\circ \Phi$)} \\
        &= \iota_2 \lt(
            -\Delta^\iota_1
            \sum_{s\in \sS} 
            \frac{\Phi_s'(q)}{\Phi_1'(q)}\cdot
            \lambda_s
            (\xi^s \circ \Phi)(q)
            +o_{\iota_1}(1)
            +O(\iota)
        \rt). && \text{(by \eqref{eq:psi-equality-discretized})}
    \end{align*}
    Since $(p,\Phi,q_0)$ is a maximizer, $F \le 0$. 
    This implies $\Delta^\iota_1 \ge -C\iota$, and by \eqref{eq:psi-equality-discretized}, $\min \Delta^\iota_1 \ge -C\iota$.
    This proves \eqref{eq:psi-goal} for $\iota > 0$. 
    The proof for $\iota < 0$ is analogous.
\end{proof}

\begin{lemma}
    \label{lem:positive-linalg-with-p}
    Let $A = (a_{i,j}) \in \bbR_{>0}^{r\times r}$, $\va, \vb \in \bbR_{>0}^r$, $\vc \in \bbR_{>0}^r$, and $c\in \bbR_{>0}$.
    Let $A_{\min},a_{\min},b_{\min}$ denote the minimal entries of $A,\va,\vb$, and $a_{\max}$ denote the maximal entry of $\va$.
    Suppose the linear system
    \[
        A\vx + \va y = \vc \odot \vx, 
        \quad 
        \la \vb,\vx\ra = c
    \]
    has solution $(y,\vx) = (y_0,\vone)$.
    If ${\vc\,}' \in \bbR_{>0}^r$ satisfies $\tnorm{\vc-{\vc\,}'}_\infty \le \eps$, then any solution $y\in \bbR_{\ge 0}$, $\vx \in \bbR_{\ge 0}^r$ to 
    \[
        A\vx + \va y = {\vc\,}' \odot \vx, 
        \quad 
        \la \vb,\vx\ra = c
    \]
    satisfies 
    \[
        |y-y_0| \le \fr{\eps c}{a_{\min}b_{\min}},
        \quad
        \tnorm{\vx-\vone}_\infty \le \fr{2a_{\max}}{a_{\min}} \cdot \fr{\eps c}{A_{\min}b_{\min}}.
    \]
\end{lemma}
\begin{proof}
    Without loss of generality let $x_1$, $x_2$ be the largest and smallest entries of $\vx$. 
    As $\la \vb,\vone\ra = \la \vb,\vx\ra = c$, 
    \[
    \fr{c}{b_{\min}}\ge x_1\ge 1\ge x_2.
    \]
    Then
    \begin{align*}
        0 
        &= a_1 y + \sum_{i=1}^r a_{1,i}x_i - c'_1x_1 \\
        &\le a_1 y + \lt(\sum_{i=1}^r a_{1,i}-c'_1\rt)x_1 - A_{\min}(x_1-x_2) \\
        &= a_1 y + \lt(c_1-c'_1-a_1y_0\rt)x_1 - A_{\min}(x_1-x_2) \\
        &\le \eps x_1 - a_1y_0(x_1-1) + a_1(y-y_0) - A_{\min}(x_1-x_2) \\
        &\le \fr{\eps c}{b_{\min}} + a_1(y-y_0) - A_{\min}(x_1-x_2).
    \end{align*}
    Analogously
    \begin{align*}
        0 
        &= a_2 y + \sum_{i=1}^r a_{2,i}x_i - c'_2x_2 \\
        &\ge a_2 y + \lt(\sum_{i=1}^r a_{2,i}-c'_2\rt)x_2 + A_{\min}(x_1-x_2) \\
        &= a_2 y + \lt(c_2-c'_2-a_2y_0\rt)x_2 + A_{\min}(x_1-x_2) \\
        &\ge -\eps x_2 - a_2y_0(x_2-1) + a_2(y-y_0) + A_{\min}(x_1-x_2) \\
        &\ge -\fr{\eps c}{b_{\min}} + a_2(y-y_0) + A_{\min}(x_1-x_2).
    \end{align*}
    Since $x_1-x_2\ge 0$, this implies
    \[
        y-y_0 
        \ge -\fr{\eps c}{a_1b_{\min}}
        \ge -\fr{\eps c}{a_{\min}b_{\min}},
        \quad 
        y-y_0 
        \le \fr{\eps c}{a_2b_{\min}}
        \le \fr{\eps c}{a_{\min}b_{\min}},
    \]
    which proves the first conclusion.
    Thus,
    \[
        A_{\min}(x_1-x_2)
        \le 
        \lt(\fr{\eps c}{b_{\min}} + a_1(y-y_0)\rt) 
        \le 
        \fr{2a_{\max}}{a_{\min}} \cdot \fr{\eps c}{b_{\min}}.
    \]
    Since $x_1\ge 1\ge x_2$, we have $\tnorm{\vx - \vone}_\infty \le x_1-x_2$ which implies the second conclusion.
\end{proof}

\begin{proposition}
    \label{prop:twice-diff}
    The functions $p'$ and $\Phi'$ are Lipschitz on $[q_0+\eps,1]$ for all $\eps>0$. 
    Thus $p''$ and $\Phi''$ are well-defined as bounded measurable functions on $[q_0+\eps,1]$.
\end{proposition}
\begin{proof}
    By Proposition~\ref{prop:psi}, $f_s$ is Lipschitz on $[q_0+\eps,1]$.
    Since it is also bounded on $[q_0+\eps,1]$ by Proposition~\ref{prop:basic-regularity}, $f_s^{-2}$ is Lipschitz as well. 
    Thus, for $q\in [q_0+\eps,1]$, $C = C(q)$, and sufficiently small $\iota \in \bbR$,
    \begin{align*}
        O(\iota)
        &\ge 
        |f_1(q+\iota)^{-2} - f_1(q)^{-2}| \\
        &= \bigg|
            \fr{p'(q+\iota)(\xi^1\circ \Phi)(q+\iota) + p(q+\iota)\sum_{s\in \sS} (\partial_{x_s} \xi^1 \circ \Phi)(q+\iota) \Phi'_s(q+\iota)}{\Phi'_1(q+\iota)} \\
            &\qquad -\fr{p'(q)(\xi^1\circ \Phi)(q) + p(q)\sum_{s\in \sS} (\partial_{x_s} \xi^1 \circ \Phi)(q) \Phi'_s(q)}{\Phi'_1(q)}
        \bigg| \\
        &= |C'_1-C_1+O(\iota)|
    \end{align*}
    for
    \begin{align*}
        C_1 &= \fr{p'(q)(\xi^1\circ \Phi)(q) + p(q)\sum_{s\in \sS} (\partial_{x_s} \xi^1 \circ \Phi)(q) \Phi'_s(q)}{\Phi'_1(q)}, \\
        C'_1 &= \fr{p'(q+\iota)(\xi^1\circ \Phi)(q) + p(q)\sum_{s\in \sS} (\partial_{x_s} \xi^1 \circ \Phi)(q) \Phi'_s(q+\iota)}{\Phi'_1(q+\iota)}.
    \end{align*}
    Thus $|C_1-C'_1| \le O(\iota)$.
    Similarly, $|C_s-C'_s|\le O(\iota)$ for analogously defined $C_s,C'_s$.
    Note that the system given by
    \begin{align}
        \label{eq:linear-eq-admissibility}
        1 &= \sum_{s\in \sS} \lambda_s \Phi'_s(q) x_s \\
        \label{eq:linear-eq-gs}
        C_1 \Phi'_1(q) x_1 &= (\xi^1\circ \Phi)(q) y + p(q)\sum_{s\in \sS} (\partial_{x_s} \xi^1 \circ \Phi)(q) \Phi'_s(q) x_s
    \end{align}
    and analogous equations to \eqref{eq:linear-eq-gs} with $s\in \sS$ in place of $1$ has solution $y=p'(q)$, $x_1=\cdots=x_r=1$.
    Moreover, the system given by \eqref{eq:linear-eq-admissibility},
    \begin{equation}
        \label{eq:linear-eq-g's}
        C'_1 \Phi'_1(q) x_1 = (\xi^1\circ \Phi)(q) y + p(q)\sum_{s\in \sS} (\partial_{x_s} \xi^1 \circ \Phi)(q) \Phi'_s(q) x_s
    \end{equation}
    and analogous equations to \eqref{eq:linear-eq-g's} with $s\in \sS$ in place of $1$ has solution $y=p'(q+\iota)$, $x_s = \Phi'_s(q+\iota)/\Phi'_s(q)$.
    Since $|C_s-C'_s|\le O(\iota)$ for all $s$, we may apply Lemma~\ref{lem:positive-linalg-with-p} 
    with $\vc=\vC,{\vc\,}'=\vC'$, $y$ taking the place of $p'(q)$ or $p'(q+\iota)$, and $A$ corresponding to the last term of \eqref{eq:linear-eq-gs} or \eqref{eq:linear-eq-g's}. The result is that
    \[
        |p'(q+\iota)-p'(q)|,
        \lt|\fr{\Phi'_s(q+\iota)}{\Phi'_s(q)}-1\rt| \le O(\iota).
    \]
    (The required constants $A_{\min},a_{\min},b_{\min},a_{\max}$ are bounded thanks to Propositions~\ref{prop:basic-regularity} and \ref{prop:p-basic}.)
    
    Since $\Phi'_s$ is bounded below by Proposition~\ref{prop:basic-regularity}, we conclude that $p', \Phi'$ are Lipschitz in a neighborhood of $q\in (q_0,1]$. 
    This Lipschitz constant is uniform on any $[q_0+\eps,1]$, thus $p',\Phi'$ are Lipschitz on these sets. 
\end{proof}

\begin{lemma}
    \label{lem:pos-linalg-diagonal-must-grow}
    Suppose $A = (a_{i,j}) \in \bbR_{>0}^{r\times r}$ and $\vb \in \bbR_{>0}^r$. 
    Let $A_{\max},A_{\min}$ be the largest and smallest entries of $A$. 
    Suppose the linear system $A\vx = \vb \odot \vx$ admits the solution $\vx = \vone$. 
    If $\vb' \preceq \vb + \eps \vone$, $A'\ge A$ entry-wise, and the system $A'\vx = \vb' \odot \vx$ admits a nontrivial solution $\vx \in \bbR^r_{\ge 0}$, then all entries of $A'-A$ are at most $\eps \cdot \fr{r A_{\max} + A_{\min} + \eps}{A_{\min}}$.
\end{lemma}
\begin{proof}
    Assume without loss of generality that $x_1$ is the smallest entry of $\vx$. 
    Let $\Delta_i = b'_i - b_i$ and $\Delta_{i,j} = a'_{i,j} - a_{i,j}$, so $\Delta_i \le \eps$, $\Delta_{i,j} \ge 0$.
    We have 
    \[
        0 
        = (b_1+\Delta_1)x_1 - \sum_{i=1}^r a'_{i,j}x_i
        \le \Delta_1 x_1 - \sum_{i=1}^r a_{i,j}(x_i-x_1).
    \]
    Thus $a_{i,j}(x_i-x_1) \le \Delta_1 x_1$ for all $i$. 
    If $x_1=0$, this implies $\vx = \vzero$, contradiction.
    Thus $x_1>0$ and we may scale $\vx$ such that $x_1=1$. 
    This implies
    \[
        1 \le x_i \le 1 + \fr{\Delta_1}{a_{i,j}} \le 1 + \fr{\eps}{A_{\min}}
    \]
    for all $i$.
    The equation $b'_jx_j = (A'\vx)_j$ implies
    \begin{align*}
        \sum_{i=1}^r \Delta_{j,i}x_i 
        = 
        b_jx_j + \Delta_jx_j - \sum_{i=1}^r a_{j,i}x_i 
        &\le 
        \lt(1 + \fr{\eps}{A_{\min}}\rt)
        \sum_{i=1}^r a_{j,i} 
        + \eps \lt(1 + \fr{\eps}{A_{\min}}\rt) - \sum_{i=1}^r a_{j,i} \\
        &\le \eps \cdot \fr{r A_{\max} + A_{\min} + \eps}{A_{\min}}.
    \end{align*}
    Since $x_i\ge 1$ for all $i$, this implies the result.
\end{proof}

Let $S\subseteq (q_0,1)$ be the set of $q$ for which \eqref{eq:psi-equality} holds, and for $q\in S$ let $\Psi(q)$ be the common value of the $\Psi_s(q)$.
Let $S_1 = \{q\in S : p'(q) > 0\}$ and $S_2 = S\setminus S_1$.
\begin{proposition}
    \label{prop:psi-negativity}
    Almost everywhere in $S_2$, $\Psi(q)<0$.
\end{proposition}
\begin{proof}
    Suppose for the sake of contradiction that $\Psi(q)\ge 0$ holds for a positive-measure set $T\subseteq S_2$. 
    Let $U\subseteq [q_0,1]$ be the set of $q$ which are Lebesgue points of $f'_s(q)$ for all $s\in \sS$.
    Since these functions are measurable and integrable on $[q_0+\eps,1]$ for all $\eps > 0$, $U$ is almost all of $[q_0,1]$.
    So $T\cap U$ has positive measure.
    Let $q \in T \cap U$.
    Thus
    \[
        \lim_{\iota \to 0^+}
        \fr{f_1(q+\iota)-f_1(q)}{\iota}
        =
        f'_1(q)
        = \Phi'_1(q)\Psi(q),
    \]
    which implies that for small $\iota>0$,
    \[
        f_1(q+\iota) 
        = 
        f_1(q) + \Phi'_1(q)\Psi(q)\iota + o(\iota)
        \ge 
        f_1(q) - o(\iota).
    \]
    Define
    \begin{align*}
        C_1 &= \fr{p(q) (\xi^1 \circ \Phi)'(q)}{\Phi'_1(q)} = f_1(q)^{-2}, \\
        C'_1 &= \fr{p(q+\iota) (\xi^1 \circ \Phi)'(q+\iota)}{\Phi'_1(q+\iota)} \le f_1(q+\iota)^{-2}.
    \end{align*}
    Thus $C'_1 \le C_1 + o(\iota)$.
    For analogously defined $C_s,C'_s$ we have $C'_s \le C_s + o(\iota)$. 
    Note that the system given by
    \[
        C_1\Phi'_1(q) x_1 = \sum_{s\in \sS} p(q) (\partial_{x_s} \xi^1 \circ \Phi)(q) \Phi'_s(q) x_q
    \]
    and analogous equations with $s\in \sS$ in place of $1$ has solution $\vx = \vone$, while the system
    \[
        C'_1\Phi'_1(q) x_1 = \sum_{s\in \sS} p(q+\iota) (\partial_{x_s} \xi^1 \circ \Phi)(q+\iota) \Phi'_s(q) x_q
    \]
    and analogous equations with $s\in \sS$ in place of $1$ has solution $x_s = \Phi'_s(q+\iota)/\Phi'_s(q)$.
    By Lemma~\ref{lem:pos-linalg-diagonal-must-grow} this implies that for all $s,s'\in \sS$, 
    \[
        p(q+\iota) (\partial_{x_s} \xi^{s'} \circ \Phi)(q+\iota) \le p(q) (\partial_{x_s} \xi^{s'} \circ \Phi)(q) + o(\iota).
    \]
    However, since $\xi$ is non-degenerate, $(\partial_{x_s} \xi^{s'} \circ \Phi)(q+\iota) \ge (\partial_{x_s} \xi^{s'} \circ \Phi)(q) + \Omega(\iota)$ for some $s,s'$. 
    This is a contradiction.
\end{proof}

\begin{lemma}
    \label{lem:s1-s2-separate}
    There exists $q_1\in [q_0,1]$ such that, up to modification by a measure zero set, $S_1 = [q_0,q_1]$ and $S_2 = [q_1,1]$.
\end{lemma}
\begin{proof}
    We will show that there do not exist positive measure subsets $I\subseteq S_1$, $J\subseteq S_2$ with $\sup J \le \inf I$.
    Suppose for contradiction that such subsets exist. 
    Define $q^* = \sup J$, $m = \int_I p'(q)~\de q$, and 
    \[
        \psi(q) = 
        \begin{cases}
            m (\int_{[q_0,q]\cap J} \de q)/(\int_J \de q) & q \le q^*, \\
            m - \int_{[q^*,q]\cap I} p'(q)~\de q & q > q^*.
        \end{cases}
    \]
    Note that $\psi$ is absolutely continuous, nonnegative-valued, and positive-valued almost everywhere in $J$. 
    Moreover $\psi(q_0) = \psi(1)=0$, and for small $\delta > 0$, the perturbation
    \begin{equation}
        \label{eq:perturb-p}
        \wtp(q) = p(q) + \delta \psi(q)
    \end{equation}
    remains increasing.
    Note that
    \[
        \fr{\de}{\de \delta}
        (p\times \xi^s \circ \Phi)'(q) = (\psi \times \xi^s \circ \Phi)'(q).
    \]
    Thus, integrating by parts,
    \begin{align*}
        F
        \equiv 
        2 \fr{\de}{\de \delta}
        \bbA(\wtp,\Phi;q_0) \Big|_{\delta=0} 
        &= 
        \sum_{s\in \sS}
        \int_{q_0}^1
        \sqrt{\fr{\Phi'_s(q)}{(p\times \xi^s \circ \Phi)'(q)}} 
        (\psi \times \xi^s \circ \Phi)'(q) 
        ~\de q\\
        &= 
        -\sum_{s\in \sS}
        \int_{q_0}^1
        \psi(q)(\xi^s\circ\Phi)(q) \Phi_s'(q) \Psi_s(q)
        ~\de q \\
        &= 
        -\sum_{s\in \sS}
        \int_{S_2}
        \psi(q)(\xi^s\circ\Phi)(q) \Phi_s'(q) \Psi(q)
        ~\de q.
    \end{align*}
    By Proposition~\ref{prop:psi-negativity}, $\Psi(q) < 0$ almost everywhere in $S_2$.
    Therefore $F > 0$ and the perturbation \eqref{eq:perturb-p} improves the value of $\bbA(p,\Phi;q_0)$, a contradiction. 
    
    Finally, define measures 
    \[
        \mu([q_0,q]) = \int_{[q_0,q]\cap S_1} \de q,
        \qquad 
        \nu([q_0,q]) = \int_{[q_0,q]\cap S_2} \de q.
    \]
    The non-existence of $I,J$ implies that $\max \supp (\mu) \le \min \supp(\nu)$.
    Since $S_1\cup S_2$ is almost all of $[q_0,1]$ the result follows.
\end{proof}

\begin{proof}[Proof of Proposition~\ref{prop:type-12}]
    That $p,\Phi_s\in W^{2,\infty}([q_0+\eps,1])$ follows from Proposition~\ref{prop:twice-diff}.
    By Lemma~\ref{lem:s1-s2-separate}, $p'>0$ almost everywhere on $[q_0,q_1]$.
    By Proposition~\ref{prop:psi}, $\Psi_s=0$ almost everywhere on $[q_0,q_1]$. 
    Since $f_s$ is Lipschitz, for all $q\in [q_0,q_1]$ we have
    \[
        f_s(q)-f_s(q_0) = \int_{q_0}^q f'_s(q)~\de q = \int_{q_0}^q \Phi'_s(q) \Psi_s(q)~\de q = 0.
    \]
    Thus $f_s(q)^{-2} = \fr{(p\times \xi^s \circ \Phi)'(q)}{\Phi'_s(q)}$ is constant on $[q_0,q_1]$.
    By Lemma~\ref{lem:s1-s2-separate} we have $p'=0$ almost everywhere on $[q_1,1]$, hence everywhere by Proposition~\ref{prop:twice-diff}. And by Proposition~\ref{prop:p-basic} we have $p(1)=1$. 
    Thus, for all $q\in [q_1,1]$, 
    \[
        p(1)-p(q) = \int_q^1 p'(q)~\de q = 0,
    \]
    so $p(q)=1$ for all $q\in [q_1,1]$. 
    Finally, by Proposition~\ref{prop:psi} and Lemma~\ref{lem:s1-s2-separate}, \eqref{eq:tree-descending-ode} is satisfied for all $s,s'$ almost everywhere on $[q_1,1]$.
\end{proof}

Given Proposition~\ref{prop:type-12}, it remains to study the behavior of $(p,\Phi)$ separately on $[q_0,q_1]$ and $[q_1,1]$ and establish the root-finding and tree-descending descriptions in Propositions~\ref{prop:root-finding-trajectory} and \ref{prop:tree-descending-trajectory}. We have seen that $(p,\Phi)$ are described by explicit differential equations on $[q_0,q_1]$ and $[q_1,1]$, and it will be important to understand both. We will refer to them as the type $\I$ and $\II$ equations respectively in Subsections~\ref{subsec:type-I-well-posed} and \ref{subsec:type-II}.

\subsection{Behavior in the Root-Finding Phase $1$: Super-solvability of $\Phi(q_1)$}

Let $q_0,q_1$ be given by Proposition~\ref{prop:type-12}, and let $L_s$ be the constant value of $(p\times \xi^s \circ \Phi)'(q)/\Phi'_s(q)$ on $[q_0,q_1]$, which exists by Proposition~\ref{prop:type-12}. 
The goal of this subsection is to prove that $\Phi(q_1)$ is super-solvable.

\begin{lemma}
    \label{lem:phi-q0-positivity}
    We have $\Phi_s(q_0)=0$ if and only if $h_s=0$. 
\end{lemma}
\begin{proof}
    Assume without loss of generality that $s=1$.
    First, suppose $h_1=0$ and $\Phi_1(q_0)>0$. 
    By admissibility, $q_0>0$.
    Consider the perturbation $\wtq_0 = q_0-\delta$, 
    \[
        \wtp(q)=
        \begin{cases}
            q-\wtq_0 & q\in [\wtq_0,q_0] \\
            \delta + (1-\delta)p(q) & q\in [q_0,1] \\
        \end{cases}
        \quad 
        \tPhi_s(q)=
        \begin{cases}
            \fr{q-\wtq_0}{\delta}\Phi_s(q_0) & q\in [\wtq_0,q_0], s=1 \\
            \Phi_s(q_0) & q\in [\wtq_0,q_0], s\neq 1 \\
            \Phi_s(q) & q\in [q_0,1] 
        \end{cases}
    \]
    for all $s\in \sS$.
    Then,
    \[
        \lambda_s
        \int_{\wtq_0}^{q_0}
        \sqrt{\tPhi'_s(q)(\wtp\times \xi^s \circ \tPhi)'(q)}~\de q 
        \ge 
        \begin{cases}
            \Omega(\delta^{1/2}) & s=1 \\
            0 & s\neq 1
        \end{cases}
    \]
    while for all $s\in \sS$, 
    \begin{align*}
        \lambda_s
        \int_{q_0}^1
        \sqrt{\tPhi'_s(q)(\wtp\times \xi^s \circ \tPhi)'(q)}~\de q 
        &\ge 
        \lambda_s
        \int_{q_0}^1
        \sqrt{\tPhi'_s(q)(\wtp\times \xi^s \circ \tPhi)'(q)}~\de q 
        -O(\delta) \\
        h_s\lambda_s \sqrt{\tPhi_s(\wtq_0)} 
        &= 
        h_s\lambda_s \sqrt{\Phi_s(q_0)}.
    \end{align*}
    Thus for small $\delta>0$ the perturbation improves the value of $\bbA$, contradiction.
    
    Conversely, suppose $h_1>0$ and $\Phi_1(q_0)=0$. 
    Consider the perturbation $(\wtp,\tPhi,\wtq_0)$ where $\wtq_0=q_0+\delta$ and $\wtp,\tPhi$ are $p,\Phi$ restricted to $[q_0+\delta,1]$.
    Note that $\tPhi_1(q_0) \ge \Omega(\delta)$ by Proposition~\ref{prop:basic-regularity}. 
    Thus 
    \begin{align*}
        h_1\lambda_1 \sqrt{\tPhi_1(q_0)} - h_1\lambda_1 \sqrt{\Phi_1(q_0)} &\ge \Omega(\delta^{1/2}), \\
        h_s\lambda_s \sqrt{\tPhi_s(q_0)} - h_s\lambda_s \sqrt{\Phi_s(q_0)} &\ge 0 \quad \forall s\neq 1.
    \end{align*}
    Furthermore, for all $s\in \sS$, 
    \begin{align*}
        &\lambda_s \int_{\wtq_0}^1\sqrt{\tPhi'_s(q)(\wtp\times \xi^s \circ \tPhi)'(q)}~\de q 
        - \lambda_s \int_{q_0}^1\sqrt{\Phi'_s(q)(p\times \xi^s \circ \Phi)'(q)}~\de q \\
        &=
        \lambda_s \int_{q_0}^{q_0+\delta}\sqrt{\Phi'_s(q)(p\times \xi^s \circ \Phi)'(q)}~\de q 
        = O(\delta).
    \end{align*}
    Thus for small $\delta>0$ the perturbation improves the value of $\bbA$, contradiction.
\end{proof}
\begin{corollary}
    \label{cor:q0-q1-nonzero}
    If $\vh \neq \vzero$, then $0<q_0<q_1$ and $\Phi(q_1) \in (0,1]^\sS$.
\end{corollary}
\begin{proof}
    Lemma~\ref{lem:phi-q0-positivity} implies $0<q_0$, so Proposition~\ref{prop:p-basic} implies $p(q_0)=0$. 
    Since $p(q_1)=1$ by Proposition~\ref{prop:type-12}, we have $q_0<q_1$.
    Proposition~\ref{prop:basic-regularity} gives $\Phi'(q) \succeq L^{-1}\vone$ for $q\in [q_0,q_1]$, so all coordinates of $\Phi(q_1)$ are positive.
\end{proof}

\begin{lemma}
    \label{lem:q0-q1-zero}
    If $\vh = \vzero$, then $q_0=q_1=0$ (and $\Phi(q_1)=\vzero$).
\end{lemma}
\begin{proof}
    By Lemma~\ref{lem:phi-q0-positivity}, $\Phi(q_0)=\vzero$ so $q_0=0$.
    Suppose that $q_1 > 0$.
    Then, for all $q\in [0,q_1]$, we have $L_s \Phi_s'(q) = (p\times \xi^s \circ \Phi)'(q)$, and by integrating $L_s \Phi_s(q) = p(q)(\xi^s \circ \Phi)(q)$.
    By Assumption~\ref{as:nondegenerate}, we can write $\xi^s(\vx) = \sum_{s'\in \sS} P_{s,s'}(\vx)x_{s'}$ where each $P_{s,s'}$ is a polynomial with nonnegative coefficients and positive constant and linear terms.
    Thus the functions $P_{s,s'} \circ \Phi$ are all strictly increasing.
    Let $0<q<q'<q_1$.
    The linear system
    \[
        L_s\Phi_s(q) x_s
        = 
        \sum_{s'\in \sS}
        p(q)(P_{s,s'} \circ \Phi)(q) \Phi_{s'}(q) x_s
        \quad 
        \forall s\in \sS
    \]
    has solution $\vx = \vone$, while the linear system
    \[
        L_s\Phi_s(q) x_s
        = 
        \sum_{s'\in \sS}
        p(q')(P_{s,s'} \circ \Phi)(q') \Phi_{s'}(q) x_s
        \quad 
        \forall s\in \sS
    \]
    has solution $x_s = \Phi_s(q')/\Phi_s(q)$.
    Monotonicity of $P_{s,s'} \circ \Phi$ implies $p(q')(P_{s,s'} \circ \Phi)(q') \ge p(q)(P_{s,s'} \circ \Phi)(q)$, so Lemma~\ref{lem:pos-linalg-diagonal-must-grow} (with $\eps=0$) implies that $p(q')(P_{s,s'} \circ \Phi)(q') = p(q)(P_{s,s'} \circ \Phi)(q)$ for all $s,s'$.
    This contradicts that the $P_{s,s'} \circ \Phi$ are strictly increasing.
\end{proof}

\begin{lemma}
    \label{lem:gs-value-with-field}
    If $h_s>0$, then $L_s = \fr{h_s^2}{\Phi_s(q_0)}$.
\end{lemma}
\begin{proof}
    Assume without loss of generality that $s=1$. 
    Consider the following perturbation $\tPhi$ of $\Phi$. 
    For all $s\neq 1$, $\tPhi_s=\Phi_s$, and $\tPhi_1(q)=\Phi_1(q)+\delta\psi(q)$  where $\psi \in C^\infty([q_0,1])$ with $\psi(q_0)=1$ and $\psi=0$ on $[q_1,1]$.
    This perturbation is not admissible, but we nonetheless have $\bbA(p,\tPhi;q_0) \le \bbA(p,\Phi;q_0)$ by Lemma~\ref{lem:admissible-optional}.
    
    Recall the calculation \eqref{eq:phi1-deriv}. 
    Integrating by parts, 
    \begin{align*}
        F_1
        &\equiv 2\lambda_1^{-1} \fr{\de}{\de \delta} \bbA(p,\tPhi;q_0)
        \Big|_{\delta=0} \\
        &= 
        \fr{h_1}{\sqrt{\Phi_1(q_0)}}
        +
        \int_{q_0}^1 
        L_1^{1/2} \psi'(q)~\de q
        +
        \sum_{s\in \sS}
        \int_{q_0}^1
        L_s^{1/2}
        (p\psi \times \partial_{x_s}\xi^1 \circ \Phi)'(q)
        = 
        \fr{h_1}{\sqrt{\Phi_1(q_0)}}
        - L_1^{1/2}.
    \end{align*}
    Recall that $\Phi'_1(q)$ is uniformly lower bounded by Proposition~\ref{prop:basic-regularity} and $\Phi_1(q_0)>0$ by Lemma~\ref{lem:phi-q0-positivity}. 
    So, this perturbation is valid for small positive and negative $\delta$. 
    Thus $F_1=0$ which implies the result. 
\end{proof}

\begin{proposition}
    \label{prop:gs-value}
    If $\vh\neq\vzero$, then for all $s$,
    \begin{equation}
        \label{eq:Ls-formula-final}
        L_s = \fr{(\xi^s \circ \Phi)(q_1) + h_s^2}{\Phi_s(q_1)},
    \end{equation}
    which is well-defined by Corollary~\ref{cor:q0-q1-nonzero}.
    Thus, $(p,\Phi)$ satisfies \eqref{eq:root-finding-ode} for all $s\in \sS$, $q\in [q_0,q_1]$ with $\vx = \Phi(q_1)$.
\end{proposition}

\begin{proof}
    Note that $\Phi_s(q_1)>0$ for all $s$ by Corollary~\ref{cor:q0-q1-nonzero} and Proposition~\ref{prop:basic-regularity}.
    Integrating the equation $(p\times \xi^s \circ \Phi)'(q) = L_s\Phi'_s(q)$ on $[q_0+\eps,q]$ and using continuity of $p$ and $\Phi$ and that $p(q_0)=0$, we find
    \begin{equation}
        \label{eq:type1-p-formula}
        p(q)(\xi^s \circ \Phi)(q) = L_s(\Phi_s(q)-\Phi_s(q_0)).
    \end{equation}
    Since $p(q_1)=1$ by Proposition~\ref{prop:type-12}, we have
    \begin{equation}
        \label{eq:type1-ode-aux}
        (\xi^s \circ \Phi)(q_1) = L_s(\Phi_s(q_1)-\Phi_s(q_0)).
    \end{equation}
    If $h_s=0$, by Lemma~\ref{lem:phi-q0-positivity} $\Phi_s(q_0)=0$, so $L_s = (\xi^s \circ \Phi)(q_1)/\Phi_s(q_1)$ as desired. 
    Otherwise, by Lemma~\ref{lem:gs-value-with-field}, $L_s = h_s^2 / (\lambda_s \Phi_s(q_0))$.
    Plugging this into \eqref{eq:type1-ode-aux} implies
    \begin{equation}
        \label{eq:type1-q0-formula}
        \Phi_s(q_0)\lt((\xi^s \circ \Phi)(q_1) + h_s^2\rt) = h_s^2 \Phi_s(q_1).
    \end{equation}
    Thus
    \[
        L_s = \fr{(\xi^s \circ \Phi)(q_1)}{\Phi_s(q_1)-\Phi_s(q_0)}
        = \fr{(\xi^s \circ \Phi)(q_1) + h_s^2}{\Phi_s(q_1)}
    \]
    as desired.
\end{proof}

\begin{corollary}
    \label{cor:alg-value}
    For $(p,\Phi;q_0)$ maximizing $\bbA$, we have
     \begin{equation}
        \label{eq:alg-functional-type1-simplification}
        \bbA(p,\Phi;q_0)
        =
        \sum_{s\in \sS}
        \lambda_s \lt[
            \sqrt{\Phi_s(q_1) (\xi^s(\Phi(q_1)) + h_s^2)}
            +
            \int_{q_1}^1
            \sqrt{\Phi'_s(q)(\xi^s \circ \Phi)'(q)} ~\de q
        \rt].
    \end{equation}
\end{corollary}
\begin{proof}
    If $\vh = \vzero$, then $q_1=0$ by Lemma~\ref{lem:q0-q1-zero}. 
    Thus, $p=1$ on $[0,1]$ by Proposition~\ref{prop:type-12}.
    Thus $(p\times \xi^s \circ \Phi)' = (\xi^s \circ \Phi)'$ and the result is clear. 
    Otherwise $\vh \neq \vzero$, and Corollary~\ref{cor:q0-q1-nonzero} implies $q_1>q_0$. 
    
    If $h_s=0$, then by Lemma~\ref{lem:phi-q0-positivity}, $\Phi_s(q_0)=0$. 
    So,
    \begin{align*}
        h_s\lambda_s \sqrt{\Phi_s(q_0)}
        +
        \lambda_s
        \int_{q_0}^{q_1}
        \sqrt{\Phi'_s(q) (p\times \xi^s \circ \Phi)'(q)}
        ~\de q 
        &= 
        \lambda_s
        \int_{q_0}^{q_1}
        \Phi'_s(q) \sqrt{L_s}
        ~\de q \\
        &= \lambda_s \Phi_s(q_1) \sqrt{L_s} = 
        \lambda_s \sqrt{\Phi_s(q_1) (\xi^s \circ \Phi)(q_1)},
    \end{align*}
    as desired. The last step uses Proposition~\ref{prop:gs-value}. 
    If $h_s>0$, then by Lemma~\ref{lem:gs-value-with-field} and Proposition~\ref{prop:gs-value},
    \begin{align*}
        h_s\lambda_s \sqrt{\Phi_s(q_0)}
        +
        \lambda_s
        \int_{q_0}^{q_1}
        \sqrt{\Phi'_s(q) (p\times \xi^s \circ \Phi)'(q)}
        ~\de q 
        &= 
        \lambda_s \lt[
            \Phi_s(q_0)\sqrt{L_s} + 
            \int_{q_0}^{q_1}\Phi'_s(q) \sqrt{L_s}~\de q 
        \rt] \\
        &= \lambda_s \Phi_s(q_1) \sqrt{L_s} \\
        &= \lambda_s \sqrt{\Phi_s(q_1) \lt((\xi^s \circ \Phi)(q_1) + h_s^2\rt)}.
    \end{align*}
\end{proof}
The following variant of this calculation determines the energy attained by $(p,\Phi;q_0)$ partway through the root-finding phase, and is used in Remark~\ref{rem:type1-partway}.
\begin{corollary}
    \label{cor:type1-partway}
    If $(p,\Phi;q_0)$ maximizes $\bbA$ and $q \in [q_0,q_1]$, then
    \[
        \sum_{s\in \sS} \lambda_s \lt[
            h_s \sqrt{\Phi_s(q_0)} + 
            \int_{q_0}^q \sqrt{\Phi'_s(t) (p\times \xi^s \circ \Phi)'(t)} ~\de t
        \rt]
        = \sum_{s\in \sS} \lambda_s \sqrt{\Phi_s(q) (p(q)(\xi^s \circ \Phi)(q) + h_s^2)}.
    \]
\end{corollary}
\begin{proof}
    If $h_s=0$, then by Lemma~\ref{lem:phi-q0-positivity}, $\Phi_s(q_0)=0$. 
    Then \eqref{eq:type1-p-formula} implies $L_s = p(q)(\xi^s \circ \Phi)(q) / \Phi_s(q)$. 
    So
    \[
        h_s \sqrt{\Phi_s(q_0)} + 
        \int_{q_0}^q \sqrt{\Phi'_s(t) (p\times \xi^s \circ \Phi)'(t)} ~\de t
        = (\Phi_s(q) - \Phi_s(q_0)) \sqrt{L_s} 
        = \sqrt{\Phi_s(q) p(q)(\xi^s \circ \Phi)(q)}.
    \]
    If $h_s>0$, \eqref{eq:type1-p-formula} implies and Lemma~\ref{lem:gs-value-with-field} imply
    \[
        p(q) (\xi^s \circ \Phi)(q) = \fr{h_s^2}{\Phi_s(q_0)} (\Phi_s(q) - \Phi_s(q_0)),
    \]
    which rearranges to 
    \[
        \fr{h_s^2 \Phi_s(q)}{\Phi_s(q_0)} = p(q)(\xi^s \circ \Phi)(q) + h_s^2.
    \]
    Then
    \begin{align*}
        h_s \sqrt{\Phi_s(q_0)} + 
        \int_{q_0}^q \sqrt{\Phi'_s(t) (p\times \xi^s \circ \Phi)'(t)} ~\de t
        &= h_s \sqrt{\Phi_s(q_0)} + (\Phi_s(q)-\Phi_s(q_0)) \sqrt{\fr{h_s^2}{\Phi_s(q_0)}} \\
        &= \fr{h_s \Phi_s(q)}{\sqrt{\Phi_s(q_0)}}
        = \sqrt{\Phi_s(q) (p(q)(\xi^s \circ \Phi)(q) + h_s^2)}.
    \end{align*}
    Summing over $s\in \sS$ completes the proof.
\end{proof}

\begin{lemma}
    \label{lem:q1-(super)-solvable}
    If $q_1=1$, then $\Phi(q_1)=\vone$ is super-solvable.
    If $q_1<1$, then $\Phi(q_1)$ is solvable.
\end{lemma}
\begin{proof}
    First suppose $q_1=1$. 
    Admissibility and the fact that $\Phi(1) \in [0,1]^\sS$ implies $\Phi(q_1)=\vone$.
    We have $p(q_1)=1$ by Proposition~\ref{prop:p-basic} and also $p'(q_1)\ge 0$.
    By Proposition~\ref{prop:gs-value},
    \begin{equation}
        \label{eq:prove-supersolvable}
        \fr{(\xi^s \circ \Phi)(q_1) + h_s^2}{\Phi_s(q_1)}
        = \fr{(p\times \xi^s \circ \Phi)'(q_1)}{\Phi'_s(q_1)}
        \ge \fr{\sum_{s'\in \sS} (\partial_{x_{s'}}\xi^s \circ \Phi)(q_1)\Phi'_{s'}(q_1)}{\Phi'_s(q_1)}.
    \end{equation}
    This implies via Corollary~\ref{cor:solvability-equivalent} (with $\Phi'$ in the role of $\vv$) that $\Phi(q_1)$ is super-solvable. 

    Now suppose $q_1<1$.
    If $\vh=\vzero$ the result follows from Lemma~\ref{lem:q0-q1-zero}, so assume $\vh\neq\vzero$. 
    Because $p(q)=1$ on $[q_1,1]$ and $p'$ is continuous (Proposition~\ref{prop:basic-regularity}), $p'(q_1)=0$. 
    So, the inequality in \eqref{eq:prove-supersolvable} is an equality.
    Thus $\Phi'(q_1)$ is in the null space of $M^*(\Phi(q_1))$, and thus (by \eqref{eq:M*sym-to-M*}) of $M^*_\sym(\Phi(q_1))$.
    So $M^*_\sym(\Phi(q_1))$ is singular and $\Phi(q_1)$ is solvable.
\end{proof}

\subsection{Behavior in the Root-Finding Phase $2$: Well-Posedness}
\label{subsec:type-I-well-posed}

In this subsection we prove Proposition~\ref{prop:root-finding-trajectory} and give a detailed characterization of $(p,\Phi)$ on $[q_0,q_1]$ in Proposition~\ref{prop:type-1}. Recalling Propositions~\ref{prop:gs-value} and \ref{lem:q1-(super)-solvable}, we consider a path $(p,\Phi)$ defined by the \textbf{type $\I$ equation}
\begin{equation}
\label{eq:type-1-traj}
\begin{aligned}
    \fr{(p\times \xi^s \circ \Phi)'(q)}{\Phi'_s(q)} 
    &= L_s= \fr{(\xi^s \circ \Phi)(q_1) + h_s^2}{\Phi_s(q_1)}
    ,\quad\forall s\in\sS
    \\  
    \Phi_s'(q)&\geq 0, \quad 
    \la \vlam, \Phi'(q)\ra = 1
\end{aligned}
\end{equation}
with super-solvable initial condition $\Phi(q_1)$ and $p(q_1)=1$.  We start by verifying the first part of Proposition~\ref{prop:root-finding-trajectory}, namely that $\vh \neq \vzero$ if and only if there exists a super-solvable point $\vx \in [0,1]^\sS$ with $\la \vlam, \vx\ra >0$.

\begin{proof}[Proof of Proposition~\ref{prop:root-finding-trajectory} (first claim)]
    First, assume $\vh \neq \vzero$.
    We will show that all $\vx \in [\delta/2,\delta]^\sS$ are super-solvable for $\delta > 0$ sufficiently small. 
    Assume without loss of generality that $h_1>0$.
    Note that for all $s\in \sS$, 
    \[
        (M^*(\vx) \vx)_s 
        = x_s \lt( h_s^2 + \xi^s(\vx) - \sum_{s'\in \sS} x_{s'} \partial_{x_{s'}} \xi^s(\vx)\rt)
        = x_s \lt( h_s^2 - O(\delta^2)\rt).
    \]
    Moreover, $(M^*(\vx) \ve_1)_1 \le h_1^2 + O(\delta)$, while for $s\neq 1$,
    \[
        (M^*(\vx) \ve_1)_s
        = 
        - x_s \partial_{x_1}\xi^s(\vx).
    \]
    Thus, for $\vv = \vx - \fr12 x_1 \ve_1$, we have 
    \[
        (M^*(\vx) \vv)_1
        \ge x_1 \lt( \fr12 h_1^2 - O(\delta)\rt)
        \ge 0
    \]
    and for $s\neq 1$, 
    \[
        (M^*(\vx) \vv)_s
        \ge 
        x_s\lt( x_1 \partial_{x_1}\xi^s(\vx) - O(\delta^2)\rt)
        \ge 0.
    \]
    This implies by Corollary~\ref{cor:solvability-equivalent} that $\vx$ is super-solvable.

    If $\vh=\vzero$, fix any $\vx \in (0,1]^\sS$. 
    Note that 
    \[
        \vx^\top M^*_\sym(\vx) \vx 
        = \sum_{s\in \sS} x_s \partial_{x_s}\xi(\vx) 
        - \sum_{s,s'\in \sS} x_s x_{s'} \partial_{x_s,x_{s'}}\xi(\vx) < 0,
    \]
    as any monomial of $\xi(\vx)$ with total degree $p\ge 2$ appears with multiplicity $p$ in the first sum and $p(p-1) \ge p$ in the second, with strict inequality for any $p>2$.
    Thus $\vx$ is strictly sub-solvable. 
\end{proof}

\begin{proposition}
\label{proposition:Lambda-def-of-type-I}
    Define for $(p(q),p'(q),\Phi(q))\in [0,1]^{\sS}\times [0,1]\times \mathbb R$ the $\sS\times\sS$ matrix $M(p(q),p'(q),\Phi(q))$ with entries
\[
    M(p(q),p'(q),\Phi(q))_{s,s'}=\frac{p(q)\partial_{x_{s'}}\xi^s\lt(\Phi(q)\rt)
    +
    \lambda_{s'} p'(q)\xi^s(\Phi(q))}{L_s}
    ,\quad\quad
    s,s'\in\sS.
\]
If $(p,\Phi)$ solves \eqref{eq:type-1-traj} then $\Lambda(M(p,p',\Phi))=1$ with Perron-Frobenius eigenvector $\Phi'(q)$.
\end{proposition}

\begin{proof}
    It suffices to expand the left-hand side of the top line of \eqref{eq:type-1-traj}:
\begin{equation}
\label{eq:type-1-traj-rewrite} 
    p'(q)\xi^s(\Phi(q))\left(\sum_{s'\in\sS}\lambda_{s'} \Phi_{s'}'(q)\right) + p(q)\sum_{s'\in\sS} \Phi_{s'}'(q) \partial_{x_{s'}} \xi^s \lt(\Phi(q)\rt)=L_s\Phi_s'(q),\quad \forall s\in\sS.
\end{equation}
    Rearranging shows that $M(p,p',\Phi)\Phi'(q)=\Phi'(q)$, and it is clear that $M$ has non-negative entries.
\end{proof}

We now show the ODE \eqref{eq:type-1-traj} is well-posed.

\begin{lemma}
\label{lem:ODE-Lipschitz}
Fix $q\in [0,1]$ and let $Y(q)=(p(q),\Phi(q))$ and $L_s>0$ be arbitrary. The equation \eqref{eq:type-1-traj} for any fixed $q$ is equivalent to
\[
    Y'(q)=F(Y(q))
\]
for a locally Lipschitz function $F:[0,1]\times \lt([0,1]^r\backslash \vzero\rt)\to \bbR^{r+1}$.
\end{lemma}

\begin{proof}
Let $M$ be as in Proposition~\ref{proposition:Lambda-def-of-type-I}. Because $\xi$ is non-degenerate,
Propositions~\ref{prop:VM} and \ref{proposition:Lambda-def-of-type-I} imply existence of $c>0$ such that 
\[
M(p,x+y,\Phi)\geq M(p,x,\Phi)+cy
\]
holds entrywise for all $x,y\geq 0$, as long as $\Phi(q)\in \bbR_{\geq 0}^{r}\backslash [0,\eps]^{r}$. Therefore a unique value $p'(q)$ solving \eqref{eq:type-1-traj} exists. Moreover $M$ is locally Lipschitz in $(p,\Phi)$, so if
\[
    M(p,p',\Phi)=M(\wt\Phi,\wtp,\wtp')
\]
then
\[
    |p'-\wtp'|\leq O(\|\Phi-\wt\Phi\|_{L^{\infty}}+|p-\wtp|).
\]
(With implicit constant depending on $\eps$ as introduced above.)
This shows that $p'$ has locally Lipschitz dependence on $Y=(p,\Phi)$.
It remains to show $\Phi'$, defined by the resulting solution to \eqref{eq:type-1-traj-rewrite}, also has locally Lipschitz dependence on $Y$. This follows by Proposition~\ref{prop:perron-eigenvector-lipschitz} below. (Note that all entries of $M$ are of the same order up to constants for $\Phi(q)\in \bbR_{\geq 0}^{r}\backslash [0,\eps]^{r}$ by non-degeneracy of $\xi$.)
\end{proof}

\begin{proposition}[{\cite[Lemma 27]{yeo2018frozen}}]
\label{prop:perron-eigenvector-lipschitz}
Let $\cM\subseteq \bbR_{\geq 0}^{r\times r}$ be a compact set of square matrices all of whose Perron-Frobenius eigenvalues have multiplicity $1$. Let $M,\wt M\in\cM$ have entrywise positive Perron-Frobenius eigenvectors $v,\wt v$, normalized so that $\|v\|_1=\|\wt v\|=1$. Then 
\[
    \|v-\wt v\|_{1}\leq O_{\cM}(\|M-\wt M\|_1).
\]
In particular, this holds for $\cM=[c,C]^{r\times r}$ for any $0<c<C<\infty$.
\end{proposition}

Lemma~\ref{lem:ODE-Lipschitz} shows that for any right endpoint $(p(q_1),\Phi(q_1))$, it is possible to solve \eqref{eq:type-1-traj} backwards in time until $q_*$ when $\Phi(q)$ reaches the boundary of $\bbR_{\geq 0}^r$, or at which $p(q)$ reaches $0$. We now show that the latter occurs first.

\begin{lemma}
\label{lem:Phi-linear-LB}
There exists $c>0$ such that for any super-solvable point $\Phi(q_1)$, the solution to the type $\I$ equation \eqref{eq:type-1-traj} on $[q_*,q_1]$ satisfies
\[
    \Phi_{s}(q)\geq cp(q)q.
\]
Moreover $p(q_*)=0$ and Lemma~\ref{lem:phi-q0-positivity} holds for $q_*$, i.e. $h_s>0$ if and only if $\Phi_s(q_*)>0$.
\end{lemma}

\begin{proof}
Observe that in \eqref{eq:type-1-traj}, we have
\[
    L_s\geq K_s\equiv \frac{\xi^s(\Phi(q_1))}{\Phi_s(q_1)}.
\]
Therefore on $q\in[q_{\eps},q_1]$, the left-hand equation in \eqref{eq:type-1-traj} implies
\[
    \frac{(p\times \xi^s\circ \Phi)(q)}{\Phi_s(q)}
    \leq 
    K_s.
\]
Recall that $\xi^s$ is non-degenerate, and so admissibility and $\Phi\succeq 0$ implies $\xi^s(\Phi(q))=\Theta(q).$ Hence for some $c>0$ and all $s\in\sS$,
\[
    \Phi_{s}(q)\geq \Omega(p(q)q/K_s)\geq cp(q)q.
\]
This concludes the proof of the first statement, which implies that $p(q_*)=0$. 

For the second, note that strict inequality holds in the first step if $h_s>0$, and so $p$ must reach $0$ before $\Phi_s$ does. On the other hand if $h_s=0$, then it is easy to see from \eqref{eq:type-1-traj} that $p$ cannot reach zero strictly sooner than $\Phi_s$, hence the numerator and denominator on the left-hand side in \eqref{eq:type-1-traj} both reach zero at time $q_*$.
\end{proof}

\begin{lemma}
\label{lem:p-concave}
If $\Phi(q_1)$ is super-solvable, then the $p$ solving \eqref{eq:type-1-traj-rewrite} is increasing and concave on $[q_*,q_1]$.  Moreover $p,\Phi_s\in C^1([q_*,q_1])$.
\end{lemma}

\begin{proof}
We claim that $p'$ is decreasing. The key point is that with $M$ as in Proposition~\ref{proposition:Lambda-def-of-type-I}, 
\[
    M(p,p',\Phi)< M(\wt\Phi,\wtp,\wtp')
\]
if $\Phi \preceq \wt\Phi$, $p \leq \wtp$ and $p'<\wtp'$. 
Indeed this is immediate by Proposition~\ref{prop:VM}.
It follows that $p'$ must increase backward in time, i.e. $p'(q)$ is a decreasing function. Since $p'(q_1)\geq 0$ by super-solvability, this completes the proof.
\end{proof}

\begin{proof}[Proof of Proposition~\ref{prop:root-finding-trajectory}, parts (\ref{it:unique-root-finding},\ref{it:unique-root-finding-q0})]
    Existence and uniqueness of the root-finding trajectory follows from Lemma~\ref{lem:ODE-Lipschitz} and Proposition~\ref{prop:ODE-well-posed}. Lemma~\ref{lem:Phi-linear-LB} ensures that the solution exists until $p$ reaches $0$. Concavity of $p$ was just shown in Lemma~\ref{lem:p-concave}. 
    This proves part (\ref{it:unique-root-finding}). 
    Part (\ref{it:unique-root-finding-q0}) follows from Lemma~\ref{lem:phi-q0-positivity} or \ref{lem:Phi-linear-LB}.
\end{proof}

\begin{lemma}
    \label{lem:vone-super-solvable}
    If $\vone$ is super-solvable, then $q_1=1$ and $\Phi(q_1)=\vone$.
    Otherwise $q_1 < 1$.
\end{lemma}
\begin{proof}
    If $\vone$ is strictly sub-solvable, Lemma~\ref{lem:q1-(super)-solvable} implies that $q_1<1$.
    Suppose $\vone$ is super-solvable.
    Let $(p^*,\Phi^*,q_0^*)$ be the root-finding trajectory with endpoint $\vone$, which exists by Proposition~\ref{prop:root-finding-trajectory}.
    By Corollary~\ref{cor:alg-value},
    \[
        \bbA(p^*,\Phi^*;q_0^*) 
        = 
        \sum_{s\in \sS} 
        \lambda_s 
        \sqrt{\xi^s(\vone) + h_s^2}.
    \]
    Suppose for contradiction that there is a different maximizer $(p,\Phi,q_0)$ of $\bbA$ with $\bbA(p,\Phi;q_0) \ge \bbA(p^*,\Phi^*;q_0^*)$. The maximizer $(p,\Phi,q_0)$ has its own value $q_1$, and we must have $q_1<1$ since for this to be a different maximizer.
    Note that for each $s\in \sS$, 
    \begin{align*}
        &\sqrt{\xi^s(\vone) + h_s^2} - \sqrt{\Phi_s(q_1)(\xi^s(\Phi(q_1)) + h_s^2)} 
        = \int_{q_1}^1 
        \fr{\de}{\de q}
        \sqrt{\Phi_s(q)((\xi^s\circ\Phi)(q) + h_s^2)}
        ~\de q \\
        &= \fr12 \int_{q_1}^1 \lt(
            \Phi'_s(q)
            \sqrt{\fr{(\xi^s\circ\Phi)(q) + h_s^2}{\Phi_s(q)}}
            + (\xi^s \circ \Phi)'(q) 
            \sqrt{\fr{\Phi_s(q)}{(\xi^s\circ\Phi)(q) + h_s^2}}
        \rt) ~\de q.
    \end{align*}
    By Corollary~\ref{cor:alg-value},
    \begin{align*}
        F &\equiv 
        \bbA(p^*,\Phi^*;q_0^*)
        -
        \bbA(p,\Phi;q_0) \\
        &=
        \sum_{s\in \sS}
        \fr{\lambda_s}{2}
        \int_{q_1}^1
        (\xi^s \circ \Phi)'(q) 
        \sqrt{\fr{\Phi_s(q)}{(\xi^s\circ\Phi)(q) + h_s^2}}
        \lt(
            \sqrt{
                \fr{\Phi'_s(q)}{(\xi^s\circ\Phi)'(q)}
                \cdot 
                \fr{(\xi^s\circ\Phi)(q) + h_s^2}{\Phi_s(q)}
            }
            - 1
        \rt)^2
        \de q
        \ge 0.
    \end{align*}
    Since $\bbA(p,\Phi;q_0) \ge \bbA(p^*,\Phi^*;q_0^*)$, we have $F=0$.
    So, for all $s\in \sS$, and almost all $q\in (q_1,1]$
    \[
        \fr{(\xi^s\circ\Phi)'(q)}{(\xi^s\circ\Phi)(q) + h_s^2}
        =
        \fr{\Phi'_s(q)}{\Phi_s(q)}
        \qquad 
        \Rightarrow 
        \qquad 
        \fr{\de}{\de q} \log \lt((\xi^s\circ\Phi)(q) + h_s^2\rt) 
        = 
        \fr{\de}{\de q} \log \Phi_s(q).
    \]
    Both sides of this equation are continuous on $(q_1,1]$, so in fact it holds for all $q\in (q_1,1]$.
    Thus there exist constants $C_s$ such that
    \[
        (\xi^s\circ\Phi)(q) + h_s^2
        = 
        C_s \Phi_s(q).
    \]
    Thus, for $q_1 < q < q+\iota \le 1$, we have
    \begin{align*}
        C_s\Phi'_s(q)
        &= 
        (\xi^s \circ \Phi)'(q) 
        =
        \sum_{s'\in \sS}
        (\partial_{x_{s'}} \xi^s \circ \Phi)(q)
        \Phi'_{s'}(q)
        \quad \forall s\in \sS, \\ 
        C_s\Phi'_s(q+\iota)
        &=
        (\xi^s \circ \Phi)'(q+\iota) 
        =
        \sum_{s'\in \sS}
        (\partial_{x_{s'}} \xi^s \circ \Phi)(q+\iota)
        \Phi'_{s'}(q+\iota)
        \quad \forall s\in \sS,
    \end{align*}
    We treat these equations as linear systems in $\Phi'(q)$ and $\Phi'(q+\iota)$.
    Since both linear systems have nonnegative solutions and $(\partial_{x_{s'}} \xi^s \circ \Phi)(q+\iota) \ge (\partial_{x_{s'}} \xi^s \circ \Phi)(q)$ for all $s,s'$, Lemma~\ref{lem:pos-linalg-diagonal-must-grow} (with $\eps=0$) implies that $(\partial_{x_{s'}} \xi^s \circ \Phi)(q+\iota) = (\partial_{x_{s'}} \xi^s \circ \Phi)(q)$ for all $s,s'$. 
    This contradicts that $\xi$ is non-degenerate and completes the proof.
\end{proof}

\begin{proposition}
    \label{prop:type-1}
    The following assertions hold. 
    \begin{enumerate}[label=(\alph*), ref=\alph*]
        \item \label{itm:s4-supsolvable} If $\vone$ is super-solvable, then $0 < q_0 < q_1 = 1$ (and thus $\Phi(q_1)=\vone$).
        \item \label{itm:s4-subsolvable-with-field} If $\vone$ is sub-solvable and $\vh \neq \vzero$, then $0 < q_0 < q_1 < 1$ and $\Phi(q_1) \in (0,1]^\sS$.
        \item \label{itm:s4-subsolvable-no-field} If $\vh = \vzero$, then $\vone$ is sub-solvable and $0=q_0=q_1$ (and thus $\Phi(q_1) = \vzero$).
    \end{enumerate}
    In cases (\ref{itm:s4-subsolvable-with-field}, \ref{itm:s4-subsolvable-no-field}), $\Phi(q_1)$ is solvable. 
    In cases (\ref{itm:s4-supsolvable}, \ref{itm:s4-subsolvable-with-field}) (and vacuously in case (\ref{itm:s4-subsolvable-no-field})) $(p,\Phi)$ restricted to $[q_0,q_1]$ is the root-finding trajectory with endpoint $\Phi(q_1)$.
\end{proposition}

\begin{proof}[Proof of Proposition~\ref{prop:type-1}]
    If $\vone$ is super-solvable, Lemma~\ref{lem:vone-super-solvable} implies $q_1=1$. 
    Comparing Corollary~\ref{cor:q0-q1-nonzero} and Lemma~\ref{lem:q0-q1-zero} gives $q_0>0$.
    If $\vone$ is sub-solvable and $\vh\neq\vzero$, Lemma~\ref{lem:vone-super-solvable} implies $q_1<1$ while Corollary~\ref{cor:q0-q1-nonzero} implies $0<q_0<q_1$ and $\Phi(q_1) \in (0,1]^\sS$. 
    If $\vh=\vzero$, Lemma~\ref{lem:q0-q1-zero} implies $0=q_0=q_1$.
    This proves assertions (\ref{itm:s4-supsolvable}, \ref{itm:s4-subsolvable-with-field}, \ref{itm:s4-subsolvable-no-field}). 

    In cases (\ref{itm:s4-subsolvable-with-field}, \ref{itm:s4-subsolvable-no-field}), since $q_1<1$, Lemma~\ref{lem:q1-(super)-solvable} implies $\Phi(q_1)$ is solvable.
    In cases (\ref{itm:s4-supsolvable}, \ref{itm:s4-subsolvable-with-field}), Proposition~\ref{prop:gs-value} implies $(p,\Phi)$ restricted to $[q_0,q_1]$ is the root-finding trajectory with endpoint $\Phi(q_1)$. 
\end{proof}

\subsection{Behavior in the Tree-Descending Phase}
\label{subsec:type-II}

The next lemma, proved in Appendix~\ref{subsec:type-II-Lipschitz}, shows the tree-descending ODE is also well-posed. 

\begin{restatable}{lemma}{lemtypeIILipschitz}
\label{lem:type-II-Lipschitz}
Fix $\eps>0$. For $\Phi(q)\in \bbR_{\geq 0}^{\sS}$ and $\Phi'(q)\in A_{\geq 0}(q)$, the \textbf{type $\II$ equation}
\begin{align*}
    \Psi_s(q) &=\Psi_{s'}(q)
    \quad\forall s,s'\in\sS
    ;
    \\
    \la \vec\lambda,\Phi''(q)\ra &= 0
\end{align*}
is equivalent (for each fixed $q$) to 
\[
    \Phi''(q)=F(\Phi(q),\Phi'(q))
\]
for a locally Lipschitz function $F:\mathbb R_{\geq 0}^{\sS}\times A_{\geq 0}^{\sS}\to\mathbb R^{\sS}$. Moreover, \begin{equation}
\label{eq:Phi-stays-increasing}
|\Phi''_s(q)|\leq O(|\Phi_s'(q)|),\quad \forall s\in \sS.
\end{equation}
with a uniform constant for bounded $\Phi'(q)$.
\end{restatable}

\begin{lemma}
\label{lem:type-II-well-posed-appendix}
    The type $\II$ equation has a unique solution on $q\in [q_1,1]$ for any initial condition $(\Phi(q_1),\Phi'(q_1))\in \bbR_{\geq 0}^{\sS}\times A_{\geq 0}$. This solution satisfies $\Phi'(q)\succeq 0$ for all $q$.
\end{lemma}

\begin{proof}
The result now follows from Proposition~\ref{prop:ODE-well-posed}, since \eqref{eq:Phi-stays-increasing} implies that $\Phi'_s(q)$ stays non-negative for all $s$, and stays strictly positive if $\Phi_s'(q_1)>0$. 
\end{proof}

\begin{proof}[Proof of Proposition \ref{prop:tree-descending-trajectory}]
    Given the above, it only remains to show existence and uniqueness of $\vv$. Consider the matrix
    \[
    M(\vx)_{s,s'}
    =
    \frac{\partial_{x_{s'}}\xi^s(\vx)}
    {\xi^s(\vx)+h_s^2}.
    \]
    Then $M$ has strictly positive entries by non-degeneracy.
    The equation $M^*_\sym(\vx)\vv=\vzero$ is equivalent to $M^*(\vx)\vv=\vzero$ by \eqref{eq:M*sym-to-M*}, which is in turn equivalent to
    \[
    M(\vx)\vv=\vv.
    \]
    Since $\vx\neq \vzero$, non-degeneracy of $\xi$ implies that $\xi^s(\vx)>0$ so there is no division by $0$. Hence any such $\vv$ is uniquely determined as the Perron-Frobenius eigenvector of $M$.
    Conversely it is easy to see that if $M$ has Perron-Frobenius eigenvector \textbf{not} equal to $1$ then $M^*$ would not be solvable, which ensures that $\vv$ as above exists. 
\end{proof}

\begin{corollary}
\label{cor:regular-final}
    $p,\Phi_s\in C^1([q_0,1])$ and their restrictions to $[q_1,1]$ are $C^2$.
\end{corollary}

\begin{proof}
    From Proposition~\ref{prop:basic-regularity}, for the first statement it suffices to verify continuity of $p',\Phi_s'$ at $q_0$. If $\vh\neq \vzero$ this follows by Lemmas~\ref{lem:ODE-Lipschitz} and \ref{lem:p-concave}. If $\vh=\vzero$ this and the second conclusion both follow from Lemmas~\ref{lem:q0-q1-zero} and \ref{lem:type-II-well-posed-appendix}.
\end{proof}

The statement of Theorem~\ref{thm:alg-optimizer} is a combination of many of the results established in this section.

\begin{proof}[Proof of Theorem~\ref{thm:alg-optimizer}]
Existence of a maximizer $(p,\Phi;q_0)$ was shown in Proposition~\ref{prop:F-max}, and such $p,\Phi$ are continuously differentiable on $[q_0,1]$ by Corollary~\ref{cor:regular-final}.
The value $q_1$ was identified in Lemma~\ref{lem:s1-s2-separate}. 
The behavior on $S_1=[q_0,q_1]$ and $S_2=[q_1,1]$ comes directly from the well-posedness of the corresponding ODEs as shown in Lemmas~\ref{lem:ODE-Lipschitz} and \ref{lem:type-II-well-posed-appendix}. 
The formula~\eqref{eq:alg-for-optimizer} was proved in Corollary~\ref{cor:alg-value}.
The last assertions follow from Proposition~\ref{prop:type-1}.
\end{proof}

We finally prove a slight generalization of Proposition~\ref{prop:type-II-locally-unique}. Recall that $\Delta^r\subseteq \bbR_{\geq 0}^r$ denotes the simplex of admissible $\Phi'$ vectors. For any initial point $\vx$ and time-increment $t>0$, solving the type $\II$ equation yields a map $F_{\vx,t}:\Delta^r\to \Delta^r$ given by
\begin{equation}
\label{eq:F-vx-t}
    F_{\vx,t}(\vv)=(\Phi(q+t)-\vx)/t
\end{equation}
where $\Phi$ solves the type $\II$ equation with initial condition $\Phi(q)=\vx$, $\Phi'(q)=\vv$.

We remark that in the case $\vx=0$ of Proposition~\ref{prop:type-II-locally-unique}, surjectivity also follows simply by taking $(p,\Phi;q_0)$ maximizing a version of $\bbA$ rescaled to have an arbitrary endpoint.

\begin{corollary}
\label{cor:type-II-locally-unique}
Assume $\xi$ is non-degenerate. For $C>0$, there exists $\eps=\eps(C)$ such that the map $F_{\vx,t}$ defined in \eqref{eq:F-vx-t} is injective for $t\in [0,\eps]$ and $\|\vx\|_1\leq C$. Moreover $F_{\vx,t}$ is always surjective.
\end{corollary}

\begin{proof}
    An easy Gr{\"o}nwall argument using \eqref{eq:Phi-stays-increasing} implies that for $0\leq t\leq \eps$, 
    \[
    \la \Phi(q+t)-\wt\Phi(q+t),\Phi'(q)-\wt\Phi'(q)\ra>0
    \]
    for any pair $(\Phi,\wt\Phi)$ of solutions to the type $\II$ equation with $\Phi(q)=\wt\Phi(q)$ and $\Phi'(q)\neq\wt\Phi'(q)$. This implies injectivity. Surjectivity follows from Lemma~\ref{lem:topology} since \eqref{eq:Phi-stays-increasing} implies that if $\vv_s=0$ then $F_{\vx,t}(\vv)_s=0$.
\end{proof}

\begin{lemma}[{\cite[Lemma 2.1]{jamison1976factoring} or \cite[Lemma 1]{karasev2009kkm}}]
\label{lem:topology}
    Let $F$ be a continuous map from $\Delta^r$ to itself such that $F(\vv)_s=0$ if $v_s=0$. Then $F$ is surjective.
\end{lemma}

\subsection{Explicit Solution for Pure Models}
\label{subsec:pure}

In this subsection we prove Theorem~\ref{thm:pure} and Corollary~\ref{cor:pure}, obtaining an explicit description of $\hALG$ in the important special case of \emph{pure} models for which
\begin{equation}
\label{eq:pure-mixture}
    \xi(x_1,\dots,x_r)=\prod_{s\in\sS} x_s^{a_s}.
\end{equation}
Due to the homogeneity and lack of external field, it is natural to expect that the optimal $(p,\Phi)$ is given by $p\equiv 1$ and $\Phi(q)=(q^{b_1},\dots,q^{b_r})$ for positive constants $b_s$. (Here we do not require $\Phi$ to be admissible, which by Lemma~\ref{lem:admissible-optional} does not make a difference.) Most of our previous results do not apply directly because $\xi$ violates the non-degeneracy condition, however as mentioned previously we can apply them after adding a small perturbation.

\begin{lemma}
\label{lem:pure-p=1}
For a pure model described by $\xi$, there exists $\Phi^*$ such that with $p\equiv 1$,
\[
   \bbA(p,\Phi^*;0)=\hALG. 
\]
\end{lemma}

\begin{proof}
Let 
\[
\xi^{(\eps)}(\vx)=\xi(\vx)+\eps\sum_{s,s'\in\sS} x_sx_{s'}+\eps\sum_{s,s',s''\in\sS}x_s x_{s'}x_{s''}.
\]
Then the preceding results show that optimal solutions $(\Phi^{(\eps)},p^{(\eps)},q_0^{(\eps)})$ for $\xi^{(\eps)}$ satisfy $p^{(\eps)}\equiv 1$ and $q_0^{\eps(\eps)}=0$. Taking a convergent subsequence $\Phi^{(\eps)}\to\Phi^*$ as $\eps\to 0$ in the space $\cM$ (shown to be compact in Appendix~\ref{subsec:maximizer-existence}) implies the result since $\hALG$ is continuous in $\xi$.
\end{proof}

We first non-rigorously guess the solution by assuming it is of the form \eqref{eq:pure-mixture} and also solves the type $\II$ equation. By homogeneity, we may assume 
\begin{equation}
\label{eq:B=1}
    \sum_{s\in\sS}a_{s}b_{s}=1.
\end{equation}
Then
\begin{align*}
    \Phi_s'(q)&=b_s q^{b_s-1},
    \\
    (\xi^s\circ\Phi)(q)
    &=
    \frac{a_s}{\lambda_s}q^{1-b_{s}}.
\end{align*}
We thus expect that for some constant $L$ independent of $s$,
\begin{align*}
    \Psi_s(q) &=
    b_s^{-1} q^{1-b_s}
    \deriv{q}\sqrt{\frac{b_s q^{b_s-1}}{q^{1-b_s-1}a_s(1-b_s)/\lambda_s}}
    \\
    &=
    \sqrt{\frac{\lambda_s}{a_s(1-b_s)b_s}}q^{1-b_s}
    \deriv{q}{q^{-\frac{1}{2}+b_s}}
    \\
    &=
    \lt(-\frac{1}{2}+b_s\rt)\sqrt{\frac{\lambda_s}{a_s(1-b_s)b_s}}
    q^{-1/2}
    \\
    &=
    -L^{-1/2}q^{-1/2}.
\end{align*}
(Recall that $\Psi_s$ should be negative.) The resulting quadratic equation in $b_s$ has solution
\begin{equation}
\label{eq:pure-L}
    b_s=\frac{1- \sqrt{\frac{a_s}{a_s+L\lambda_s}}}{2}.
\end{equation}
Finally $L$ is chosen to satisfy \eqref{eq:B=1}; it is easy to see there is a unique such choice.

Our next step is to verify the computation above and prove uniqueness.

\begin{proof}[Proof of Theorem~\ref{thm:pure}]
~\\
\paragraph{Part $1$: Value of $\ALG$}
Here we assume $p\equiv 1$, relying on Lemma~\ref{lem:pure-p=1}, and determine the value $\ALG$. Using the purity of $\xi$, a simple scaling argument shows the $\hALG$ value with endpoint $\vx=(x_1,\dots,x_r)$ (cf. Remark~\ref{rem:normalization})
is given by
\begin{equation}
\label{eq:pure-scaling}
    \hALG(\vx)=\hALG(\vone)
    \cdot
    \prod_{s\in \sS}x_s^{a_s/2}.
\end{equation}
(Recall that $\xi$ is a covariance, hence the factor $1/2$ in the exponent on the right-hand side.)
Set $\phi_{D-1}^s=1-b_s \delta$ for small $\delta$ and $\vb\succeq 0$ satisfying \eqref{eq:B=1}. This is a fully general choice for $\phi_{D-1}$ as in Section~\ref{sec:uc}. In light of Proposition~\ref{prop:what-F-is}, we obtain that for small $\delta>0$,
\begin{equation}
\label{eq:pure-DP}
    \hALG(\vone)
    =
    \max_{\vb\,:\,\eqref{eq:B=1}}
    \big(\hALG(\phi_{D-1})
    +
    \delta\sum_{s\in\sS} \lambda_s\sqrt{\lambda_s^{-1}a_sb_s(1-b_s)}
    \big)
    +o(\delta).
\end{equation}
Denoting $\hALG=\hALG(\vone)$ and using \eqref{eq:pure-scaling}, we find 
\begin{align*}
    \hALG&=
    \max_{\vb\,:\,\eqref{eq:B=1}}
    \Big(\hALG\cdot\prod_{s\in\sS} (1-b_s\delta)^{a_i/2} + \delta\sum_{s\in\sS} \lambda_s\sqrt{\lambda_s^{-1}a_sb_s(1-b_s)}
    \Big)
    +o(\delta)
    \\
    &=
    \max_{\vb\,:\,\eqref{eq:B=1}}
    \bigg(\lt(1-\frac{\delta}{2}\rt) \hALG+\delta\sum_{s\in\sS} \sqrt{\lambda_s a_s b_s (1-b_s)}
    \bigg)+o(\delta).
\end{align*}
Rearranging and sending $\delta\to 0$ yields
\begin{equation}
\label{eq:ALG-pure-max}
    \hALG=
    2\max_{\vb\,:\,\eqref{eq:B=1}} 
    \sum_{s\in\sS} 
    \sqrt{\lambda_s a_s b_s (1-b_s)}.
\end{equation}
First, it is easy to see that any maximizing $\vb^*$ has $b_s^*>0$ for all $s$, since otherwise the derivative of the right-hand side in $b_s$ would be infinite. By Lagrange multipliers, for some $C>0$ any solution will have
\begin{equation}
\label{eq:LM-formula}
\begin{aligned}
    \sqrt{\frac{a_s}{L\lambda_s}}
    &=\deriv{b_s}\lt(\sqrt{b_s(1-b_s)}\rt)
    \\
    &=\frac{\frac{1}{2}-b_s}{\sqrt{b_s(1-b_s)}}
\end{aligned}
\end{equation}
for some $L\in [0,\infty]$ (where division by $\infty$ gives $0$).

Let us first assume $\sum_{s\in\sS} a_s\geq 3$. Then \eqref{eq:B=1} implies that $b_s<1/2$ for some $s$, hence for all $s$ since the signs have to match in \eqref{eq:LM-formula}. In particular we have $L<\infty$, and \eqref{eq:pure-L} above easily follows from \eqref{eq:LM-formula}.
The resulting formula is as desired:
\begin{align*}
    \hALG
    &=
    2\sum_{s\in \sS}
    \sqrt{\lambda_s a_s}\cdot \lt(\frac{1}{2}-b_s\rt)\sqrt{\frac{L\lambda_s}{a_s}}
    \\
    &=
    \sum_{s\in\sS}
    \lambda_s
    \sqrt{\frac{ L a_s}{L\lambda_s+a_s}}
    .
\end{align*}

The only remaining case is $\xi(x_1,x_2)=x_1 x_2$. Then it is clear from \eqref{eq:ALG-pure-max} that $b_1=b_2=1/2$ and 
\[
\hALG=\sqrt{\lambda_1}+\sqrt{\lambda_2}.
\]
(This case of Theorem~\ref{thm:pure} is stated with $b_1=b_2=1$ which is an equivalent parametrization.)

\paragraph{Part $2$: Uniqueness Assuming $p\equiv 1$}
Next we show the optimal trajectory $\Phi^*(q)=(q^{b_1},\dots,q^{b_r})$ is unique up to reparametrization when $p\equiv 1$. The maximization problem in \eqref{eq:ALG-pure-max} is strictly convex on the affine subspace defined by \eqref{eq:B=1}, and hence has a unique minimizer. It follows that if $\phi_d$ in the preceding equation is defined by any choice $\vb$ bounded away from the optimal one, the obtained value would be strictly worse than $\ALG$. In other words, any optimal trajectory where $p\equiv 1$ must satisfy $\Phi'(1)=\vb$. By scale-invariance, we conclude that $\Phi(q)=(q^{b_1},\dots,q^{b_r})$ is the unique optimal such trajectory.

\paragraph{Part $3$: Uniqueness of Optimal $p$}
Finally we prove that all optimal solutions actually satisfy $p\equiv 1$.
Suppose another maximizer $(p,\Phi)$ exists. Let
\[
    q_*=\inf_{q>0}\{q~:~\min_{s\in\sS}\Phi_s(q)>0\}.
\]
The definition of $p$ on $[0,q_*)$ is irrelevant so we assume without loss of generality that $p$ is constant on $[0,q_*]$ and continuous at $q_*$. It is easy to see that such a maximizing $p$ must be continuous on all of $[0,1]$ and satisfy $p(1)=1$; otherwise $p$ could be strictly increased while keeping $p'$ constant for the purposes of $\bbA$. The proof of Lemma~\ref{lem:p-AC} implies that $p$ is uniformly Lipschitz on $[q_*+\eps,1]$ for any $\eps>0$, so that $p'$ makes sense as a measurable function. 

We have seen that if $p\equiv 1$ then $\ALG$ is achieved by a unique $\Phi$, so we remains to show that no optimal $(p,\Phi)$ satisfies $p\not\equiv 1$  Assuming that $p\not\equiv 1$ we may choose $q>q_*$ a Lebesgue point for both $p'$ and $\Phi'$ such that
\[
    p'(q)>0.
\]
We now derive a contradiction by expanding $\hALG$ around $q$ as in \eqref{eq:pure-DP}. In particular, consider $\phi_d=\Phi(q-\delta)$ and $p_d=p(q-\delta)$.
Let $\Delta_s=\Phi_s(q)-\phi_{d,s}$ and $\Delta_p=p(q)-p_d$. Since $q$ is a Lebesgue point, we have $\Delta_s=\Phi_s'(q)\delta+o(\delta)$ and $\Delta_p=p'(q)+o(\delta)$.

The computation above for the value $\hALG$ implies 
\[
    \hALG(p_d,\phi_d)
    =
    \hALG(\phi_d)\sqrt{p_d}
    .
\]
Here $\hALG(p_d,\phi_d)$ denotes the analog of \eqref{eq:alg} with endpoint value $\Phi(q)=\phi_d$ rather than $q=1^{\sS}$, and $p(q)=p_d$.
Therefore
\begin{align*}
    \hALG(p(q),\Phi(q))
    &=
    \hALG(\phi_d)
    \sqrt{p_d}
    +
    \sum_{s\in\sS}
    \lambda_s
    \sqrt{\Delta_s \lt(\Delta_p \xi^s(\phi_d)+p_d \sum_{s'\in\sS} \partial_{x_{s'}}\xi^s(\phi_d)\Delta_{s'} \rt)}
    +
    o(\delta)
    \\
    &=
     \hALG(p(q),\Phi(q))
    \cdot
    \lt(1-\frac{\delta}{2}\times\lt(\frac{p'(q)}{p(q)}+\sum_{s\in\sS} \frac{a_s \Phi_s'(q)}{\Phi_s(q)}\rt)\rt)
    \\
    &\quad\quad
    +
    \delta\sum_{s\in\sS}
    \lambda_s
    \sqrt{\Phi'_s(q) \lt(p'(q) (\xi^s\circ\Phi)(q)+p(q) \sum_{s'\in\sS} \partial_{x_{s'}}(\xi^s\circ\Phi)(q)\Phi_{s'}'(q) \rt)}
    +
    o(\delta)
    .
\end{align*}
Rearranging and sending $\delta\to 0$ implies
\begin{equation}
\label{eq:ALG-pure-recursion}
    \hALG(p(q),\Phi(q))/2
    =
    \frac{
    \sum_{s\in\sS}
    \lambda_s
    \sqrt{\Phi'_s(q) \lt(p'(q) (\xi^s\circ\Phi)(q)+p(q) \sum_{s'\in\sS} \partial_{x_{s'}}(\xi^s\circ\Phi)(q)\Phi_{s'}'(q) \rt)}
    }
    {
    \frac{p'(q)}{p(q)}+\sum_{s\in\sS} \frac{a_s \Phi_s'(q)}{\Phi_s(q)}
    }
    \,.
\end{equation}
We claim that \eqref{eq:ALG-pure-recursion} forces $p'(q)=0$, which completes the proof of uniqueness since $q$ was an arbitrary choice of Lebesgue point. Note that from any solution to \eqref{eq:ALG-pure-recursion} we immediately get a maximizing $(\Phi,p)$ for $\bbA$ where $p(q)$ and each $\Phi_s(q)$ is a monomial of the form $aq^b$.

The right-hand side above has maximum value $\hALG(p(q),\Phi(q))/2$, and we already know from Lemma~\ref{lem:pure-p=1} there exists $(p'(q),\Phi'(q))$ achieving this value with $p'(q)=0$. Supposing another maximizing $(\wt p'(q),\wt\Phi'(q))$ with $\wt p'(q)>0$ exists, 
we suppress the input $q$ and consider a general solution 
\[
    (p_a',\Phi_a')=\big(ap_1'-(a-1)p_0',a\Phi_1'-(a-1)\Phi_0' \big).
\]
We always restrict to $a$ such that all derivatives are non-negative. The denominator of the right-hand side of \eqref{eq:ALG-pure-recursion} is affine in $a$, while Lemma~\ref{lem:sqrt-xy-concave} implies the numerator is concave. Since $(p_0',\Phi_0')$ and $(p_1',\Phi_1')$ both maximize the right-hand side we deduce that it takes the constant value $\hALG(p(q),\Phi(q))/2$ on $(p_a',\Phi_a')$ for all $a\in [0,1]$. In particular using again Lemma~\ref{lem:sqrt-xy-concave} we find that each of the $r$ terms in the numerator is actually a linear function of $a$ on the interval such that 
\begin{equation}
\label{eq:good-a}
    p_a'(q)\geq 0, \quad\text{and}\quad 
    \min_s \Phi_{a,s}'(q)\geq 0.
\end{equation}
This means equality is achieved for $p_a$ for $a$ satisfying \eqref{eq:good-a} (even if $a>1$) and implies that $\Phi'(q)\neq \wt\Phi'(q)$. Let $a_*>0$ be the maximal value satisfying \eqref{eq:good-a}, so that $\min_s \Phi_{a_*,s}'(q)=\Phi_{a_*,s_*}'(q)=0$. Then clearly the $s_*$ term of the numerator is not affine on $a\in [a_*-\eps,a_*]$; since $\Phi_{a_*}'$ satisfies admissibility it does not equal $\vzero$. This gives a contradiction, so we conclude that $p\equiv 1$ holds for all optimal $(p,\Phi)$.  
\end{proof}

\begin{proof}[Proof of Corollary~\ref{cor:pure} ]
    Here we have $\lambda_s=\frac{a_s}{\sum_{s\in\sS} a_s}$ in the preceding formulas. It is easy to see from \eqref{eq:pure-L} that the values $b_s$ are all equal. From \eqref{eq:B=1} we find $b_s=\frac{1}{\sum_{s\in\sS} a_s}$ and so
\begin{align*}
    \hALG
    &=
    2\sum_{s\in \sS}\sqrt{\lambda_s a_sb_s(1-b_s)}
    \\
    &=
    2\sum_{s\in \sS} \frac{a_s}{\sqrt{\sum_{s\in\sS} a_s}}\cdot \frac{\sqrt{\big(\sum_{s\in\sS} a_s\big)-1}}{\sum_{s\in\sS} a_s}
    \\
    &=
    2\sqrt{\frac{\big(\sum_{s\in\sS} a_s\big)-1}{\sum_{s\in\sS} a_s}}.
\end{align*}
\end{proof}

We finally show Corollary~\ref{cor:E-infty}, recalling the formula for $E_{\infty}$ from \cite{mckenna2021complexity} and verifying it equals $\ALG$ for pure models. It is given as follows, where $\bbH=\{z\in\bbC~:~{\mathsf {Im}}(z)> 0\}$ denotes the complex open upper half plane. We recall (a slight generalization of) \cite[Lemma 2.2]{mckenna2021complexity}; as written
only the bipartite case was considered therein but the general multi-species case is no different.
Additionally we point out that the constants $\alpha_s$ appearing in \cite{mckenna2021complexity} continue to vanish in pure models for general $r$, which we take advantage of in the statement below.

Informally, the point below is simply that $\sum_s \lambda_s M_s$ is the Stieltjes transform of the bulk spectral distribution of an $N\times N$ random matrix with variance profile $\partial_{x_s,x_{s'}}\xi$ with diagonal species-dependent shift $E\xi^s(\vone)$. This essentially corresponds to the behavior of the Riemannian Hessian $\nabla^2_{\sph}H_N(\bsig)$ at a point $\bsig$ with $H_N(\bsig)=E$, where the diagonal shift corresponds to the induced radial derivative of $H_N$.

\begin{proposition}[{Adaptation of \cite[Lemma 2.2]{mckenna2021complexity} with $r$ species and pure $\xi$}]
\label{prop:E-infty-mckenna}
    For $z\in\bbH$ (resp. $-\bbH$), there is a unique solution $\vec M\in \bbH^{\sS}$ (resp. $-\bbH^{\sS}$) to the matrix Dyson equation
    \[
    1+M_s\lt(
    \big(z-E\xi^s(\vone) \big) 
    +
    \partial_{x_{s}}\xi^{s}(\vone)
    M_{s}
    +
    \sum_{s'\neq s}
    (\partial_{x_{s'}}\xi^s(\vone))
    M_{s'}
    \rt)
    =0,\quad\forall s\in\sS.
    \]
    The threshold $E_{\infty}\geq 0$ is the smallest value such that with $z=0$, $\vec M(E)$ extends analytically and continuously at the boundary to $E\in [E_{\infty},\infty)$ (and is real-valued on this interval).
\end{proposition}

When $\xi(\vx)=\prod_{s\in\sS} x_s^{a_s}$ is pure and $z=0$, the vector Dyson equation simplifies to
\begin{equation}
\label{eq:matrix-dyson-pure}
    1+a_s M_s\bigg(
    E-\lambda_s M_s + 
    \sum_{s'\in\sS}
    \lambda_{s'}a_{s'}M_{s'}
    \bigg)
    =0,\quad
    \forall\,s\in\sS.
\end{equation}

\begin{proof}[Proof of Corollary~\ref{cor:E-infty}]
For convenience we omit the case $\xi(x_1,x_2)=x_1x_2$ and assume $\sum_s a_s\geq 3$.
Setting
\begin{align*}
    K_s&=a_sM_s,
    \\
    K&=\sum_{s\in\sS} \lambda_s K_s    
\end{align*}
the system \eqref{eq:matrix-dyson-pure} can be rearranged to
\[
    A\equiv K+E = \frac{\lambda_s K_s}{a_s}-\frac{1}{K_s},\quad \forall\,s\in\sS.
\]
With $B_s=\frac{a_s}{\lambda_s}$ we find that at $E=E_{\infty}$,
\[
    K_s=\frac{A B_s - \sqrt{A^2 B_s^2+4B_s}}{2}.
\]
Here the choice of sign is forced by $M_s<0$; this easily holds for sufficiently large $E$ (where one can give a power series expansion), and follows by continuity since $K_s\neq 0$ in general.

Note that $A$ above determines each $K_s$, hence $K$ and hence $E=A-K$.
Viewing $E$ as a function the $A$, its derivative must vanish and so:
\begin{align}
\nonumber
    0&=\frac{\de E}{\de A}
    \\
\label{eq:A-positive}
    &=
    1-\frac{1}{2}
    \sum_{s\in\sS}
    a_s\lt(
    1-\frac{A B_s}{\sqrt{A^2 B_s^2 + 4B_s}}
    \rt)
    \\
\label{eq:A-zero}
    &=
    1-\frac{1}{2}
    \sum_{s\in\sS}
    a_s\lt(
    1-\frac{1}{\sqrt{1 + 4/(A^2 B_s)}}
    \rt)
    \\
\nonumber
    &\stackrel{\eqref{eq:B=1}}{=}
    1-\frac{1}{2}
    \sum_{s\in\sS}
    a_s\lt(
    1-\sqrt{\frac{a_s}{a_s+L\lambda_s}}
    \rt)
    \\
\label{eq:L-zero}
    &=
    1-\frac{1}{2}
    \sum_{s\in\sS}
    a_s\lt(
    1-\sqrt{\frac{1}{1+L/B_s}}
    \rt)
    .
\end{align}
Here we used $\sum_s a_s\geq 3$ to deduce from \eqref{eq:A-positive} that $A>0$, thus implying the next line.
By monotonicity, equality of \eqref{eq:A-zero} and \eqref{eq:L-zero} now implies $A=2/\sqrt{L}$. 
Turning to the desired equality, we first write
\begin{align*}
    E_{\infty}&=A-K
    \\
    &=
    \frac{2}{\sqrt{L}}
    -
    \frac{1}{2}
    \sum_s 
    \lambda_s 
    \lt(
    \frac{2a_s}{\lambda_s \sqrt{L}}
    -
    \sqrt{
    \frac{4a_s^2}{L\lambda_s^2}
    +
    \frac{4a_s}{\lambda_s}
    }
    \rt)
    \\
    &=
    \frac{2}{\sqrt{L}}
    -
    \sum_s 
    \frac{a_s}{\sqrt{L}}
    \lt(1-
    \sqrt{
    \frac{a_s+L\lambda_s}{a_s}
    }
    \rt).
\end{align*}
With $V_s\equiv\sqrt{a_s+L\lambda_s}$, adding and subtracting $\sum_s \frac{a_s^{3/2}}{V_s\sqrt{L}}$ to get the second equality, we compute
\begin{align*}
    E_{\infty}-\ALG
    &=
    \frac{2}{\sqrt{L}}
    +\sum_s
    \lt(
    -\frac{a_s}{\sqrt{L}}
    +
    \frac{V_s\sqrt{a_s}}{\sqrt{L}}
    -
    \frac{\lambda_s \sqrt{L a_s}}{V_s}
    \rt)
    \\
    &=
    \frac{1}{\sqrt{L}}\lt(
    2
    +
    \sum_s
    \lt(
    -a_s
    +
    \frac{a_s^{3/2}}
    {V_s}
    \rt)
    \rt)
    +
    \sum_s
    \frac{\sqrt{a_s}}{V_s \sqrt{L}}
    \lt(
    V_s^2 - a_s-L\lambda_s
    \rt)
    \\
    &=0.
\end{align*}
Here in the last step, we used \eqref{eq:B=1} to handle the first contribution (summed over $s\in\sS$) and the definition of $V_s$ for the second (for each $s\in\sS$). 
\end{proof}

%% file: final-tex/7-ack.tex
\section*{Acknowledgements}

We thank Mehtaab Sawhney for pointing us to \cite{bandeira2021matrix}, and Jean-Christophe Mourrat and Nike Sun for helpful feedback. 
B.H. was supported by an NSF Graduate Research Fellowship, a Siebel scholarship, NSF awards CCF-1940205 and DMS-1940092, and NSF-Simons collaboration grant DMS-2031883.
M.S. was supported by an NSF
graduate research fellowship, the William R. and Sara Hart Kimball Stanford graduate fellowship, and NSF
award CCF-2006489 and was a member at the IAS while parts of this work were completed. 
The initial ideas for this work were generated while the authors were visiting the Computational Complexity of Statistical Inference program at the Simons Institute in Fall 2021. 

%% file: final-tex/a0-equivalence-of-bogps.tex
\section{Equivalence of $\BOGP$ and $\BOGP_{\loc,0}$}
\label{sec:equivalence-of-bogps}

In this section, we prove Proposition~\ref{prop:bogp-equivalent} that $\BOGP = \BOGP_{\loc,0}$.
We introduce two other limits $\BOGP_{\den}$ and $\BOGP_{\loc}$, as follows (restating $\BOGP$ and $\BOGP_{\loc,0}$ for convenience). 
\begin{align*}
    \BOGP
    &= 
    \lim_{D\to\infty}
    \lim_{\eta\to 0}
    \lim_{k\to\infty}
    \sup_{\vchi \in \bbI(0,1)^\sS}
    \inf_{\uvphi=\vchi(\up)}
    \limsup_{N\to\infty}
    \fr{1}{N} 
    \bbE \sup_{\ubsig \in \cQ(\eta)}
    \cH_N(\ubsig), \\
    \BOGP_{\den}
    &= 
    \lim_{D\to\infty}
    \lim_{\eta\to 0}
    \lim_{k\to\infty}
    \sup_{\substack{
        \vchi \in \bbI(0,1)^\sS \\ 
        \text{$1/D^2$-separated}
    }}
    \inf_{\substack{
        \uvphi=\vchi(\up) \\
        \text{$6r/D$-dense}
    }}
    \limsup_{N\to\infty}
    \fr{1}{N} 
    \bbE \sup_{\ubsig \in \cQ(\eta)}
    \cH_N(\ubsig), \\
    \BOGP_{\loc}
    &=
    \lim_{D\to\infty}
    \lim_{\eta\to 0}
    \lim_{k\to\infty}
    \sup_{\substack{
        \vchi \in \bbI(0,1)^\sS \\ 
        \text{$1/D^2$-separated}
    }}
    \inf_{\substack{
        \uvphi=\vchi(\up) \\
        \text{$6r/D$-dense}
    }}
    \limsup_{N\to\infty}
    \fr{1}{N} 
    \bbE \sup_{\ubsig \in \cQ_{\loc}(\eta)}
    \cH_N(\ubsig), \\
    \BOGP_{\loc,0}
    &=
    \lim_{D\to\infty}
    \lim_{k\to\infty}
    \sup_{\substack{
        \vchi \in \bbI(0,1)^\sS \\ 
        \text{$1/D^2$-separated}
    }}
    \inf_{\substack{
        \uvphi=\vchi(\up) \\ 
        \text{$6r/D$-dense}
    }}
    \limsup_{N\to\infty}
    \fr{1}{N} 
    \bbE \sup_{\ubsig \in \cQ_{\loc}(0)}
    \cH_N(\ubsig).
\end{align*}
In the last three lines, the limits in $k, \eta$ are clearly decreasing, but the limits in $D$ are not, so the existence of these limits needs to be proven.
Proposition~\ref{prop:bogp-equivalent} follows from the following propositions.
\begin{proposition}
    \label{prop:bogp-den}
    The limit $\BOGP_{\den}$ exists and $\BOGP = \BOGP_{\den}$. 
\end{proposition}

\begin{proposition}
    \label{prop:bogp-loc}
    The limit $\BOGP_{\loc}$ exists and $\BOGP_{\den} = \BOGP_{\loc}$. 
\end{proposition}

\begin{proposition}
    \label{prop:bogp-loc0}
    The limit $\BOGP_{\loc,0}$ exists and $\BOGP_{\loc} = \BOGP_{\loc,0}$. 
\end{proposition}

\subsection{Equivalence of $\BOGP$ and $\BOGP_{\den}$}

Let $\up' = (p_0,\ldots,p_{D'})$ and $\uvphi' = (\vphi'_0,\ldots,\vphi'_{D'})$. 
Say $(\up',\uvphi')$ \textbf{refines} $(\up,\uvphi)$ if there exists $0 \le a_0 < \cdots < a_D \le D'$ such that $(p_d,\vphi_d) = (p'_{a_d}, \vphi'_{a_d})$ for all $0\le d\le D$.

\begin{lemma}
    \label{lem:refinement}
    The value $\bbE \max_{\ubsig \in \cQ(\eta)} \cH_N(\ubsig)$ is decreasing under refinement. 
    That is, if $(\up',\uvphi')$ refines $(\up,\uvphi)$, then for any $k,\eta$,
    \[
        \bbE \max_{\ubsig \in \cQ^{k,D,\uvphi}(\eta)} \cH_N^{k,D,\up} (\ubsig)
        \ge \bbE \max_{\ubsig \in \cQ^{k,D',\uvphi'}(\eta)} \cH_N^{k,D',\up'} (\ubsig).
    \]
\end{lemma}
\begin{proof}
    Let $I = \{a_0,\ldots,a_D\}$ and $J = [D'] \setminus I$. 
    Define an equivalence relation $\bowtie$ on $\bbL(k,D')$ by $u\bowtie v$ if $u_d=v_d$ for all $d\in J$. 
    Let 
    \[
        \cQ' = 
        \lt\{
            \ubsig \in \cB_N^{\bbL(k,D')} : 
            \norm{\vR(\bsig(u^1),\bsig(u^2)) - \vphi'_{u^1\wedge u^2}}_\infty 
            \le \eta,
            ~\forall u^1 \bowtie u^2
        \rt\}
    \]
    be the superset of $\cQ^{k,D',\uvphi'}$ where we only enforce the overlap constraint for $u^1\bowtie u^2$.
    Then
    \begin{align*}
        \bbE \max_{\ubsig \in \cQ^{k,D',\uvphi'}(\eta)} \cH_N^{k,D',\up'} (\ubsig)
        &\le \bbE \max_{\ubsig \in \cQ'} \cH_N^{k,D',\up'} (\ubsig) \\
        &= \bbE \max_{\ubsig \in \cQ'} \fr{1}{k^{D'-D}} \sum_{u_J \in [k]^{D'-D}} \fr{1}{k^D} \sum_{u_I \in [k]^D} \cH_N^{k,D',\up'} (\ubsig) \\
        &= \bbE \max_{\ubsig \in \cQ^{k,D,\uvphi}(\eta)} \cH_N^{k,D,\up} (\ubsig).
    \end{align*}
\end{proof}

\begin{proof}[Proof of Proposition~\ref{prop:bogp-den}]
    Let $\BOGP_{\den}^+$ and $\BOGP_{\den}^-$ be $\BOGP_{\den}$ where the outer limit in $D$ is replaced by $\limsup$ and $\liminf$, respectively.
    We will separately prove $\BOGP \ge \BOGP_{\den}^+$ and $\BOGP \le \BOGP_{\den}^-$.

    Fix any $D,k,\eta$, $1/D^2$-separated $\vchi$, and (not necessarily $6r/D$-dense) $\up,\uvphi$ satisfying $\uvphi = \vchi(\up)$.
    Let $\delta = (r+1)/D$.
    Let $\tp_0 = p_0$, and define a sequence $\tp_1,\ldots,\tp_{\tilde D}$, where $\tp_{d+1}$ is the smallest $p\in [\tp_d,1]$ such that 
    \[
        \max\lt(p-\tp_d, \norm{\vchi(p) - \vchi(\tp_d)}_\infty\rt) \ge \delta
    \]
    if such $p$ exists.
    Let $\tilde D$ be the first index $d$ such that no such $p$ exists.
    Note that if $\Sigma_d = \tp_d + \tnorm{\vchi(\tp_d)}_1$, then $0\le \Sigma_d \le r+1$ and $\Sigma_{d+1} - \Sigma_{d} \ge \delta$ for all $d$.
    Thus $\tilde D \le (r+1)/\delta = D$.

    Consider either $D'=2D$ or $D'=2D+1$.
    Let $\up'$ be the sorted union of $\{p_0,\ldots,p_D\}$, $\{\tp_1,\ldots,\tp_{\tilde D}\}$, and (if necessary) additional arbitrary $p\in [0,1]$, so that $\up'$ has length $D'$.
    Define $\uvphi' = \vchi(\up')$.
    Since $(\up',\uvphi')$ refines $(\up,\uvphi)$, Lemma~\ref{lem:refinement} implies
    \[
        \bbE \max_{\ubsig \in \cQ^{k,D,\uvphi}(\eta)} \cH_N (\ubsig)
        \ge \bbE \max_{\ubsig \in \cQ^{k,D',\uvphi'}(\eta)} \cH_N^{k,D',\up'} (\ubsig).
    \]
    Moreover, one can check that $\delta \le 6r/D'$, so $(\up',\uvphi')$ is $6r/D'$-dense.
    Thus, if $f(D)$ and $g(D)$ are the quantities inside the outer limits of $\BOGP$ and $\BOGP_{\den}$, we have shown $f(D) \ge g(2D), g(2D+1)$ (as taking the supremum over not necessarily $1/D^2$-separated $\vchi$ can only increase $f(D)$). 
    This implies $\BOGP \ge \BOGP_{\den}^+$.

    For the other direction, fix $D,k,\eta$ and (not necessarily $1/D^2$-separated) $\vchi$. 
    Define
    \[
        \vchi'(p) = (1-D^{-2}) \vchi(p) + D^{-2} \vone,
    \]
    so $\vchi'$ is $1/D^2$-separated.
    Consider any $6r/D$-dense $(\up,\uvphi')$ with $\uvphi' = \vchi'(\up)$, and let $\uvphi = \vchi(\up)$. 

    Let $\ubsig \in \cQ^{k,D,\uvphi}(\eta)$.
    Let $\bx$ satisfy $\vR(\bx,\bx) = \vone$ and $\vR(\bx,\bsig(u))=\vzero$ for all $u\in \bbL$.
    Define
    \[
        \brho(u) = \sqrt{1-D^{-2}}\bsig(u) + D^{-1} \bx,
    \]
    so that for all $u, v \in \bbL$,
    \[
        \norm{\vR(\brho(u),\brho(v)) - \vphi'_{u\wedge v}}_\infty
        = (1-D^{-2}) \norm{\vR(\bsig(u),\bsig(v)) - \vphi_{u\wedge v}}_\infty
        \le \eta.
    \]
    Thus $\ubrho \in \cQ^{k,D,\uvphi'}(\eta)$, and we can easily check that
    \[
        \fr{1}{\sqrt{N}} \norm{\brho(u) - \bsig(u)}_2 = O(D^{-2})
    \]
    for all $u\in \bbL$.
    By Proposition~\ref{prop:gradients-bounded}, with probability $1-e^{-\Omega(N)}$ we have $\HNp{u}\in K_N$ for all $u\in \bbL$. 
    On this event, 
    \[
        \lt|\fr{1}{N} \cH_N(\ubrho) - \fr{1}{N} \cH_N(\ubsig)\rt| \le CD^{-2}
    \]
    for some $C>0$, and so 
    \[
        \fr{1}{N} \sup_{\ubrho \in \cQ^{k,D,\uvphi'}(\eta)} \cH_N(\ubrho)
        + CD^{-2}
        \ge 
        \fr{1}{N} \sup_{\ubsig \in \cQ^{k,D,\uvphi}(\eta)} \cH_N(\ubsig).
    \]
    By Lemma~\ref{lem:bogp-subgaussian}, both sides of this inequality are subgaussian with fluctuations $O(N^{-1/2})$, so the contribution from the complement of this event is $o_N(1)$, and 
    \[
        \limsup_{N\to\infty} 
        \fr{1}{N} \bbE \sup_{\ubrho \in \cQ^{k,D,\uvphi'}(\eta)} \cH_N(\ubrho)
        + CD^{-2}
        \ge 
        \limsup_{N\to\infty}
        \fr{1}{N} \bbE \sup_{\ubsig \in \cQ^{k,D,\uvphi}(\eta)} \cH_N(\ubsig).
    \]
    Thus $f(D) \le g(D) + CD^{-2}$ (as taking the infimum over $(\up,\uvphi)$ that are not necessarily the image of a $6r/D$-dense $(\up,\uvphi')$ under the above transformation can only decrease $f(D)$).
    This implies $\BOGP \le \BOGP_{\den}^-$.
\end{proof}

\subsection{Equivalence of $\BOGP_{\den}$ and $\BOGP_{\loc}$}

\begin{lemma}
    \label{lem:recursive-barycenter}
    We have that $\cQ(\eta) \subseteq \cQ_{\loc}(\eta+\fr{2}{k})$.
\end{lemma}
\begin{proof}
    Consider any $\ubsig \in \cQ(\eta)$. 
    We define $\ubrho \in \cB_N^\bbT$ by $\brho(u) = \bsig(u)$ if $u\in \bbL$, and
    \[
        \brho(u) = \fr{1}{k} \sum_{i=1}^k \brho(ui)
    \]
    for $u\in \bbT \setminus \bbL$.
    We will show that $\ubrho \in \cQ_{\loc+}(\eta+\fr{2}{k})$, and so $\ubsig \in \cQ_{\loc}(\eta+\fr{2}{k})$.
    
    Let $v\succeq u$ denote that $v$ is a descendant of $u$ in $\bbT$.
    Consider any non-leaf $u\in \bbT$ and two of its children $ui,uj$, for $i\neq j$. 
    For any $s\in \sS$,
    \begin{equation}
        \label{eq:sibling-overlap}
        |R_s(\brho(ui),\brho(uj)) - \phi_{|u|}^s|
        \le 
        \fr{1}{k^{2(D-|u|)}}
        \sum_{\substack{v,v' \in \bbL \\ v \succeq ui, v' \succeq uj}}
        |R_s(\bsig(v),\bsig(v')) - \phi_{|u|}^s|
        \le \eta.
    \end{equation}
    Moreover,
    \[
        |R_s(\brho(ui),\brho(u)) - \phi_{|u|}^s|
        \le 
        \fr1k 
        \sum_{j=1}^k 
        |R_s(\bsig(ui),\bsig(uj)) - \phi_{|u|}^s|
        \le 
        \eta + \fr{2}{k},
    \]
    where we bounded the terms $j\neq i$ by \eqref{eq:sibling-overlap} and the term $j=i$ crudely by $2$.
    Thus,
    \[
        |R_s(\brho(u),\brho(u)) - \phi_{|u|}^s|
        \le 
        \fr{1}{k}
        \sum_{i=1}^k
        |R_s(\bsig(ui),\bsig(u)) - \phi_{|u|}^s|
        \le 
        \eta + \fr{2}{k}.
    \]
\end{proof}

For $k'\le k$, define a \emph{$k'$-ary subtree} of $\bbT$ to be a subset $T\subseteq \bbT$ isometric to $\bbT(k',D)$.
The following fact is clear from the definition of $\cQ_{\loc}(\eta)$. 
\begin{fact}
    \label{fac:subtree-satisfy-local}
    Let $T\subseteq \bbT$ be a $k'$-ary subtree with leaf set $L$.
    If $\ubsig \in \cQ_{\loc}(\eta)$, then $(\bsig(u))_{u\in L} \in \cQ_{\loc}^{k',D,\uvphi}(\eta)$.
\end{fact}
\begin{proof}
    There exists $\ubrho \in \cQ_{\loc+}(\eta)$ such that $\brho(u) = \bsig(u)$ for all $u\in \bbL$.
    Then $(\brho(u))_{u\in T} \in \cQ_{\loc+}^{k',D,\uvphi}(\eta)$, which implies the result.
\end{proof}

\begin{lemma}
    \label{lem:prune-global-constraints}
    Let $k'$ be the largest integer solution to $D(k')^D \le \min(\sqrt{k}, \eta^{-1})$. 
    If $\ubsig \in \cQ_{\loc}(\eta)$, there exists a $k'$-ary subtree $T$ of $\bbT$ with leaf set $L$ such that $(\bsig(u))_{u\in L} \in \cQ^{k',D,\uvphi}(CD^2 (k^{-1/4} + \eta^{1/4}))$, for some $C >0$. 
\end{lemma}
\begin{proof}
    Let $\ubrho \in \cQ_{\loc+}(\eta)$ such that $\brho(u) = \bsig(u)$ for all $u\in \bbL$. 
    We will construct $T$ by a breadth-first search: we start from $T = \{\emptyset\}$ and each step \emph{process} a leaf $u$ of $T$ by adding $k'$ children of $u$ to $T$, until all leaves of $T$ are of depth $D$.
    
    Suppose we are currently processing vertex $u$. 
    Let $V =  \{\brho(v) : v\in T\}$ and $S = \text{span}(V)$; note that $|V| \le D(k')^D \le \min(\sqrt{k}, \eta^{-1})$.
    Let $P_S$ denote the projection operator onto $S$. 
    For $i\in [k]$, write $\bx^i = \fr{1}{\sqrt{N}} (\brho(ui) - \brho(u))$, and note $\norm{\bx^i}_2 \le 2$.
    Then
    \[
        \fr{1}{N} \sum_{i=1}^k \norm{P_S(\brho(ui)-\brho(u))}_2^2
        = 
        \sum_{i=1}^k \norm{P_S \bx^i}_2^2
    \]
    is upper bounded by the sum of the top $|V|$ eigenvalues of the Gram matrix $\bM = (\la \bx^i,\bx^j\ra)_{i,j=1}^k$.
    However, for $i\neq j$, 
    \[
        |\la \bx^i,\bx^j\ra|
        = \fr{1}{N} |\la \brho(ui)-\brho(u), \brho(uj)-\brho(u)\ra|
        \le \sum_{s\in \sS} \lambda_s |R_s(\brho(ui)-\brho(u), \brho(uj)-\brho(u))| 
        \le 4\eta,
    \]
    while $|\la \bx^i,\bx^i\ra| \le 4$.
    So, if we let $\bM = \bD + \bA$ where $\bD = \diag(\bM)$, and let $a_1 \ge \cdots \ge a_{|V|}$ be the top $|V|$ eigenvalues of $\bA$, then the sum of the top $|V|$ eigenvalues of $\bM$ is upper bounded by $4|V| + \sum_{i=1}^{|V|} a_i$. 
    However,
    \[
        \sum_{i=1}^{|V|} a_i
        \le 
        \sqrt{|V| \sum_{i=1}^{|V|} a_i^2}
        \le 
        \sqrt{|V| \norm{\bA}_F^2} 
        \le 4k\eta \sqrt{|V|}.
    \]
    It follows that
    \[
        \fr{1}{kN} \sum_{i=1}^k \norm{P_S(\brho(ui)-\brho(u))}_2^2
        \le 
        \fr{4|V|}{k} + 4\eta \sqrt{|V|}
        \le 4(k^{-1/2} + \eta^{1/2}),
    \]
    where the last step follows from $|V| \le \min(\sqrt{k}, \eta^{-1})$.
    Thus there are $k'$ children $ui$ of $u$ such that
    \[
        \fr{1}{\sqrt{N}} \norm{P_S(\brho(ui)-\brho(u))}_2 
        \le 3(k^{-1/4} + \eta^{1/4}).
    \]
    We choose these as the children of $u$ in $T$.
    By constructing $T$ in this manner, we get that for all distinct edges $(u,ui)$, $(v,vj)$ in $T$, 
    \[
        \fr{1}{N} |\la \bsig(ui)-\bsig(u), \bsig(vj)-\bsig(v) \ra|, 
        \fr{1}{N} |\la \bsig(ui)-\bsig(u), \bsig(\emptyset) \ra| 
        \le 6(k^{-1/4} + \eta^{1/4}).
    \]
    whence 
    \[
        \norm{\vR(\bsig(ui),\bsig(u), \bsig(vj)-\bsig(v))}_\infty, 
        \norm{\vR(\bsig(ui),\bsig(u), \bsig(\emptyset))}_\infty 
        \le \fr{6(k^{-1/4} + \eta^{1/4})}{\min_s \lambda_s}.
    \]
    We now verify that $(\bsig(u))_{u\in L} \in \cQ^{k',D,\uvphi}(CD^2(k^{-1/4} + \eta^{1/4}))$.
    Consider any $u,v\in L$ with least common ancestor $w$, and let $|w|=d$. 
    Let $(u_0,\ldots,u_{D-d})$ and $(v_0,\ldots,v_{D-d})$ be the paths from $w$ to $u,v$, with $u_0=v_0=w$ and $u_{D-d}=u$, $v_{D-d}=v$, and let $(w_0,\ldots,w_d)$ be the path from $\emptyset$ to $w$, with $w_0=\emptyset$, $w_d = w$. 
    Also define as convention $\bsig(w_{-1}) = \bzero$.
    Then,
    \begin{align*}
        \norm{\vR(\bsig(u),\bsig(v)) - \vphi_d}_\infty 
        &\le 
        \norm{\vR(\bsig(w),\bsig(w)) - \vphi_d}_\infty
        + \sum_{i=1}^{D-d} \sum_{\ell=0}^d \norm{\vR(\bsig(w_\ell) - \bsig(w_{\ell-1}), \bsig(u_i)-\bsig(u_{i-1}))}_\infty \\
        &\quad + \sum_{j=1}^{D-d} \sum_{\ell=0}^d \norm{\vR(\bsig(w_\ell) - \bsig(w_{\ell-1}), \bsig(v_j)-\bsig(v_{j-1}))}_\infty \\
        &\quad + \sum_{i,j=1}^{D-d} \norm{\vR(\bsig(u_i)-\bsig(u_{i-1}),\bsig(v_j)-\bsig(v_{j-1}))}_\infty \\
        &\le CD^2(k^{-1/4} + \eta^{1/4}).
    \end{align*}
\end{proof}

\begin{lemma}
    \label{lem:prune-all-leaves-good}
    There exists $C>0$ such that with probability $1-e^{-\Omega(N)}$ over the Hamiltonians $\HNp{u}$ the following holds.
    If $\eps > 0$, $\ubsig \in \cQ_{\loc}(\eta)$, and 
    \[
        \fr{1}{N} \cH(\ubsig) \ge E,
    \]
    then for $k' = \lfloor k\eps/3CD\rfloor$, there exists a $k'$-ary subtree $T$ of $\bbT$ with leaf set $L$ such that
    \[
        \fr{1}{N} \HNp{u}(\bsig(u)) \ge E-\eps
    \]
    for all $u\in L$.
\end{lemma}
\begin{proof}
    We consider the event that $\HNp{u} \in K_N$ for all $u\in \bbL$, for $K_N$ defined in Proposition~\ref{prop:gradients-bounded}. 
    This holds with probability $1-e^{-\Omega(N)}$, and on this event, $|\HNp{u}(\bsig(u))| \le C$ for all $u\in \bbL$.
    For $u\in \bbT$ define
    \[
        F(u) = \fr{1}{Nk^{D-|u|}} 
        \sum_{\substack{v\in \bbL \\ v\succeq u}}
        H^{(v)}(\bsig(v)).
    \]
    We will show that for any $u\in \bbT \setminus \bbL$, we may find $k'$ distinct children $ui_1,\ldots,ui_{k'}$ such that $F(ui_j) \ge F(u) - \eps/D$ for all $j$.
    Indeed, we have
    \[
        F(u) = \fr{1}{k} \sum_{i=1}^k F(ui),
    \]
    and $|F(ui)| \le C$ for all $i$, so the claim follows from Markov's inequality.

    We construct the subtree $T$ recursively starting from $\emptyset$, using the above claim to select the $k'$ children of each node.
    Thus, for all $u,ui\in T$ with $ui$ a child of $u$, we have $F(ui) \ge F(u) - \eps/D$. 
    Since $F(\emptyset) = \fr{1}{N} \cH_N(\ubsig) \ge E$, the result follows.
\end{proof}

\begin{proof}[Proof of Proposition~\ref{prop:bogp-loc}]
    Let $\BOGP_{\loc}^+$ and $\BOGP_{\loc}^-$ be $\BOGP_{\loc}$ where the outer limit in $D$ is replaced by $\limsup$ and $\liminf$, respectively.
    Lemma~\ref{lem:recursive-barycenter} gives $\BOGP_{\den} \le \BOGP_{\loc}^-$, so it suffices to prove $\BOGP_{\den} \ge \BOGP_{\loc}^+$.

    Fix arbitrary $\eps>0$, $D,k,\eta$, $1/D^2$-dense $\vchi$, and $6r/D$-dense $(\up,\uvphi)$ satisfying $\uvphi = \vchi(\up)$.
    If $\ubsig \in \cQ_{\loc}(\eta)$ and $\fr{1}{N} \cH_N(\ubsig) \ge E$, then on an event with probability $1-e^{-\Omega(N)}$, Lemma~\ref{lem:prune-all-leaves-good} gives a $k'$-ary subtree $T\subseteq \bbT$ with leaf set $L$ such that $\fr{1}{N} \HNp{u}(\bsig(u)) \ge E-\eps$ for all $u\in L$.
    However, $(\bsig(u))_{u\in L}$ is itself an element of $\cQ_{\loc}^{k',D,\uvphi}(\eta)$ by Fact~\ref{fac:subtree-satisfy-local}, so Lemma~\ref{lem:prune-global-constraints} gives a $k''$-ary subtree $T' \subseteq T$ with leaf set $L'$ such that $(\bsig(u))_{u \in L'} \in \cQ^{k'',D,\uvphi}(\eta')$.
    Here $k' = \lfloor \eps / 3CD \rfloor$, $k''$ is the largest solution to $D(k'')^D \le \min(\sqrt{k'}, \eta^{-1})$, and $\eta' = CD^2 ((k')^{-1/4} + \eta^{1/4})$.
    
    It follows that for all $E$,
    \[
        \bbP\lt[
            \fr{1}{N}
            \sup_{\ubsig \in \cQ^{k'',D,\uvphi}(\eta')}
            \cH_N^{k'',D,\up}(\ubsig) \ge E - \eps
        \rt]
        \ge 
        \binom{k}{k''}^{-D}
        \bbP\lt[
            \fr{1}{N}
            \sup_{\ubsig \in \cQ_{\loc}^{k,D,\uvphi}(\eta)}
            \cH_N^{k,D,\up}(\ubsig) \ge E
        \rt]
        - e^{-\Omega(N)}.
    \]
    By Lemma~\ref{lem:bogp-subgaussian}, the random variables in these two probabilities are both subgaussian with fluctuations $O(N^{-1/2})$.
    So
    \[
        \limsup_{N\to\infty} 
        \fr{1}{N}
        \bbE 
        \sup_{\ubsig \in \cQ^{k'',D,\uvphi}(\eta')}
        \cH_N^{k'',D,\up}(\ubsig)
        + \eps
        \ge 
        \limsup_{N\to\infty} 
        \fr{1}{N}
        \bbE
        \sup_{\ubsig \in \cQ_{\loc}^{k,D,\uvphi}(\eta)}
        \cH_N^{k,D,\up}(\ubsig).
    \]
    For fixed $D$, as $k\to\infty$ and $\eta\to 0$, we have $k'' \to \infty$ and $\eta' \to 0$.
    Then taking $D\to\infty$ shows $\BOGP_{\den} + \eps \ge \BOGP^+_{\loc}$. 
    Since $\eps$ was arbitrary, the result follows. 
\end{proof}

\begin{remark}
    \label{rem:min-BOGP}
    A byproduct of Lemma~\ref{lem:prune-all-leaves-good} is that defining $\cH_N$ as the  \textbf{minimum} over $u\in\bbL$ of the energies $H_N^{(u)}(\bsig(u))$, and $\BOGP$ in terms of this $\cH_N$, gives the same threshold as our definition \eqref{eq:grand-hamiltonian} of $\cH_N$ as the average of these energies. The minimal energy is actually more directly connected to our proof of Theorem~\ref{thm:main-ogp-oc}, as seen in the definition \eqref{eq:S-events} of $\Ssolve$. However the average energy is more convenient for our analysis in Section~\ref{sec:uc}.
\end{remark}

\subsection{Equivalence of $\BOGP_{\loc}$ and $\BOGP_{\loc,0}$}

\begin{lemma}
    \label{lem:gram-schmidt}
    Let $k\in \bbN$, $0 < q_0 \le q \le 1$ and $q',\eps \in [0,1]$. 
    There exists $\eps' = \eps'(\eps,k,q_0)$, where $\eps'\to 0$ as $\eps\to 0$ for fixed $k,q_0$, such that the following holds for all $q,q'$. 
    Suppose that $\bx,\by^1,\ldots,\by^k \in \bbR^N$ and 
    \[
        \bY = \begin{bmatrix}\bx & \by^1 & \cdots & \by^k\end{bmatrix}
    \]
    satisfies $\bY^\top \bY = D + E$, where $D = \diag(q,q',\ldots,q')$, all entries of $E$ have magnitude at most $\eps$, and $E_{1,1}=0$.
    There exist $\bz^1,\ldots,\bz^k$ such that for
    \[
        \bZ = \begin{bmatrix}\bx & \bz^1 & \cdots & \bz^k \end{bmatrix},
    \]
    we have $\bZ^\top \bZ = D$ and $\norm{\bz^i-\by^i}_2 \le \eps'$ for all $i\in [k]$.
\end{lemma}
\begin{proof}
    We will take 
    \[
        \eps' = 
        \begin{cases}
            2 & k^2 \eps \ge q_0, \\
            3k^{3/2} \eps^{1/2} & \text{otherwise}.
        \end{cases}
    \]
    If $k^3 \eps \ge q_0$, we let $\bz^1,\ldots,\bz^k$ be any orthogonal vectors of norm $\sqrt{q'}$ orthogonal to $\bx$ and each other.
    As $\norm{\by^i}_2,\norm{\bz^i}_2 \le 1$, the result follows.
    Similarly, if $k^3 \eps \ge q'$, then 
    \[
        \norm{\by^i-\bz^i}_2
        \le 
        \norm{\by^i}_2 + \norm{\bz^i}_2
        = \sqrt{q'+\eps} + \sqrt{q'}
        \le 3k^{3/2} \eps^{1/2}.
    \]
    It remains to address the case $k^3 \eps \le \min(q_0,q')$.
    We define $\bz^1,\ldots,\bz^k$ by the Gram-Schmidt algorithm, i.e. 
    \[
        \hbz^i = \by^i - \fr{\la \by^i, \bx \ra}{\norm{\bx}_2^2} \bx - \sum_{j=1}^{i-1} \fr{\la \by^i, \bz^j\ra}{\norm{\bz^j}_2} \bz^j,
        \qquad
        \bz^i = \fr{\sqrt{q'}}{\norm{\hbz^i}_2} \hbz^i.
    \]
    Let $\eps_i = \eps (1 + 3k^{-2})^i$, and note that $\eps \le \eps_i \le 2\eps$ for all $0\le i\le k$.
    We will show by induction over $i$ that for all $j\le i<\ell$,
    \begin{equation}
        \label{eq:gram-schmidt-induction-goal}
        |\la \by^\ell, \bz^j\ra| \le \eps_i,
    \end{equation}
    where as the base case this vacuously holds for $i=0$.
    Suppose the inductive hypothesis holds for $i-1$.
    It suffices to prove \eqref{eq:gram-schmidt-induction-goal} for $j=i$ because the assertion for the remaining $j$ is implied by the inductive hypothesis, as $\eps_{i-1} \le \eps_i$.
    We have
    \[
        \norm{\hbz^i}_2^2
        = 
        \norm{\by^i}_2^2
        - \fr{\la \by^i, \bx\ra^2}{\norm{\bx}_2^2}
        - \sum_{j=1}^{i-1} \fr{\la \by^i, \bz^j \ra^2}{\norm{\bz^j}_2^2}
    \]
    so
    \[
        \lt|\fr{\norm{\hbz^i}_2^2}{q'}-1\rt|
        \le 
        \lt|\fr{\norm{\by^i}_2^2}{q'}-1\rt|
        + \fr{\la \by^i, \bx\ra^2}{q'\norm{\bx}_2^2}
        + \sum_{j=1}^{i-1} \fr{\la \by^i, \bz^j \ra^2}{q'\norm{\bz^j}_2^2}
        \le 
        \fr{\eps_{i-1}}{q'} + \fr{\eps_{i-1}^2}{q'q_0} + \fr{k\eps_{i-1}^2}{(q')^2}
        \le \fr{2}{k^3}.
    \]
    Thus $\norm{\hbz^i}_2 \ge \sqrt{q'}(1-2k^{-3})$.
    For any $\ell > i$,
    \begin{align*}
        |\la \hbz^i, \by^\ell \ra| 
        &\le 
        |\la \by^i, \by^\ell\ra| 
        + \fr{|\la \by^i, \bx \ra||\la \bx, \by^\ell \ra|}{\norm{\bx}_2^2}  
        + \sum_{j=1}^{i-1} \fr{|\la \by^i, \bz^j\ra| |\la \bz^j, \by^\ell \ra|}{\norm{\bz^j}_2^2} \\
        &\le \eps_{i-1}  + \fr{\eps_{i-1}^2}{q_0} + \fr{k\eps_{i-1}^2}{q'_0} 
        \le \eps_{i-1} \lt(1 + \fr{2}{k^2}\rt)
    \end{align*}
    Thus 
    \[
        |\la \bz^i, \by^\ell \ra| 
        \le \eps_{i-1} \cdot \fr{1 + 2k^{-2}}{1 - 2k^{-3}} 
        \le \eps_i,
    \]
    completing the induction.
    Finally, note that
    \[
        \norm{\hbz^i-\by^j}_2^2
        =
        \fr{\la \by^i, \bx \ra^2}{\norm{\bx}_2^2} 
        + \sum_{j=1}^{i-1} \fr{\la \by^i, \bz^j\ra^2}{\norm{\bz^j}_2^2}
        \le \fr{\eps_{i-1}^2}{q_0} + \fr{k\eps_{i-1}^2}{q'} 
        \le \fr{5\eps}{k^2}
    \]
    and
    \[
        \lt|\norm{\hbz^i}_2 - \sqrt{q'}\rt|
        \le 
        \fr{\lt|\norm{\hbz^i}_2^2 - q'\rt|}{\sqrt{q'}}
        \le 
        \eps_{i-1} + \fr{\eps_{i-1}^2}{q'q_0} + \fr{k\eps_{i-1}^2}{(q')^2}
        \le 3\eps.
    \]
    Thus
    \[
        \norm{\bz^i-\by^i}_2
        \le 
        \norm{\hbz^i-\by^i}_2 + \lt|\norm{\hbz^i}_2 - \sqrt{q'}\rt|
        \le 
        \fr{\sqrt{5\eps}}{k} + 3\eps
        \le \eps'.
    \]
\end{proof}

\begin{proof}[Proof of Proposition~\ref{prop:bogp-loc0}]
    Let $\BOGP_{\loc,0}^+$ and $\BOGP_{\loc,0}^-$ be $\BOGP_{\loc,0}$ where the outer limit in $D$ is replaced by $\limsup$ and $\liminf$, respectively.
    It is clear that $\BOGP_{\loc} \ge \BOGP_{\loc,0}^+$, so it suffices to prove $\BOGP_{\loc} \le \BOGP_{\loc,0}^-$. 

    Fix $D,k,\eta$, $1/D^2$-separated $\vchi$, and $6r/D$-dense $(\up,\uvphi)$ with $\uvphi = \vchi(\up)$.
    Consider $\ubsig \in \cQ_{\loc}(\eta)$ and let $\ubrho \in \cQ_{\loc+}(\eta)$ such that $(\brho(u))_{u\in \bbL} = \ubsig$.
    Define $\eps_0 = \eta D$ and $\eps_d = \eps'(6\eps_{d-1} + 4\eta,k,D^{-2})$ for $1\le d\le D$, where $\eps'$ is given by Lemma~\ref{lem:gram-schmidt}.
    We will now construct $\ubtau \in \cQ_{\loc+}(0)$ approximating $\ubrho$ in the sense that for all $u\in \bbT, s\in \sS$,
    \begin{equation}
        \label{eq:gram-schmidt-invariant}
        \sqrt{R_s(\btau(u)-\brho(u),\btau(u)-\brho(u))}
        \le \eps_{|u|}.
    \end{equation}
    We define $\btau(\emptyset)$ by
    \[
        \btau(\emptyset)_s 
        = \brho(\emptyset)_s 
        \sqrt{\fr{\phi_0^s}{R_s(\brho(\emptyset),\brho(\emptyset))}}
    \]
    for all $s\in \sS$.
    Thus $R_s(\btau(\emptyset),\btau(\emptyset)) = \phi_0^s$ and 
    \[
        \sqrt{R_s(\btau(\emptyset)-\brho(\emptyset),\btau(\emptyset)-\brho(\emptyset))}
        =
        \lt|\sqrt{R_s(\brho(\emptyset),\brho(\emptyset))} - \sqrt{\phi_0^s}\rt|
        \le 
        \fr{\eta}{\sqrt{\phi_0^s}}
        \le 
        \eps_0,
    \]
    where the second-last inequality holds for all sufficiently small $\eta > 0$.
    This proves \eqref{eq:gram-schmidt-invariant} for $u=\emptyset$.
    We construct $\btau(u)$ for the remaining $u\in \bbT$ recursively.
    Suppose we have constructed $\btau(u)$ satisfying \eqref{eq:gram-schmidt-invariant}. 
    Then, for each $s\in \sS$, $i,j\in [k]$, 
    \begin{align*}
        R_s(\brho(ui)-\btau(u), \brho(uj)-\btau(u))
        &=
        R_s(\brho(ui)-\brho(u), \brho(uj)-\brho(u)) \\
        &\quad + R_s(\brho(ui)-\brho(u), \brho(u)-\btau(u)) \\
        &\quad + R_s(\brho(u)-\btau(u), \brho(uj)-\brho(u)) \\
        &\quad + R_s(\brho(u)-\btau(u), \brho(u)-\btau(u)),
    \end{align*}
    so 
    \[
        |R_s(\brho(ui)-\btau(u), \brho(uj)-\btau(u)) - (\phi_{|u|+1}^s - \phi_{|u|}^s) \ind\{i=j\}| \le 6\eps_{|u|} + 4\eta.
    \]
    Similarly,
    \begin{align*}
        |R_s(\brho(ui)-\btau(u), \btau(u))|
        &= 
        |R_s(\brho(ui)-\brho(u), \brho(u))| \\
        &\quad + |R_s(\brho(u)-\btau(u), \brho(u))| \\
        &\quad + |R_s(\brho(ui)-\brho(u), \btau(u)-\brho(u))| \\
        &\quad + |R_s(\brho(u)-\btau(u), \btau(u)-\brho(u))| \\
        &\le 6\eps_{|u|} + 4\eta.
    \end{align*}
    We apply Lemma~\ref{lem:gram-schmidt} on the vectors
    \[
        \fr{\btau(u)_s}{\sqrt{\lambda_s N}},
        \fr{\brho(u1)_s-\btau(u)_s}{\sqrt{\lambda_s N}},
        \ldots,
        \fr{\brho(uk)_s-\btau(u)_s}{\sqrt{\lambda_s N}}
    \]
    with $q = \phi_{|u|}^s \ge D^{-2}$, $q' = \phi_{|u|+1}^s - \phi_{|u|}^s$, and $\eps = 6\eps_{|u|} + 4\eta$. 
    This gives us $\btau(u1)_s,\ldots,\btau(uk)_s$ satisfying \eqref{eq:gram-schmidt-invariant}, such that 
    \[
        R_s(\btau(ui)-\btau(u), \btau(ui)-\btau(u))
        = \phi_{|u|+1}^s - \phi_{|u|}^s
    \]
    and the vectors $\btau(u)_s$, $\btau(u1)_s-\btau(u)_s$, $\btau(uk)_s-\btau(u)_s$ are pairwise orthogonal. 
    From this we can see that
    \begin{align*}
        \vR(\btau(ui),\btau(u)) &= \vphi_{|u|}, \\
        \vR(\btau(ui),\btau(ui)) &= \vphi_{|u|+1}, \\
        \vR(\btau(ui),\btau(uj)) &= \vphi_{|u|} \quad \text{if}~i\neq j. 
    \end{align*}
    Thus the $\ubtau$ constructed this way is an element of $\cQ_{\loc+}(0)$.
    Finally, let $\ubsig' = (\btau(u))_{u\in \bbL}$, so $\ubsig' \in \cQ_{\loc}(0)$.
    Equation \eqref{eq:gram-schmidt-invariant} implies that for all $u\in \bbL$, 
    \[
        \fr{1}{\sqrt{N}} \norm{\bsig'(u)-\bsig(u)}_2 \le \eps_D.
    \]
    By Proposition~\ref{prop:gradients-bounded}, with probability $1-e^{-\Omega(N)}$ we have $\HNp{u}\in K_N$ for all $u\in \bbL$. 
    On this event, 
    \[
        \lt|\fr{1}{N} \cH_N(\ubsig') - \fr{1}{N} \cH_N(\ubsig)\rt| \le C\eps_D
    \]
    for some $C >0$, and so 
    \[
        \fr{1}{N} \sup_{\ubsig' \in \cQ_{\loc}^{k,D,\uvphi'}(\eta)} \cH_N(\ubsig')
        + C\eps_D
        \ge 
        \fr{1}{N} \sup_{\ubsig \in \cQ_{\loc}^{k,D,\uvphi}(0)} \cH_N(\ubsig).
    \]
    By Lemma~\ref{lem:bogp-subgaussian}, both sides of this inequality are subgaussian with fluctuations $O(N^{-1/2})$, so the contribution from the complement of this event is $o_N(1)$, and 
    \[
        \limsup_{N\to\infty} 
        \fr{1}{N} \bbE \sup_{\ubsig' \in \cQ_{\loc}^{k,D,\uvphi'}(0)} \cH_N(\ubsig')
        + C\eps_D
        \ge 
        \limsup_{N\to\infty}
        \fr{1}{N} \bbE \sup_{\ubsig \in \cQ_{\loc}^{k,D,\uvphi}(\eta)} \cH_N(\ubsig).
    \]
    Taking $\eta \to 0$ (which forces $\eps_D \to 0$) followed by $D,k\to\infty$ implies $\BOGP_{\loc} \le \BOGP_{\loc,0}^-$, as desired.
\end{proof}

%% file: final-tex/a1-sk-to-alg.tex
\section{Ground State Energy of Multi-Species Spherical SK With External Field}
\label{sec:sk-ext-field}

We adopt the notations of Lemma~\ref{lem:sk-ext-field}.
In this section, we will prove this lemma by showing that 
\[
    \limsup_{N\to\infty} \bbE \GS_N(W,\vv,k)
    \le 
    \sum_{s\in \sS}
    \lambda_s
    \sqrt{v_s^2 + 2\sum_{s'\in \sS} \lambda_{s'} w_{s,s'}^2}
    \le 
    \liminf_{N\to\infty} \bbE \GS_N(W,\vv,k).
\]

\subsection{Upper Bound for $\vv=\vzero$, $k=1$}

The following (exact) upper bound for the case $\vv = \vzero$, $k=1$ follows from the results of \cite{bandeira2021matrix}.
We will prove Lemma~\ref{lem:sk-ext-field} using only this result and elementary techniques.

\begin{proposition}
    \label{prop:free-prob}
    For $W$ as in Lemma~\ref{lem:sk-ext-field},
    \[
        \limsup_{N\to\infty} \bbE \GS_N(W,\vzero,1)
        \le
        \sum_{s\in \sS} \lambda_s \sqrt{2\sum_{s'\in \sS} \lambda_{s'} w_{s,s'}^2}.
    \]
\end{proposition}

\begin{proof}
    In this proof, abbreviate $\GS_N = \GS_N(W,\vv,k)$.
    Let $\bG \in \bR^{N\times N}$ have i.i.d. standard Gaussian entries.
    Thus $G = \fr12 \lt(\bG + \bG^\top\rt)$ is symmetric with $\cN(0,1)$ diagonal entries, $\cN(0, 1/2)$ off-diagonal entries, and independent entries on and above the diagonal.
    Define $M\in \bbR^{N\times N}$ by $M_{i,j} = N^{-1/2} w_{s(i),s(j)} G_{i,j}$.
    It is clear by homogeneity that 
    \[
        \GS_N =
        \fr{1}{N} 
        \max_{\bsig \in \cB_N}
        \bsig^\top M \bsig.
    \]
    Let $\vC \in \bbR_{>0}^\sS$ be a vector of constants we will set later.
    We consider the rescaled matrix $\wtM = \sqrt{\vC}^{\otimes 2} \diamond M$.
    This can be generated by $\wtM = \hM + \oM$, where $\hM$ is a random symmetric matrix with independent entries on and above the diagonal 
    \[
        \hM_{i,j} \sim \cN\lt(0, \fr{C_{s(i)}C_{s(j)}w_{s(i),s(j)}^2}{2N}\rt)
    \]
    and $\oM$ is a random diagonal matrix with independent entries
    \[
        \oM_{i,i} \sim \cN\lt(0, \fr{C_{s(i)}^2 w_{s(i),s(i)}^2}{2N}\rt).
    \]
    Clearly $\bbE \tnorm{\oM}_{\op} = O(\sqrt{N^{-1} \log N})$. 
    \cite[Theorem 1.2]{bandeira2021matrix} states that
    \[
        \bbE\tnorm{\hM}_{\op}
        \leq
        \tnorm{X_{\free}}_{\op}+O\lt(v^{1/2}\sigma^{1/2}(\log N)^{3/4}\rt),
    \]
    where $\tnorm{X_{\free}}_{\op}, \sigma, v$ are defined as follows.
    We have
    \[
        \sigma
        =
        \sqrt{\bbE\tnorm{\hM^2}_{\op}}=O(1),
        \qquad
        v
        =
        \sqrt{\tnorm{\Cov(\hM)}_{\op}},
    \]
    where $\Cov(\hM)\in \bbR^{N^2\times N^2}$ is the covariance matrix of the entries of $\hM$ and has operator norm $O(1/N)$. 
    It follows that the error term $v^{1/2}\sigma^{1/2}(\log N)^{3/4}$ contributes $o_N(1)$. Finally \cite[Lemma 3.2]{bandeira2021matrix} states that in our setting,
    \[
        \norm{X_{\free}}_{\op}=
        2\sup_{\substack{a\in [0,1]^N\\ \sum_i a_i=1}}
        \sum_{i\in [N]}
        \sqrt{a_{i}\sum_{i'\in [N]} \frac{C_{s(i)}C_{s(i')}w_{s(i),s(i')}^2 a_{i'}}{2N}}
    \]
    It is not difficult to see by concavity of the square-root that, for $\lambda_{s,N} = |\cI_s|/N$ (so $\lambda_{s,N} \to \lambda_s$) replacing all $a_i$ such that $i\in \cI_s$ with
    \[
        A_s=\lambda_{s,N}^{-1}\sum_{i:s(i)=s} a_i
    \]
    only improves the right-hand side. Substituting $B_s=C_sA_s$, we conclude that
    \begin{align*}
        \norm{X_{\free}}_{\op}
        &=
        \sup_{\substack{\vA\in \bbR_{\geq 0}^{\sS}\\ \sum_s \lambda_{s,N} A_s=1}}
        \sum_{s\in\sS}
        \lambda_{s,N}
        \sqrt{2A_s\sum_{s'\in\sS}\lambda_{s'} C_s C_{s'}w_{s,s'}^2 A_{s'}} \\
        &=
        \sup_{\substack{\vB\in \bbR_{\geq 0}^{\sS}\\ \sum_s C_s^{-1}\lambda_{s,N} B_s=1}}
        \sum_{s\in\sS}
        \lambda_{s,N}
        \sqrt{2B_s\sum_{s'\in\sS}\lambda_{s'}  w_{s,s'}^2 B_{s'}}.
    \end{align*}
    From the above discussion, $\tnorm{\wtM}_{\op} \le \tnorm{X_{\free}}_{\op} + o_N(1)$.
    Moreover we observe that
    \begin{align*}
        \GS_N
        &=
        \frac{1}{N}
        \max_{\norm{\bsig_s}_2^2 \le \lambda_s N}
        \bsig^{\top}M\bsig
        =
        \frac{1}{N}
        \max_{\norm{\bsig_s}_2^2 \le C_s^{-1}\lambda_s N}
        \bsig^{\top}\wtM\bsig \\
        &\leq
        \frac{1}{N}
        \max_{\norm{\bsig}_2^2 \le \sum_{s\in\sS} C_s^{-1}\lambda_s N}
        \bsig^{\top}\wtM\bsig
        =
        \lt(\sum_{s\in\sS} C_s^{-1}\lambda_s\rt)
        \tnorm{\wtM}_{\op} 
        \,.
    \end{align*}
    Combining and using homogeneity, we find
    \begin{align}
        \notag
        \bbE \GS_N
        &\leq 
        \lt(\sum_{s\in\sS} C_s^{-1}\lambda_s\rt)
        \bbE \tnorm{\wtM}_{\op} \\
        \notag
        &=
        \lt(\sum_{s\in\sS} C_s^{-1}\lambda_s\rt)
        \sup_{\substack{
            \vB \in \bbR_{\geq 0}^{\sS} \\ 
            \sum_s C_s^{-1}\lambda_{s,N} B_s=1
        }}
        \sum_{s\in\sS}
        \lambda_{s,N}
        \sqrt{2B_s \sum_{s'\in\sS} \lambda_{s',N} w_{s,s'}^2 B_{s'}}
        +o_N(1) \\
        \label{eq:rvh-supremum}
        &=
        \fr{\sum_{s\in\sS} C_s^{-1}\lambda_s}
        {\sum_{s\in\sS} C_s^{-1}\lambda_{s,N}}
        \cdot 
        \sup_{\substack{
            \vD \in \bbR_{\geq 0}^{\sS} \\ 
            \sum_s C_s^{-1} \lambda_{s,N} D_s = \sum_{s\in\sS} C_s^{-1} \lambda_{s,N}
        }}
        \sum_{s\in\sS}
        \lambda_{s,N}
        \sqrt{2 D_s \sum_{s'\in\sS} \lambda_{s',N} w_{s,s'}^2 D_{s'}}
        +o_N(1).
    \end{align}
    If the supremum in \eqref{eq:rvh-supremum} is attained at $\vD = \vone$, then (because $\lambda_{s,N} \to \lambda_s$) we get the desired bound
    \[
        \bbE \GS_N 
        \le 
        \sum_{s\in\sS} \lambda_s \sqrt{2\sum_{s'\in\sS} \lambda_{s'} w_{s,s'}^2} 
        + o_N(1).
    \]
    Crucially, we observe that the expression 
    \begin{equation}
    \label{eq:concave-dude}
        F(\vD)
        =
        \sum_{s\in\sS}
        \lambda_{s,N}
        \sqrt{2 D_s \sum_{s'\in\sS} \lambda_{s',N} w_{s,s'}^2 D_{s'}}
    \end{equation}
    is concave in $\vD$. 
    Therefore if $\vD=\vone$ is a critical point of $F$ within the set satisfying $\sum_s C_s^{-1} \lambda_{s,N} D_s = \sum_{s\in\sS} C_s^{-1} \lambda_{s,N}$, then it also attains the supremum in \eqref{eq:rvh-supremum}.
    For the choice $C_s=\frac{\lambda_{s,N}}{\partial_{D_s} F}$, $\vD=\vone$ is a critical point of $F$. 
    This concludes the proof.
\end{proof}

\subsection{General Upper Bound}

In this subsection, we will prove the following upper bound for the case $k=1$.
\begin{proposition}
    \label{prop:sk-ub-one-rep}
    For $W, \vv$ as in Lemma~\ref{lem:sk-ext-field},
    \[
        \limsup_{N\to\infty} \bbE \GS_N(W,\vv,1)
        \le
        \sum_{s\in \sS} \lambda_s \sqrt{v_s^2 + 2\sum_{s'\in \sS} \lambda_{s'} w_{s,s'}^2}.
    \]
\end{proposition}
By slight abuse of notation, let $H_N = H_{N,1}^1$ and $\GS_N(W,\vv) = \GS_N(W,\vv,1)$.
Recall that 
\[
    H_N(\bsig)
    =
    \la \vv \diamond \bg, \bsig \ra
    + \wtH_N(\bsig),
    \qquad
    \wtH_N(\bsig)
    =
    \fr{1}{\sqrt{N}}
    \la W \diamond \bG, \bsig^{\otimes 2} \ra 
\]
where $\bg \in \bbR^N$, $\bG\in \bbR^{N\times N}$ have i.i.d. standard Gaussian entries. 
Define
\[
    A(W,\vv) = \limsup_{N\to\infty} 
    \bbE \GS_N(W, \vv).
\]
We first establish some basic properties of this limit.
\begin{lemma}
    \label{lem:A-basic-properties}
    $A$ satisfies the following properties.
    \begin{enumerate}[label=(\alph*), ref=\alph*]
        \item \label{itm:A-homogenity} For any $c>0$, $A(cW, c\vv) = cA(W,\vv)$.
        \item \label{itm:A-linear} $A(0,\vv) = \sum_{s\in \sS} \lambda_s v_s$.
        \item \label{itm:A-quadratic} $A(W,\vzero) \le \sum_{s\in \sS} \lambda_s \sqrt{2\sum_{s'\in \sS} \lambda_{s'} w_{s,s'}^2}$.
        \item \label{itm:A-subadditive} $A(W,\vv) \le A(W,\vzero) + A(0,\vv)$.
    \end{enumerate}
\end{lemma}
\begin{proof}
    Part (\ref{itm:A-homogenity}) is obvious. Part (\ref{itm:A-linear}) follows from
    \[
        \bbE \GS_N(0,\vv) 
        =
        \fr{1}{N} 
        \bbE \max_{\bsig \in \cS_N}
        \la \vv \diamond \bg, \bsig \ra
        =
        \fr{1}{N} 
        \sum_{s\in \sS}
        \sqrt{\lambda_s N} v_s
        \bbE \norm{\bg_s}_2
        =
        \sum_{s\in \sS}
        \lambda_s v_s 
        + o_N(1).
    \]
    Part (\ref{itm:A-quadratic}) follows from Proposition~\ref{prop:free-prob}.
    Part (\ref{itm:A-subadditive}) follows from 
    \begin{align}
        \notag
        \GS_N(W,\vv)
        &=
        \fr{1}{N} 
        \max_{\bsig \in \cS_N}
        \lt(
            \la \vv \diamond \bg, \bsig \ra + 
            \fr{1}{\sqrt{N}} 
            \la W \diamond \bG, \bsig^{\otimes 2} \ra
        \rt) \\
        \notag
        &\ge
        \fr{1}{N} 
        \max_{\bsig \in \cS_N}
        \la \vv \diamond \bg, \bsig \ra 
        +
        \fr{1}{N} 
        \max_{\bsig \in \cS_N}
        \fr{1}{\sqrt{N}} 
        \la W \diamond \bG, \bsig^{\otimes 2} \ra \\
        \label{eq:A-subadditive}
        &=
        \GS_N(W,\vzero) + \GS_N(0,\vv).
    \end{align}
\end{proof}

Next we show some a priori regularity conditions on $A$.

\begin{proposition}
    \label{prop:sk-concentration}
    Let
    \[
        C(W,\vv)
        =
        4 \lt(
            \sum_{s\in \sS}
            \lambda_{s}
            v_s^2 
            +
            \sum_{s,s'\in \sS}
            \lambda_{s}
            \lambda_{s'}
            w_{s,s'}^2
        \rt).
    \]
    Then, for sufficiently large $N$ and all $t>0$,
    \[
        \bbP\lt[
            \lt|\GS_N(W,\vv) - \bbE \GS_N(W,\vv)\rt| > t
        \rt]
        \le 
        2\exp\lt(-\fr{Nt^2}{C(W,\vv)}\rt).
    \]
\end{proposition}
\begin{proof}
    Let $C = C(W,\vv)$.
    For any $\bsig \in \cS_N$, 
    \begin{align*}
        \bbE H_N(\bsig)^2 
        &= 
        \norm{\vv \diamond \bsig}_2^2 + \fr{1}{N} \norm{W \diamond \bsig^{\otimes 2}}_F^2 \\
        &=
        N \lt(
            \sum_{s\in \sS}
            \lambda_{s,N}
            v_s^2 
            +
            \sum_{s,s'\in \sS}
            \lambda_{s,N}
            \lambda_{s',N}
            w_{s,s'}^2
        \rt) 
        \le 
        \fr{CN}{2}.
    \end{align*}
    for large enough $N$.
    By the Borell-TIS inequality, $\max_{\bsig \in \cS_N} H_N(\bsig)$ is $CN/2$-subgaussian, so $\GS_N(W,\vv)$ is $C/2N$-subgaussian, which implies the result.
\end{proof}

For $\va = (a_s)_{s\in \sS'} \in [0,1]^\sS$, define  $W(W,\vv,\va) = (w'_{s,s'})_{s,s'\in \sS}$ and $\vv(W,\vv,\va) = (v'_s)_{s\in \sS}$ where
\[
    w'_{s,s'} 
    = 
    \sqrt{(1-a_s)(1-a_{s'})} 
    w_{s,s'},
    \qquad
    v'_s 
    = 
    \sqrt{
        2(1-a_s) \lt(
            \sum_{s'\in \sS} 
            \lambda_{s'} a_{s'} 
            w_{s,s'}^2
        \rt)
    }.
\]
We will prove Proposition~\ref{prop:sk-ub-one-rep} using the following recursive upper bound in $A$.
\begin{lemma}
    \label{lem:A-fn-ineq}
    For $W, \vv$ as in Lemma~\ref{lem:sk-ext-field},
    \begin{equation}
        \label{eq:A-fn-ineq}
        A(W,\vv) 
        \le 
        \max_{\va \in [0,1]^{\sS}}
        \sum_{s\in \sS}
        \lambda_s v_s \sqrt{a_s}
        +
        A\lt(W(W,\vv,\va), \vv(W,\vv,\va)\rt).
    \end{equation}
\end{lemma}
\begin{proof}
    Define $\hbg \in \cS_N$ by $\hbg_s = \fr{\sqrt{\lambda_s N} \bg_s}{\norm{\bg_s}_2}$ for each $s\in \sS$.
    For $\va\in [0,1]^{\sS}$, define
    \[
        \GS_N(W,\vv;\va)
        = 
        \fr{1}{N} 
        \max_{\bsig \in \cR_N(\va)}
        H_N(\bsig),
        \qquad
        \cR_N(\va) = \lt\{
            \bsig \in \cS_N: 
            R(\bsig, \hbg) = \sqrt{\va}
        \rt\}.
    \]
    For a non-random $\va$ and any $\bsig \in \cR_N(\va)$,
    \[
        \la \vv \diamond \bg, \bsig \ra
        =
        N \sum_{s\in \sS}
        \lambda_s v_s \sqrt{a_s}
        \fr{\norm{\bg_s}_2}{\sqrt{\lambda_s N}}.
    \]
    For $\bsig\in \cR_N(\va)$, we may write $\bsig = \sqrt{\va} \diamond \hbg + \sqrt{\vone-\va} \diamond \brho$ for $\brho \in \cR_N(\vzero)$.
    Define the Gaussian process $\hH_N^{\va}(\brho) = \wtH_N\big(\sqrt{\va} \diamond \hbg + \sqrt{\vone - \va} \diamond \brho\big)$, which is supported on $\cR_N(\vzero)$.
    We next calculate the covariance of this process.
    Recall that the covariance of $\wtH_N$ is 
    \[
        \bbE \wtH_N(\bsig)\wtH_N(\bsig') = N\xi(R(\bsig,\bsig')), 
        \qquad
        \xi(\vx) = \lt\la W \odot W, (\vlam \odot \vx)^{\otimes 2} \rt\ra.
    \]
    Because $\bg, \bG$ are independent, the covariance of $\hH_N^{\va}$ is
    \begin{equation}
        \label{eq:sk-hH-covariance}
        \bbE \hH_N^{\va}(\brho)\hH_N^{\va}(\brho')
        = 
        N\xi_{\va}(R(\brho,\brho')), 
    \end{equation}
    where, for $W' = W(W, \vv, \va)$ and $\vv' = \vv(W,\vv,\va)$,
    \begin{align}
        \notag
        \xi_{\va}(\vx) 
        &= 
        \xi\lt(\va + (\vone - \va) \odot \vx\rt)
        = 
        \lt\la
            W \odot W,
            (\vlam \odot \va + \vlam \odot (1-\va) \odot \vx)^{\otimes 2}
        \rt\ra \\
        \label{eq:sk-hH-covariance-calculation}
        &= 
        \lt\la 
            W' \odot W',
            (\vlam \odot \vx)^{\otimes 2}
        \rt\ra
        +
        \lt\la
            \vv' \odot \vv', \vlam \odot \vx
        \rt\ra
        +
        \lt\la
            W \odot W,
            (\vlam \odot \va)^{\otimes 2}
        \rt\ra.
    \end{align}
    We may construct a Gaussian process $\oH_N^{\va}$ (conditional on $\bg$) on $\cS_N$ with covariance \eqref{eq:sk-hH-covariance} whose restriction to $\cR_N(\vzero)$ agrees with $\hH_N^{\va}$.
    Thus
    \begin{align*}
        \GS_N(W,\vv;\va)
        &=
        \sum_{s\in \sS}
        \lambda_s v_s \sqrt{a_s} \fr{\norm{\bg_s}_2}{\sqrt{\lambda_s N}}
        +
        \fr{1}{N} 
        \max_{\brho \in \cR_N(\vzero)}
        \hH_N^{\va}(\brho) \\
        & \le 
        \sum_{s\in \sS}
        \lambda_s v_s \sqrt{a_s} \fr{\norm{\bg_s}_2}{\sqrt{\lambda_s N}}
        +
        \fr{1}{N} 
        \max_{\brho \in \cS_N}
        \oH_N^{\va}(\brho).
    \end{align*}
    Moreover, 
    \[
        \fr{1}{N} \max_{\brho \in \cS_N} \oH_N^{\va}(\brho)
        =_d 
        \GS(W(W,\vv,\va), \vv(W,\vv,\va)) + 
        \fr{1}{\sqrt{N}} 
        \lt\la
            W \odot W,
            (\vlam \odot \va)^{\otimes 2}
        \rt\ra^{1/2}
        Z
    \]
    for an independent $Z \sim \cN(0, 1)$.
    Let $\cD = \{0, \fr{1}{N}, \ldots, \fr{N-1}{N}, 1\}^\sS$.
    Let $\cE$ be the event that
    \begin{enumerate}[label=(\alph*), ref=\alph*]
        \item \label{itm:condition-lipschitz} For a constant $L$, $H_N(\bsig)$ is $L\sqrt{N}$-Lipschitz on $\bsig \in \cS_N$. 
        By Proposition~\ref{prop:gradients-bounded}, this occurs with probability $1-\exp(-CN)$.
        \item For all $s\in \sS$, $|\norm{\bg_s}_2 - \sqrt{\lambda_{s,N}N}| \le N^{1/4}$; by standard concentration inequalities this holds with probability $1-r\exp(-CN^{1/2})$.
        \item For all $\va \in \cD$, $|\fr{1}{N} \max_{\brho \in \cS_N} \oH_N^{\va}(\brho) - \bbE \GS_N(W(W,\vv,\va), \vv(W,\vv,\va))| \le N^{-1/4}$; by Proposition~\ref{prop:sk-concentration} and standard tail bounds on $Z$ this holds with probability $1-2(N+1)^r \exp(-CN^{1/2})$.
        Here we use that for $\va \in \cD$, the constants $C(W(W,\vv,\va), \vv(W,\vv,\va))$ in Proposition~\ref{prop:sk-concentration} are uniformly upper bounded.
    \end{enumerate}
    By adjusting $C$, $\bbP(\cE) \ge 1-\exp(-CN^{1/2})$.
    On $\cE$, if $\bsig\in \cS_N$ maximizes $H_N$, we can find $\bsig' \in \bigcup_{\va \in \cD} \cR_N(\va)$ with $\norm{\bsig'-\bsig}_2 \le O(1/\sqrt{N})$.
    By the Lipschitz condition (\ref{itm:condition-lipschitz}), $|H(\bsig)-H(\bsig')| \le O(1)$.
    So, 
    \begin{align*}
        \GS_N(W,\vv)
        &= 
        \fr{1}{N}
        H_N(\bsig)
        \le 
        \fr{1}{N} H_N(\bsig') + O(1/N) \\
        &\le 
        \max_{\va \in \cD}
        \GS_N(W,\vv;\va) + O(1/N) \\
        &\le 
        \max_{\va \in \cD} \lt(
            \sum_{s\in \sS} \lambda_s v_s \sqrt{a_s}
            +
            \bbE \GS_N(W(W,\vv,\va), \vv(W,\vv,\va))
        \rt)
        + o_N(1).
    \end{align*}
    The subgaussianity from Proposition~\ref{prop:sk-concentration} implies that the contribtion to $\bbE \GS_N(W,\vv)$ from $\cE^c$ is $o_N(1)$, so 
    \begin{align*}
        \bbE \GS_N(W,\vv)
        &\le 
        \max_{\va \in \cD} \lt(
            \sum_{s\in \sS} \lambda_s v_s \sqrt{a_s}
            +
            \bbE \GS_N(W(W,\vv,\va), \vv(W,\vv,\va))
        \rt)
        + o_N(1) \\
        &\le 
        \max_{\va \in [0,1]^\sS} \lt(
            \sum_{s\in \sS} \lambda_s v_s \sqrt{a_s}
            +
            \bbE \GS_N(W(W,\vv,\va), \vv(W,\vv,\va))
        \rt)
        + o_N(1).
    \end{align*}
    Taking $\limsup_{N\to\infty}$ on both sides yields the result.
\end{proof}

\begin{proof}[Proof of Proposition~\ref{prop:sk-ub-one-rep}]
    We will show that any $A$ satisfying the properties in Lemma~\ref{lem:A-basic-properties} and the bound \eqref{eq:A-fn-ineq} must satisfy
    \[
        A(W,\vv) 
        \le 
        A_*(W,\vv)
        \equiv
        \sum_{s\in \sS} 
        \lambda_s
        \sqrt{v_s^2 + 2\sum_{s'\in \sS} \lambda_{s'} w_{s,s'}^2}.
    \]
    Clearly $A_*$ satisfies the conclusions of Lemma~\ref{lem:A-basic-properties}, with equality in assertion (\ref{itm:A-quadratic}).
    For any $\va \in [0,1]^{\sS}$,
    \begin{align*}
        A_*\lt(W(W,\vv,\va),\vv(W,\vv,\va)\rt)
        &= 
        \sum_{s\in \sS}
        \lambda_s
        \sqrt{
            2(1-a_s) \sum_{s'\in \sS} a_{s'} \lambda_{s'} w_{s,s'}^2 + 
            2\sum_{s'\in \sS} \lambda_{s'} (1-a_s)(1-a_{s'}) w_{s,s'}^2
        } \\
        &= 
        \sum_{s\in \sS}
        \lambda_s
        \sqrt{
            2(1-a_s) \sum_{s'\in \sS} \lambda_{s'} w_{s,s'}^2
        },
        \\
        \implies
        \sum_{s\in \sS}
        \lambda_s v_s \sqrt{a_s}
        +
        A_*\lt(W(W,\vv,\va),\vv(W,\vv,\va)\rt)
        &= 
        \sum_{s\in \sS}
        \lambda_s \lt(
            \sqrt{a_s} v_s + 
            \sqrt{1-a_s}
            \sqrt{2\sum_{s'\in \sS} \lambda_{s'} w_{s,s'}^2} 
        \rt) \\
        &\le 
        \sum_{s\in \sS}
        \lambda_s \sqrt{
            v_s^2 +
            2\sum_{s'\in \sS} \lambda_{s'} w_{s,s'}^2
        } 
        = 
        A_*(W,\vv)
    \end{align*}
    by Cauchy-Schwarz. Equality holds when
    \begin{equation}
        \label{eq:sk-a-opt}
        a_s = \fr{v_s^2}{v_s^2 + 2\sum_{s'\in \sS} \lambda_{s'} w_{s,s'}^2}
    \end{equation}
    for all $s\in \sS$, and so $A_*$ satisfies \eqref{eq:A-fn-ineq} with equality.
    
    Suppose $A$ satisfies the conclusions of Lemma~\ref{lem:A-basic-properties} and the inequality \eqref{eq:A-fn-ineq}, and there exists $(W,\vv)$ with $A(W,\vv) > A_*(W,\vv)$.
    By homogeneity (Lemma~\ref{lem:A-basic-properties}(\ref{itm:A-homogenity})), we can assume $1 = \norm{W}_1 \equiv \sum_{s,s'\in \sS} w_{s,s'}$.
    For any small $\delta > 0$, we may choose $(W^*, \vv^*)$ such that $\norm{W^*}_1 = 1$ and 
    \[
        A(W^*,\vv^*) - A_*(W^*,\vv^*)
        \ge 
        (1-\delta)
        \sup_{(W,\vv) : \norm{W}_1=1}
        \lt(
            A(W,\vv) - A_*(W,\vv)
        \rt) 
        > 0.
    \]
    Set 
    \[
        \va^* 
        = 
        \argmax_{\va \in [0,1]^\sS}
        \sum_{s\in \sS}
        \lambda_s v^*_s \sqrt{a_s} + A\lt(W(W^*,\vv^*,\va), \vv(W^*,\vv^*,\va)\rt),
    \]
    and $W' = W(W^*,\vv^*,\va^*), \vv' = \vv(W^*,\vv^*,\va^*)$, so
    \begin{align*}
        A(W^*,\vv^*) 
        &\le 
        \sum_{s\in \sS} 
        \lambda_s v^*_s \sqrt{a_s^*} +
        A(W',\vv'), \\
        A_*(W^*,\vv^*)
        &\ge 
        \sum_{s\in \sS} 
        \lambda_s v^*_s \sqrt{a_s^*} +
        A_*(W',\vv').
    \end{align*}
    Here, the second inequality uses that $A_*$ satisfies \eqref{eq:A-fn-ineq} with equality.
    Therefore
    \begin{equation}
        \label{eq:A-max-diff}
        A(W',\vv') - A_*(W',\vv')
        \ge 
        A(W^*,\vv^*) - A_*(W^*,\vv^*)
        \ge 
        (1-\delta)
        \sup_{(W,\vv) : \norm{W}_1=1}
        \lt(
            A(W,\vv) - A_*(W,\vv)
        \rt).
    \end{equation}
    By homogeneity, this implies $\norm{W'}_1 \ge 1-\delta$.
    Let $\sS_0 \subseteq \sS$ be the set of $s$ for which there exists $s'$ with $w_{s,s'} \ge \delta_1 \equiv \sqrt{2\delta}$.
    For such $s,s'$, 
    \[
        \delta 
        \ge 
        \norm{W^*}_1 - \norm{W'}_1
        \ge 
        w^*_{s,s'} - w'_{s,s'}
        \ge 
        \lt(1 - \sqrt{1-a_s}\rt) w_{s,s'}
        \ge 
        \fr12 a_s w_{s,s'}
        \ge 
        \fr12 a_s \delta_1.
    \]
    Thus, for $s\in \sS_0$, $a_s \le \delta_1$.
    Of course, for $s\in \sS \setminus \sS_0$, $w_{s,s'} \le \delta_1$ for all $s'\in \sS$.
    Thus for all $s\in \sS$, 
    \[
        v'_s \le \sqrt{2\sum_{s'\in \sS} \lambda_{s'} \delta_1} = \sqrt{2\delta_1} \equiv \delta_2.
    \]
    By parts (\ref{itm:A-subadditive}), (\ref{itm:A-linear}), and (\ref{itm:A-quadratic}) of  Lemma~\ref{lem:A-basic-properties},
    \[
        A(W',\vv') 
        \le 
        A(W',\vzero) + A(0,\vv')
        \le 
        A(W',\vzero) + \delta_2
        \le 
        A_*(W', \vzero) + \delta_2.
    \]
    By inspection, $A_*(W',\vv') \ge A_*(W',\vzero)$.
    Thus
    \[
        A(W',\vv') - A_*(W',\vv')
        \le 
        \delta_2.
    \]
    For small enough $\delta > 0$, this contradicts \eqref{eq:A-max-diff}.
\end{proof}

Finally, the upper bound for $k=1$ directly implies the upper bound for general $k$.
\begin{corollary}
    \label{cor:sk-ub}
    For $W, \vv$ as in Lemma~\ref{lem:sk-ext-field},
    \[
        \limsup_{N\to\infty} \bbE \GS_N(W,\vv,k)
        \le
        \sum_{s\in \sS} \lambda_s \sqrt{v_s^2 + 2\sum_{s'\in \sS} \lambda_{s'} w_{s,s'}^2}.
    \]
\end{corollary}
\begin{proof}
    Note that (recall \eqref{eq:bbtperp})
    \[
        \GS_N(W,\vv,k)
        = 
        \fr{1}{kN} 
        \max_{\vbsig \in \cS_N^{k,\perp}} 
        H_{N,k}
        \le 
        \fr{1}{k} 
        \sum_{i=1}^k
        \fr{1}{N}
        \max_{\bsig^i \in \cS_N} 
        H^i_{N,k}(\bsig^i).
    \]
    Taking expectations yields $\bbE \GS_N(W,\vv,k) \le \bbE \GS_N(W,\vv,1)$.
    This and Proposition~\ref{prop:sk-ub-one-rep} imply the result.
\end{proof}

\begin{remark}
    The proof of Proposition~\ref{prop:sk-ub-one-rep} via the recursive inequality \eqref{eq:A-fn-ineq} extends to the ground state energies in multi-species spherical spin glasses with general (non-quadratic) interactions. 
    It thus gives an elementary way to upper bound the ground state energy for spin glasses with external field given the ground state energy of spin glasses without external field, when the latter is known. 
    As we will see in the next subsection, it is possible to construct points where this recursive inequality holds with (approximate) equality, so the upper bound is sharp.
\end{remark}

\subsection{Lower Bound}

In this subsection, we will constructively prove the matching lower bound to Corollary~\ref{cor:sk-ub}.
\begin{proposition}
    \label{prop:sk-lb}
    For $W,\vv$ as in Lemma~\ref{lem:sk-ext-field}, 
    \[
        \liminf_{N\to\infty} \bbE \GS_N(W,\vv,k)
        \ge
        \sum_{s\in \sS} \lambda_s \sqrt{v_s^2 + 2\sum_{s'\in \sS} \lambda_{s'} w_{s,s'}^2}.
    \]
\end{proposition}

\begin{lemma}   
    \label{lem:sk-gram-schmidt}
    Let $S_N = \{\bx\in \bbR^N : \norm{\bx}_2 = \sqrt{N}\}$.
    Suppose $\by^1,\ldots,\by^k \in S_N$ satisfy $|\la \by^i, \by^j\ra| \le N^{2/3}$ for all $i\neq j$. 
    Then there exist pairwise orthogonal $\bz^1,\ldots,\bz^k \in S_N$ such that $\Span(\bz^1,\ldots,\bz^k) = \Span(\by^1,\ldots,\by^k)$ and $\la \by^i,\bz^i\ra \ge N - 4kN^{1/3}$.
\end{lemma}
\begin{proof}
    We define $\bz^1,\ldots,\bz^k$ by applying the Gram-Schmidt algorithm to $\by^1,\ldots,\by^k$: let $\bz^1=\by^1$, and for $2\le i\le k$, let
    \[
        \tby^i = \by^i - \sum_{j=1}^{i-1} \fr{\la \bz^j, \by^i\ra}{N} \bz^j,
        \qquad
        \bz^i = \fr{\sqrt{N}}{\norm{\tby^i}_2} \tby^i.
    \]
    Clearly $\Span(\bz^1,\ldots,\bz^k) = \Span(\by^1,\ldots,\by^k)$.
    We will prove by induction on $i$ that for all $\ell >i$, $|\la \bz^i, \by^\ell\ra| \le 2N^{2/3}$.
    The base case $i=1$ is true by hypothesis.
    For $i>1$, we have
    \[
        \norm{\tby^i}_2^2
        =
        N \lt(1 - \sum_{j=1}^{i-1} \fr{\la \bz^j, \by^i\ra^2}{N}\rt)
        \in 
        \lt[N(1-4kN^{-2/3}),N\rt],
    \]
    using the inductive hypothesis.
    Moreover, for $\ell >i$, 
    \[
        |\la \tby^i, \by^\ell\ra|
        \le 
        |\la \by^i, \by^\ell\ra|
        +
        \sum_{j=1}^{i-1} \fr{|\la \bz^j, \by^i\ra| |\la \bz^j, \by^\ell\ra|}{N}
        \le 
        N^{2/3}\lt(1 + 4kN^{-1/3}\rt).
    \]
    Therefore
    \[
        |\la \bz^i, \by^\ell\ra| 
        \le 
        N^{2/3} 
        \lt(1 + 4kN^{-1/3}\rt) \lt(1-kN^{-2/3}\rt)^{-1/2}
        \le 2N^{2/3},
    \]
    completing the induction.
    Now $\la \tby^i, \by^i\ra = \norm{\tby^i}_2^2$, so 
    \[
        \la \bz^i, \by^i\ra 
        =
        \sqrt{N} \norm{\tby^i}_2
        \ge 
        N \lt(1-4kN^{-2/3}\rt)^{1/2}
        \ge 
        N-4kN^{1/3}.
    \]
\end{proof}
Recall that $\lambda_{s,N} = |\cI_s|/N$.
Let $\delta_N = \max_{s\in \sS} |\fr{\lambda_{s,N}}{\lambda_s} - 1|$.
\begin{lemma}
    \label{lem:sk-multi-gram-schmidt}
    There exists an event $\cE \in \sigma(\bg^1,\ldots,\bg^k)$ with $\bbP(\cE) \ge 1-\exp(-CN^{1/3})$ such that on this event, there exists $\vbx = (\bx^1,\ldots,\bx^k)\in \cS_N^{k,\perp}$ such that the following properties hold.
    \begin{enumerate}[label=(\alph*), ref=\alph*]
        \item \label{itm:x-norms-good} For all $i$, $\norm{R(\bg^i,\bg^i) - \vone}_\infty \le \delta_N + N^{-1/4}$.
        \item \label{itm:x-approx-g} For all $i$, $\norm{R(\bg^i,\bx^i) - \vone}_\infty \le \delta_N + N^{-1/4}$. 
        \item \label{itm:x-unit} For all $i$, $R(\bx^i,\bx^i) = \vone$.
        \item \label{itm:x-span} For all $s\in \sS$, $\Span(\bx^1_s,\ldots,\bx^k_s) = \Span(\bg^1_s,\ldots,\bg^k_s)$.
    \end{enumerate}
\end{lemma}
\begin{proof}
    By standard concentration inequalities, for each $i\in [k]$ and $s\in \sS$, $|\la \bg^i_s, \bg^i_s\ra - \lambda_{s,N} N| \le N^{2/3}$ with probability $1-\exp(-CN^{1/3})$, which implies
    \[
        \lt|\fr{\la \bg^i_s, \bg^i_s\ra}{\lambda_s N}-1\rt|
        \le 
        \lt|\fr{\lambda_{s,N}}{\lambda_s}-1\rt| + \fr{N^{2/3}}{\lambda_s N}
        \le 
        \delta_N + N^{-1/4}.
    \]
    When this holds for all $i\in [k]$, $s\in \sS$, part (\ref{itm:x-norms-good}) follows.
    
    For each $i\in [k]$, define $\hbg^i \in \bbR^N$ by $\hbg^i_s = \fr{\sqrt{\lambda_{s,N} N}}{\norm{\bg^i_s}_2} \bg^i_s$ for all $s\in \sS$.
    Note that each $\hbg^i_s$ is a uniformly random point on the sphere of radius $\sqrt{\lambda_{s,N} N}$ supported on the coordinates $\cI_s$.
    
    Fix $s\in \sS$. 
    By standard concentration inequalities, for each pair of distinct $i, j\in [k]$, $|\la \hbg^i_s, \hbg^j_s\ra| \le (\lambda_{s,N}N)^{2/3}$ with probability $1-\exp(-CN^{1/3})$. 
    If this holds for all $s,i,j$, Lemma~\ref{lem:sk-gram-schmidt} implies the existence of orthogonal $\bz^1_s,\ldots,\bz^k_s$ on the sphere of radius $\sqrt{\lambda_{s,N}N}$ supported on coordinates $\cI_s$ with 
    \begin{equation}
        \label{eq:sk-span}
        \Span(\bz^1_s,\ldots,\bz^k_s) = \Span(\hbg^1_s,\ldots,\hbg^k_s)
    \end{equation}
    and
    \[
        \lambda_{s,N}N - 4k(\lambda_{s,N}N)^{1/3} \le \la \bz^k_s, \hbg^i_s\ra \le \lambda_{s,N}N.
    \]
    Let $\bx^i_s = \bz^i_s \cdot \sqrt{\lambda_s / \lambda_{s,N}}$, so
    \[
        \fr{\la \bx^i_s, \bg^i_s\ra}{\lambda_s N}
        = 
        \fr{\la \bz^i_s, \hbg^i_s\ra}{\lambda_{s,N} N} \cdot \sqrt{\fr{\lambda_{s,N}}{\lambda_s }} \cdot \fr{\norm{\bg^i_s}}{\sqrt{\lambda_{s,N}N}}
        =
        (1 + O(N^{-1/3}) \sqrt{\fr{\lambda_{s,N}}{\lambda_s }}.
    \]
    Thus
    \[
        \lt|\fr{\la \bx^i_s, \bg^i_s\ra}{\lambda_s N}-1\rt|
        \le 
        \lt|\sqrt{\fr{\lambda_{s,N}}{\lambda_s }} - 1\rt| + O(N^{-1/3}) 
        \le 
        \delta_N + N^{-1/4}.
    \]
    If this holds for all $s$, part (\ref{itm:x-approx-g}) follows.
    By a union bound, adjusting $C$ as necessary, the above events simultaneously hold with probability $1-\exp(-CN^{1/3})$.
    By construction, $R(\bx^i,\bx^i) = \vone$ and $R(\bx^i,\bx^j) = \vzero$ for all $i\neq j$, which implies part (\ref{itm:x-unit}) and $\vbx \in \cS_N^{k,\perp}$.
    The relation \eqref{eq:sk-span} implies part (\ref{itm:x-span}).
\end{proof}

The following recursive lower bound for $\bbE \GS_N(W,\vv,k)$ is a converse to Lemma~\ref{lem:A-fn-ineq} and is the main step in the proof of Proposition~\ref{prop:sk-lb}.
\begin{lemma}
    \label{lem:sk-fe-lb}
    Let $W, \vv$ be as in Lemma~\ref{lem:sk-ext-field} and $\va \in [0,1]^\sS$, and set $W' = W(W,\vv,\va)$, $\vv' = \vv(W,\vv,\va)$. Then,
    \[
        \bbE \GS_N(W,\vv,k)
        \ge 
        \sum_{s\in \sS}
        \lambda_s v_s \sqrt{a_s}
        + 
        \bbE \GS_{N-kr}(W',\vv',k)
        - o_N(1),
    \]
    where $\GS_{N-kr}$ denotes the ground state energy (see \eqref{eq:sk-gsn}) of a dimension $N-kr$ multi-species quadratic spin glass with species sizes $\tilde \cI_s = \cI_s - k$.
\end{lemma}
\begin{proof}
    Suppose for now the event $\cE$ in Lemma~\ref{lem:sk-multi-gram-schmidt} holds and let $\vbx = (\bx^1,\ldots,\bx^k)$ be as in this lemma.
    Let
    \begin{equation}
        \label{eq:bbtnperp}
        \cS_{N,\perp} 
        \equiv
        \lt\{
            \brho \in \cS_N : 
            R(\brho, \bx^i) = \vzero
            ~\forall i\in [k]
        \rt\}
        =
        \lt\{
            \brho \in \cS_N : 
            R(\brho, \bg^i) = \vzero
            ~\forall i\in [k]
        \rt\}
    \end{equation}
    where the second equality follows from Lemma~\ref{lem:sk-multi-gram-schmidt}(\ref{itm:x-span}) and 
    \begin{equation}
        \label{eq:bbtnperpkperp}
        \cS_{N,\perp}^{k,\perp}
        \equiv
        \lt\{
            \vbrho = (\brho^1, \ldots, \brho^k) \in \cS_{N,\perp}^k : 
            R(\brho^i,\brho^j) = \vzero ~\forall i\neq j
        \rt\}.
    \end{equation}
    For each $i\in [k]$ let 
    \begin{equation}
        \label{eq:sk-recursive-bsig}
        \bsig^i = \sqrt{\va} \diamond \bx^i + \sqrt{\vone - \va} \diamond \brho^i    
    \end{equation}
    where $\vbrho = (\brho^1, \ldots, \brho^k) \in \cS_{N,\perp}^{k,\perp}$.
    The orthogonality relations in \eqref{eq:bbtnperp} and \eqref{eq:bbtnperpkperp} imply $\vbsig = (\bsig^1, \ldots, \bsig^k) \in \cS_N^{k,\perp}$.
    Then, 
    \begin{equation}
        \label{eq:sk-k-step}
        \fr{1}{N} H_{N,k}(\vbsig)
        =
        \fr{1}{kN}
        \sum_{i=1}^k
        \la \vv \diamond \bg^i, \sqrt{\va} \diamond \bx^i \ra + 
        \fr{1}{kN^{3/2}}
        \sum_{i=1}^k
        \lt\la W \diamond \bG, (\sqrt{\va} \diamond \bx^i + \sqrt{\vone - \va} \diamond \brho^i)^{\otimes 2} \rt\ra.
    \end{equation}
    By Lemma~\ref{lem:sk-multi-gram-schmidt}(\ref{itm:x-norms-good}, \ref{itm:x-approx-g}),
    \[
        \fr{1}{N} 
        \la W \diamond \bG, \sqrt{\va} \diamond \bx^i\ra
        = 
        \sum_{s\in \sS}
        \lambda_s v_s \sqrt{a_s}
        + o_N(1).
    \]
    Note that the state space $\cS_{N,\perp}$ is $\cS_N$ with $k$ fewer dimensions in each species, and these dimensions (and $\vbx$) are independent of $\bG$. 
    So, optimizing the second term of \eqref{eq:sk-k-step} over $\vbrho \in \cS_{N,\perp}^{k,\perp}$ is equivalent to optimizing a dimension $N-kr$ multi-species quadratic spin glass.
    The same covariance calculation as \eqref{eq:sk-hH-covariance-calculation} shows that
    \[
        \sup_{\vbrho \in \cS_{N,\perp}^{k,\perp}}
        \fr{1}{N^{3/2}}
        \lt\la W \diamond \bG, (\sqrt{\va} \diamond \bx^i + \sqrt{\vone - \va} \diamond \brho^i)^{\otimes 2} \rt\ra
        =_d
        \sqrt{\fr{N-kr}{N}}
        \GS_{N-kr}(W', \vv')
        +
        O(N^{-1/2}) Z,
    \]
    where $Z\sim \cN(0,1)$ is independent of $\GS_{N-kr}$.
    Thus
    \[
        \bbE \GS_N(W,\vv,k)
        \ge 
        \bbE \ind\{\cE\} \fr{1}{N} H_{N,k}(\vbsig)
        \ge 
        \sum_{s\in \sS}
        \lambda_s v_s \sqrt{a_s}
        + 
        \bbE \GS_{N-kr}(W', \vv',k)
        - o_N(1).
    \]    
\end{proof}

Lemma~\ref{lem:sk-fe-lb} suggests a natural way to construct an approximate ground state of $H_{N,k}$.
First, use Gram-Schmidt orthogonalization to produce $\vbx = (\bx^1,\ldots,\bx^k)$ from the external fields $\bg^1,\ldots,\bg^k$, as in Lemma~\ref{lem:sk-multi-gram-schmidt}.
Choose $\va \in [0,1]^\sS$ and set $\vbsig$ as in \eqref{eq:sk-recursive-bsig}, for $\vbrho \in \cS_{N,\perp}^{k,\perp}$ to be determined.
The correlations of the $\bsig^i$ with the external fields $\bg^i$ contribute energy $\sum_{s\in \sS} \lambda_s v_s \sqrt{a_s}$, while the optimization over $\vbrho$ is equivalent to optimizing another quadratic multi-species spin glass, whose parameters depend on $\va$.
Finally, recursively optimize $\vbrho$.
The following proof demonstrates that when $\vv > \vzero$, there exists a sequence of choices of $\va$ such that running this algorithm to a large constant recursion depth finds a near ground state $\vbsig \in \cS_N^{k,\perp}$ of $H_{N,k}$.
(If some entries of $\vv$ are zero, the algorithm succeeds after first introducing a small artificial external field.)

\begin{proof}[Proof of Proposition~\ref{prop:sk-lb}]
    Assume for now that $\vv \succ \vzero$ where the inequality is strict in each coordinate.
    Define $W^{(0)} = W$, $\vv^{(0)}=\vv$. 
    Denote the relation \eqref{eq:sk-a-opt} by $\va = \va(W,\vv)$.
    Let $T$ be a large constant to be determined, and for $0\le t\le T-1$ define
    \[
        \va^{(t)} = \va(W^{(t)}, \vv^{(t)}),
        \qquad
        W^{(t+1)} = W(W^{(t)}, \vv^{(t)}, \va^{(t)}),
        \qquad
        \vv^{(t+1)} = \vv(W^{(t)}, \vv^{(t)}, \va^{(t)}).
    \]
    Further define
    \[
        E^{(t)} = \sum_{s\in \sS} \lambda_s \sqrt{(v^{(t)}_s)^2 + 2\sum_{s'\in \sS} \lambda_{s'} (w_{s,s'}^{(t)})^2}, 
        \qquad
        F^{(t)} = \sum_{s\in \sS} \lambda_s v_s^{(t)} \sqrt{a_s^{(t)}}.
    \]
    Let $\delta > 0$ be arbitrary; we will show that $\bbE \GS_N(W,\vv) \ge E^{(0)} - \delta$ for all sufficiently large $N$.
    Lemma~\ref{lem:sk-fe-lb} with the choice $\va = \va^{(t)}$ implies that
    \[
        \bbE \GS_{N-tkr}(W^{(t)}, \vv^{(t)}) 
        \ge 
        F^{(t)} + \bbE \GS_{N-(t+1)kr}(W^{(t+1)}, \vv^{(t+1)}) - o_N(1),
    \]
    and summing yields
    \[
        \bbE \GS_N(W,\vv)
        \ge 
        \sum_{t=0}^{T-1} F^{(t)} - o_N(1).
    \]
    Note that
    \begin{align*}
        F^{(t)} &= \sum_{s\in \sS} \lambda_s \sqrt{(v^{(t)}_s)^2 + 2\sum_{s'\in \sS} \lambda_{s'} (w_{s,s'}^{(t)})^2} \cdot a_s^{(t)}, \\
        E^{(t+1)} &= \sum_{s\in \sS} \lambda_s \sqrt{(v^{(t)}_s)^2 + 2\sum_{s'\in \sS} \lambda_{s'} (w_{s,s'}^{(t)})^2} \cdot (1-a_s^{(t)}), 
    \end{align*}
    so $F^{(t)} = E^{(t)} - E^{(t+1)}$.
    Thus
    \[
        \bbE \GS_N(W,\vv)
        \ge 
        E^{(0)} - E^{(T)} - o_N(1).
    \]
    Since
    \[
        (v_s^{(t+1)})^2 + 2\sum_{s'\in \sS} \lambda_{s'} (w_{s,s'}^{(t+1)})^2 
        = 
        2 (1-a_s^{(t)}) \sum_{s'\in \sS} \sum_{s'} \lambda_{s'} (w_{s,s'}^{(t)})^2
    \]
    we have
    \[
        a_s^{(t+1)} 
        = 
        \fr{(v_s^{(t+1)})^2}{(v_s^{(t+1)})^2 + 2\sum_{s'\in \sS} \lambda_{s'} (w_{s,s'}^{(t+1)})^2}
        =
        \fr{\sum_{s'\in \sS} a_{s'}^{(t)} \lambda_{s'} (w_{s,s'}^{(t)})^2}{\sum_{s'\in \sS} \lambda_{s'} (w_{s,s'}^{(t)})^2}.
    \]
    It follows that, for $\alpha^{(t)} = \min_{s\in \sS} a_s^{(t)}$, we have $\alpha^{(t+1)} \ge \alpha^{(t)}$.
    The assumption $\vv > \vzero$ ensures $\alpha^{(0)} > 0$.
    Because $E^{(t+1)}/E^{(t)} \le 1 - \alpha^{(t)}$, we have
    \[
        E^{(T)} \le E^{(0)} (1 - \alpha^{(0)})^T < \delta/2
    \]
    for sufficiently large constant $T$.
    This implies $\bbE \GS_N(W,\vv) \ge E^{(0)} - \delta/2 - o_N(1) \ge E^{(0)} - \delta$.
    This proves the result when $\vv \succ \vzero$.
    
    If some coordinates of $\vv$ are zero, we apply this result to $\vv + \eta \vone$ for small $\eta > 0$.
    By \eqref{eq:A-subadditive},
    \[
        \GS_N(W,\vv)
        \ge 
        \GS_N(W,\vv+\eta \vone)
        -
        \GS_N(0,\eta \vone),
    \]
    so for sufficiently large $N$,
    \[
        \bbE \GS_N(W,\vv)
        \ge 
        \sum_{s\in \sS}
        \lambda_s 
        \sqrt{v_s^2 + 2\sum_{s'\in \sS} \lambda_{s'} w_{s,s'}^2}
        - \delta - \eta.
    \]
    As this holds for any $\delta, \eta > 0$ the result follows.
\end{proof}

\begin{proof}[Proof of Lemma~\ref{lem:sk-ext-field}]
    Follows from Corollary~\ref{cor:sk-ub} and Proposition~\ref{prop:sk-lb}.
\end{proof}

%% file: final-tex/a2-details-of-alg.tex
\section{Deferred Proofs From Section~\ref{sec:alg}}
\label{app:alg}

\subsection{Existence of a Maximizer: Proof of Proposition~\ref{prop:F-max}}
\label{subsec:maximizer-existence}

\propFmax*

Given $(p,\Phi,q_0)\in\cM$ we extend $p,\Phi$ to domain $[0,1]$ by setting $p(q)=0$ for $q\in [0,q_0)$ and making $\Phi$ linear on $[0,q_0]$ with $\Phi(0)=\vzero$. Using this canonical extension we equip $\cM$ with the metric 
    \begin{equation}
    \label{eq:metric}
    d\lt((p^1,\Phi^1,q_0^1),(p^2,\Phi^2,q_0^2)\rt)
    =
    \norm{p^1-p^2}_{L^1([0,1])}
    +
    \norm{\Phi^1-\Phi^2}_{L^1([0,1])}
    + 
    |q_0^1-q_0^2|.
\end{equation}

We will prove that $\cM$ is a compact space on which $\bbA$ is upper semi-continuous. Existence of a triple $(p,\Phi;q_0)\in\cM$ maximizing \eqref{eq:alg} within this space then follows.

\begin{proposition}
\label{prop:compact-space}
The space $\cM$ with metric \eqref{eq:metric} is compact.
\end{proposition}

\begin{proof}
Given an infinite sequence $(p^n,\Phi^n,q_0^n)_{n\geq 0}$ of points in $\cM$, we show there is a limit point. First find a subsequence $(a_n)$ along which the convergence $q_0^{a_n}\to q_0$ holds. Then the subsequence $(p^{a_n})_{n\geq 0}$ has a subsubsequential limit in the space $L^1([q_0,1])$; similarly for $(\Phi_s^{a_n})_{n\geq 0}$, for each $s\in\sS$. Thus we may choose a subsequence $b_n$ of $a_n$ on which $p^{b_n}\to p$ and $\Phi_s^{b_n}\to \Phi_s$ (for all $s\in \sS$) in $L^1([q_0,1])$. It is easy to see that $p$ and each $\Phi_s$ vanishes on $[0,q_0)$, and that $\Phi$ satisfies admissibility. It is easy to see that 
\[
    \|p^{b_n}-p\|_{L^1([0,1])}
    \leq
    \|p^{b_n}-p\|_{L^1([q_0,1])}
    +
    |q_0-q^{b_n}_0|
\]
and 
\[
    \|\Phi_s^{b_n}-\Phi_s\|_{L^1([0,1])}
    \leq
    \|\Phi_s^{b_n}-\Phi_s\|_{L^1([q_0,1])}
    +
    |q_0-q^{b_n}_0|.
\]
It follows that $(p^{b_n},\Phi^{b_n},q_0^{b_n})\to (p,\Phi,q_0)$ in $\cM$. This completes the proof.
\end{proof}

\begin{proposition}
\label{prop:F-bounded}
The function $\bbA$ is uniformly bounded on $\cM$.
\end{proposition}

\begin{proof}
For any admissible $\Phi$ we have by Cauchy-Schwarz
\begin{align}
\nonumber
    \sum_{s\in\sS}\lambda_s\int_0^1 \sqrt{{\Phi}'_s(q)(p\times \xi^s\circ\Phi)'(q)}\de q
    &\leq 
    \sum_{s\in\sS}\lambda_s\int_0^1 \lt({\Phi}_s'(q)+(p\times \xi^s\circ\Phi)'(q)\rt)\de q
    \\
\label{eq:F-bounded}
    &\leq
    \sum_{s\in\sS}\lambda_s \lt(1+\xi^s(\vone)-\xi^s(\vzero)\rt).
\end{align}
The first term of $\bbA$ is clearly uniformly bounded, so the result follows.
\end{proof}

\begin{proposition}
\label{prop:F-usc}
$\bbA$ is upper semi-continuous on $\cM$.
\end{proposition}

\begin{proof}
Suppose $(p^{b_n},\Phi^{b_n},q_0^{b_n})\to (p,\Phi,q_0)$ in $\cM$. We write
\begin{align*}
    |\bbA(p^{b_n},\Phi^{b_n};q_0^{b_n})-\bbA(p,\Phi;q_0)|
    &\leq
    \int_{q_0}^1
    \lt|\sqrt{(\Phi_s^{b_n})'(q)(p^{b_n}\times \xi^s\circ\Phi^{b_n})'(q)}
    -
    \sqrt{\Phi'_s(q)(p\times \xi^s\circ\Phi)'(q)}\rt|\de q
    \\
    &  \quad +
    C_{\lambda}
    \lt(
    \lt|\int_{q_0^{b_n}}^{q_0}
    \sqrt{p'(q)+1}
    \,\de q\,
    \rt|
    +
    \sum_{s\in \sS}
    \lt|\sqrt{\Phi_s^n(q_0^n)}-\sqrt{\Phi_s(q_0)}\rt|
    \rt)
    .
\end{align*}
The sum over $s\in\sS$ obviously tends to $0$. Moreover by Cauchy--Schwarz, 
\begin{align*}
    \lt|\int_{q_0^{b_n}}^{q_0}
    \sqrt{p'(q)+1}
    \,\de q\,
    \rt|
    &\leq
    |q_0-q_0^{b_n}|^{1/2}\cdot \sqrt{p(q_0)-p(q_0^{b_n})+1}
    \\
    &\leq
    C'_{\lambda}(q_0-q_0^{b_n})^{1/2}.
\end{align*}
Therefore it suffices to show the first term above tends to $0$. Since the map
\[
    (p,\Phi)\mapsto (p\times \xi^s\circ\Phi)
\]
from $L^1([0,1])^{|\sS|+1}\to L^1([0,1])$ is continuous and returns a non-decreasing function, it suffices to show that
\[
    G(f,g)=\int_0^1 \sqrt{f'(q)g'(q)}\,\de q
\]
is upper semi-continuous on $L^1([0,1])\times L^1([0,1])$ when restricted to non-decreasing functions. This is essentially equivalent to upper semi-continuity of Hellinger distance which is well-known.
\end{proof}

Combining the results above implies Proposition~\ref{prop:F-max}.

\subsection{A Priori Regularity of Maximizers}
\label{subsec:regularity-for-4.1}

Let $(p,\Phi,q_0)\in\cM$ be a maximizer of $\bbA$, which exists by Proposition~\ref{prop:F-max}. In this subsection we will prove the following two propositions. 

\propBasicRegularity*

\propPBasic*

\begin{lemma}
    \label{lem:p-AC}
    The function $p$ is absolutely continuous and $p(1)=1$. Moreover $p'$ is uniformly bounded on compact subsets of $(q_0,1)$.
\end{lemma}
\begin{proof}
    Given any increasing $p:[q_0,1]\to [q_0,1]$, we may view $p'$ as a positive measure of the form
    \begin{equation}
    \label{eq:positive-measure-of-the-form}
        p'(x)dx = f(x)dx + \mu(dx)
    \end{equation}
    for $\mu$ a singular-plus-atomic measure and $f\in L^1([q_0,1];\mathbb R_{\geq 0})$. We may then replace $p$ by $\bar p$ such that 
    \[
        \bar p'(x)dx=f(x)dx,\quad\text{ and }\quad \bar p(1)=1.
    \]
    Then $\bar p(x)\geq p(x)$ for all $x\in [q_0,1]$, and $\bar p'(x)$ agrees with $p'(x)$ except for a singular-plus-atomic part. It follows that 
    \[
        \bbA(p,\Phi;q_0)\leq \bbA(\bar p,\Phi;q_0).
    \]
    Moreover it is easy to see that strict inequality $\bbA(p,\Phi;q_0)< \bbA(\bar p,\Phi;q_0)$ holds whenever $p\neq \bar p$. We conclude that $p$ is absolutely continuous and $p(1)=1$.

    To show the latter statement, we use a similar argument with more care. Let $q\in (q_0,1)$ and choose a large constant $C=C(q_0,q)$. Recalling \eqref{eq:positive-measure-of-the-form}, suppose $\|f(x)\|_{L^{\infty}([q,1])}> C$ for a large constant $C$ and let 
    \[
    c\equiv\frac{\int_q^1 (f(x)-C)_+~\de x}{q-q_0}.
    \]
    We may replace $f$ by
    \[
        f_C(x)=
        \begin{cases}
        f(x),\quad x\in [0,q_0)\\
        f(x)+c,\quad x\in [q_0,q)\\
        \min(C,f(x)),\quad x\in [q,1]
        \end{cases}
    \]
    and similarly replace $p$ by $p_C$ with 
    \[
        \bar p_C'(x)dx=f_C(x)~\de x,\quad\text{ and }\quad \bar p(1)=1.
    \]
    It is easy to see that $p_C(x)\geq p(x)$ for each $x\in [0,1]$. Keeping $\Phi$ the same, we consider the change in $\bbA$. The decrease in $\bbA$ on $[q,1]$ is at most
    \begin{equation}
    \label{eq:q-1}
    \begin{aligned}
        &\sum_{s\in\sS}
        \lambda_s
        \int_q^1
        \sqrt{\Phi_s'(x)(p\times \xi^s\circ\Phi)'(x)}
        -
        \sqrt{\Phi_s'(x)(p_C\times \xi^s\circ\Phi)'(x)}
        ~\de x
        \\
        &=
        \sum_{s\in\sS}
        \lambda_s
        \int_q^1
        \sqrt{\Phi_s'(x)\lt(p'(x)\xi^s\big(\Phi(x)\big)+p(x)\langle \Phi'(x),\nabla\xi^s\big(\Phi(x)\big)\rangle\rt)}
        \\
        &\quad\quad\quad\quad\quad\quad -
        \sqrt{\Phi_s'(x)\lt(p_C'(x)\xi^s\big(\Phi(x)\big)+p_C(x)\langle \Phi'(x),\nabla\xi^s\big(\Phi(x)\big)\rangle\rt)}
        ~\de x
        \\
        &\leq
        \sum_{s\in\sS}
        \lambda_s
        \int_q^1
        \sqrt{\Phi_s'(x)\lt(p'(x)\xi^s\big(\Phi(x)\big)+p(x)\langle \Phi'(x),\nabla\xi^s\big(\Phi(x)\big)\rangle\rt)}
        \\
        &\quad\quad\quad\quad\quad\quad -
        \sqrt{\Phi_s'(x)\lt(p_C'(x)\xi^s\big(\Phi(x)\big)+p(x)\langle \Phi'(x),\nabla\xi^s\big(\Phi(x)\big)\rangle\rt)}
        ~\de x
        \\
        &\leq
        \sum_{s\in\sS}
        \lambda_s
        \int_q^1
        \sqrt{\Phi_s'(x) p'(x)\cdot\xi^s\big(\Phi(x)\big)}
        -
        \sqrt{\Phi_s'(x)p_C'(x)\cdot\xi^s\big(\Phi(x)\big)}
        ~\de x
        \\
        &\leq
        O(1)\cdot
        \int_q^1
        \sqrt{p'(x)}-\sqrt{p_C'(x)}
        \de x
        \\
        &\leq
        O(1)\cdot \int_{q}^1 C^{-1/2}(f(x)-C)_+ ~\de x
        \\
        &\leq
        O\lt(\frac{c(q-q_0)}{\sqrt{C}}\rt)
        .
    \end{aligned}
    \end{equation}
    (In the second inequality we used $\sqrt{x+z}-\sqrt{y+z}\leq \sqrt{x}-\sqrt{y}$ for $x\geq y\geq 0$, and in the third we used that $\Phi_s'$ is uniformly bounded by admissibility.)
    On $x\in [q_0,q]$, we find that changing from $p$ to $p_C$ increases the value of $\bbA$:
    \begin{align*}
        &\sum_{s\in\sS}
        \lambda_s
        \int_{q_0}^q
        \sqrt{\Phi_s'(x)(p_C\times \xi^s\circ\Phi)'(x)}
        -
        \sqrt{\Phi_s'(x)(p\times \xi^s\circ\Phi)'(x)}
        ~\de x
        \\
        &=
        \sum_{s\in\sS}
        \lambda_s
        \int_{q_0}^q
        \sqrt{\Phi_s'(x)\lt(p_C'(x)\xi^s\big(\Phi(x)\big)+p_C(x)\langle \Phi'(x),\nabla\xi^s\big(\Phi(x)\big)\rangle\rt)}
        \\
        &\quad\quad -
        \sqrt{\Phi_s'(x)\lt(p'(x)\xi^s\big(\Phi(x)\big)+p(x)\langle \Phi'(x),\nabla\xi^s\big(\Phi(x)\big)\rangle\rt)}
        ~\de x
        \\
        &\geq
        \sum_{s\in\sS}
        \lambda_s
        \int_{q_0}^q
        \sqrt{\Phi_s'(x)\lt((p'(x)+c)\xi^s\big(\Phi(x)\big)+p(x)\langle \Phi'(x),\nabla\xi^s\big(\Phi(x)\big)\rangle\rt)}
        \\
        &\quad\quad -
        \sqrt{\Phi_s'(x)\lt(p'(x)\xi^s\big(\Phi(x)\big)+p(x)\langle \Phi'(x),\nabla\xi^s\big(\Phi(x)\big)\rangle\rt)}
        ~\de x
        \\
        &\geq
        \Omega(c)
        \cdot
        \int_{q_0}^q  
        \sum_{s\in\sS}
        \frac{\de x}{ \sqrt{\Phi_s'(x)\lt((p'(x)+c)\xi^s\big(\Phi(x)\big)+p(x)\langle \Phi'(x),\nabla\xi^s\big(\Phi(x)\big)\rangle\rt)}}\,.
    \end{align*}
    By Markov's inequality, $p'(x)\leq \frac{2}{(q-q_0)}$ on a set of $x\in [q_0,q]$ of measure at least $\frac{q-q_0}{2}$. For each such $x$, we have $\Phi_s'(x)\leq O(1)$ and $\xi^s(\Phi(x))\leq O(1)$. We thus find
    \[
        \Omega(c)
        \cdot
        \int_{q_0}^q  
        \sum_{s\in\sS}
        \frac{\de x}{ \sqrt{\Phi_s'(x)\lt((p'(x)+c)\xi^s\big(\Phi(x)\big)+p(x)\langle \Phi'(x),\nabla\xi^s\big(\Phi(x)\big)\rangle\rt)}}
        \geq
        \Omega\lt(\frac{c(q-q_0)}{\sqrt{q-q_0+c}}\rt)
    \]
    Since $c\leq \frac{1}{q-q_0}$, for $C$ sufficiently large, combining with \eqref{eq:q-1} above implies that
    \[
    \sum_{s\in\sS}
        \lambda_s
        \int_{q}^1
        \sqrt{\Phi_s'(x)(p_C\times \xi^s\circ\Phi)'(x)}
        -
        \sqrt{\Phi_s'(x)(p\times \xi^s\circ\Phi)'(x)}
        ~\de x>0.
    \]
    Since $p(x)=p_C(x)$ for $x\leq q_0$, we find $\bbA(p,\Phi;q_0)<\bbA(p_C,\Phi;q_0)$, contradicting maximality of $\bbA(p,\Phi;q_0)$. Having reached a contradiction for $C$ sufficiently large, we conclude that $p'$ is uniformly bounded on $[q,1]$ for each $q\in (q_0,1)$ as desired.
\end{proof}

\begin{lemma}
    \label{lem:p-positive}
    $p(q)>0$ holds for all $q>q_0$.
\end{lemma}
\begin{proof}
    Suppose not. Then $p(q)=0$ for all $q\in [q_0,q_0+\eps]$, for some $\eps>0$. For $\delta>0$ small, define
    \[
    p_{\delta}(q)=\delta+(1-\delta)p(q).
    \]
    Then 
    \begin{align*}
        \sum_{s\in \sS}\lambda_s\int_{q_0}^{q_0+\eps}\sqrt{\Phi_s'(q)(p_{\delta}\times\xi^s\circ\Phi)'(q)}\de q
        &=
        \delta^{1/2}\sum_{s\in \sS}\lambda_s\int_{q_0}^{q_0+\eps} \sqrt{\Phi_s'(q)(\xi^s\circ\Phi)'(q)}\de q
        \\
        &\geq
        \delta^{1/2}c(\xi)\sum_{s\in\sS}\lambda_s\int_{q_0}^{q_0+\eps}\sqrt{\Phi_s'(q)^2}\de q\\
        &=
        \delta^{1/2}c(\xi).
    \end{align*}
    while
    \[
        \sum_{s\in \sS}\lambda_s\int_{q_0}^{q_0+\eps}\sqrt{\Phi_s'(q)(p\times\xi^s\circ\Phi)'(q)}\de q
        =
        0.
    \]
    On the other hand since $p_{\delta}(q)\geq p(q)$ for all $q\in [q_0,1]$ and $(p_{\delta})'=(1-\delta)(p)'$ as measures, we obtain
    \begin{align*}
        \sum_{s\in \sS}\lambda_s\int_{q_0+\eps}^1\sqrt{\Phi_s'(q)(p_{\delta}\times\xi^s\circ\Phi)'(q)}\de q
        &\geq 
        (1-\delta)\sum_{s\in \sS}\lambda_s\int_{q_0+\eps}^1 \sqrt{\Phi_s'(q)(p\times \xi^s\circ\Phi)'(q)}\de q.
    \end{align*}
    Combining the above implies $\bbA(p_{\delta},\Phi;q_0)>\bbA(p,\Phi;q_0)$ for small enough $\delta$, a contradiction. 
\end{proof}

\begin{lemma}
\label{lem:Phi-q0-neq-1}
    For all $s\in\sS$ and $q\in (0,1)$, we have $\Phi_s(q)<1$.
\end{lemma}

\begin{proof}
    Suppose $\Phi_{s_0}(q_*)=1$; this implies $0<q_0\leq q_*<1$. For small $\delta>0$ we consider the perturbation $\Phi_{\delta}$ with $\Phi_{\delta,s}=\Phi_s$ for $s\neq s_0$ and:
    \[
    \Phi_{\delta,s_0}'(q)
    =
    \begin{cases}
    \Phi_{s_0}'(q)\cdot (1-\delta(1-q_*)),\quad q\in [0,q_*],
    \\
    \delta,\quad\quad\quad\quad\quad\quad\quad\quad\quad\quad q\in [q_*,1]
    \end{cases}
    \]
    Note that $\Phi_{\delta,s}'(q)\geq (1-O(\delta))\Phi_{s}'(q)$ and so also $\Phi_{\delta,s}(q)\geq (1-O(\delta))\Phi_{s}(q)$ for all $s\in\sS$ and $q\in [0,1]$. As $\bbA$ is uniformly bounded, we can thus bound
    \begin{align*}
    \bbA(p,\Phi_{\delta};q_0)-\bbA(p,\Phi;q_0)
    &=
    \sum_{s\in \sS}
    h_s \lambda_s \sqrt{\Phi_{\delta,s}(q_0)}
    +
    \lambda_s 
    \int_{q_0}^1
    \sqrt{\Phi'_{\delta,s}(q) (p\times \xi^s \circ \Phi_{\delta})'(q)}
    ~\de q
    \\
    &\quad\quad
    -
    \sum_{s\in \sS}
    h_s \lambda_s \sqrt{\Phi_s(q_0)}
    -
    \lambda_s 
    \int_{q_0}^1
    \sqrt{\Phi'_s(q) (p\times \xi^s \circ \Phi)'(q)}
    ~\de q
    \\
    &\geq
    -O(\delta)
    +
    \lambda_{s_0}
    \int_{\frac{1+q_*}{2}}^1
    \sqrt{\Phi'_{\delta,s}(q) (p\times \xi^s \circ \Phi)'(q)}
    -
    \sqrt{\Phi'_s(q) (p\times \xi^s \circ \Phi)'(q)}
    ~\de q
    .
    \end{align*}
    Using Lemma~\ref{lem:p-positive}, admissibility and non-degeneracy of $\xi$, we find that $(p\times \xi^s \circ \Phi)'(q)\geq \Omega(q)$ for all $q\geq \frac{1+q_*}{2}$. Therefore 
    \[
    \lambda_{s_0}
    \int_{\frac{1+q_*}{2}}^1
    \sqrt{\Phi'_{\delta,s}(q) (p\times \xi^s \circ \Phi)'(q)}
    -
    \sqrt{\Phi'_s(q) (p\times \xi^s \circ \Phi)'(q)}
    ~\de q
    \geq \Omega(\delta^{1/2})
    \]
    for small $\delta$. Since $\delta^{1/2}$ is of larger order than $\delta$ we conclude that $\bbA(p,\Phi_{\delta};q_0)>\bbA(p,\Phi;q_0)$. This is a contradiction (recall Lemma~\ref{lem:admissible-optional}) and completes the proof.
\end{proof}

Next we turn our attention to $\Phi'$. Similarly to Lemma~\ref{lem:p-positive}, the idea is that the square root function has infinite derivative at $0$.

\begin{lemma}
    \label{lem:Phi-inc}
    There exists $\eta>0$ such that $\Phi'(q) \succeq \eta \vone$ almost everywhere in $q\in [q_0,1]$.
\end{lemma}
\begin{proof}
    First, given $(p,\Phi;q_0)$ choose for some $s\in\sS$ (specified below) a Lebesgue point $q_s\in (q_0,1)$ of $\Phi'$ with 
    \begin{equation}
    \label{eq:Phi-s-1}
        \Phi_s'(q_s)\geq a
    \end{equation}
    for $a>0$. Lemma~\ref{lem:Phi-q0-neq-1} ensures this is possible for some $a$ depending only on $q_0$ and $\Phi(q_0)$ (as long as $q_0<1$, else there is nothing to prove). In fact we can actually find two distinct such points $q_s^{(1)},q_s^{(2)}$ (which will be helpful below).

    Next for small $\eps>0$ depending only on $(p,\Phi)$, define the interval
    \begin{align*}
        J_{s,\eps}&=(q_s-\eps,q_s+\eps).
    \end{align*}
    By \eqref{eq:Phi-s-1} and the fact that $q_s$ is a Lebesgue point of $\Phi'$, there is a subset $I_{s,\eps}\subseteq J_{s,\eps}$ of Lebesgue measure at least $|I_{s,\eps}|\geq \frac{|J_{s,\eps}|}{2}=\eps$ such that
    \begin{equation}
    \label{eq:Phi-s-half}
        \Phi_s'(q)\geq \frac{a}{2} 
        ,\quad
        \forall q\in I_{s,\eps}.
    \end{equation}
    as long as $\eps>0$ is chosen sufficiently small.
    A simple consequence is the estimate
    \begin{equation}
    \label{eq:C-eps}
        C_{\eps}:=\Phi_s(q_s+\eps)-\Phi_s(q_s-\eps)=\int_{q_s-\eps}^{q_s+\eps} \Phi_s'(q)\de q\geq \frac{a\eps}{2}.
    \end{equation}

    With the setup above complete (except that $s$ is not yet specified), suppose the conclusion is false let $\eta$ be sufficiently small depending on $(p,\Phi,\eps)$ for $\eps$ as above. Then there exist $s,{s_0}\in\sS$ and 
    $\hat q_{s_0}\in (q_0,1)$ which is a Lebesgue point for $\nabla \Phi$ such that 
    \begin{align}
    \label{eq:Phi-s-0}
        \Phi_s'(\hat q_{s_0})&\leq \eta,
        \\
    \label{eq:Phi-r-1}
        \Phi_{s_0}'(\hat q_{s_0})&\geq 1.
    \end{align}
    Indeed if $q$ is any Lebesgue point of $\nabla \Phi$ satisfying \eqref{eq:Phi-s-0} for some $s$, then \eqref{eq:Phi-r-1} holds for some ${s_0}\neq s$ by admissibility and we define $\hat q_{s_0}=q$ this way. The bound \eqref{eq:Phi-s-0} determines the species $s$ chosen initially.
    
    As $\hat q_{s_0}$ is also a Lebesgue point of $\Phi'$, in light of \eqref{eq:Phi-s-0} and \eqref{eq:Phi-r-1}, there exists a set $I_{{s_0},\eta}\subseteq J_{{s_0},\eps}=(\hat q_{s_0}-\eps,\hat q_{s_0}+\eps)$ of positive Lebesgue measure such that the inequalities
    \begin{align}
    \label{eq:Phi-s-eta}
        \Phi_s'(q)&\leq 2\eta,
        \\
    \label{eq:Phi-r-half}
        \Phi_{s_0}'(q)&\geq \frac{a}{2}.
    \end{align}
    both hold for all $q\in I_{{s_0},\eta}$.
    Moreover we can assume $J_{s,\eps},J_{s_0,\eps}$ are disjoint, i.e. $|q_s-\hat q_{s_0}|> 2\eps$. Indeed as noted earlier we can choose two candidate points $q_s^{(1)},q_s^{(2)}$. If $\eps<|q_s^{(1)}-q_s^{(2)}|/5$ is taken, at least one of them suffices for any $\hat q_{s_0}\in (q_0,1)$.

    Next choose $\delta\in (0,\eta)$ small and consider the perturbation $\Phi_{\delta}$ with $\Phi_{\delta}(q_0)=\Phi(q_0)$ and
    \[
       \Phi_{\delta,s}'(q)
        =
        \begin{cases}
        \Phi_s'(q)+\delta,\quad q\in I_{s_0,\eta}
        \\
        \Phi_s'(q)\lt(1-\frac{\delta |I_{s_0,\eta}|}{C_{\eps}}\rt),\quad \forall q\in J_{s,\eps}
        \\
        \Phi_s'(q),\quad \text{otherwise}
        \end{cases}
    \]
    and $\Phi_{\delta,s'}=\Phi_{s'}$ for all $s'\in \sS\backslash \{s\}$. (Note we used disjointness of $J_{s,\eps},J_{s_0,\eps}$ for this definition to make sense.)
    By Lemma~\ref{lem:admissible-optional}, we must have $\bbA(p,\Phi_{\delta};q_0)\geq \bbA(p,\Phi;q_0)$ although $\Phi_{\delta}$ may not be admissible. Then for $\delta\leq\eta$,
    \begin{align}
    \nonumber
        \int_{I_{s_0,\eta}}
        \sqrt{\Phi_{\delta,s}'(q)(p\times\xi^s\circ\Phi)'(q)}
        -
        &
        \sqrt{\Phi_s'(q) (p\times\xi^s\circ\Phi)'(q)}
        \de q
        \\
    \nonumber
        &\stackrel{\eqref{eq:Phi-s-eta}}{\geq}
        (\sqrt{2\eta+\delta}-\sqrt{2\eta})
        \int_{I_{s_0,\eta}}\sqrt{(p\times\xi^s\circ\Phi)'(q)}\de q
        \\
    \nonumber
        &\geq
        \frac{\delta p(q_s-\eps)^{1/2}}{10\eta^{1/2}}
        \int_{I_{s_0,\eta}}
        \sqrt{(\xi^s\circ\Phi)'(q)}\de q
        \\
    \nonumber
        &\geq 
        \frac{\delta p(q_s/2)^{1/2}c(\xi)}
        {10\eta^{1/2}}
        \int_{I_{s_0,\eta}}
        \sqrt{\Phi_{s_0}'(q)}\de q
        \\
    \label{eq:Phi-gain}
        &\stackrel{\eqref{eq:Phi-r-half}}{\geq}
        \frac{\delta a^{1/2} p(q_s/2)^{1/2}c(\xi)|I_{s_0,\eta|}}
        {20\eta^{1/2}}
        .
    \end{align}
    We used non-degeneracy of $\xi$ in the penultimate step. On the other hand recalling \eqref{eq:C-eps}, it follows that for all $\wt{s}\in\sS$ and almost all $q\in [q_0,1]$:
    \begin{equation}
    \label{eq:Phi-comp-1}
        \Phi_{\delta,\wt{s}}'(q)\geq \lt(1-O\lt(\frac{\delta|I_{s_0,\eta}|}{\eps}\rt)\rt)
        \Phi_{\wt{s}}'(q).
    \end{equation}
    Integrating on $[q_0,q]$, we find 
    \begin{equation}
    \label{eq:Phi-comp-2}
        \Phi_{\delta,\wt{s}}(q)\geq \lt(1-O\lt(\frac{\delta|I_{s_0,\eta}|}{\eps}\rt)\rt)
        \Phi_{\wt{s}}(q)
    \end{equation}
    for all $q\in [q_0,1]$. By the chain rule we similarly obtain that for all $\wt{s}\in\sS$,
    \begin{align}
    \label{eq:Phi-comp-3}
        (p\times \xi^{\wt{s}}\circ\Phi_{\delta})'
        &\geq 
        \lt(1-O\lt(\frac{\delta|I_{s_0,\eta}|}{\eps}\rt)\rt)(p\times \xi^{\wt{s}}\circ\Phi)',
        \\
    \label{eq:Phi-comp-4}
        (p\times \xi^{\wt{s}}\circ\Phi_{\delta})(q)
        &\geq 
        \lt(1-O\lt(\frac{\delta|I_{s_0,\eta}|}{\eps}\rt)\rt)(p\times \xi^{\wt{s}}\circ\Phi)(q).
    \end{align}
    It follows from \eqref{eq:Phi-comp-1}, \eqref{eq:Phi-comp-2}, \eqref{eq:Phi-comp-3}, \eqref{eq:Phi-comp-4} that
    \begin{align}
    \nonumber
        \int_{I_{s_0,\eta}}
        \sqrt{\Phi_{\delta,s}'(q)(p\times\xi^s\circ\Phi_{\delta})'(q)}
        &\geq
        \lt(1-O\lt(\frac{\delta|I_{s_0,\eta}|}{\eps}\rt)\rt)
        \int_{I_{s_0,\eta}}
        \sqrt{\Phi_{\delta,s}'(q) (p\times\xi^s\circ\Phi)'(q)}
        \de q
        \\
    \label{eq:Phi-extra-term}
        &\stackrel{\eqref{eq:F-bounded}}{\geq}
        \int_{I_{s_0,\eta}}
        \sqrt{\Phi_{\delta,s}'(q) (p\times\xi^s\circ\Phi)'(q)} \de q
        - 
        O\lt(\frac{\delta|I_{s_0,\eta}|}{\eps}\rt).
    \end{align}
    Since $\Phi_{\delta}$ and $\Phi$ differ only inside $[q_0,1]$ we use $\bbA_{[q_0,1]}$ below to denote the second term of $\bbA$. We have:
    \begin{align*}
        \bbA_{[q_0,1]}(p,\Phi)
        &=
        \sum_{\wt{s}\in\sS}\lambda_{\wt{s}}\int_{q_0}^1 \sqrt{\Phi_{\wt{s}}'(q)(p\times \xi^{\wt{s}}\circ\Phi)'(q)}\de q
        \\
        &=
        \lambda_s\int_{I_{s_0,\eta}} \sqrt{\Phi_s'(q)(p\times \xi^{s}\circ\Phi)'(q)}\de q
        +
        \lambda_s\int_{[q_0,1]\backslash I_{s_0,\eta}} \sqrt{\Phi_s'(q)(p\times \xi^{s}\circ\Phi)'(q)}\de q
        \\
        &
        \qquad 
        +
        \sum_{\wt{s}\in\sS\backslash\{s\}}\lambda_{\wt{s}}\int_{q_0}^1 \sqrt{\Phi_{\wt{s}}'(q)(p\times \xi^{\wt{s}}\circ\Phi)'(q)}\de q
        \\
        &\equiv
        \I+\II+\III.
    \end{align*}
    Similarly for $J$ instead of $I$, 
    \begin{align*}
        \bbA_{[q_0,1]}(p,\Phi_{\delta})
        &=
        \sum_{\wt{s}\in\sS}\lambda_{\wt{s}}\int_{q_0}^1 \sqrt{\Phi_{\delta,\wt{s}}'(q)(p\times \xi^{\wt{s}}\circ\Phi_{\delta})'(q)}\de q
        \\
        &=
        \lambda_s\int_{I_{s_0,\eta}} \sqrt{(\Phi_{\delta,s})'(q)(p\times \xi^{s}\circ\Phi_{\delta})'(q)}\de q
        +
        \lambda_s\int_{[q_0,1]\backslash I_{s_0,\eta}} \sqrt{(\Phi_{\delta,s})'(q)(p\times \xi^{s}\circ\Phi_{\delta})'(q)}\de q
        \\
        &
        \qquad 
        +
        \sum_{\wt{s}\in\sS\backslash\{s\}}\lambda_{\wt{s}}\int_{q_0}^1 \sqrt{\Phi_{\delta,\wt{s}}'(q)(p\times \xi^{\wt{s}}\circ\Phi_{\delta})'(q)}\de q
        \\
        &\equiv
        \I_{\delta}+\II_{\delta}+\III_{\delta}.
    \end{align*}
    Using \eqref{eq:Phi-comp-1}, \eqref{eq:Phi-comp-2}, \eqref{eq:Phi-comp-3}, \eqref{eq:Phi-comp-4} again, we obtain
    \begin{align*}
        \II_{\delta}&\geq \lt(1-O\lt(\frac{\delta|I_{s_0,\eta}|}{\eps}\rt)\rt)\II,
        \\
        \III_{\delta}&\geq \lt(1-O\lt(\frac{\delta|I_{s_0,\eta}|}{\eps}\rt)\rt)\III.
    \end{align*}
    Meanwhile \eqref{eq:Phi-gain} and \eqref{eq:Phi-extra-term} imply that for $\delta$ small compared to $\eta$,
    \[
    \I_{\delta}\geq \lt(1-O\lt(\frac{\delta|I_{s_0,\eta}|}{\eps}\rt)\rt)\I + \frac{\delta a^{1/2}p(q_s/2)^{1/2}c(\xi)|I_{s_0,\eta}|}
        {20\eta^{1/2}}.
    \]
    Combining, we find
    \[
        \bbA_{[q_0,1]}(p,\Phi_{\delta})\geq \bbA_{[q_0,1]}(p,\Phi)+
        \frac{\delta a^{1/2}p(q_s/2)^{1/2}c(\xi)|I_{s_0,\eta}|}
        {20\eta^{1/2}}-O\lt(\frac{\delta|I_{s_0,\eta}|}{\eps}\rt).
    \]
    Taking $\eta\ll \eps^2 ap(q_s/2)c(\xi)^2$ and then $\delta$ sufficiently small contradicts the maximality of $(p,\Phi,q_0)$, thus completing the proof.
\end{proof}

\begin{proposition}
    \label{prop:p-q0-0}
    If $q_0 > 0$, then $p(q_0)=0$.
\end{proposition}
\begin{proof}
    Assume that $p(q_0)>0$. Consider the perturbation
    \[
        \wtp(q)=
        \begin{cases}
        p(q)+(q-q_0-\eps)\delta,\quad q<q_0+\eps
        \\
        p(q),\quad\quad\quad\quad\quad\quad\quad q\geq q_0+\eps.
        \end{cases}
    \]
    The function $\tilde p$ is increasing, and is non-negative for sufficiently small $\eps,\delta>0$. For $q<q_0+\eps$ we find
    \begin{align*}
        \deriv{\delta} (p\times \xi^s\circ\Phi)'(q)
        &=
        \deriv{\delta}\lt(
            p'(q)\xi^s(\Phi(q))
            +
            p(q)(\xi^s\circ\Phi)'(q)
        \rt)
        \\
        &=
        \xi^s(\Phi(q))
        -
        (q_0+\eps-q)
        (\xi^s\circ\Phi)'(q)
        \\
        &\geq 
        \xi^s(\Phi(q))-O(\eps).
    \end{align*}
    If $q_0>0$, then $\xi^s(\Phi(q))\geq c(q_0)>0$ by admissibility and non-degeneracy of $\Phi$. This contradicts optimality of $(p,\Phi,q_0)$ and completes the proof.
\end{proof}

\begin{proof}[Proof of Proposition~\ref{prop:p-basic}]
    Follows from Lemmas~\ref{lem:p-AC} and \ref{lem:p-positive} and Proposition~\ref{prop:p-q0-0}.
\end{proof}

\subsubsection{Continuous Differentiability on $(q_0,1]$}
\label{subsubsec:C1-on-compact-subsets}

Here we show that $p$ and $\Phi$ are continuously differentiable on compact subsets of $(q_0,1]$ using another local perturbation argument. 

\begin{lemma}
    \label{lem:sqrt-xy-concave}
    The function $f(x,y)=\sqrt{xy}$ is concave on $\bbR_{>0}^2$, with strict concavity on all lines except for those passing through the origin.
\end{lemma}

\begin{proof}
    Given $x_0,y_0,x_1,y_1>0$ with $(x_0,y_0)\neq (x_1,y_1)$ and $c\in (0,1)$, we have
    \begin{align*}
        (x_0 y_1 - x_1 y_0)^2
        &\geq 
        0
        \\
        \implies 
        x_0^2 y_1^2 + x_1^2 y_0^2 + 2x_0 x_1 y_0 y_1
        &\geq 
        4 x_0 x_1 y_0 y_1
        \\
        \implies 
        (x_0 y_1 + x_1 y_0)
        &\geq 
        2\sqrt{x_0 x_1 y_0 y_1}
        \\
        \implies 
        c(1-c)\cdot (x_0 y_1 + x_1 y_0)
        &\geq 
        2c(1-c)\sqrt{x_0 x_1 y_0 y_1}
        \\
        \implies 
        c^2 x_0 y_0 + (1-c)^2 x_0 y_0) + c(1-c)\cdot (x_0 y_1 + x_1 y_0)
        &\geq 
        c^2 x_0 y_0 + (1-c)^2 x_0 y_0) + 
        2c(1-c)\sqrt{x_0 x_1 y_0 y_1}
        \\
        \implies 
        \sqrt{(cx_0+(1-c)x_1)(cy_0+(1-c)y_1) }
        &\geq 
        c\sqrt{x_0y_0}+(1-c)\sqrt{x_1y_1}.
    \end{align*}
    Moreover equality holds if and only if it holds in the first step.
\end{proof}

\begin{lemma}
    \label{lem:derivatives-continuous-apriori}
    Both $p$ and $\Phi$ are continuously differentiable on compact subsets of $(q_0,1]$.
\end{lemma}

\begin{proof}

We assume that $q_0<1$ (else there is nothing to prove), and recall Lemma~\ref{lem:p-positive} throughout. Admissibility implies that $\Phi$ is uniformly Lipschitz, and Lemma~\ref{lem:p-AC} shows that $p$ is uniformly Lipschitz on compact subsets of $(q_0,1)$. Hence both $p'(x)$ and $\Phi'_s$ exist as non-negative, integrable functions which are uniformly bounded away from $q_0$.

By an elementary result of \cite{zaanen1986continuity}, if a measurable function $[q_0,1]\to\bbR$ does not agree with any continuous function on a full measure set, then it possesses a genuine point of discontinuity $q_*\in (q_0,1)$ such that $F$ cannot be made continuous at $q_*$ even by modification on a measure zero set. We fix such a point $q_*$ for sake of contradiction. By definition, this means that for some $\eta>0$ depending only on $(p,\Phi,q_*)$ and for arbitrarily small $\eps>0$, there exist measurable sets $I,J\subseteq (q_*-\eps,q_*+\eps)$ and $a\in \bbR$ such that:
\begin{equation}
\label{eq:eta-discontinuous-f}
\begin{aligned}
    |I|&= \eps_1>0,\\
    |J|&= \eps_1>0,\\
    f(q)&\geq a+\eta,\quad \forall q\in I,\\
    f(q)&\leq a-\eta,\quad \forall q\in J.
\end{aligned}
\end{equation}
Here $f(q)=p'(q)$ or $f(q)=\Phi'_s(q)$ for some $s\in\sS$.

Let $\gamma_I:[0,\eps_1]\to I$ and $\gamma_J:[0,\eps_1]\to J$ be increasing, measure-preserving bijections (and note that their inverse functions are also measurable). For convenience we set $q_{I,x}=\gamma_I(x)$ and $q_{J,x}=\gamma_J(x)$. We construct perturbations $\tilde p,\tilde\Phi$ of $p$ and $\Phi$ by averaging derivatives on $q_{I,x}$ and $q_{J,x}$:
\begin{align*}
   \tilde p'(q_{I,x})&= 
   \tilde p'(q_{J,x})
   = \frac{p'(q_{I,x})+ p'(q_{J,x})}{2}
    \,;\\
    \tilde p'(q) &= p'(q),\quad q\notin I\cup J
    \,;\\
    \tilde \Phi_s'(q_{I,x})&= 
   \tilde \Phi_s'(q_{J,x})
   = \frac{\Phi_s'(q_{I,x})+ \Phi_s'(q_{J,x})}{2}
    \,;\\
    \tilde \Phi_s'(q) &= \Phi_s'(q),\quad q\notin I\cup J.
\end{align*}
We claim that for fixed $q_*,\eta$ and sufficiently small $\eps>0$, we have
\begin{equation}
\label{eq:derivative-continuous}
    \bbA(\tilde p,\tilde\Phi;q_0)
    >
    \bbA(p,\Phi;q_0).
\end{equation}
This contradicts maximality of $(p,\Phi)$ and thus implies the desired continuity of $(p',\Phi')$.

To begin proving \eqref{eq:derivative-continuous}, recall from Lemma~\ref{lem:p-AC} that $p'$ is uniformly bounded away from $q_0$, hence on $(q_*-\eps,q_*+\eps)$. Moreover $\Phi'$ is uniformly bounded by definition. It follows that for all $s\in\sS$ and $q\in (q_*-\eps,q_*+\eps)$,
\begin{equation}
\label{eq:almost-constant}
\begin{aligned}
    |p(q)-\tilde p(q)|&\leq O(\eps_1),\\
    |p(q)-p(q_*)|&\leq O(\eps),
    \\
    |\Phi_s(q)-\tilde\Phi_{s}(q)|&\leq O(\eps_1),\\
    |\xi^s(\Phi(q))-\xi^s(\tilde\Phi(q))|&\leq O(\eps_1),
    \\
    |\Phi_s(q)-\Phi_s(q_*)|&\leq O(\eps),
    \\
    |\xi^s(\Phi(q))-\xi^s(\Phi(q_*))|&\leq O(\eps).
\end{aligned}
\end{equation}
These estimates will let us treat the above functions as almost constant while proving \eqref{eq:derivative-continuous}, so we can focus on the more important changes in their derivatives. First for $q\notin [q_*-\eps,q_*+\eps]$, we have $p(q)=\tilde p(q)$ and $\Phi(q)=\tilde\Phi(q)$, so it suffices to analyze the discrepancy within $q\in [q_*-\eps,q_*+\eps]$. Next, the estimates \eqref{eq:almost-constant} together with the fact that $\Phi_s'$ is uniformly bounded below (by Lemma~\ref{lem:Phi-inc}) imply that
\begin{equation}
\label{eq:good-outside-IJ}
    \lt|
    \sqrt{\Phi_s'(q)(p\times \xi^s\circ\Phi)'(q)}
    -
    \sqrt{\tilde \Phi_{s}'(q)(\tilde p\times \xi^s\circ\tilde \Phi)'(q)}
    \rt|
    \leq
    O(\eps_1),
    \quad
    \forall~
    q\in [q_*-\eps,q_*+\eps]\backslash (I\cup J).
\end{equation}
Integrating, we obtain
\begin{equation}
\label{eq:good-outside-IJ-2}
    \int_{q\in [q_*-\eps,q_*+\eps]\backslash (I\cup J)}\lt|
    \sqrt{\Phi_s'(q)(p\times \xi^s\circ\Phi)'(q)}
    -
    \sqrt{\tilde\Phi_{s}'(q)(\tilde p\times \xi^s\circ\tilde \Phi)'(q)}
    \rt|
    ~\de q
    \leq
    O(\eps_1 \eps).
\end{equation}
Next we fix $x\in [0,\eps_1]$ and analyze the joint effect of the pertubation at the pair of points $q_{I,x}$ and $q_{J,x}$. This is given by
\begin{equation}
\label{eq:joint-effect}
\begin{aligned}
    &\sqrt{\tilde\Phi_{s}'(q_{I,x})(\tilde p\times \xi^s\circ\tilde\Phi)'(q_{I,x})}
    -
    \sqrt{\Phi_s'(q_{I,x})(p\times \xi^s\circ\Phi)'(q_{I,x})}
    \\
    \quad\quad
    &+
    \sqrt{\tilde\Phi_{s}'(q_{J,x})(\tilde p\times \xi^s\circ\tilde\Phi)'(q_{J,x})}
    -
    \sqrt{\Phi_s'(q_{J,x})(p\times \xi^s\circ\Phi)'(q_{J,x})}
    .
\end{aligned}
\end{equation}
Recalling again \eqref{eq:almost-constant}, we have 
\begin{equation}
\label{eq:recalling-again-blah}
\begin{aligned}
    (\tilde p\times \xi^s\circ\tilde\Phi)'(q_{I,x})
    &=
    \tilde p (q_{I,x}) \sum_{s'\in \sS}\partial_{x_{s'}}\xi^s(\tilde \Phi(q_{I,x}))\cdot \tilde\Phi_{s'}'(q_{I,x})
    +
    \tilde p'(q_{I,x})\cdot \xi^s(\tilde\Phi(q_{I,x}))
    \\
    &=
    p (q_*) \sum_{s'\in \sS}\partial_{x_{s'}}\xi^s(\Phi(q_*))\cdot \tilde\Phi_{s'}'(q_{I,x})
    +
    p'(q_{I,x})\cdot \xi^s(\Phi(q_*))
    \pm O(\eps).
\end{aligned}
\end{equation}
Similarly to \eqref{eq:good-outside-IJ}, we now control the first two terms of \eqref{eq:joint-effect}:
\begin{equation}
\label{eq:approx-Phi-diff}
\begin{aligned}
    &
    \sqrt{\tilde\Phi_{s}'(q_{I,x})(\tilde p\times \xi^s\circ\tilde\Phi)'(q_{I,x})}
    -
    \sqrt{\Phi_s'(q_{I,x})(p\times \xi^s\circ\Phi)'(q_{I,x})}
    +
    O(\eps)
    \\
    &\stackrel{\eqref{eq:recalling-again-blah}}{\geq}
    \sqrt{\tilde\Phi_{s}'(q_{I,x})\cdot \lt(p (q_*) \sum_{s'\in \sS}\partial_{x_{s'}}\xi^s(\Phi(q_*))\cdot \tilde\Phi_{s'}'(q_{I,x})
    +
    \tilde p'(q_{I,x})\cdot \xi^s(\Phi(q_*))\rt)}
    \\
    &\quad -
    \sqrt{\Phi_{s}'(q_{I,x})\cdot \lt(p(q_*) \sum_{s'\in \sS}\partial_{x_{s'}}\xi^s(\Phi(q_*))\cdot \Phi'_{s'}(q_{I,x})
    +
    p'(q_{I,x})\cdot \xi^s(\Phi(q_*))\rt)}
\end{aligned}
\end{equation}
and analogously for $J$ instead of $I$.

It remains to lower-bound the right hand side of \eqref{eq:approx-Phi-diff}. We break into cases depending on whether $\Phi'$ is continuous (if so, then $p'$ must be discontinuous). In both cases, the idea is to argue that the concavity of the square root function yields an increase in the value of $\bbA$.

\paragraph{Case $1$: $\Phi'$ is continuous at $q_*$}

In this case $p'$ is discontinuous, and \eqref{eq:eta-discontinuous-f} applies with $f=p$. We estimate the right-hand side of \eqref{eq:approx-Phi-diff}: as $|\Phi_s'(q)-\tilde \Phi_{s}'(q')|\leq o_{\eps\to 0}(1)$ uniformly in $q,q'\in (q_*-\eps,q_*+\eps)$ by definition, 
\begin{align}
    \nonumber
    &
    \sqrt{\tilde\Phi_{s}'(q_{I,x})\cdot \lt(p (q_*) \sum_{s'\in \sS}\partial_{x_{s'}}\xi^s(\Phi(q_*))\cdot \tilde\Phi_{s'}'(q_{I,x})
    +
    \tilde p'(q_{I,x})\cdot \xi^s(\Phi(q_*))\rt)}
    \\
    \nonumber
    &\quad -
    \sqrt{\Phi_{s}'(q_{I,x})\cdot \lt(p(q_*) \sum_{s'\in \sS}\partial_{x_{s'}}\xi^s(\Phi(q_*))\cdot \Phi'_{s'}(q_{I,x})
    +
    p'(q_{I,x})\cdot \xi^s(\Phi(q_*))\rt)}
    \\
    \nonumber
    &=
    \sqrt{\Phi_{s}'(q_*)\cdot \lt(p (q_*) \sum_{s'\in \sS}\partial_{x_{s'}}\xi^s(\Phi(q_*))\cdot \Phi'_{s'}(q_*)
    +
    \tilde p'(q_{I,x})\cdot \xi^s(\Phi(q_*))\rt)}
    \\
    \label{eq:I-big-long-equation}
    &\quad -
    \sqrt{\Phi_{s}'(q_*)\cdot \lt(p(q_*) \sum_{s'\in \sS}\partial_{x_{s'}}\xi^s(\Phi(q_*))\cdot \Phi'_{s'}(q_*)
    +
    p'(q_{I,x})\cdot \xi^s(\Phi(q_*))\rt)} \pm o_{\eps\to 0}(1)
    .
\end{align}
We analyze the last term, combined with the analogous expression for $J$, using the strict concavity in Lemma~\ref{lem:sqrt-xy-concave} of $x\mapsto \sqrt{x}$ together with \eqref{eq:eta-discontinuous-f} applied to $p$. We find that
\begin{equation}
\label{eq:big-gain-case-1}
\begin{aligned}
    &\sqrt{\Phi_{s}'(q_*)\cdot \lt(p (q_*) \sum_{s'\in \sS}\partial_{x_{s'}}\xi^s(\Phi(q_*))\cdot \Phi'_{s'}(q_*)
    +
    \tilde p'(q_{I,x})\cdot \xi^s(\Phi(q_*))\rt)}
    \\
    &\quad -
    \sqrt{\Phi_{s}'(q_*)\cdot \lt(p(q_*) \sum_{s'\in \sS}\partial_{x_{s'}}\xi^s(\Phi(q_*))\cdot \Phi'_{s'}(q_*)
    +
    p'(q_{I,x})\cdot \xi^s(\Phi(q_*))\rt)}
    \\
    &\quad
    +
    \sqrt{\Phi_{s}'(q_*)\cdot \lt(p (q_*) \sum_{s'\in \sS}\partial_{x_{s'}}\xi^s(\Phi(q_*))\cdot \Phi'_{s'}(q_*)
    +
    \tilde p'(q_{J,x})\cdot \xi^s(\Phi(q_*))\rt)}
    \\
    &\quad -
    \sqrt{\Phi_{s}'(q_*)\cdot \lt(p(q_*) \sum_{s'\in \sS}\partial_{x_{s'}}\xi^s(\Phi(q_*))\cdot \Phi'_{s'}(q_*)
    +
    p'(q_{J,x})\cdot \xi^s(\Phi(q_*))\rt)}
    \\
    &\geq
    c(\eta).
\end{aligned}
\end{equation}
Indeed, all quantities except $p'(\cdot)$ and $\tilde p'(\cdot)$ are the same in the four expressions and are bounded away from $0$ and infinity. Furthermore all other expressions differ by $O(\eps_1)$ thanks to \eqref{eq:almost-constant}, which is small compared to the discrepancy $\eta$ between the values of $p'$ and $\wt p$'. Hence for $\eta$ fixed and $\eps$ small enough, they are bounded away from the equality cases of Lemma~\ref{lem:sqrt-xy-concave}.

Combining \eqref{eq:approx-Phi-diff}, \eqref{eq:I-big-long-equation}, and \eqref{eq:big-gain-case-1} implies that for each $x\in [0,\eps_1]$ and small enough $\eps$,
\begin{align*}
    &\sqrt{\tilde\Phi_{s}'(q_{I,x})(\tilde p\times \xi^s\circ\tilde\Phi)'(q_{I,x})}
    -
    \sqrt{\Phi_s'(q_{I,x})(p\times \xi^s\circ\Phi)'(q_{I,x})}
    \\
    \quad\quad
    &+
    \sqrt{\tilde\Phi_{s}'(q_{J,x})(\tilde p\times \xi^s\circ\tilde\Phi)'(q_{J,x})}
    -
    \sqrt{\Phi_s'(q_{J,x})(p\times \xi^s\circ\Phi)'(q_{J,x})}
    \\
    &\geq c(\eta)-o_{\eps\to 0}(1)
    \\
    &\geq 
    c(\eta)/2.
\end{align*}
Integrating over $x\in [0,\eps_1]$ and combining with \eqref{eq:good-outside-IJ-2}, we conclude that \eqref{eq:derivative-continuous} holds in Case $1$.

\paragraph{Case $2$: $\Phi'$ is discontinuous at $q_*$.} 
(Note that $p'$ might also be discontinuous.) 

Define for each $s\in\sS$ the function
\[
    F_s(A_1,\dots,A_r,B)
    =
    \sqrt{A_s\cdot \lt(p(q_*) \sum_{s'\in \sS}\partial_{x_{s'}}\xi^s(\Phi(q_*)) A_{s'}
    +
    B\xi^s(\Phi(q_*))\rt)}.
\]
Lemma~\ref{lem:sqrt-xy-concave} implies that each function $F_s$ is concave on $\bbR_{\geq 0}^{r+1}$, since both $A_s$ and $p(q_*) \sum_{s'\in \sS}\partial_{x_{s'}}\xi^s(\Phi(q_*)) A_{s'}+B\xi^s(\Phi(q_*))$ are linear functions of $(A_1,\dots,A_r,B)$. In particular, for each $(s,x)\in \sS\times [0,\eps_1]$ the function
\[
    f_{s,x}(t)
    \equiv
    F_s\lt(\frac{(1-t)\Phi_{1}'(q_{I,x})+t\Phi_{1}'(q_{J,x})}{2},\dots,
    \frac{(1-t)\Phi_{r}'(q_{I,x})+t\Phi_{r}'(q_{J,x})}{2},
    \frac{(1-t)p'(q_{I,x})+tp'(q_{J,x})}{2}\rt)
\]
is concave for $t\in [0,1]$. Recalling the definitions of $\tilde p$ and $\tilde \Phi$, we expand the inequality $2f_{s,x}(1/2)\geq f_{s,x}(0)+f_{s,x}(1)$ to obtain
\begin{equation}
\label{eq:concave-in-each-s}
\begin{aligned}
    &
    \sqrt{\tilde\Phi_{s}'(q_{I,x})\cdot \lt(p (q_*) \sum_{s'\in \sS}\partial_{x_{s'}}\xi^s(\Phi(q_*))\cdot \tilde\Phi_{s'}'(q_{I,x})
    +
    \tilde p'(q_{I,x})\cdot \xi^s(\Phi(q_*))\rt)}
    \\
    &\quad -
    \sqrt{\Phi_{s}'(q_{I,x})\cdot \lt(p(q_*) \sum_{s'\in \sS}\partial_{x_{s'}}\xi^s(\Phi(q_*))\cdot \Phi'_{s'}(q_{I,x})
    +
    p'(q_{I,x})\cdot \xi^s(\Phi(q_*))\rt)}
    \\
    &\quad +
    \sqrt{\tilde\Phi_{s}'(q_{J,x})\cdot \lt(p (q_*) \sum_{s'\in \sS}\partial_{x_{s'}}\xi^s(\Phi(q_*))\cdot \tilde\Phi_{s'}'(q_{J,x})
    +
    \tilde p'(q_{J,x})\cdot \xi^s(\Phi(q_*))\rt)}
    \\
    &\quad -
    \sqrt{\Phi_{s}'(q_{J,x})\cdot \lt(p(q_*) \sum_{s'\in \sS}\partial_{x_{s'}}\xi^s(\Phi(q_*))\cdot \Phi'_{s'}(q_{J,x})
    +
    p'(q_{J,x})\cdot \xi^s(\Phi(q_*))\rt)}
    \\
    &\geq
    0.
\end{aligned}
\end{equation}
In light of \eqref{eq:approx-Phi-diff}, this means that perturbing $(p,\Phi)\to(\tilde p,\tilde \Phi)$ can only hurt the contribution from a given $s\in\sS$ by $O(\eps)$. To complete the proof we will show that the contribution from some $s\in\sS$ is positive and of a larger order. Which of these must occur will depend on the ratio $\frac{p'(q_{I,x})}{p'(q_{J,x})}$.

We will get this contribution from either $s_{\max}$ or $s_{\min}$, defined now. For each $x\in[0,\eps_1]$, let
\begin{align*}
    s_{\max}(x)
    &=
    \argmax_{s\in\sS}
    \frac{\Phi_s'(q_{I,x})}{\Phi_s'(q_{J,x})}
    ,
    \\
    s_{\min}(x)
    &=
    \argmin_{s\in\sS}
    \frac{\Phi_s'(q_{I,x})}{\Phi_s'(q_{J,x})}
    .
\end{align*}
(Both are defined up to almost everywhere equivalence if ties are broken lexicographically.) Recall the functions $\Phi_s'(x)$ are uniformly bounded above and below. It follows from \eqref{eq:eta-discontinuous-f} that 
\begin{equation}
\label{eq:Phi-s-minmax}
    \frac{\Phi_{s_{\min}}'(q_{I,x})}{\Phi_{s_{\min}}'(q_{J,x})}
    \leq
    1-\eta'
    \leq
    1+\eta'
    \leq
    \frac{\Phi_{s_{\max}}'(q_{I,x})}{\Phi_{s_{\max}}'(q_{J,x})}
\end{equation}
for some $\eta'$ depending only on $(\eta,\xi,h)$. (Discontinuity of $\Phi'$ gives one side, and admissibility forces another $s\in\sS$ to change in the opposite direction.)

Without loss of generality, suppose that
\begin{equation}
\label{eq:p'-shrink-wlog}
    \frac{p'(q_{I,x})}{p'(q_{J,x})}\leq 1.
\end{equation}
In this case, the assumption \eqref{eq:p'-shrink-wlog} implies
\[
    \frac{
    p(q_*) \sum_{s'\in \sS}\partial_{x_{s'}}\xi^{s_{\max}}(\Phi(q_*))\cdot \Phi'_{s'}(q_{I,x})
    +
    p'(q_{I,x})\cdot \xi^{s_{\max}}(\Phi(q_*))
    }
    {
    p(q_*) \sum_{s'\in \sS}\partial_{x_{s'}}\xi^{s_{\max}}(\Phi(q_*))\cdot \Phi'_{s'}(q_{J,x})
    +
    p'(q_{J,x})\cdot \xi^{s_{\max}}(\Phi(q_*))
    }
    \leq
    \frac{\Phi_{s_{\max}}'(q_{I,x})}
    {\Phi_{s_{\max}}'(q_{J,x})}-\eta_1
\]
for a constant $\eta_1>0$ depending only on $(\eta,q_*,\xi,h)$. Since all quantities are bounded away from $0$ and infinity, applying a simple compactness argument to the equality case in Lemma~\ref{lem:sqrt-xy-concave} implies
\begin{equation}
\label{eq:f-concave-with-eta1-gain}
    2f_{s_{\max},x}(1/2)\geq f_{s_{\max},x}(0)+f_{s_{\max},x}(1)+c(\eta_1).
\end{equation}
Similarly if \eqref{eq:p'-shrink-wlog} does not hold, then we find \eqref{eq:f-concave-with-eta1-gain} with $s_{\min}$ in place of $s_{\max}$.

Combining the above with $\eps\ll \eta$, we find that for each $x\in[0,\eps_1]$,
\begin{align*}
    \sum_{s\in\sS} 2f_{s,x}(1/2)
    &\geq
    \sum_{s\in\sS} 
    \Big(f_{s,x}(0)+f_{s,x}(1)\Big)+c(\eta_1)/2.
\end{align*}
Integrating over $x$ and recalling \eqref{eq:good-outside-IJ-2} and \eqref{eq:approx-Phi-diff}, we conclude that \eqref{eq:derivative-continuous} also holds in Case $2$. This completes the proof.
\end{proof}

\begin{proof}[Proof of Proposition~\ref{prop:basic-regularity}]
    Follows from Lemmas~\ref{lem:p-AC}, \ref{lem:Phi-inc}, and \ref{lem:derivatives-continuous-apriori}.
    The upper bound on $\Phi'$ comes from admissibility \eqref{eq:admissible}, which implies that $\Phi'_s \le \lambda_s^{-1}$.
\end{proof}

\subsection{Type $\II$ Solutions}
\label{subsec:type-II-Lipschitz}

Here we show that the type $\II$ equation implicitly takes the form of a second order ordinary differential equation in which $\Phi''(q)$ is Lipschitz in $(\Phi(q),\Phi'(q))$. It follows that a unique type $\II$ solution exists given any first-order initial condition $(\Phi(q_1),\Phi'(q_1))$, and that the type $\II$ ODE is satisfied at \emph{all} points in $(q_1,1)$. We will often enforce the admissibility conditions
\begin{align}
\label{eq:derivative-admissible}
    \langle \vlam, \vec\Phi'(q)\rangle &=1,
    \\
\label{eq:second-derivative-admissible}
    \langle \vlam, \vec\Phi''(q)\rangle &=0.
\end{align}
In particular, we denote by $A_{\geq 0}$ the set of vectors $v\in \bbR_{\geq 0}^{\sS}$ satisfying $\langle \vlam, v \rangle=1$. The following important but rather lengthy Lemma~\ref{lem:type-II-Lipschitz} ensures that type $\II$ solutions are described by a Lipschitz ODE. In it, the value $q$ is actually irrelevant and just serves as a placeholder. Importantly there is no issue when $\Phi_s(q)$ or $\Phi_s'(q)$ is near zero, thanks to non-degeneracy.

\lemtypeIILipschitz*

\begin{proof}
Write $\Psi(q)$ for $\Psi_s(q)$, which is independent of $s\in\sS$ by assumption. We assume throughout that $\Phi(q)$ lies in a bounded set in writing $O(\cdot)$ and $\Omega(\cdot)$ expressions. Note that $\vec\Phi''(q)$ exists as an $L^1$ function for $q\in (q_1,1]$ since $\vec\Phi'$ is absolutely continuous. We write
\begin{equation}
\label{eq:expand-Psi-type-II}
\begin{aligned}
    2\Psi(q)
    &=
    \frac{2}{\Phi_s'(q)}
    \deriv{q}{\sqrt{\frac{\Phi'_s(q)}{(\xi^s\circ\Phi)'(q)}}}
    \\
    &=
    \sqrt{\frac{(\xi^s\circ\Phi)'(q)}{\Phi_s'(q)^3}}
    \deriv{q}{\frac{\Phi'_s(q)}{(\xi^s\circ\Phi)'(q)}}
    \\
    &=
    \sqrt{\frac{(\xi^s\circ\Phi)'(q)}{\Phi_s'(q)^3}}
    \cdot
    \frac{\Phi_s''(q)(\xi^s\circ \Phi)'(q) - \Phi_s'(q)(\xi^s\circ\Phi)''(q)}{(\xi^s\circ\Phi)'(q)^2}
    \\
    &=
    \frac{1}{\sqrt{\Phi_s'(q)^3 (\xi^s\circ\Phi)'(q)^3}}
    \cdot
    \lt(
    \Phi_s''(q)(\xi^s\circ \Phi)'(q) - \Phi_s'(q)(\xi^s\circ\Phi)''(q)
    \rt)\,.
\end{aligned}
\end{equation}
Moreover we have
\begin{align*}
    &\Phi_s''(q)(\xi^s\circ \Phi)'(q) - \Phi_s'(q)(\xi^s\circ\Phi)''(q)
    \\
    &=
    \Phi_s''(q)
    \sum_{s'\in\sS}\partial_{x_{s'}}\xi^s(\Phi(q)) \cdot \Phi'_{s'}(q) 
    \\
    &\quad
    - 
    \Phi_s'(q)
    \lt(
    \sum_{s'\in\sS}\partial_{x_{s'}}\xi^s(\Phi(q)) \cdot \Phi''_{s'}(q) 
    +
    \sum_{s',s''\in\sS} 
    \partial_{x_{s'}}\partial_{x_{s''}}
    \xi^s(\Phi(q)) \cdot \Phi'_{s'}(q) 
    \rt)
\end{align*}
Let 
\[
B_s(q)=\sum_{s'\in\sS}\partial_{x_{s'}}\xi^s(\Phi(q)) \cdot \Phi'_{s'}(q).
\]
Note that by non-degeneracy each $\partial_{x_{s'}}\xi^s(\Phi(q))$ is bounded away from $0$ and $\infty$ for all $\Phi(q)\in [0,1]^{\sS}$. Meanwhile $\sum_{s\in\sS}\lambda_s\Phi'_s(q)=1$. Thus for $\Phi'(q)$ obeying \eqref{eq:derivative-admissible}, each $B_s(q)$ is uniformly bounded away from $0$ and $\infty$.

Next let $M(q)\in\mathbb R^{\sS\times\sS}$ be a square matrix with entries
\[
    M(q)_{s,s'}
    =
    \frac{\Phi_s'(q) \cdot \partial_{x_{s'}}\xi^s(\Phi(q))}{B_s(q)}
    \,
\]
and let $I$ denote the identity $\sS\times\sS$ matrix. Then the above equations for all $s\in\sS$ can be expressed more succinctly as
\begin{equation}
\label{eq:type-II-Lipschitz}
    (M-I)\Phi''(q) 
    =
    -w_1(\Phi(q),\Phi'(q))-\Psi(q)\cdot w_2(\Phi(q),\Phi'(q))
\end{equation}
for Lipschitz functions $w_1,w_2:[0,1]^{2r}\to\mathbb R_{>0}^{r}$ given explicitly by
\begin{equation}
\label{eq:w1w2}
\begin{aligned}
    (w_1)_s
    &=
    \frac{
    \Phi_s'(q)\sum_{s',s''\in\sS} 
    \partial_{x_{s'}} \partial_{x_{s''}}
    \xi^s(\Phi(q)) \cdot \Phi'_{s'}(q)  
    }
    {B_s(q)}
    ;
    \\
    (w_2)_s
    &=
    \frac{2\sqrt{\Phi_s'(q)^3 (\xi^s\circ\Phi)'(q)^3}}{B_s(q)}.
\end{aligned}
\end{equation}
Since $B_s$ is bounded below, both $w_1$ and $w_2$ have uniformly bounded entries. Moreover $B$ and $w_1,w_2$ are uniformly Lipschitz in $(\Phi(q),\Phi'(q))$. Note also that $w_2$ is entry-wise non-negative.

As a first observation, observe that
\[
    (M-I)\Phi'(q)=0.
\]
Because $\Phi'(q)\succeq 0$ and $M$ has positive entries, this means $\Phi'(q)$ is the unique right Perron-Frobenius eigenvector of $M$, and thus $\rank(M-I)=r-1$. It follows that for given $(\Phi(q),\Phi'(q))$, a unique solution $(\Phi''(q),\Psi(q))$ to \eqref{eq:type-II-Lipschitz} exists so long as 
\begin{equation}
\label{eq:w2-notin-range-MI}
w_2\notin \range(M-I).
\end{equation}
In fact \eqref{eq:w2-notin-range-MI} is always true. To see this, note that $M$ has a left Perron-Frobenius eigenvector $v\in\mathbb R_{>0}^{r}$ with $v(M-I)=0$. Then if $w_2=(M-I)w$ for $w\in\bbR^{\sS}$, we find $\langle v,w_2\rangle=0$.
This is a contradiction: $\langle v,w_2\rangle>0$ since all entries are strictly positive in both vectors. We denote by $\Lambda(q)\in\bbR^\sS$ the value of $\Phi''(q)$ in the aforementioned unique solution.

Our primary aim is now to show that $\Lambda(q)$ is a Lipschitz function of $(\Phi(q),\Phi'(q))\in\bbR^{\sS}\times A_{\geq 0}$. We would like to apply Perron-Frobenius arguments to $M$, but the fact that $M_{s,s'}\asymp \Phi_s'(q)$ may be very small poses an issue. To rectify this, we define $\wt M(q)$ with entries
\begin{equation}
\label{eq:wtM}
    \wt M(q)_{s,s'}
    =
    \frac{\Phi'_{s'}(q)\partial_{x_{s'}}\xi^s(\Phi(q))}{B_s(q)}
    \,
    .
\end{equation}
Then defining the diagonal $\sS\times\sS$ matrix $D(\Phi'(q))$ with entries
\[
    D(\Phi'(q))_{s,s}=\Phi_s'(q)
\]
we have
\[
\wt M(q)=D(\Phi'(q))^{-1}\, M \, D(\Phi'(q)).
\]
The key property obeyed by $\wt M$ but not $M$ is that for any $v\in \bbR_{>0}^{\sS}$, the entries of $\wt M v$ are of the same order. Namely, all ratios $\frac{(\wt M v)_s}{(\wt M v)_{s'}}$ are uniformly bounded because the ratios $M_{s,s'}/M_{s'',s'}$ are uniformly bounded.
In particular Lemma~\ref{lem:Phi'=0isOK} and hence Lemma~\ref{lem:PF-closed-range} (see below) apply to $\wt M$.

Note that $\wt M$ has Perron-Frobenius eigenvector $\vone$ and $\wt M$ is Lipschitz in $(\Phi(q),\Phi'(q))$. We set 
\begin{equation}
\label{eq:wtV-defn}
\begin{aligned}
    \wt V(q)&=D(\Phi'(q))^{-1}\Lambda(q),\quad \text{i.e.}~ \wt V(q)_s=\frac{\Lambda_s(q)}{\Phi_s'(q)};
    \\
    V(q)_s&=\wt V(q)_s - \frac{\sum_{s'\in\sS} \wt V(q)_{s'}}{r}.
\end{aligned}
\end{equation}
By construction, $\sum_s V(q)_s=0$. Moreover 
\begin{equation}
\label{eq:V-wtV-behave-same}
(\wt M(q)-I)V(q)=(\wt M(q)-I)\wt V(q)
\end{equation}
since $V(q)-\wt V(q)$ is proportional to $\vone$. 

\paragraph{A priori estimate on $\Lambda(q)$}
We now prove \eqref{eq:Phi-stays-increasing}, which will also serve as a useful intermediate step. Note first that $w_1$ satisfies $|(w_1)_s|= O(\Phi'_s(q))$ (recall that $B_s$ is bounded below), while all entries of $w_2$ are non-negative. Therefore the entries of $w_1(q)+\Psi(q) w_2(q)$
are bounded either above or below by $O(\Phi'_s(q))$. Furthermore by definition, 
\begin{align*}
    -w_1(q)-\Psi(q) w_2(q)
    &=
    (M(q)-I)\Lambda(q)
    \\
    &=
    D(\Phi'(q))(\wt M(q)-I)\wt V
    \\
    &\stackrel{\eqref{eq:V-wtV-behave-same}}{=}
    D(\Phi'(q))(\wt M(q)-I)V.
\end{align*}
We conclude that
\[
    \min\lt(\|((\wt M(q)-I) V)_+\|_1,\|((\wt M(q)-I) V)_-\|_1\rt) \leq O(1).
\]
Lemma~\ref{lem:PF-closed-range} below now implies that 
\begin{equation}
\label{eq:V-bounded}
\|V(q)\|_1\leq O(1).
\end{equation}

Note that $\langle \wt V,\lambda\odot\Phi'(q)\rangle=0$ by \eqref{eq:second-derivative-admissible} and \eqref{eq:wtV-defn}. The second part of the latter also implies $V(q)-\wt V(q)$ is proportional to $\vone$, and so
\begin{equation}
\label{eq:V-wtV-1}
\begin{aligned}
    \lt|(V(q)-\wt V(q))_s\rt|
    &=
    \lt|
    \langle 
    V(q)-\wt V(q)
    ,
    \lambda\odot\Phi'(q)
    \rangle
    \rt|
    \\
    &=
    \lt|
    \langle 
    V(q)
    ,
    \lambda\odot\Phi'(q)
    \rangle
    \rt|
    \\
    &\stackrel{\eqref{eq:V-bounded}}{\leq} O(1)
\end{aligned}
\end{equation}
Using again \eqref{eq:V-bounded} and \eqref{eq:wtV-defn} we find that $\|\wt V(q)\|_1\leq O(1)$ as well. Finally since $\Lambda(q)=\wt V(q)\odot \Phi'(q)$, we get \eqref{eq:Phi-stays-increasing} as desired.

\paragraph{Controlling $\Psi$}
We take a second detour to show that $\Psi(q)$ is bounded and Lipschitz.
Using that $\|w_1\|_1\leq O(1)$ and $\|w_2\|_1\geq \Omega(1)$ in the first step below, we find
\begin{align*}
    \Omega(|\Psi(q)|)-O(1)
    &\leq 
    \|w_1(q)+\Psi(q) w_2(q)\|_1
    \\
    &=
    \|(M(q)-I)\Lambda(q)\|_1
    \\
    &\stackrel{\eqref{eq:type-II-Lipschitz}}{\leq}
    O(1).
\end{align*}
The just-proved estimate \eqref{eq:Phi-stays-increasing} implies the weaker bound $\|\Lambda(q)\|_1\leq O(1)$, which was used in the last step.

We conclude that $\Psi(q)$ is uniformly bounded:
\begin{equation}
\label{eq:Psi-bounded}
    |\Psi(q)|\leq O(1).
\end{equation}
Next we show that $\Psi(q)$ is Lipschitz in $(\Phi(q),\Phi'(q))$. We begin by writing
\begin{align*}
    (M(q)-I)\Lambda(q)
    -
    (M(q')-I)\Lambda(q')
    &=
    w_1(q')-w_1(q)
    +
    \Psi(q')w_2(q')
    -
    \Psi(q)w_2(q)
    \\
    &=
    w_1(q')-w_1(q)
    +
    \Psi(q')\big(w_2(q')-w_2(q)\big)
    +
    \big(\Psi(q')-\Psi(q)\big)w_2(q)
    \\
    &=
    O\big(\|\Phi(q)-\Phi(q')\|+\|\Phi'(q)-\Phi'(q')\|\big)
    +
    \big(\Psi(q')-\Psi(q)\big)w_2(q)
    .
\end{align*}
(Note that the latter $O(\cdot)$ notation hides a vector in $\bbR^r$.) We will rely on the fact that $w_2(q)$ is entrywise positive and $\|w_2(q)\|\geq \Omega(1)$. To analyze the left-hand side above, we write
\begin{align*}
    (M(q)-I)\Lambda(q)
    -
    (M(q')-I)\Lambda(q')
    &=
    (M(q)-M(q'))\Lambda(q)
    +
    (M(q')-I)\big(\Lambda(q)-\Lambda(q')\big)
    \\
    &\leq
    O\big(\|\Phi(q)-\Phi(q')\|+\|\Phi'(q)-\Phi'(q')\|\big)
    +
    (M(q')-I)\big(\Lambda(q)-\Lambda(q')\big)
    .
\end{align*}
The latter step holds since $M(q)$ is Lipschitz in $(\Phi(q),\Phi(q'))$ and $\|\Lambda(q)\|_1\leq O(1)$ from \eqref{eq:Phi-stays-increasing}. Now, let $v$ be the left Perron-Frobenius eigenvector of $M(q')$, so $v(M(q')-I)=0$, normalized so that $v\succeq 0$ and $\|v\|_1=1.$ Combining the previous displays implies that 
\[
    (\Psi(q')-\Psi(q))\cdot \langle v, w_2(q)\rangle 
    =
    O\big(\|\Phi(q)-\Phi(q')\|+\|\Phi'(q)-\Phi'(q')\|\big).
\]
Finally we show that $\langle v,w_2(q)\rangle$ is bounded away from $0$. Indeed both vectors are entrywise positive, and $\|w_2(q)\|_1\geq \Omega(1)$ while $\min_s v_s\geq \Omega(1)$. The latter statement holds for similar reasons to the right eigenvector properties of $\wt M$ explained above: for \emph{any} $v\in\bbR_{>0}^{\sS}$, the ratios $\frac{(vM)_{s}}{(vM)_{s'}}$ are uniformly bounded, and this ratio is simply $v_s/v_{s'}$ when $v$ is the left Perron-Frobenius eigenvector. We conclude that
\begin{equation}
\label{eq:Psi-Lip}
    |\Psi(q)-\Psi(q')|\leq O\big(\|\Phi(q)-\Phi(q')\|+\|\Phi'(q)-\Phi'(q')\|\big)
\end{equation}
which ends this second detour.

\paragraph{Finishing the Proof}
Having established \eqref{eq:Phi-stays-increasing} and \eqref{eq:Psi-bounded}, we return to showing that $\Lambda(q)$ is Lipschitz in $(\Phi(q),\Phi'(q))$. Fix a different pair 
\[
(\Phi(q'),\Phi'(q'))\neq (\Phi(q),\Phi'(q)).
\]
Accordingly define $w_1(q'),w_2(q'),M(q'),V(q')$ and so on using $(\Phi(q'),\Phi'(q'))$. (Since we don't require admissibility but only its differential version \eqref{eq:derivative-admissible}, there is no loss of generality here; $q'$ like $q$ is just a place-holder variable so e.g. $\Phi(q)=\Phi(q')$ is possible.)

Then Lemma~\ref{lem:PF-closed-range} implies:
\begin{equation}
\label{eq:type-II-reverse-Lipschitz}
    \|(\wt M(q)-I)V(q)-(\wt M(q)-I)V(q')\|_1 
    \geq
    \Omega\big(\|V(q)-V(q')\|_1\big).
\end{equation}
Using the reverse triangle inequality in the first step, we find the lower bound
\begin{align*}
    \|(\wt M(q)-I) V(q)
    -
    (\wt M(q')-I) V(q')\|_1
    &\geq
    \|(\wt M(q)-I) V(q)
    -
    (\wt M(q)-I) V(q')\|_1
    \\
    &\quad\quad
    -
    \|(\wt M(q)-I) V(q')
    -
    (\wt M(q')-I) V(q')\|_1
    \\
    &\stackrel{\eqref{eq:type-II-reverse-Lipschitz}}{\geq}
    \Omega\big(\| V(q)- V(q')\|_1\big)
    -
    O\big(\|\wt M(q)- \wt M(q')\|_1\big)
    \\
    &\geq
    \Omega\big(\| V(q)- V(q')\|_1\big)
    -
    O\big(\|\Phi(q)-\Phi(q')\|_1 + \|\Phi'(q)-\Phi'(q')\|_1\big).
\end{align*}
By \eqref{eq:type-II-Lipschitz}, \eqref{eq:w1w2}, \eqref{eq:Psi-bounded} \eqref{eq:Psi-Lip}, and the simple estimate $\max\big(|w_1(q)_s|,|w_2(q)_s|\big)\leq O(\Phi'_s(q))$, the left-hand side above is upper bounded by
\begin{align*}
   &\|(\wt M(q)-I) V(q)
    -
    (\wt M(q')-I) V(q')\|_1
    \\
    &=
    \|(\wt M(q)-I) \wt V(q)
    -
    (\wt M(q')-I) \wt V(q')\|_1  
    \\
    &=\Big\|
    D(\Phi'(q))^{-1}
    \Big(
    (M(q)-I) \Lambda(q)
    \Big)
    -
    D(\Phi'(q'))^{-1}
    \Big(
    (M(q')-I) \Lambda(q')\Big)
    \Big\|_1
    \\
    &=
    \Big\|
    D(\Phi'(q))^{-1}
    \Big(
    w_1(q)+\Psi(q)w_2(q)
    \Big)
    -
    D(\Phi'(q'))^{-1}
    \Big(
    w_1(q')+\Psi(q')w_2(q')
    \Big)
    \Big\|_1
    \\
    &\leq
    O\big(\|\Phi(q)-\Phi(q')\|_1 + \|\Phi'(q)-\Phi'(q')\|_1\big).
\end{align*}
Rearranging the previous two displays implies that
\begin{equation}
\label{eq:V-initial-bound}
    \|V(q)-V(q')\|_1
    \leq
    O\big(\|\Phi(q)-\Phi(q')\|_1 + \|\Phi'(q)-\Phi'(q')\|_1\big).
\end{equation}
It remains to unwind the transformations to conclude the same for $\Lambda$. Mimicking \eqref{eq:V-wtV-1} in the first step, 
\begin{align*}
    \lt|\big(V_s(q)-V_s(q')\big)-\big(\wt V_s(q)-\wt V_s(q')\big)\rt|
    &=
    \lt|\sum_s \lambda_s \big(\Phi_s'(q) V(q)-\Phi_s'(q')V(q')\big)\rt| 
    \\
    &\leq
    O\big(\|\Phi'(q)\|\cdot \|V(q)-V(q')\| \big)
    +
    O\big(\|\Phi'(q)-\Phi'(q')\|\cdot \|V(q')\| \big)
    \\
    &\stackrel{\eqref{eq:V-initial-bound},\eqref{eq:V-bounded}}{\leq}
    O\big(\|\Phi(q)-\Phi(q')\|_1 + \|\Phi'(q)-\Phi'(q')\|_1\big)
    \\
    &\quad\quad
    +
    O\big(\|\Phi'(q)-\Phi'(q')\big)
    .
\end{align*}
Combining the previous two displays, we conclude that 
\begin{align*}
    \|\wt V(q)-\wt V(q')\|_1
    &\leq
    \|V(q)-V(q')\|_1
    +
    \|(V(q)-V(q'))-(\wt V(q)-\wt V(q'))\|
    \\
    &\leq 
    O\big(\|\Phi(q)-\Phi(q')\|_1 + \|\Phi'(q)-\Phi'(q')\|_1\big).
\end{align*}
Finally since $\Lambda(q)=\wt V(q) \odot \Phi'(q)$ and $\|\wt V(q')\|_1, \|\Phi'(q)\|_1\leq O(1)$, we obtain the desired:
\begin{align*}
    \|\Lambda(q)-\Lambda(q')\|_1
    &\leq
    O\big(\|\wt V(q)-\wt V(q')\|_1 \cdot \|\Phi'(q)\|_1\big)
    +
    O\big(\|\wt V(q')\|_1 \cdot \|\Phi'(q)-\Phi'(q')\|_1\big)
    \\
    &\leq
    O\big(\|\wt V(q)-\wt V(q')\|_1 + \|\Phi'(q)-\Phi'(q')\|_1\big)
    \\
    &\leq
    O\big(\|\Phi(q)-\Phi(q')\|_1 + \|\Phi'(q)-\Phi'(q')\|_1\big).
\end{align*}
This concludes the proof.
\end{proof}

\begin{lemma}
\label{lem:PF-closed-range}
Let $\cM\subseteq \bbR_{\geq 0}^{\sS\times\sS}$ be a compact set of entry-wise non-negative matrices with unique Perron-Frobenius eigenvector $\vone$ and associated eigenvalue $1$.

Then for all  $v\in\bbR^{\sS}$ with $\sum_{s\in\sS} v_s=0$, we have
\begin{align*}
    \|((M-I)v)_+\|_1 &\geq \Omega_{\cM,r}(\|v\|_1),
    \\
    \|((M-I)v)_-\|_1 &\geq \Omega_{\cM,r}(\|v\|_1).
\end{align*}
\end{lemma}

\begin{proof}
The two statements are equivalent under negation so we assume the first is false and derive a contradiction. If it is false, by taking a convergent sequence of approximate counterexamples $(M^i,v^i)\to (\hM,\hv)$ with $M^i\in\cM$ and $\|v^i\|_1=1$, we have:
\begin{enumerate}
    \item $\hM\in \cM$.
    \item $\hM$ has Perron-Frobenius eigenvector $\vone$ and eigenvalue $1$. 
    \item $\sum_{s\in\sS}\hv_s= 0$.
    \item $\|\hv\|_1=1$.
    \item $\hM\hv\preceq \hv$ (since $((\hM-I)\hv)_+=0$).
\end{enumerate}
Since $\hM$ has simple Perron-Frobenius eigenvalue $1$, for $\hM\hv\preceq \hv$ to hold we must actually have $\hM \hv=\hv$. Therefore $\hv=\vone/r$ is a multiple of the right Perron-Frobenius eigenvector, contradicting $\sum_{s\in\sS}\hv_s= 0$.
\end{proof}

\begin{lemma}
\label{lem:Phi'=0isOK}
For $C>0$, let $\cM_C\subseteq \bbR_{\geq 0}^{\sS\times\sS}$ consist of all matrices $M$ such that:
\begin{enumerate}
    \item $M_{s,s'}\in [0,C]$ for all $s,s'\in\sS$.
    \item $M_{s,s'}\leq CM_{s'',s'}$ for all $s,s',s''\in\sS$.
    \item $M\vone=\vone$.
    \item $\sum_{s,s'\in\sS}M_{s,s'}\geq 1/C$.
\end{enumerate}
Then $\cM=\cM_C$ satisfies the conditions of Lemma~\ref{lem:PF-closed-range}.
\end{lemma}

\begin{proof}
    The only thing to show is that $\vone$ is the \textbf{unique} right Perron-Frobenius eigenvector associated to the eigenvalue $1$ of any $M\in\cM_C$, even though $M$ may include zero entries. Thus, suppose that $w\in\bbR^{\sS}$ satisfies $Mw=w$; we will show that $w$ has all equal entries. Let $S'\subseteq \sS$ be the non-empty set of $s'$ such that $M_{s,s'}>0$ (which does not depend on $s$ by definition of $\cM_C$). Then letting $M'$ and $w'$ be the $S'\times S'$ and $S'$-dimensional restrictions of $M$ and $w$, we have $M'w'=w'$. Since $M'$ has strictly positive entries, we conclude that $w'$ has all entries proportional. Hence for some $a\geq 0$, we have $w_s=a$ for all $s\in S'$. By definition of $S'$ we obtain $w=Mw=Ma^{\sS}=a^{\sS}$. This concludes the proof.
\end{proof}

%% file: final.bbl
\newcommand{\etalchar}[1]{$^{#1}$}
\begin{thebibliography}{KMRT{\etalchar{+}}07}

\bibitem[ABA13]{auffinger2013complexity}
Antonio Auffinger and G{\'e}rard Ben~Arous.
\newblock Complexity of random smooth functions on the high-dimensional sphere.
\newblock {\em The Annals of Probability}, 41(6):4214--4247, 2013.

\bibitem[ABA{\v{C}}13]{auffinger2013random}
Antonio Auffinger, G{\'e}rard Ben~Arous, and Ji{\v{r}}{\'\i} {\v{C}}ern{\'y}.
\newblock Random matrices and complexity of spin glasses.
\newblock {\em Communications on Pure and Applied Mathematics}, 66(2):165--201,
  2013.

\bibitem[ABXY22]{adhikari2022spectral}
Arka Adhikari, Christian Brennecke, Changji Xu, and Horng-Tzer Yau.
\newblock {Spectral Gap Estimates for Mixed $ p $-Spin Models at High
  Temperature}.
\newblock {\em arXiv preprint arXiv:2208.07844}, 2022.

\bibitem[AC17]{auffinger2017parisi}
Antonio Auffinger and Wei-Kuo Chen.
\newblock Parisi formula for the ground state energy in the mixed $p$-spin
  model.
\newblock {\em The Annals of Probability}, 45(6b):4617--4631, 2017.

\bibitem[ACO08]{achlioptas2008phasetransitions}
Dimitris Achlioptas and Amin Coja-Oghlan.
\newblock Algorithmic barriers from phase transitions.
\newblock In {\em Proceedings of 49th FOCS}, pages 793--802, 2008.

\bibitem[ADG01]{arous2001aging}
G~Ben Arous, Amir Dembo, and Alice Guionnet.
\newblock Aging of spherical spin glasses.
\newblock {\em Probability theory and related fields}, 120:1--67, 2001.

\bibitem[AG95]{arous1995large}
G~Ben Arous and Alice Guionnet.
\newblock Large deviations for langevin spin glass dynamics.
\newblock {\em Probability Theory and Related Fields}, 102:455--509, 1995.

\bibitem[AG97]{arous1997symmetric}
G~Ben Arous and Alice Guionnet.
\newblock Symmetric langevin spin glass dynamics.
\newblock {\em The Annals of Probability}, 25(3):1367--1422, 1997.

\bibitem[AG20]{auffinger2020number}
Antonio Auffinger and Julian Gold.
\newblock The number of saddles of the spherical $p$-spin model.
\newblock {\em arXiv preprint arXiv:2007.09269}, 2020.

\bibitem[AJK{\etalchar{+}}22]{anari2021entropic}
Nima Anari, Vishesh Jain, Frederic Koehler, Huy~Tuan Pham, and Thuy-Duong
  Vuong.
\newblock Entropic independence: optimal mixing of down-up random walks.
\newblock In {\em Proceedings of the 54th Annual ACM SIGACT Symposium on Theory
  of Computing}, pages 1418--1430, 2022.

\bibitem[ALR87]{aizenman1987some}
Michael Aizenman, Joel~L Lebowitz, and David Ruelle.
\newblock {Some rigorous results on the Sherrington-Kirkpatrick spin glass
  model}.
\newblock {\em Communications in mathematical physics}, 112:3--20, 1987.

\bibitem[AMS21]{ams20}
Ahmed~El Alaoui, Andrea Montanari, and Mark Sellke.
\newblock {Optimization of Mean-Field Spin Glasses}.
\newblock {\em The Annals of Probability}, 49(6):2922--2960, 2021.

\bibitem[AMS23]{alaoui2021local}
Ahmed~El Alaoui, Andrea Montanari, and Mark Sellke.
\newblock {Local algorithms for Maximum Cut and Minimum Bisection on locally
  treelike regular graphs of large degree}.
\newblock {\em Random Structures \& Algorithms}, 2023.

\bibitem[AS22]{alaoui2022algorithmic}
Ahmed~El Alaoui and Mark Sellke.
\newblock {Algorithmic Pure States for the Negative Spherical Perceptron}.
\newblock {\em Journal of Statistical Physics}, 189(2):27, 2022.

\bibitem[ASS03]{aizenman2003extended}
Michael Aizenman, Robert Sims, and Shannon~L Starr.
\newblock {Extended variational principle for the Sherrington-Kirkpatrick
  spin-glass model}.
\newblock {\em Physical Review B}, 68(21):214403, 2003.

\bibitem[BADG06]{arous2006cugliandolo}
G{\'e}rard Ben~Arous, Amir Dembo, and Alice Guionnet.
\newblock {Cugliandolo-Kurchan equations for dynamics of spin-glasses}.
\newblock {\em Probability Theory and Related Fields}, 136(4):619--660, 2006.

\bibitem[BAGJ20]{arous2020bounding}
G{\'e}rard Ben~Arous, Reza Gheissari, and Aukosh Jagannath.
\newblock Bounding flows for spherical spin glass dynamics.
\newblock {\em Communications in Mathematical Physics}, 373(3):1011--1048,
  2020.

\bibitem[BASZ20]{arous2020geometry}
G{\'e}rard Ben~Arous, Eliran Subag, and Ofer Zeitouni.
\newblock Geometry and temperature chaos in mixed spherical spin glasses at low
  temperature: the perturbative regime.
\newblock {\em Communications on Pure and Applied Mathematics},
  73(8):1732--1828, 2020.

\bibitem[BBvH21]{bandeira2021matrix}
Afonso~S Bandeira, March~T Boedihardjo, and Ramon van Handel.
\newblock Matrix concentration inequalities and free probability.
\newblock {\em arXiv preprint arXiv:2108.06312}, 2021.

\bibitem[BCMT15]{barra2015multi}
Adriano Barra, Pierluigi Contucci, Emanuele Mingione, and Daniele Tantari.
\newblock Multi-species mean field spin glasses. rigorous results.
\newblock In {\em Annales Henri Poincar{\'e}}, volume~16, pages 691--708.
  Springer, 2015.

\bibitem[B{\v{C}}NS21]{belius2021triviality}
David Belius, Ji{\v{r}}{\'\i} {\v{C}}ern{\'y}, Shuta Nakajima, and Marius
  Schmidt.
\newblock Triviality of the geometry of mixed $ p $-spin spherical hamiltonians
  with external field.
\newblock {\em arXiv preprint arXiv:2104.06345}, 2021.

\bibitem[BGT10]{bayati2010combinatorial}
Mohsen Bayati, David Gamarnik, and Prasad Tetali.
\newblock Combinatorial approach to the interpolation method and scaling limits
  in sparse random graphs.
\newblock In {\em Proceedings of the 42nd ACM symposium on Theory of
  computing}, pages 105--114. ACM, 2010.

\bibitem[BH22]{bresler2021ksat}
Guy Bresler and Brice Huang.
\newblock {The Algorithmic Phase Transition of Random $k$-SAT for Low Degree
  Polynomials}.
\newblock In {\em Proceedings of 62nd FOCS}, pages 298--309. IEEE, 2022.

\bibitem[BL20]{baik2020free}
Jinho Baik and Ji~Oon Lee.
\newblock {Free Energy of Bipartite Spherical Sherrington--Kirkpatrick Model}.
\newblock {\em Annales de l’Institut Henri Poincar{\'e}-Probabilit{\'e}s et
  Statistiques}, 56(4):2897--2934, 2020.

\bibitem[BS22]{bates2022free}
Erik Bates and Youngtak Sohn.
\newblock Free energy in multi-species mixed p-spin spherical models.
\newblock {\em Electronic Journal of Probability}, 27:1--75, 2022.

\bibitem[CCM21]{celentano2021high}
Michael Celentano, Chen Cheng, and Andrea Montanari.
\newblock The high-dimensional asymptotics of first order methods with random
  data.
\newblock {\em arXiv preprint arXiv:2112.07572}, 2021.

\bibitem[Ces12]{cesari2012optimization}
Lamberto Cesari.
\newblock {\em Optimization—theory and applications: problems with ordinary
  differential equations}, volume~17.
\newblock Springer Science \& Business Media, 2012.

\bibitem[CGPR19]{chen2019suboptimality}
Wei-Kuo Chen, David Gamarnik, Dmitry Panchenko, and Mustazee Rahman.
\newblock Suboptimality of local algorithms for a class of max-cut problems.
\newblock {\em The Annals of Probability}, 47(3):1587--1618, 2019.

\bibitem[CK94]{cugliandolo1994out}
Leticia~F. Cugliandolo and Jorge Kurchan.
\newblock {On the out-of-equilibrium relaxation of the Sherrington-Kirkpatrick
  model}.
\newblock {\em Journal of Physics A: Mathematical and General}, 27(17):5749,
  1994.

\bibitem[CLR03]{crisanti2003complexity}
Andrea Crisanti, Luca Leuzzi, and Tommaso Rizzo.
\newblock {The complexity of the spherical $p$-spin spin glass model,
  revisited}.
\newblock {\em The European Physical Journal B-Condensed Matter and Complex
  Systems}, 36(1):129--136, 2003.

\bibitem[CLR05]{crisanti2005complexity}
Andrea Crisanti, Luca Leuzzi, and Tommaso Rizzo.
\newblock Complexity in mean-field spin-glass models: Ising p-spin.
\newblock {\em Physical Review B}, 71(9):094202, 2005.

\bibitem[CN95]{comets1995sherrington}
Francis Comets and Jacques Neveu.
\newblock {The Sherrington-Kirkpatrick model of spin glasses and stochastic
  calculus: the high temperature case}.
\newblock {\em Communications in Mathematical Physics}, 166:549--564, 1995.

\bibitem[CS17]{chen2017parisi}
Wei-Kuo Chen and Arnab Sen.
\newblock Parisi formula, disorder chaos and fluctuation for the ground state
  energy in the spherical mixed $p$-spin models.
\newblock {\em Communications in Mathematical Physics}, 350(1):129--173, 2017.

\bibitem[CS21]{chatterjee2019average}
Sourav Chatterjee and Leila Sloman.
\newblock Average {G}romov hyperbolicity and the {P}arisi ansatz.
\newblock {\em Advances in Mathematics}, 376:107417, 2021.

\bibitem[DG21]{dembo2021diffusions}
Amir Dembo and Reza Gheissari.
\newblock {Diffusions interacting through a random matrix: universality via
  stochastic Taylor expansion}.
\newblock {\em Probability Theory and Related Fields}, 180:1057--1097, 2021.

\bibitem[DLZ21]{dembo2021universality}
Amir Dembo, Eyal Lubetzky, and Ofer Zeitouni.
\newblock Universality for langevin-like spin glass dynamics.
\newblock {\em The Annals of applied probability}, 31(6):2864--2880, 2021.

\bibitem[DMS17]{dembo2017extremal}
Amir Dembo, Andrea Montanari, and Subhabrata Sen.
\newblock Extremal cuts of sparse random graphs.
\newblock {\em The Annals of Probability}, 45(2):1190--1217, 2017.

\bibitem[DS20]{dembo2020dynamics}
Amir Dembo and Eliran Subag.
\newblock Dynamics for spherical spin glasses: disorder dependent initial
  conditions.
\newblock {\em Journal of Statistical Physics}, 181:465--514, 2020.

\bibitem[DZ95]{deuschel1995limiting}
Jean-Dominique Deuschel and Ofer Zeitouni.
\newblock {Limiting curves for IID records}.
\newblock {\em The Annals of Probability}, pages 852--878, 1995.

\bibitem[EKZ21]{eldan2021spectral}
Ronen Eldan, Frederic Koehler, and Ofer Zeitouni.
\newblock A spectral condition for spectral gap: fast mixing in
  high-temperature ising models.
\newblock {\em Probability Theory and Related Fields}, pages 1--17, 2021.

\bibitem[FKS87a]{fyodorov1987antiferromagnetic}
Ya~V Fyodorov, I~Ya Korenblit, and EF~Shender.
\newblock Antiferromagnetic ising spin glass.
\newblock {\em Journal of Physics C: Solid State Physics}, 20(12):1835, 1987.

\bibitem[FKS87b]{fyodorov1987phase}
Ya~V Fyodorov, I~Ya Korenblit, and EF~Shender.
\newblock Phase transitions in frustrated metamagnets.
\newblock {\em Europhysics Letters}, 4(7):827, 1987.

\bibitem[Fyo13]{fyodorov2013high}
Yan~V. Fyodorov.
\newblock High-dimensional random fields and random matrix theory.
\newblock {\em arXiv preprint arXiv:1307.2379}, 2013.

\bibitem[Gam21]{gamarnik2021survey}
David Gamarnik.
\newblock The overlap gap property: A topological barrier to optimizing over
  random structures.
\newblock {\em Proceedings of the National Academy of Sciences}, 118(41), 2021.

\bibitem[GJ21]{gamarnik2019overlap}
David Gamarnik and Aukosh Jagannath.
\newblock The overlap gap property and approximate message passing algorithms
  for $ p $-spin models.
\newblock {\em The Annals of Probability}, 49(1):180--205, 2021.

\bibitem[GJW20]{gamarnik2020optimization}
David Gamarnik, Aukosh Jagannath, and Alexander~S. Wein.
\newblock Low-degree hardness of random optimization problems.
\newblock In {\em Proceedings of 61st FOCS}, pages 131--140. IEEE, 2020.

\bibitem[GK21]{gamarnik2021partitioning}
David Gamarnik and Eren~C. K{\i}z{\i}lda\u{g}.
\newblock Algorithmic obstructions in the random number partitioning problem.
\newblock {\em arXiv preprint arXiv:2103.01369}, 2021.

\bibitem[GKPX22]{gamarnik2022algorithms}
David Gamarnik, Eren~C K{\i}z{\i}lda{\u{g}}, Will Perkins, and Changji Xu.
\newblock Algorithms and barriers in the symmetric binary perceptron model.
\newblock In {\em 2022 IEEE 63rd Annual Symposium on Foundations of Computer
  Science (FOCS)}, pages 576--587. IEEE, 2022.

\bibitem[GS17a]{gamarnik2014limits}
David Gamarnik and Madhu Sudan.
\newblock Limits of local algorithms over sparse random graphs.
\newblock {\em Annals of Probability}, 45(4):2353--2376, 2017.

\bibitem[GS17b]{gamarnik2017performance}
David Gamarnik and Madhu Sudan.
\newblock Performance of sequential local algorithms for the random
  {NAE}-{$K$}-sat problem.
\newblock {\em SIAM Journal on Computing}, 46(2):590--619, 2017.

\bibitem[GT02]{guerra2002thermodynamic}
Francesco Guerra and Fabio~Lucio Toninelli.
\newblock The thermodynamic limit in mean field spin glass models.
\newblock {\em Communications in Mathematical Physics}, 230(1):71--79, 2002.

\bibitem[HS21]{huang2021tight}
Brice Huang and Mark Sellke.
\newblock {Tight Lipschitz Hardness for Optimizing Mean Field Spin Glasses}.
\newblock {\em arXiv preprint arXiv:2110.07847}, 2021.

\bibitem[HS23a]{huang2023optimization}
Brice Huang and Mark Sellke.
\newblock Optimization algorithms for multi-species spherical spin glasses.
\newblock {\em arXiv preprint arXiv:2308.09672}, 2023.

\bibitem[HS23b]{huang2023strong}
Brice Huang and Mark Sellke.
\newblock Strong topological trivialization of multi-species spherical spin
  glasses.
\newblock {\em arXiv preprint arXiv:2308.09677}, 2023.

\bibitem[Jag17]{jagannath2017approximate}
Aukosh Jagannath.
\newblock Approximate ultrametricity for random measures and applications to
  spin glasses.
\newblock {\em Communications on Pure and Applied Mathematics}, 70(4):611--664,
  2017.

\bibitem[JMSS23]{jones2022random}
Chris Jones, Kunal Marwaha, Juspreet~Singh Sandhu, and Jonathan Shi.
\newblock {Random Max-CSPs Inherit Algorithmic Hardness from Spin Glasses}.
\newblock {\em Proceedings of the 14th conference on Innovations in theoretical
  computer science}, 2023.

\bibitem[Joe92]{joe1992generalized}
Harry Joe.
\newblock Generalized majorization orderings and applications.
\newblock {\em Lecture Notes-Monograph Series}, pages 145--158, 1992.

\bibitem[JR76]{jamison1976factoring}
Robert~E Jamison and William~H Ruckle.
\newblock Factoring absolutely convergent series.
\newblock {\em Mathematische Annalen}, 224:143--148, 1976.

\bibitem[Kar09]{karasev2009kkm}
Roman Karasev.
\newblock {KKM-type theorems for products of simplicegues and cutting sets and
  measures by straight lines}.
\newblock {\em arXiv preprint arXiv:0909.0604}, 2009.

\bibitem[KC75]{kincaid1975phase}
John~M Kincaid and Ezechiel Godert~David Cohen.
\newblock Phase diagrams of liquid helium mixtures and metamagnets: experiment
  and mean field theory.
\newblock {\em Physics Reports}, 22(2):57--143, 1975.

\bibitem[Kiv21]{kivimae2021ground}
Pax Kivimae.
\newblock The ground state energy and concentration of complexity in spherical
  bipartite models.
\newblock {\em arXiv preprint arXiv:2107.13138}, 2021.

\bibitem[KMRT{\etalchar{+}}07]{krzakala2007gibbs}
Florent Krzakala, Andrea Montanari, Federico Ricci-Tersenghi, Guilhem
  Semerjian, and Lenka Zdeborov{\'a}.
\newblock Gibbs states and the set of solutions of random constraint
  satisfaction problems.
\newblock {\em Proceedings of the National Academy of Sciences},
  104(25):10318--10323, 2007.

\bibitem[KS85]{korenblit1985spin}
I~Ya Korenblit and EF~Shender.
\newblock Spin glass in an lsing two-sublattice magnet.
\newblock {\em Zh. Eksp. Teor. Fiz}, 89:1785--1795, 1985.

\bibitem[McK21]{mckenna2021complexity}
Benjamin McKenna.
\newblock Complexity of bipartite spherical spin glasses.
\newblock {\em arXiv preprint arXiv:2105.05043}, 2021.

\bibitem[Mon21]{mon18}
Andrea Montanari.
\newblock {Optimization of the Sherrington--Kirkpatrick Hamiltonian}.
\newblock {\em SIAM Journal on Computing}, (0):FOCS19--1, 2021.

\bibitem[Mou20]{mourrat2020free}
Jean-Christophe Mourrat.
\newblock Free energy upper bound for mean-field vector spin glasses.
\newblock {\em arXiv preprint arXivbates:2010.09114}, 2020.

\bibitem[MV85]{mezard1985microstructure}
Marc M{\'e}zard and Miguel~Angel Virasoro.
\newblock The microstructure of ultrametricity.
\newblock {\em Journal de Physique}, 46(8):1293--1307, 1985.

\bibitem[Pan13]{panchenko2013parisi}
Dmitry Panchenko.
\newblock {The Parisi ultrametricity conjecture}.
\newblock {\em Annals of Mathematics}, pages 383--393, 2013.

\bibitem[Pan15]{panchenko2015free}
Dmitry Panchenko.
\newblock The free energy in a multi-species sherrington--kirkpatrick model.
\newblock {\em The Annals of Probability}, 43(6):3494--3513, 2015.

\bibitem[Pan18]{panchenko2018k}
Dmitry Panchenko.
\newblock On the {K}-sat model with large number of clauses.
\newblock {\em Random Structures \& Algorithms}, 52(3):536--542, 2018.

\bibitem[Par79]{parisi1979infinite}
Giorgio Parisi.
\newblock Infinite number of order parameters for spin-glasses.
\newblock {\em Physical Review Letters}, 43(23):1754, 1979.

\bibitem[Par06]{parisi2006computing}
Giorgio Parisi.
\newblock Computing the number of metastable states in infinite-range models.
\newblock {\em arXiv preprint arXiv:cond-mat/0602349}, 2006.

\bibitem[Rou13]{roubivcek2013nonlinear}
Tom{\'a}{\v{s}} Roub{\'\i}{\v{c}}ek.
\newblock {\em Nonlinear partial differential equations with applications},
  volume 153.
\newblock Springer Science \& Business Media, 2013.

\bibitem[Rud70]{rudin1970real}
Walter Rudin.
\newblock {\em Real and Complex Analysis P. 2}.
\newblock McGraw-Hill, 1970.

\bibitem[Rue87]{ruelle1987mathematical}
David Ruelle.
\newblock A mathematical reformulation of {D}errida's {REM} and {GREM}.
\newblock {\em Communications in Mathematical Physics}, 108(2):225--239, 1987.

\bibitem[RV17]{rahman2017independent}
Mustazee Rahman and B\'alint Vir\'ag.
\newblock Local algorithms for independent sets are half-optimal.
\newblock {\em The Annals of Probability}, 45(3):1543--1577, 2017.

\bibitem[Sel21]{sellke2021optimizing}
Mark Sellke.
\newblock {Optimizing Mean Field Spin Glasses with External Field}.
\newblock {\em arXiv preprint arXiv:2105.03506}, 2021.

\bibitem[Sel23]{sellke2023threshold}
Mark Sellke.
\newblock {The Threshold Energy of Low Temperature Langevin Dynamics for Pure
  Spherical Spin Glasses}.
\newblock {\em arXiv preprint arXiv:2305.07956}, 2023.

\bibitem[SK75]{sherrington1975solvable}
David Sherrington and Scott Kirkpatrick.
\newblock Solvable model of a spin-glass.
\newblock {\em Physical review letters}, 35(26):1792, 1975.

\bibitem[Sub17]{subag2017complexity}
Eliran Subag.
\newblock The complexity of spherical $ p $-spin models—a second moment
  approach.
\newblock {\em The Annals of Probability}, 45(5):3385--3450, 2017.

\bibitem[Sub18]{subag2018free}
Eliran Subag.
\newblock Free energy landscapes in spherical spin glasses.
\newblock {\em arXiv preprint arXiv:1804.10576}, 2018.

\bibitem[Sub21a]{subag2018following}
Eliran Subag.
\newblock {Following the Ground States of Full-RSB Spherical Spin Glasses}.
\newblock {\em Communications on Pure and Applied Mathematics},
  74(5):1021--1044, 2021.

\bibitem[Sub21b]{subag2021tap}
Eliran Subag.
\newblock {TAP Approach for Multi-Species Spherical Spin Glasses II: the Free
  Energy of the Pure Models}.
\newblock {\em arXiv preprint arXiv:2111.07134}, 2021.

\bibitem[SZ81]{sompolinsky1981dynamic}
Haim Sompolinsky and Annette Zippelius.
\newblock Dynamic theory of the spin-glass phase.
\newblock {\em Physical Review Letters}, 47(5):359, 1981.

\bibitem[SZ21]{subag2021concentration}
Eliran Subag and Ofer Zeitouni.
\newblock Concentration of the complexity of spherical pure $p$-spin models at
  arbitrary energies.
\newblock {\em Journal of mathematical physics}, 62(12):123301, 2021.

\bibitem[Tal06a]{talagrand2006spherical}
Michel Talagrand.
\newblock Free energy of the spherical mean field model.
\newblock {\em Probability Theory and Related Fields}, 134:339--382, 03 2006.

\bibitem[Tal06b]{talagrand2006parisi}
Michel Talagrand.
\newblock {The Parisi formula}.
\newblock {\em Annals of Mathematics}, pages 221--263, 2006.

\bibitem[Tal10]{TalagrandVolI}
Michel Talagrand.
\newblock {\em Mean Field Models for Spin Glasses: Volume I}.
\newblock Springer-Verlag, Berlin, 2010.

\bibitem[Tau49]{taussky1949recurring}
Olga Taussky.
\newblock A recurring theorem on determinants.
\newblock {\em The American Mathematical Monthly}, 56(10P1):672--676, 1949.

\bibitem[Wei22]{wein2020independent}
Alexander~S Wein.
\newblock Optimal low-degree hardness of maximum independent set.
\newblock {\em Mathematical Statistics and Learning}, 4(3):221--251, 2022.

\bibitem[Yeo18]{yeo2018frozen}
Dominic Yeo.
\newblock Frozen percolation on inhomogeneous random graphs.
\newblock {\em arXiv preprint arXiv:1810.02750}, 2018.

\bibitem[Zaa86]{zaanen1986continuity}
AC~Zaanen.
\newblock Continuity of measurable functions.
\newblock {\em The American Mathematical Monthly}, 93(2):128--130, 1986.

\bibitem[Zie12]{ziemer2012weakly}
William~P Ziemer.
\newblock {\em Weakly differentiable functions: Sobolev spaces and functions of
  bounded variation}, volume 120.
\newblock Springer Science \& Business Media, 2012.

\bibitem[ZK07]{zdeborova2007phase}
Lenka Zdeborov{\'a} and Florent Krzakala.
\newblock Phase transitions in the coloring of random graphs.
\newblock {\em Physical Review E}, 76(3):031131, 2007.

\end{thebibliography}
